\documentclass[11pt, twoside]{article}

\usepackage{amsmath}
\usepackage{amssymb}
\usepackage{titletoc}
\usepackage{mathrsfs}
\usepackage{amsthm}
\usepackage{indentfirst}
\usepackage{color}
\usepackage{txfonts}
\usepackage{enumerate}

\usepackage{anysize}

\allowdisplaybreaks

\newtheorem{theorem}{Theorem}[section]
\newtheorem{lemma}[theorem]{Lemma}
\newtheorem{corollary}[theorem]{Corollary}
\newtheorem{proposition}[theorem]{Proposition}

\theoremstyle{definition}
\newtheorem{remark}[theorem]{Remark}
\newtheorem{definition}[theorem]{Definition}

\numberwithin{equation}{section}

\begin{document}

\title{\bf\Large  Besov--Triebel--Lizorkin-Type Spaces with Matrix $A_\infty$ Weights
\footnotetext{\hspace{-0.3cm}2020 {\it Mathematics Subject Classification}.
Primary 46E35; Secondary 47A56, 42B25, 42B20, 42C40, 35S05, 42B35.\endgraf
{\it Key words and phrases.}
matrix weight,
Besov-type space,
Triebel--Lizorkin-type space,
$A_{p,\infty}$-dimension,
almost diagonal operator,
molecule,
wavelet,
trace,
pseudo-differential operator,
pointwise multiplier,
Calder\'on--Zygmund operator.\endgraf
This project is supported by
the National Key Research and Development Program of China
(Grant No.\ 2020Y-FA0712900),
the National Natural Science Foundation of China
(Grant Nos.\ 12431006 and 12371093),
the Fundamental Research Funds for the Central Universities
(Grant No.\ 2233300008),
and the Research Council of Finland
(Grant Nos.\ 346314 and 364208).}}
\date{}
\author{Fan Bu, Tuomas Hyt\"onen,
Dachun Yang\footnote{Corresponding author,
E-mail: \texttt{dcyang@bnu.edu.cn}/{\color{red}
January 6, 2025}/Final version.} \ and Wen Yuan}

\maketitle

\vspace{-0.8cm}

\begin{center}
\begin{minipage}{13cm}
{\small {\bf Abstract}\quad
Introduced by A. Volberg, matrix $A_{p,\infty}$ weights provide a suitable generalization of Muckenhoupt $A_\infty$ weights from the classical theory. In our previous work, we established new characterizations of these weights.
Here, we use these results to study inhomogeneous
Besov-type and Triebel--Lizorkin-type spaces with such weights.
In particular, we characterize these spaces, in terms of the $\varphi$-transform, molecules, and wavelets,
and obtain the boundedness of almost diagonal operators, pseudo-differential operators, trace operators, pointwise multipliers, and Calder\'on--Zygmund operators on these spaces.
This is the first systematic study of inhomogeneous Besov--Triebel--Lizorkin-type
spaces with $A_{p,\infty}$-matrix weights, but some of the results are new even when specialized to the scalar unweighted case.
}
\end{minipage}
\end{center}

\vspace{0.2cm}

\tableofcontents

\vspace{0.2cm}

\section{Introduction}

Besov and Triebel--Lizorkin spaces have a long history, some of which
is summarized in the introduction of our previous article \cite{bhyy}.
Inspired by the theory of weighted Hardy spaces,
the first systematic study of {\em weighted} Besov and Triebel--Lizorkin spaces
was carried out by Bui in 1982 \cite{b82}.
It was already realised at that point that much of this theory
can be developed in the generality of $A_\infty$ weights;
in particular, it is not necessary to impose the stronger assumption that
the weight should be in the Muckenhoupt class $A_p$
for the same $p$ as the integrability exponent of
the Besov or the Triebel--Lizorkin space under consideration.

Our present interest is in {\em matrix-weighted} versions of these spaces and their generalizations.
Matrix weights originate from the study of multivariate stationary stochastic processes and Toeplitz operators
\cite{tv97,v97,wm58}. More recently, matrix-weighted Sobolev spaces
and matrix-weighted Poincar\'e inequalities have been studied for applications to
degenerate elliptic systems \cite{cmr16,crr18,im19}.

Following the development of the theory of $A_p$-matrix weights
and related $L^p(W)$ spaces by Nazarov, Treil, and Volberg \cite{nt96,tv97,v97}
shortly before the turn of the century
and Bownik, Christ, and Goldberg \cite{b01,cg01,g03} shortly after,
matrix-weighted Besov spaces $\dot B^s_{p,q}(W)$ were introduced
and studied by Frazier and Roudenko \cite{fr04,fr08,ro03,ro04}.
The fact that $L^p(W)$ can be identified with a matrix-weighted
Triebel--Lizorkin space $\dot F^{0}_{p,2}(W)$ was already obtained
in \cite{nt96,v97}, and Isralowitz \cite{i21} recently
gave another proof of this result.
However, a systematic study of the full scale of
matrix-weighted Triebel--Lizorkin spaces $\dot F^s_{p,q}(W)$
was only recently started by Frazier and Roudenko \cite{fr21}
and continued by Bai and Xu \cite{bx24a,bx24b}, Wang et al. \cite{wyz24}, as well as the present authors
in our previous articles \cite{bhyy,bhyyp2,bhyyp3}. There, we also treated matrix-weighted versions of
the more general Besov-type $\dot B^{s,\tau}_{p,q}$ and Triebel--Lizorkin-type $\dot F^{s,\tau}_{p,q}$ spaces
with an additional ``Morrey index'' $\tau\in[0,\infty)$ as previously studied in \cite{yy08,yy10,ysy10}; the classical spaces $\dot B^{s}_{p,q}=\dot B^{s,0}_{p,q}$ and $\dot F^{s}_{p,q}=\dot F^{s,0}_{p,q}$ are contained as the special case $\tau=0$.
Another generalisation of matrix-weighted Besov and Triebel--Lizorkin spaces
with logarithmic smoothness is recently due to Li et al. \cite{lyy24,lyy25}.

However, in contrast to the results for $A_\infty$ weights by \cite{b82}
and several subsequent contributions dealing with scalar-weighted spaces,
most of the above-quoted matrix-weighted theory was developed in the context of
matrix $A_p$ weights only, where $p$ should match the integrability exponent
of the Besov-type or the Triebel--Lizorkin-type space under consideration.
(Actually, some of the Besov space results of \cite{fr04,fr08,ro03,ro04}
are also proved for so-called doubling matrix weights.
While this allows for {\em some} weights beyond $A_p$ and even beyond $A_\infty$,
these results are mostly not genuine generalizations of the $A_\infty$ theory
because a condition that the doubling exponent $\beta$ must satisfy $\beta<p$
appears in \cite{fr08,ro03,ro04}, while $\beta$ may be arbitrarily large
for an $A_\infty$ weight. In the scalar-valued case, some aspects of the theory
of Besov and Triebel--Lizorkin spaces have even been developed
in a non-doubling setting \cite{dhy06,hy04}.)

The first goal of this article is to extend the theory of
matrix-weighted Besov-type and Triebel--Lizorkin-type spaces
to the full generality of matrix $A_\infty$ weights.
This class of matrix weights was already introduced by
Volberg in 1997 \cite{v97}, and it has been extensively studied
in a recent work of the present authors \cite{bhyy2}.
Actually, in the matrix-weighted setting,
the classical $A_\infty$ class splits into a family of weight classes
$\{A_{p,\infty}\}_{p\in(0,\infty)}$, all of which reduce to
the usual $A_\infty$ in the case of scalar weights,
and hence they are the natural replacement of the latter
in the matrix-weighted setting.
Thus, the weight class that we use, for Besov-type and
Triebel--Lizorkin-type spaces with integrability exponent $p$,
will still depend on $p$ in the matrix-weighted setting,
but this class is genuinely more general than the class of matrix $A_p$ weights
for which analogous results have been developed
in the previous contributions \cite{bhyy,bhyyp2,bhyyp3,fr21,wyz24}.

The other goal of this article is to treat the case of
{\em inhomogeneous} matrix-weighted Besov-type and Triebel--Lizorkin-type
$
A^{s,\tau}_{p,q}(W)\in\{B^{s,\tau}_{p,q}(W),F^{s,\tau}_{p,q}(W)\},
$
while only the homogeneous versions of these spaces
were discussed in \cite{bhyy,bhyyp2,bhyyp3}. Thus, while we consider
matrix weights $W\in A_{p,\infty}$, the results for non-trivial Morrey index $\tau\in(0,\infty)$
are also new when specialized to $W\in A_p$; but note that several results in the inhomogeneous spaces $A^s_{p,q}(W)=A^{s,0}_{p,q}(W)$ with $W\in A_p$ are already contained in \cite{fr04,fr08,fr21,ro03,ro04}.

Recently, Kakaroumpas and Soler i Gibert \cite{ks24}
established the matrix-weighted version of the classical Fefferman--Stein
vector-valued maximal inequality and also gave some applications
to the studies on product dyadic matrix-weighted $\mathrm{BMO}$ spaces
(see also \cite{ks23} for some related studies on
matrix-weighted $\mathrm{BMO}$ spaces).
Combining this result with \cite[Corollary 3.2]{im19}, we find that
the aforementioned matrix-weighted Fefferman--Stein
vector-valued maximal inequality holds if and only if
the matrix weight $W$ under consideration belongs to the class $A_p$.
Thus, there no longer exists any
Fefferman--Stein vector-valued maximal inequality on
$A_{p,\infty}$-matrix weights, which causes some essential difficulties.
To overcome them, we borrow some ideas from Frazier and Roudenko \cite{fr21},
that is, we establish some equivalent relations between
matrix-weighted Besov--Triebel--Lizorkin-type spaces
and corresponding averaging spaces and we work on the latter.
However, Frazier and Roudenko \cite{fr21} considered the matrix $A_p$ class,
whereas we study the matrix $A_{p,\infty}$ classes which strictly contain
the former. This essential difference means that many arguments in \cite{fr21} are unusable
and hence we need to develop some new ones, such as an analogue
of the discrete Calder\'on reproducing formulae in Proposition \ref{reproL} and
properties of the dimensions of the matrix $A_\infty$ class from our previous work \cite{bhyy2}.
Note that the concept of $A_{p,\infty}$-dimensions
originates from $A_p$-dimensions introduced in \cite{bhyy}.
The $A_p$-dimension had been used
to obtain the sharp boundedness of almost diagonal operators on
matrix-weighted Besov--Triebel--Lizorkin-type sequence spaces \cite{bhyy},
to characterize the precompactness of
matrix-weighted Bourgain--Morrey spaces \cite{bx24},
and to establish the boundedness of some pointwise multipliers
on matrix-weighted Besov--Triebel--Lizorkin spaces
with logarithmic smoothness \cite{lyy24},
showing it is a powerful tool.

Let us briefly mention some other recent trends in the matrix-weighted theory, although we will not pursue them further in the present work. Motivated by the $A_2$ conjecture on {\em sharp weighted inequalities}, solved in the scalar case in \cite{Hyt:A2}, there has been considerable interest in obtaining analogous results for matrix weights. In surprising contrast to the linear weighted bound for Calder\'on--Zygmund operators on scalar-weighted $L^2(w)$ \cite{Hyt:A2}, the best estimate in the matrix-weighted case turned out to have power $\frac32$  of the matrix $A_2$ constant $[W]_{A_2}$ instead: the upper bound was already obtained in \cite{nptv} but only recently shown to be sharp \cite{DPTV}. See also \cite{HPV,i20,tre23} for sharp bounds on the dyadic square function and \cite{DPHL,dly21,ipr21,kns,Lau23,LLOR23,LLOR24,mr22} for a sample of other quantitative matrix-weighted estimates. The question of sharp dependence on the weight is perhaps more natural in spaces like $L^p(W)$, where the choice of the norm is ``canonical''; in $A^{s,\tau}_{p,q}(W)$, a major topic is the equivalence of various different norms, and the sharp constant in any inequality would typically depend on the particular choice of the norms under consideration.

Another major topic about classical weights is the {\em extrapolation theory} started by Rubio de Francia \cite{RdF} and extended into several direction, many of them discussed in the monograph \cite{cump}; its matrix-weighted extension is due to \cite{bcu22}. Moreover, {\em multi-parameter} (also known as product space) theory of matrix weights has been developed in \cite{DKPS,kn24,ks23,ks24,Vuo23}.

The organization of the remainder of this article is as follows.

In Section \ref{sec:background},
we recall some basic concepts and results from the theory of matrix weights
obtained in \cite{bhyy2}, including various equivalent characterizations
of $A_{p,\infty}$-matrix weights and a sharp estimate on the norm of
the composition of reducing operators, which are widely used in this article.
In the proof of this sharp estimate,
$A_{p,\infty}$-dimensions (introduced in \cite{bhyy2}) play an irreplaceable role.

In Section \ref{BF(W)},
we first introduce inhomogeneous averaging matrix-weighted Besov-type
and Triebel--Lizorkin-type spaces
$A^{s,\tau}_{p,q}(\mathbb A)\in\{B^{s,\tau}_{p,q}
(\mathbb A),F^{s,\tau}_{p,q}(\mathbb A)\}$
and corresponding sequence spaces
$a^{s,\tau}_{p,q}(\mathbb A)\in\{b^{s,\tau}_{p,q}(\mathbb A),f^{s,\tau}_{p,q}(\mathbb A)\}$.
Using the characterization of $A_{p,\infty}$-matrix weights
and the sharp estimate from Section \ref{sec:background},
we establish the $\varphi$-transform characterization of $A^{s,\tau}_{p,q}(\mathbb A)$.
Later, we introduce inhomogeneous matrix-weighted Besov-type and Triebel--Lizorkin-type spaces
$A^{s,\tau}_{p,q}(W)\in\{B^{s,\tau}_{p,q}(W),F^{s,\tau}_{p,q}(W)\}$
and corresponding sequence spaces
$a^{s,\tau}_{p,q}(W)\in\{b^{s,\tau}_{p,q}(W),f^{s,\tau}_{p,q}(W)\}$
and then prove that $A^{s,\tau}_{p,q}(W)=A^{s,\tau}_{p,q}(\mathbb A)$
and $a^{s,\tau}_{p,q}(W)=a^{s,\tau}_{p,q}(\mathbb A)$.
Using these and the $\varphi$-transform characterization of $A^{s,\tau}_{p,q}(\mathbb A)$,
we obtain the $\varphi$-transform characterization of $A^{s,\tau}_{p,q}(W)$.
Compared to $A_p$-matrix weights, $A_{p,\infty}$-matrix weights
have significant differences, such as the lack of
some ``dual'' estimates (see Remark \ref{rem RHI} below).
To overcome this difficulty, we develop some new arguments,
including establishing an analogue of discrete Calder\'on reproducing formulae.

In Section \ref{Almost Diagonal Operators},
we first recall homogeneous matrix-weighted Besov-type
and Triebel--Lizorkin-type sequence spaces
$\dot a^{s,\tau}_{p,q}(W)\in\{\dot b^{s,\tau}_{p,q}(W),\dot f^{s,\tau}_{p,q}(W)\}$.
In Subsection \ref{Homogeneous},
we establish the boundedness of almost diagonal operators on the homogeneous
$\dot a^{s,\tau}_{p,q}(W)$ under conditions that are shown to be sharp, at least for certain parameter ranges.
The corresponding results on the inhomogeneous spaces $a^{s,\tau}_{p,q}(W)$ are obtained as a simple corollary of this in Subsection \ref{Inhomogeneous}.

When $m=1$ and $\tau=0$,
the space $\dot f^{s,\tau}_{p,q}(W)$ reduces to the scalar-weighted
Triebel--Lizorkin sequence space $\dot f^s_{p,q}(w)$.
These spaces are special cases of both scalar-weighted anisotropic
Triebel--Lizorkin sequence spaces and matrix-weighted
Triebel--Lizorkin-type sequence spaces.
In this context, our result improves \cite[Theorem 4.1]{bh06}
when $p\leq\min\{1,q\}$; we detail this comparison in Subsection \ref{sec:BH}.

In Section \ref{molecules and more} we apply the boundedness on almost diagonal operators to obtain several further characterizations of the space $A^{s,\tau}_{p,q}(W)$.
In Subsection \ref{molecular characterization},
we establish the molecular characterization of $A^{s,\tau}_{p,q}(W)$
by using the boundedness of almost diagonal operators on $a^{s,\tau}_{p,q}(W)$.
Applying this and the $\varphi$-transform characterization of $A^{s,\tau}_{p,q}(W)$,
we further obtain the boundedness of pseudo-differential
operators on $A^{s,\tau}_{p,q}(W)$ in Subsection \ref{psiDO}, while,
in Subsection \ref{wavelet decomposition}, we obtain the wavelet
decomposition and the atomic characterization of $A^{s,\tau}_{p,q}(W)$.

In the final Section \ref{classical}, we apply the previous results to obtain
the boundedness of various classical operators on $A^{s,\tau}_{p,q}(W)$:
trace and extension operators in Subsection \ref{trace theorems},
pointwise multipliers in Subsection \ref{Pointwise Multipliers},
and Calder\'on--Zygmund operators in Subsection \ref{C-Z operators}.

When $m=1$ and $W\equiv1$, our various results reduce to their unweighted versions and improve on the previous related results of \cite{ysy10}.
Our results on Calder\'on--Zygmund operators are new even
in this special case, while our results on almost diagonal operators,
pseudo-differential operators, trace theorems, and pointwise multipliers
improve, respectively, on \cite[Theorem 3.1]{ysy10},
\cite[Theorem 5.1]{ysy10}, \cite[Theorem 6.8]{ysy10},
and \cite[Theorem 6.1]{ysy10}, each of which imposes an upper
bound on the parameter $\tau$.

Notice that Besov--Triebel--Lizorkin-type was further generalized
by Bui et al. \cite{b20a,b20b,bbd22,bd15,bd21a,bd21b,bd21c},
Ho \cite{Ho1,Ho2,Ho3}, Haroske et al. \cite{Haroske1,Haroske2,Haroske3,Haroske4},
and Sun et al. \cite{syy}. In forthcoming articles,
we will study their corresponding matrix-weighted theories.

At the end of this section, we make some conventions on notation.
Through the whole article, we work on $\mathbb R^n$.
The \emph{ball} $B$ of $\mathbb{R}^n$,
centered at $x\in\mathbb{R}^n$ with radius $r\in(0,\infty)$,
is defined by setting
$
B:=\{y\in\mathbb{R}^n:\ |x-y|<r\}=:B(x,r);
$
moreover, for any $\lambda\in(0,\infty)$, $\lambda B:=B(x,\lambda r)$.
A \emph{cube} $Q$ of $\mathbb{R}^n$ always has finite edge length
and edges of cubes are always assumed to be parallel to coordinate axes,
but $Q$ is not required to be open or closed.
For any cube $Q$ of $\mathbb{R}^n$,
let $c_Q$ be its center and $\ell(Q)$ its edge length.
For any $\lambda\in(0,\infty)$ and any cube $Q$ of $\mathbb{R}^n$,
let $\lambda Q$ be the cube with the same center as $Q$ and the edge length $\lambda\ell(Q)$.
For any $r\in\mathbb{R}$, we define $r_+:=\max\{0,r\}$
and $r_-:=\max\{0,-r\}$.
For any $t\in(0,\infty)$, let $\log_+t:=\max\{0,\log t\}$.
For any $a,b\in\mathbb{R}$, we write $a\wedge b:=\min\{a,b\}$ and $a\vee b:=\max\{a,b\}$.
The symbol $C$ denotes a positive constant that is independent
of the main parameters involved, but may vary from line to line.
The symbol $A\lesssim B$ means that $A\leq CB$ for some positive constant $C$,
while $A\sim B$ means $A\lesssim B\lesssim A$.
Let $\mathbb N:=\{1,2,\ldots\}$, $\mathbb Z_+:=\mathbb N\cup\{0\}$, and $\mathbb Z_+^n:=(\mathbb Z_+)^n$.
For any multi-index $\gamma:=(\gamma_1,\ldots,\gamma_n)\in\mathbb Z_+^n$
and any $x:=(x_1,\ldots,x_n)\in\mathbb R^n$,
we define the length $|\gamma|:=\gamma_1+\cdots+\gamma_n$,
the power $x^\gamma:=x_1^{\gamma_1}\cdots x_n^{\gamma_n}$,
and the partial derivative $\partial^\gamma:=(\frac{\partial}{\partial x_1})^{\gamma_1}
\cdots(\frac{\partial}{\partial x_n})^{\gamma_n}$.
We use $\mathbf{0}$ to denote the \emph{origin} of $\mathbb{R}^n$.
For any set $E\subset\mathbb{R}^n$,
we use $\mathbf{1}_E$ to denote its \emph{characteristic function}.
For any $p\in(0,\infty]$, the \emph{Lebesgue space} $L^p(\mathbb{R}^n)$
has the usual meaning, and the \emph{locally integrable Lebesgue space}
$L^p_{\mathrm{loc}}(\mathbb{R}^n)$ is defined to be the set of
all measurable functions $f$ on $\mathbb{R}^n$ such that,
for any bounded measurable set $E$,
$$
\|f\|_{L^p(E)}:=\|f\mathbf{1}_E\|_{L^p(\mathbb{R}^n)}<\infty.
$$
In what follows, we denote $L^p(\mathbb{R}^n)$ and $L^p_{\mathrm{loc}}(\mathbb{R}^n)$
simply by $L^p$ and $L^p_{\mathrm{loc}}$, respectively.
For any measurable function $w$ on $\mathbb{R}^n$
and any measurable set $E\subset\mathbb{R}^n$, let
$
w(E):=\int_Ew(x)\,dx
$
and, if moreover $|E|\in(0,\infty)$, let
$$
\fint_E w(x)\,dx:=\frac{1}{|E|}\int_E w(x)\,dx.
$$
The \emph{Hardy--Littlewood maximal operator} $\mathcal{M}$ is defined by setting,
for any $f\in L^1_{\mathrm{loc}}$ and $x\in\mathbb{R}^n$,
\begin{equation}\label{maximal}
\mathcal{M}(f)(x):=\sup_{\mathrm{ball}\,B\ni x}\fint_B|f(y)|\,dy,
\end{equation}
where the supremum is taken over all balls $B$ that contain $x$.
For any space $X$, the \emph{product space $X^m$} with $m\in\mathbb{N}$
is defined by setting
\begin{equation*}
X^m:=\left\{\vec f:=(f_1,\ldots,f_m)^{\mathrm{T}}:\
\text{for any}\ i\in\{1,\ldots,m\},\ f_i\in X\right\}.
\end{equation*}

Next, we recall the concept of dyadic cubes.
For any $j\in\mathbb{Z}$ and $k\in\mathbb{Z}^n$, let
$$
Q_{j,k}:=2^{-j}([0,1)^n+k),\
\mathscr{Q}:=\{Q_{j,k}:\ j\in\mathbb{Z},\ k\in\mathbb{Z}^n\},
$$
$$
\mathscr{Q}_+:=\{Q_{j,k}:\ j\in\mathbb{Z}_+,\ k\in\mathbb{Z}^n\},
\text{ and }
\mathscr{Q}_{j}:=\{Q_{j,k}:\ k\in\mathbb{Z}^n\}.
$$
For any $Q:=Q_{j,k}\in\mathscr{Q}$, we let $j_Q:=j$ and $x_Q:=2^{-j}k$.

When we prove a theorem (and the like),
in its proof we always use the same symbols as in
the theorem (and the like) statement itself.

\section{Background on Matrix Weights}\label{sec:background}

In this section, we recall some basic concepts and results from the theory of matrix weights.

For any $m,n\in\mathbb{N}$,
the set of all $m\times n$ complex-valued matrices is denoted by $M_{m,n}(\mathbb{C})$
and $M_{m,m}(\mathbb{C})$ is simply denoted by $M_{m}(\mathbb{C})$.
The identity matrix in $M_m(\mathbb C)$ is denoted by $I_m$, the zero matrix in $M_{m,n}(\mathbb{C})$ by $O_{m,n}$ and $O_{m,m}$ simply by $O_m$.
In what follows, we regard $\mathbb{C}^m$ as $M_{m,1}(\mathbb{C})$
and let $\vec{\mathbf{0}}:=(0,\ldots,0)^\mathrm{T}\in\mathbb{C}^m$.
The conjugate transpose of $A$ is denoted by $A^*$ and the operator norm by
$\|A\|:=\sup_{\vec z\in\mathbb{C}^m,\,|\vec z|=1}|A\vec z|$.
A matrix $A\in M_m(\mathbb C)$ is said to be
\emph{Hermitian} if $A^*=A$,
\emph{unitary} if $A^*A=I_m$,
\emph{positive definite} if, for any $\vec z\in\mathbb{C}^m\setminus\{\vec{\mathbf{0}}\}$, $\vec z^*A\vec z>0$, and
\emph{nonnegative definite} if, for any $\vec z\in\mathbb{C}^m$, $\vec z^*A\vec z\geq0$.
It is well known that any nonnegative definite matrix is always Hermitian (see, for instance, \cite[Theorem 4.1.4]{hj13}).
For a positive definite matrix $A\in M_m(\mathbb{C})$,
its powers $A^\alpha$ for any real $\alpha\in\mathbb R$
are well defined by setting
$$
A^\alpha:=U\operatorname{diag}\left(\lambda_1^\alpha,\ldots,\lambda_m^\alpha\right)U^*,
$$
where $\{\lambda_i\}_{i=1}^m$ are the (necessarily positive) eigenvalues of $A$
and $U\in M_m(\mathbb{C})$ is a unitary matrix
such that the above formula holds with $\alpha=1$.

Now, we recall the concept of matrix weights (see, for instance, \cite{nt96,tv97,v97}).

\begin{definition}
A function $W:\ \mathbb{R}^n\to M_m(\mathbb{C})$ is called
a \emph{matrix weight} if
\begin{enumerate}[\rm(i)]
\item for every $x\in\mathbb{R}^n$, $W(x)$ is nonnegative definite;
\item for almost every $x\in\mathbb{R}^n$, $W(x)$ is invertible;
\item the entries of $W$ are all locally integrable.
\end{enumerate}
A \emph{scalar weight} is a matrix weight with $m=1$.
\end{definition}

Next, we recall the concept of reducing operators
(see, for instance, \cite[(3.1)]{v97}).

\begin{definition}\label{reduce}
Let $p\in(0,\infty)$, $W$ be a matrix weight,
and $E\subset\mathbb{R}^n$ a bounded measurable set satisfying $|E|\in(0,\infty)$.
The matrix $A_E\in M_m(\mathbb{C})$ is called a \emph{reducing operator} of order $p$ for $W$
if $A_E$ is positive definite and,
for any $\vec z\in\mathbb{C}^m$,
\begin{equation}\label{equ_reduce}
\left|A_E\vec z\right|
\sim\left[\fint_E\left|W^{\frac{1}{p}}(x)\vec z\right|^p\,dx\right]^{\frac{1}{p}},
\end{equation}
where the positive equivalence constants depend only on $m$ and $p$.
\end{definition}

\begin{remark}
In Definition \ref{reduce}, the existence of $A_E$ is guaranteed by
\cite[Proposition 1.2]{g03} and \cite[p.\,1237]{fr04}; we omit the details.
\end{remark}

The following lemma is just \cite[Lemma 2.11]{bhyy}.

\begin{lemma}\label{reduceM}
Let $p\in(0,\infty)$, $W$ be a matrix weight,
and $E\subset\mathbb{R}^n$ a bounded measurable set satisfying $|E|\in(0,\infty)$.
If $A_E$ is a reducing operator of order $p$ for $W$,
then, for any matrix $M\in M_m(\mathbb{C})$,
\begin{align*}
\|A_EM\|\sim\left[\fint_E\left\|W^{\frac{1}{p}}(x)M\right\|^p\,dx\right]^{\frac{1}{p}},
\end{align*}
where the positive equivalence constants depend only on $m$ and $p$.
\end{lemma}

Corresponding to the classical weight class $A_p(\mathbb R^n)$
(see, for instance, \cite[Definitions 7.1.1 and 7.1.3]{g14c}),
we have the following concept of $A_p$-matrix weights
(see, for instance, \cite[p.\,490]{fr21}).

\begin{definition}
Let $p\in(0,\infty)$. A matrix weight $W$ on $\mathbb{R}^n$
is called an $A_p(\mathbb{R}^n,\mathbb{C}^m)$-\emph{matrix weight}
if $W$ satisfies that, when $p\in(0,1]$,
\begin{align*}
[W]_{A_p(\mathbb{R}^n,\mathbb{C}^m)}
:=\sup_{\mathrm{cube}\,Q}\mathop{\mathrm{\,ess\,sup\,}}_{y\in Q}
\fint_Q\left\|W^{\frac{1}{p}}(x)W^{-\frac{1}{p}}(y)\right\|^p\,dx
<\infty
\end{align*}
or that, when $p\in(1,\infty)$,
\begin{align*}
[W]_{A_p(\mathbb{R}^n,\mathbb{C}^m)}
:=\sup_{\mathrm{cube}\,Q}
\fint_Q\left[\fint_Q\left\|W^{\frac{1}{p}}(x)W^{-\frac{1}{p}}(y)\right\|^{p'}
\,dy\right]^{\frac{p}{p'}}\,dx
<\infty,
\end{align*}
where $\frac{1}{p}+\frac{1}{p'}=1$.
The $A_p(\mathbb{R}^n,\mathbb{C}^m)$-matrix weights
reduce to $A_{p\vee 1}(\mathbb R^n)$-weights when $m=1$.
\end{definition}

If there is no confusion,
we will denote $A_p(\mathbb{R}^n,\mathbb{C}^m)$ simply by $A_p$.
The following lemma is contained in \cite[Corollary 2.16]{bhyy}.

\begin{lemma}\label{Ap dual}
Let $p\in(1,\infty)$, $\frac{1}{p}+\frac{1}{p'}=1$, and $W\in A_p$.
Let $Q$ be a cube and $A_Q$ the reducing operators of order $p$ for $W$.
Then, for any $M\in M_m(\mathbb C)$,
$$
\left\|A_Q^{-1}M\right\|
\sim\left[\fint_Q\left\|W^{-\frac{1}{p}}(x)M\right\|^{p'}
\,dx\right]^{\frac{1}{p'}},
$$
where the positive equivalence constants
depend only on $m$, $p$, and $[W]_{A_p}$.
\end{lemma}

\subsection{Matrix $A_\infty$ Weights}

Corresponding to scalar-valued $A_\infty(\mathbb R^n)$
(see, for instance, \cite[Definition 7.3.1]{g14c}),
there is a family of matrix-valued weight classes $A_{p,\infty}$
with an extra parameter $p\in(0,\infty)$.
We find it convenient to use the following definition
from our recent work \cite{bhyy2} but, as we will shortly see,
this is equivalent to various other conditions,
including the original definition of this class introduced by Volberg \cite{v97}.

\begin{definition}\label{def ap,infty}
Let $p\in(0,\infty)$. A matrix weight $W$ on $\mathbb{R}^n$
is called an $A_{p,\infty}(\mathbb{R}^n,\mathbb{C}^m)$-\emph{matrix weight}
if $W$ satisfies, for every cube $Q\subset\mathbb{R}^n$, the integrability condition
$\log_+(\fint_Q\|W^{\frac{1}{p}}(x)W^{-\frac{1}{p}}(\cdot)\|^p\,dx)\in L^1(Q)$,
and the following quantity is finite:
\begin{align*}
[W]_{A_{p,\infty}}
:=[W]_{A_{p,\infty}(\mathbb{R}^n,\mathbb{C}^m)}
:=\sup_{\mathrm{cube}\,Q}\exp\left(\fint_Q\log\left(\fint_Q
\left\|W^{\frac{1}{p}}(x)W^{-\frac{1}{p}}(y)\right\|^p\,dx\right)\,dy\right).
\end{align*}
\end{definition}

The $A_{p,\infty}(\mathbb{R}^n,\mathbb{C}^m)$-matrix weights
reduce to $A_\infty(\mathbb R^n)$-weights when $m=1$.

If there is no confusion,
we will denote $A_{p,\infty}(\mathbb{R}^n,\mathbb{C}^m)$ simply by $A_{p,\infty}$.
In contrast to the scalar case, the following result is not entirely obvious
from the definitions, but nevertheless true also for matrix weights.

\begin{proposition}\label{constant}
Let $p\in(0,\infty)$ and $W$ be a matrix weight.
Then there is a positive constant $C$,
depending only on $p$ and $m$, and equal to $1$ if $p\in(0,1]$, such that
$$
C[W]_{A_p}\geq [W]_{A_{p,\infty}}\geq 1
\text{ and }
[W]_{A_p}\geq 1.
$$
\end{proposition}

\begin{proof}
For $W\in A_{p,\infty}$, \cite[Proposition 5.5]{bhyy2}
states that $[W]_{A_{p,\infty}}\in[1,\infty)$,
while, for $W\notin A_{p,\infty}$, we have $[W]_{A_{p,\infty}}=\infty$.
The bound $[W]_{A_p}\geq 1$ follows in the same way.
For $W\in A_p$, the bound $C[W]_{A_p}\geq [W]_{A_{p,\infty}}$
is contained in \cite[Proposition 4.2]{bhyy2},
and for $W\notin A_p$ it is obvious.
This finishes the proof of Proposition \ref{constant}.
\end{proof}

The following result gives a list of several conditions equivalent to $W\in A_{p,\infty}$.
Condition \eqref{W1} below is the original definition of
$A_{p,\infty}$-matrix weights by Volberg \cite[(2.2)]{v97}.

\begin{proposition}\label{compare}
Let $p\in(0,\infty)$ and $W$ be a matrix weight.
For each cube $Q\subset\mathbb R^n$, let $A_Q$ be the reducing operator of order $p$ for $W$.
Then the following conditions, each to hold for some positive constant $C$ independent of the cube $Q$ but not necessarily the same in each condition, are equivalent:
\begin{enumerate}[\rm(i)]
\item\label{W1}
$
\exp(\fint_Q\log|W^{-\frac{1}{p}}(x)\vec z|\,dx)
\leq C\sup_{\vec u\in\mathbb{C}^m\setminus\{\vec{\mathbf{0}}\}}|(\vec z,\vec u)|
[\fint_Q|W^{\frac{1}{p}}(x)\vec u|^p\,dx]^{-\frac{1}{p}}$ for every $\vec z\in\mathbb{C}^m
$;
\item
$
\exp(\fint_Q\log|W^{-\frac{1}{p}}(x)\vec z|\,dx)
\leq C|A_Q^{-1}\vec z|
$
for every $\vec z\in\mathbb C^m$;
\item
$
\exp(\fint_Q\log_+|W^{-\frac{1}{p}}(x)A_Q\vec v|\,dx)\leq C
$
for every $\vec v\in\mathbb C^m$ with $|\vec v|=1$;
\item
$
\exp(\fint_Q\log_+\|W^{-\frac{1}{p}}(x)A_Q U\|\,dx)\leq C
$
for every $U\in M_m(\mathbb C)$ with $\|U\|=1$;
\item
$
\exp(\fint_Q\log_+\|W^{-\frac{1}{p}}(x)A_Q\|\,dx)\leq C
$;
\item
$ \exp(\fint_Q\log_+(\fint_Q\|W^{\frac1p}(x)W^{-\frac{1}{p}}(y)\|^p\,dx)\,dy)\leq C
$;
\item\label{W6}
$W\in A_{p,\infty}$ in the sense of Definition \ref{def ap,infty};
\item\label{W5}
$
\exp(\fint_Q\log\|W^{-\frac{1}{p}}(x)M\|\,dx)
\leq C\|A_Q^{-1}M\|
$
for every $M\in M_m(\mathbb C)$;
\item\label{W5b}
$
C^{-1}\|A_Q^{-1}M\|^p
\leq\exp\{\fint_Q\log[\|W^{-\frac{1}{p}}(x)M\|^p]\,dx\}
\leq C\|A_Q^{-1}M\|^p
$
for every $M\in M_m(\mathbb C)$;
\item\label{WAsL}
there exists $u\in(0,\infty)$ such that
\begin{align}\label{WAs}
\sup_{\mathrm{cube}\,Q\subset\mathbb{R}^n}
\fint_Q\left\|W^{-\frac{1}{p}}(x)A_Q\right\|^u\,dx<\infty.
\end{align}
\end{enumerate}
\end{proposition}

\begin{proof}
The equivalence of \eqref{W1} through \eqref{W5} is \cite[Proposition 3.7]{bhyy2}.
The implication \eqref{W6} $\Rightarrow$ \eqref{W5b} is \cite[Lemma 3.5]{bhyy2},
and the implication \eqref{W5b} $\Rightarrow$ \eqref{W5} is obvious.
Finally, the equivalence of \eqref{W6} and \eqref{WAsL}
is \cite[Proposition 4.1]{bhyy2}.
This finishes the proof of Proposition \ref{compare}.
\end{proof}

$A_{p,\infty}$-matrix weights also satisfy the following distributional estimates.

\begin{lemma}\label{8 prepare}
Let $p\in(0,\infty)$, $W\in A_{p,\infty}$,
and $\{A_Q\}_{\mathrm{cube}\,Q}$ be a family of
reducing operators of order $p$ for $W$.
Then there exists a positive constant $C$,
depending only on $m$ and $p$, such that,
for any cube $Q\subset\mathbb R^n$ and any $M\in(0,\infty)$,
\begin{align}\label{eq 8 var}
\left|\left\{y\in Q:\ \left\|A_Q W^{-\frac1p}(y)\right\|^p\geq e^M\right\}\right|
\leq \frac{\log(C[W]_{A_{p,\infty}})}{M} |Q|
\end{align}
and,
if $\{Q_j\}_{j\in J}\subset Q$ are pairwise disjoint cubes such that
$\|A_QA_{Q_j}^{-1}\|^p\geq e^M$ for every $j\in J$,
then
\begin{align}\label{eq 8 prepare}
\sum_{j\in J}|Q_j|\leq\frac{\log(C[W]_{A_{p,\infty}})}{M}|Q|,
\end{align}
where $J$ is a countable set of indices.
\end{lemma}

\begin{proof}
The estimate \eqref{eq 8 var} is \cite[Corollary 3.9]{bhyy2},
while \eqref{eq 8 prepare} is \cite[Lemma 3.1]{v97},
restated as \cite[Lemma 3.10]{bhyy2}.
This finishes the proof of Lemma \ref{8 prepare}.
\end{proof}

While Definition \ref{def ap,infty} of $A_{p,\infty}$ is modeled
after a characterization of classical $A_\infty$ weights
due to Hru\v{s}\v{c}ev \cite{h84}, the following variant \eqref{Fujii}
related to the work of Fujii \cite{Fujii} and Wilson \cite{w87}
will also be relevant to us. The matrix-valued version \eqref{wz}
can be found in \cite[(1.5)]{nptv} or \cite[p.\,3087]{ipr21},
while the variant \eqref{wM} was introduced in \cite[Lemma 5.3]{bhyy2}.

\begin{definition}
For any scalar weight $w$ and any matrix weight $W$, we define the weight constants
\begin{align}
[w]_{A_\infty(\mathbb R^n)}^*
&:=\sup_{\mathrm{cube}\,Q}\frac{1}{w(Q)}\int_Q\mathcal{M}\left(w\mathbf{1}_Q\right)(x)\,dx,  \label{Fujii} \\
[W]_{A_{p,\infty}(\mathbb R^n,\mathbb C^m)}^{\mathrm{sc}}
&:=\sup_{M\in M_m(\mathbb C)\setminus\{O_m\}}[w_M]_{A_\infty(\mathbb R^n)}^*,
\text{ where } w_M:=\left\|W^{\frac{1}{p}}M\right\|^p, \label{wM} \\
\widetilde{[W]}_{A_{p,\infty}(\mathbb R^n,\mathbb C^m)}^{\mathrm{sc}}
&:=\sup_{\vec z\in\mathbb{C}^m\setminus\{\vec{\mathbf{0}}\}}
\left[w_{\vec z}\right]_{A_\infty(\mathbb R^n)}^*, \text{ where }
w_{\vec z}:=\left|W^{\frac{1}{p}}\vec z\right|^p, \label{wz}
\end{align}
where $\mathcal{M}$ is the maximal operator defined in \eqref{maximal}.
\end{definition}

The two constants \eqref{wM} and \eqref{wz} are comparable
and both dominated by $[W]_{A_{p,\infty}(\mathbb R^n,\mathbb C^m)}$.

\begin{proposition}\label{2.18}
Let $p\in(0,\infty)$. Then there exist two positive constants $C_1$ and $C_2$,
depending only on $m$ and $p$, such that, for any matrix weight $W$,
$$
C_1[W]_{A_{p,\infty}(\mathbb R^n,\mathbb C^m)}^{\mathrm{sc}}
\leq \widetilde{[W]}_{A_{p,\infty}(\mathbb R^n,\mathbb C^m)}^{\mathrm{sc}}
\leq[W]_{A_{p,\infty}(\mathbb R^n,\mathbb C^m)}^{\mathrm{sc}}
\leq C_2[W]_{A_{p,\infty}(\mathbb R^n,\mathbb C^m)}.
$$
\end{proposition}

\begin{proof}
The first two bounds are \cite[Proposition 2.24]{bhyy2},
while the last one is \cite[Lemma 2.23]{bhyy2}.
This finishes the proof of Proposition \ref{2.18}.
\end{proof}

One of the key properties of classical $A_\infty$ weights is
the reverse H\"older inequality that goes back to \cite{cf74}.
For $A_{p,\infty}$-matrix weights, we have the following version.

\begin{proposition}\label{RHI}
Let $p\in(0,\infty)$ and $W\in A_{p,\infty}(\mathbb R^n,\mathbb C^m)$.
Then, for any
$$
r\in\left[1,1+\frac{1}{2^{n+1}[W]_{A_{p,\infty}(\mathbb R^n,\mathbb C^m)}^{\mathrm{sc}}-1}\right],
$$
the following estimates hold:
\begin{enumerate}[\rm(i)]
\item\label{RHI 2} for any cube $Q\subset\mathbb{R}^n$ and any matrix $M\in M_m(\mathbb{C})$,
\begin{align*}
\fint_Q\left\|W^{\frac{1}{p}}(x)M\right\|^{pr}\,dx
\leq2\left[\fint_Q\left\|W^{\frac{1}{p}}(x)M\right\|^{p}\,dx\right]^r;
\end{align*}
\item\label{8-1} there exists a positive constant $C$,
depending only on $m$ and $p$, such that
\begin{align*}
\sup_{\mathrm{cube}\,Q}\left[\fint_Q
\left\|W^{\frac{1}{p}}(x)A_Q^{-1}\right\|^{pr}\,dx\right]^{\frac{1}{r}}
\leq C;
\end{align*}
\item\label{8-2} there exists a positive constant $C$,
depending only on $m$, $p$, and $[W]_{A_{p,\infty}}$, such that
\begin{align*}
\sup_{Q\in\mathscr{Q}}\left[\fint_Q\sup_{R\in\mathscr{Q},\,x\in R\subset Q}
\left\|W^{\frac{1}{p}}(x)A_R^{-1}\right\|^{pr}\,dx\right]^{\frac{1}{r}}
\leq C.
\end{align*}
\end{enumerate}
\end{proposition}

\begin{proof}
Item \eqref{RHI 2} is \cite[Proposition 5.6]{bhyy2},
while items \eqref{8-1} and \eqref{8-2} are \cite[Corollary 5.7]{bhyy2}.
This finishes the proof of Proposition \ref{RHI}.
\end{proof}

\begin{remark}\label{rem RHI}
The consequences of the $A_{p,\infty}$ condition
listed in Proposition \ref{RHI} are exact analogous of similar properties
satisfied by $A_p$-matrix weights; see \cite[Lemmas 3.2 and 3.3]{fr21}
or \cite[Lemma 2.19]{bhyy}. Thus, whenever these properties are used
in an argument, we can simply borrow the existing steps from the results
already available for $A_p$-matrix weights. However, $A_p$-matrix weights
also satisfy some extra related estimates involving the inverse quantities
$\|W^{-\frac{1}{p}}(x)A_Q \|$ in place of $\|W^{\frac{1}{p}}(x)A_Q^{-1}\|$ above;
see again \cite[Lemmas 3.2 and 3.3]{fr21} or \cite[Lemma 2.19]{bhyy}.
These ``dual'' estimates lack a counterpart in the context of
$A_{p,\infty}$-matrix weights, and circumventing their use is a key difficulty
that requires us to use some new methods compared to
\cite{bhyy,bhyyp2,bhyyp3,fr21} in our treatment of
the matrix-weighted function spaces with $A_{p,\infty}$ weights below.
\end{remark}

As a consequence of Proposition \ref{RHI}, one obtains the following corollary
about pointwise multipliers on sequence spaces, which will play an important role
in our study of Besov-type and Triebel--Lizorkin-type spaces below.

\begin{corollary}\label{46x}
Let $p\in(0,\infty)$, $q\in(0,\infty]$, $W\in A_{p,\infty}$,
and $\{A_Q\}_{Q\in\mathscr{Q}}$ be a sequence of
reducing operators of order $p$ for $W$.
Then there exists a positive constant $C$ such that,
for any sequence $\{f_Q\}_{Q\in\mathscr Q}$ of complex numbers indexed by the dyadic cubes,
\begin{equation*}
\left\|\left(\sum_{j\in\mathbb Z}\sum_{Q\in\mathscr{Q}_j}
\left\|W^{\frac{1}{p}}A_Q^{-1}\right\|^q
|f_Q|^q \mathbf 1_Q\right)^{\frac 1q}\right\|_{L^p}
\leq C\left\|\left(\sum_{j\in\mathbb Z}\sum_{Q\in\mathscr{Q}_j} |f_Q|^q
\mathbf 1_Q\right)^{\frac 1q}\right\|_{L^p},
\end{equation*}
where $\mathscr Q_j$ is the collection of dyadic cubes of edge-length $2^{-j}$
for any $j\in\mathbb Z$.
\end{corollary}

\begin{proof}
This is a simple reformulation \cite[Corollary 5.8]{bhyy2},
there stated for sequences of functions $\{f_j\}_{j\in\mathbb Z}$
instead of the numbers $\{f_Q\}_{Q\in\mathscr Q_j}$ for any $j\in\mathbb Z$.
The present version follows by taking $f_j=\sum_{Q\in\mathscr Q_j}f_Q\mathbf 1_Q$
for any $j\in\mathbb Z$. This finishes the proof of Corollary \ref{46x}.
\end{proof}

\subsection{Upper and Lower Dimensions of $A_{p,\infty}$ Weights}

Several results on matrix-weighted function spaces,
obtained in the previous work \cite{bhyy,bhyyp2,bhyyp3,fr21,ro04},
depend on estimates for norms of the form $\|A_QA_R^{-1}\|$ involving reducing operators on two different cubes $Q$ and $R$. In order to relate
such bounds to the $A_p$ condition, we introduced in \cite{bhyy}
the concept of the $A_p$-dimension of a matrix weight.
In \cite{bhyy2}, we extended these considerations to the larger class
of $A_{p,\infty}$-matrix weights. In this subsection, we collect
the relevant definitions and results that we will need further below.
We first recall \cite[Definition 6.2]{bhyy2}.

\begin{definition}\label{AinftyDim}
Let $p\in(0,\infty)$ and $d\in\mathbb{R}$.
A matrix weight $W$ is said to have \emph{$A_{p,\infty}$-lower dimension $d$}
if there exists a positive constant $C$ such that,
for any $\lambda\in[1,\infty)$ and any cube $Q\subset\mathbb{R}^n$,
\begin{align*}
\exp\left(\fint_{\lambda Q}\log\left(
\fint_Q\left\|W^{\frac{1}{p}}(x)W^{-\frac{1}{p}}(y)\right\|^p\,dx\right)\,dy\right)
\leq C\lambda^d.
\end{align*}
A matrix weight $W$ is said to have \emph{$A_{p,\infty}$-upper dimension $d$}
if there exists a positive constant $C$ such that,
for any $\lambda\in[1,\infty)$ and any cube $Q\subset\mathbb{R}^n$,
\begin{align*}
\exp\left(\fint_Q\log\left(
\fint_{\lambda Q}\left\|W^{\frac{1}{p}}(x)W^{-\frac{1}{p}}(y)\right\|^p\,dx\right)\,dy\right)
\leq C\lambda^d.
\end{align*}
\end{definition}

The following conclusion is a reformulation of
\cite[Propositions 6.3 and 6.4]{bhyy2}.

\begin{lemma}\label{Ap dim prop}
Let $p\in(0,\infty)$ and $W$ be a matrix weight.
Then the following statements hold.
\begin{enumerate}[\rm(i)]
\item\label{lower} The possible forms of the set
$\{d\in\mathbb R:\ W\text{ has }A_{p,\infty}\text{-lower dimension }d\}$
are the empty set and all the intervals of the form $(a,\infty)$ and $[a,\infty)$, where $a\in[0,n)$.
\item\label{upper} The possible forms of the set
$\{d\in\mathbb R:\ W\text{ has }A_{p,\infty}\text{-upper dimension }d\}$
are the empty set and all the intervals of the form $(b,\infty)$ and $[b,\infty)$, where $b\in[0,\infty)$.
\item
In both \eqref{lower} and \eqref{upper}, the empty set corresponds to $W\notin A_{p,\infty}$.
\end{enumerate}
\end{lemma}

For any matrix weight $W\in A_{p,\infty}$, we then define
\begin{align}\label{ApLower}
d_{p,\infty}^{\mathrm{lower}}(W) &:=\inf\{d\in[0,n):\ W\text{ has }A_{p,\infty}\text{-lower dimension }d\}, \\
d_{p,\infty}^{\mathrm{upper}}(W) &:=\inf\{d\in[0,\infty):\ W\text{ has }A_{p,\infty}\text{-upper dimension }d\},\notag \\
[\![d_{p,\infty}^{\mathrm{lower}}(W),n) &:=\begin{cases}
[d_{p,\infty}^{\mathrm{lower}}(W),n) &\text{if }W\text{ has }A_{p,\infty}\text{-lower dimension }d_{p,\infty}^{\mathrm{lower}}(W),\\
(d_{p,\infty}^{\mathrm{lower}}(W),n) &\text{otherwise}, \end{cases}\notag \\
[\![d_{p,\infty}^{\mathrm{upper}}(W),\infty) &:=\begin{cases}
[d_{p,\infty}^{\mathrm{upper}}(W),\infty) &\text{if }W\text{ has }A_{p,\infty}\text{-upper dimension }d_{p,\infty}^{\mathrm{upper}}(W),\\
(d_{p,\infty}^{\mathrm{upper}}(W),\infty)&\text{otherwise}.\notag \end{cases}
\end{align}
With these concepts, we have the following sharp estimate.

\begin{lemma}\label{sharp}
Let $p\in(0,\infty)$, $W\in A_{p,\infty}$,
and $\{A_Q\}_{\mathrm{cube}\,Q}$ be
reducing operators of order $p$ for $W$.
Let $d_1\in[\![d_{p,\infty}^{\mathrm{lower}}(W),n)$ and
$d_2\in[\![d_{p,\infty}^{\mathrm{upper}}(W),\infty)$.
Then there exists a positive constant $C$ such that:
\begin{enumerate}[{\rm(i)}]
\item\label{sd} for any $Q,R\in\mathscr{Q}$ with ``lower left'' corners $x_Q$ and $x_R$,
\begin{align*}
\left\|A_QA_R^{-1}\right\|^p
\leq C\max\left\{\left[\frac{\ell(R)}{\ell(Q)}\right]^{d_1},
\left[\frac{\ell(Q)}{\ell(R)}\right]^{d_2}\right\}
\left[1+\frac{|x_Q-x_R|}{\ell(Q)\vee\ell(R)}\right]^{d_1+d_2};
\end{align*}
\item\label{wd} for any $j\in\mathbb Z$ and $Q,R\in\mathscr{Q}_j$,
dyadic cubes of edge-length $2^{-j}$,
\begin{align*}
\left\|A_QA_R^{-1}\right\|^p
\leq C\left(1+2^j|x_Q-x_R|\right)^{d_1+d_2};
\end{align*}
\item for any cubes $Q,R\subset\mathbb{R}^n$ with $Q\cap R\neq\emptyset$,
\begin{align*}
\left\|A_QA_R^{-1}\right\|^p
\leq C\max\left\{\left[\frac{\ell(R)}{\ell(Q)}\right]^{d_1},
\left[\frac{\ell(Q)}{\ell(R)}\right]^{d_2}\right\}.
\end{align*}
\end{enumerate}
The estimates are sharp in the sense that the conclusions are false, in general, if either $d_1$ or $d_2$ is outside the said intervals.
\end{lemma}

\begin{proof}
The estimates are contained in \cite[Lemma 6.7]{bhyy2}
and their sharpness in \cite[Lemma 7.6]{bhyy2}.
This finishes the proof of Lemma \ref{sharp}.
\end{proof}

Sometimes, it is enough to know that a family of positive definite matrices $A_Q$
(not necessarily arising as reducing operators of some $W\in A_{p,\infty}$)
satisfies the conclusions of Lemma \ref{sharp}.
In analogy with \cite[Definition 2.1]{fr21}, we give the following definition.

\begin{definition}\label{doubling}
We say that a sequence of positive definite matrices $\{A_Q\}_{Q\in\mathscr{Q}_+}$ is
\begin{enumerate}[\rm(i)]
\item \emph{strongly doubling} of order $(d_1,d_2;p)$
if it satisfies Lemma \ref{sharp}\eqref{sd};
\item \emph{weakly doubling} of order $(d_1,d_2;p)$
if it satisfies Lemma \ref{sharp}\eqref{wd}.
\end{enumerate}
\end{definition}

\section{Matrix-Weighted Besov-Type and Triebel--Lizorkin-Type Spaces}
\label{BF(W)}

In this section, we introduce the inhomogeneous matrix-weighted Besov-type
and Triebel--Lizorkin-type spaces and obtain their $\varphi$-transform characterization.
More precisely, we will define two types of (inhomogeneous) matrix-weighted function spaces:
\begin{itemize}
\item the (pointwise) weighted spaces $A^{s,\tau}_{p,q}(W)$,
where $W:\ \mathbb R^n\to M_m(\mathbb C)$ is a matrix weight, and
\item the ``averaging'' weighted spaces $A^{s,\tau}_{p,q}(\mathbb A)$, where the weights come from a discrete family $\mathbb A=\{A_Q\}_{Q\in\mathscr Q_+}$ of matrices $A_Q\in M_m(\mathbb C)$ indexed by the dyadic cubes,
\end{itemize}
as well as sequence spaces $a^{s,\tau}_{p,q}(W)$ and $a^{s,\tau}_{p,q}(\mathbb A)$
related to both versions. Our main interest will be in the case that
$W\in A_{p,\infty}(\mathbb R^n,\mathbb C^m)$ and each $A_Q$
is the reducing operator of order $p$ for $W$ on the corresponding cube $Q$,
and we will show that the two types of spaces actually coincide in this case.
For $A_Q$ like this, Lemma \ref{sharp} proves, in the terminology of
Definition \ref{doubling}, that $\mathbb A$ is strongly doubling of
order $(d_1,d_2;p)$ for some $d_1,d_2\in[0,\infty)$.
It turns out that the averaging spaces $A^{s,\tau}_{p,q}(\mathbb A)$
and $a^{s,\tau}_{p,q}(\mathbb A)$, as well as their mutual relation
via the so-called $\varphi$-transform, can be studied under
this strong doubling condition alone,
without postulating any connection to a matrix weight $W$.

In our previous work \cite{bhyy,bhyyp2,bhyyp3},
we investigated (the homogeneous versions of)
these spaces under the stronger condition that $W\in A_p$.
(The difference between the homogeneous and the inhomogeneous versions
is a relatively minor technicality, while the difference of
the weight classes $A_p$ and $A_{p,\infty}$ is more substantial.)
Even when dealing with the averaging spaces $A^{s,\tau}_{p,q}(\mathbb A)$
and $a^{s,\tau}_{p,q}(\mathbb A)$, the results of \cite{bhyy,bhyyp2,bhyyp3}
are stated under the assumption that $\mathbb A$ arises as the family of
reducing operators for a weight $W$. However, an inspection of the proofs
shows that it is often only the strong doubling property that is actually used.
Thus, for the part of the theory dealing with the averaging spaces only,
we can borrow some of the work already done in \cite{bhyy,bhyyp2,bhyyp3}
with relatively little effort. On the other hand, when it comes to
the pointwise weighted spaces, the differences of $A_p$ and $A_{p,\infty}$
arise more prominently, and we will need to
develop some new arguments to deal with the latter class.

Accordingly, we will organize our treatment of these spaces somewhat differently from \cite{bhyy},
where the two types of spaces were studied largely in parallel.
Here, we will first develop the theory of the averaging spaces in their own right
(Subsection \ref{phi Transform of A}) and only then turn to
the pointwise weighted versions of these spaces (Subsection \ref{phi-transform}).
But first, we begin with some preliminaries (Subsection \ref{fsprel}).

\subsection{Preliminaries for Function Spaces}\label{fsprel}

Let $\mathcal{S}$ be the space of all Schwartz functions on $\mathbb{R}^n$,
equipped with the well-known topology determined by a countable family of norms,
and let $\mathcal{S}'$ be the set of all continuous linear functionals on $\mathcal{S}$,
equipped with the weak-$*$ topology.
For any $f\in L^1$ and $\xi\in\mathbb{R}^n$, let
$
\widehat{f}(\xi):=\int_{\mathbb{R}^n}f(x)e^{-ix\cdot\xi}\,dx
$
denote its \emph{Fourier transform}
and $\check{f}(\xi):=\widehat{f}(-\xi)$
its \emph{inverse Fourier transform}.

For any complex-valued function $g$ on $\mathbb{R}^n$, let
$\operatorname{supp}g:=\overline{\{x\in\mathbb{R}^n:\ g(x)\neq0\}}$.
For any $f\in\mathcal{S}'$, let
$$
\operatorname{supp}f:=\bigcap\left\{\text{closed set }K\subset\mathbb{R}^n:\
\langle f,\varphi\rangle=0\text{ for any }\varphi\in\mathcal{S}
\text{ with}\operatorname{supp}\varphi\subset\mathbb{R}^n\setminus K\right\},
$$
which can be found in \cite[Definition 2.3.16]{g14c}.

Let $\Phi\in\mathcal{S}$ satisfy
\begin{equation}\label{19}
\begin{cases}
\operatorname{supp}\widehat\Phi
\subset\left\{\xi\in\mathbb{R}^n:\ |\xi|\leq2\right\},\\
\displaystyle\left|\widehat\Phi(\xi)\right|\geq C>0
\text{ if }\xi\in\mathbb{R}^n\text{ with }|\xi|\leq\frac53
\end{cases}
\end{equation}
and $\varphi\in\mathcal{S}$ satisfy
\begin{equation}\label{20}
\begin{cases}
\displaystyle\operatorname{supp}\widehat\varphi
\subset\left\{\xi\in\mathbb{R}^n:\ \frac12\leq|\xi|\leq2\right\},\\
\displaystyle\left|\widehat\varphi(\xi)\right|\geq C>0
\text{ if }\xi\in\mathbb{R}^n\text{ with }\frac35\leq|\xi|\leq\frac53,
\end{cases}
\end{equation}
where $C$ is a positive constant independent of $\xi$.
It is easy to prove that
\begin{align*}
\varphi\in\mathcal{S}_\infty:=\left\{\varphi\in\mathcal{S}:\
\int_{\mathbb R^n}x^\gamma\varphi(x)\,dx=0
\text{ for any }\gamma\in\mathbb{Z}_+^n\right\}.
\end{align*}
By \cite[pp.\,130-131]{fj90}, we conclude that
there exist $\Psi\in\mathcal{S}$ satisfying \eqref{19}
and $\psi\in\mathcal{S}$ satisfying \eqref{20} such that, for any $\xi\in\mathbb{R}^n$,
\begin{equation}\label{21}
\overline{\widehat\Phi(\xi)}\widehat\Psi(\xi)
+\sum_{j=1}^\infty\overline{\widehat\varphi\left(2^{-j}\xi\right)}
\widehat{\psi}\left(2^{-j}\xi\right)=1.
\end{equation}

Let $\varphi$ be a complex-valued function on $\mathbb{R}^n$.
For any $j\in\mathbb{Z}$ and $x\in\mathbb{R}^n$, let
$\varphi_j(x):=2^{jn}\varphi(2^jx)$.
For any $Q:=Q_{j,k}\in\mathscr{Q}$ and $x\in\mathbb{R}^n$, let
$
\varphi_Q(x)
:=|Q|^{-\frac12}\varphi(2^jx-k)
=|Q|^{\frac12}\varphi_j(x-x_Q).
$

For functions as above, we have the following Calder\'on reproducing formula from \cite[(12.4)]{fj90}.

\begin{lemma}\label{7}
Let $\Phi,\Psi\in\mathcal{S}$ satisfy \eqref{19} and
$\varphi,\psi\in\mathcal{S}$ satisfy \eqref{20} such that \eqref{21} holds.
Then, for any $f\in\mathcal{S}'$, we have the following identity,
with convergence of the series in $\mathcal{S}'$:
\begin{align*}
f=\sum_{Q\in\mathscr{Q}_0}\left\langle f,\Phi_Q\right\rangle\Psi_Q
+\sum_{j=1}^\infty\sum_{Q\in\mathscr{Q}_j}\left\langle f,\varphi_Q\right\rangle\psi_Q.
\end{align*}
\end{lemma}

We also record the following elementary lemmas that will be frequently used when working with expressions like those arising from the use of Lemma \ref{7}.
The following two lemmas, whose proofs are simple computations, are repeated from \cite[Lemmas 3.10 and 3.11]{bhyy}.

\begin{lemma}\label{253x}
Let $\alpha\in(n,\infty)$ and $y\in\mathbb{R}^n$.
\begin{enumerate}[{\rm(i)}]
\item For any $j\in\mathbb{Z}$, one has
$
\int_{\mathbb{R}^n}\frac{2^{jn}}{(1+|2^jx+y|)^\alpha}\,dx\sim1
$.
\item For any $j\in\mathbb{Z}$ with $j\leq0$, one has
$
\sum_{k\in\mathbb{Z}^n}\frac{2^{jn}}{(1+|2^jk+y|)^\alpha}\sim1
$.
\end{enumerate}
Here all the positive equivalence constants depend only on $\alpha$ and $n$.
\end{lemma}

For any $k:=(k_1,\ldots,k_n)\in\mathbb{Z}^n$,
let $\|k\|_{\infty}:=\max_{i\in\{1,\ldots,n\}}|k_i|$.

\begin{lemma}\label{33}
Let $P\in\mathscr{Q}$ and $k\in\mathbb{Z}^n$ with $\|k\|_{\infty}\geq2$.
Then, for any $j\in\{j_P,j_P+1,\ldots\}$, $x\in P$, and $y\in P+k\ell(P)$,
$
1+2^j|x-y|\sim2^{j-j_P}|k|,
$
where the positive equivalence constants depend only on $n$.
\end{lemma}

In what follows, we use $C^\infty$
to denote the set of all infinitely differentiable functions on $\mathbb R^n$.

The following standard identity is a slight extension of a version given
in the proof of \cite[Theorem 2.4]{fr21} and reproduced in
\cite[Lemma 3.15]{bhyy}, where the case $\alpha=2$ was considered.
We will need its slightly more general version
and give the details for the convenience of the reader.
Obviously, Lemma \ref{10x} is of independent interest.

\begin{lemma}\label{10x}
Let $\alpha\in(0,\pi)$, let $\gamma\in\mathcal{S}$ satisfy
$\widehat{\gamma}(\xi)=1$ for any $\xi\in\mathbb{R}^n$ with $|\xi|\leq\alpha$, and let
$
\operatorname{supp}\widehat{\gamma}\subset\{\xi\in\mathbb{R}^n:\ |\xi|<\pi\}.
$
Then, for any $j\in\mathbb{Z}$ and $f\in\mathcal{S}'$
with $\operatorname{supp}\widehat{f}\subset\{\xi\in\mathbb{R}^n:\ |\xi|\leq\alpha\cdot 2^j\}$,
one has $f\in C^\infty$ and, for any $x,y\in\mathbb{R}^n$, the pointwise identity
\begin{equation}\label{10x1}
f(x)=\sum_{R\in\mathscr{Q}_j}2^{-jn}f(x_R+y)\gamma_j(x-x_R-y).
\end{equation}
\end{lemma}

\begin{proof}
The claim that $f\in C^\infty$ follows from the assumption that $f\in\mathcal{S}'$ has a compactly supported Fourier transform by \cite[Theorem 2.3.21]{g14c}; the same theorem shows that $f$ has at most polynomial growth, which implies the absolute convergence of the series in the claimed identity, uniformly on compact sets.

To prove the actual identity, suppose first that, in addition to the assumptions of the lemma, we have $f\in\mathcal{S}$ (instead of just $f\in\mathcal{S}'$). Denoting the right-hand side of the claimed identity by
\begin{equation*}
h(x):=\sum_{R\in\mathscr{Q}_j}2^{-jn}f(x_R+y)\gamma_j(x-x_R-y)
=\sum_{k\in\mathbb Z^n}g(k)\gamma_j\left(x-2^{-j}k-y\right),
\end{equation*}
where $g(z):=2^{-jn}f(2^{-j} z+y)$,
its Fourier transform is
\begin{equation*}
\widehat h(\xi)
=\left[\sum_{k\in\mathbb Z^n}g(k)\exp\left(-ik\cdot 2^{-j}\xi\right)\right]
\widehat\gamma_j(\xi)e^{-i y\cdot\xi},
\ \forall\,\xi\in\mathbb R^n.
\end{equation*}
With our normalization of the Fourier transform, the Poisson summation formula takes the form
\begin{equation*}
\sum_{k\in\mathbb Z^n}\theta(x+2\pi k)=(2\pi)^{-n}\sum_{k\in\mathbb Z^n}\widehat\theta(k)e^{ik\cdot x},
\end{equation*}
and we find that, for any $\xi\in\mathbb R^n$,
\begin{equation*}
\sum_{k\in\mathbb Z^n}g(k)\exp\left(-ik\cdot 2^{-j}\xi\right)
=(2\pi)^n\sum_{k\in\mathbb Z^n}\check g\left(-2^{-j}\xi-2\pi k\right)
=\sum_{k\in\mathbb Z^n}\widehat g\left(2^{-j}\xi+2\pi k\right).
\end{equation*}
Since $\widehat g(\eta)=\widehat f(2^j\eta)\exp(i y\cdot 2^j\eta)$
for any $\eta\in\mathbb R^n$, it follows that,
for any $k\in\mathbb Z^n$ and $\xi\in\mathbb R^n$,
\begin{equation*}
\widehat g\left(2^{-j}\xi+2\pi k\right)
=\widehat f\left(\xi+2\pi 2^j k\right)e^{iy\cdot\xi}
\exp\left(iy\cdot 2\pi 2^j k\right).
\end{equation*}
Thus, for any $\xi\in\mathbb R^n$,
\begin{equation*}
\widehat h(\xi)=\left[\sum_{k\in\mathbb Z^n}\widehat f\left(\xi+2\pi 2^j k\right)
\exp\left(iy\cdot 2\pi 2^{-j}k\right)\right]\widehat\gamma_j(\xi).
\end{equation*}
On the support of $\widehat\gamma_j$, we have $|\xi|<\pi 2^j$. Since $\widehat f$ is also supported on this set, only the term $k=0$ contributes to the sum, and we are left with
\begin{equation*}
\widehat h(\xi)=\widehat f(\xi)\widehat\gamma_j(\xi)=\widehat f(\xi),
\end{equation*}
where the last equality follows from the fact that $\widehat\gamma_j\equiv 1$ on the support of $\widehat f$. Hence $h=f$, proving the lemma in the case $f\in\mathcal S$ under consideration.

Let us finally consider the general form of the assumptions. Let $g\in\mathcal{S}$ satisfy $\operatorname{supp}\widehat{g}\subset B(\mathbf{0},1)$,
$\widehat{g}(\xi)\geq0$ for any $\xi\in\mathbb R^n$, and $g(\mathbf{0})=1$.
These, combined with the inverse Fourier transform formula, further imply that,
for any $x\in\mathbb R^n$,
\begin{align*}
|g(x)|
\leq\int_{\mathbb{R}^n}|\widehat{g}(\xi)|\,d\xi
=g(\mathbf{0})
=1.
\end{align*}
For any $\delta\in(0,\infty)$, let $ f_\delta(\cdot):= f(\cdot)g(\delta\cdot)$.
Then
\begin{align*}
\operatorname{supp}\widehat{ f_\delta}
\subset\operatorname{supp}\widehat{ f}+\operatorname{supp}\widehat{g(\delta\cdot)}
\subset\left\{\xi\in\mathbb R^n:\ |\xi|\leq \alpha+\delta\right\},
\end{align*}
where $\alpha+\delta<\pi$ as soon as $\delta$ is small enough.
From \cite[Theorem 2.3.20]{g14c}, we infer that $f_\delta\in \mathcal{S}$.
By the special case of the lemma that we already proved,
the claimed identity holds with $f_\delta$ in place of $f$, i.e.,
\begin{equation*}
f(x)g(\delta x)=\sum_{R\in\mathscr{Q}_j}2^{-jn}f(x_R+y)g(\delta(x_R+y))\gamma_j(x-x_R-y).
\end{equation*}
Letting $\delta\to 0$, the original claim follows by
the dominated convergence theorem because
$0\leq g(\delta(x_R+y))\leq g(\mathbf 0)=1$,
and we have already observed in the beginning of the proof that
the series converges absolutely with $\delta=0$.
This finishes the proof of Lemma \ref{10x}.
\end{proof}

The following lemma is \cite[Appendix B.1]{g14}.

\begin{lemma}\label{lucas}
Let $a,b\in\mathbb{R}^n$, $\mu,v\in\mathbb{R}$, and $M,N\in(n,\infty)$.
Then there exists a positive constant $C$,
depending only on $n$, $M$, and $N$, such that
\begin{align*}
\int_{\mathbb{R}^n}\frac{2^{\mu n}}{(1+2^\mu|x-a|)^M}\frac{2^{vn}}{(1+2^v|x-b|)^N}\,dx
\leq C\frac{2^{(\mu\wedge v)n}}{(1+2^{\mu\wedge v}|a-b|)^{M\wedge N}}.
\end{align*}
\end{lemma}

As a simple application of Lemma \ref{lucas},
we obtain the following conclusion; we omit the details.

\begin{lemma}\label{lucas2}
Let $a,b\in\mathbb{R}^n$, $\mu,v\in(-\infty,0]$, and $M,N\in(n,\infty)$.
Then there exists a positive constant $C$,
depending only on $n$, $M$, and $N$, such that
\begin{align*}
\sum_{k\in\mathbb Z^n}\frac{2^{\mu n}}{(1+2^\mu|k-a|)^M}\frac{2^{vn}}{(1+2^v|k-b|)^N}
\leq C\frac{2^{(\mu\wedge v)n}}{(1+2^{\mu\wedge v}|a-b|)^{M\wedge N}}.
\end{align*}
\end{lemma}

The following lemma is \cite[Lemma 2.4]{ysy10}.

\begin{lemma}\label{85}
Let $M\in\mathbb{Z}_+$, $\phi\in\mathcal{S}$,
and $\varphi\in\mathcal{S}$ satisfy
$\int_{\mathbb{R}^n}x^\gamma\varphi(x)\,dx=0$
for every multi-index $\gamma\in\mathbb{Z}_+^n$ with $|\gamma|\leq M$.
Then there exists a positive constant $C$,
depending only on $M$ and $n$, such that,
for any $j\in\mathbb{Z}_+$ and $x\in\mathbb{R}^n$,
$
|(\varphi_j*\phi)(x)|
\leq C\|\varphi\|_{\mathcal{S}_{M+1}}\|\phi\|_{\mathcal{S}_{M+1}}
2^{-jM}(1+|x|)^{-(n+M)},
$
where, for any $\phi\in\mathcal{S}$,
\begin{equation}\label{SM}
\|\phi\|_{S_M}
:=\sup_{\gamma\in\mathbb{Z}_+^n,\,|\gamma|\leq M}
\sup_{x\in\mathbb{R}^n}|\partial^\gamma\phi(x)|(1+|x|)^{n+M+|\gamma|}.
\end{equation}
\end{lemma}

As a simple application of Lemma \ref{85}, we obtain the following estimate;
we omit the details.

\begin{corollary}\label{85x}
Let $M\in\mathbb{Z}_+$, $\phi\in\mathcal{S}$,
and $\varphi\in\mathcal{S}$ satisfy
$\int_{\mathbb{R}^n}x^\gamma\varphi(x)\,dx=0$
for any multi-index $\gamma\in\mathbb{Z}_+^n$ with $|\gamma|\leq M$.
Then, for any $Q\in\mathscr{Q}_+$,
$
|\langle\varphi_Q,\phi\rangle|
\leq C\|\varphi\|_{S_{M+1}}\left\|\phi\right\|_{S_{M+1}}
[\ell(Q)]^{M+\frac{n}{2}}(1+|x_Q|)^{-(n+M)},
$
where $C$ is the same as in Lemma \ref{85}.
\end{corollary}

Let $s\in\mathbb{R}$, $\tau\in[0,\infty)$, and $p,q\in(0,\infty]$.
For any sequence $\{f_j\}_{j\in\mathbb Z}$ of measurable functions on $\mathbb{R}^n$,
any subset $J\subset\mathbb Z$, and any measurable set $E\subset\mathbb{R}^n$, let
$$
\|\{f_j\}_{j\in\mathbb Z}\|_{LB_{pq}(E\times J)}
:=\|\{f_j\}_{j\in\mathbb Z}\|_{\ell^qL^p(E\times J)}
:=\left[\sum_{j\in J}\|f_j\|_{L^p(E)}^q\right]^{\frac{1}{q}}
$$
and
$$
\|\{f_j\}_{j\in\mathbb Z}\|_{LF_{pq}(E\times J)}
:=\|\{f_j\}_{j\in\mathbb Z}\|_{L^p\ell^q(E\times J)}
:=\left\|\left(\sum_{j\in J}|f_j|^q\right)^{\frac{1}{q}}\right\|_{L^p(E)}
$$
with the usual modification made when $q=\infty$.

We may drop the domain $E\times J$ from these symbols,
when it is the full space $E\times J=\mathbb R^n\times\mathbb Z$.
We use $LA_{pq}\in\{LB_{pq},LF_{pq}\}$ as a generic notation in statements that apply to both types of spaces.
For any $J\subset\mathbb Z$, we identify $\{f_j\}_{j\in J}$ with $\{\mathbf 1_J(j)f_j\}_{j\in\mathbb Z}$; in particular, for any $k\in\mathbb{Z}$, we identify $\{f_j\}_{j\geq k}$ with
$\{\mathbf{1}_{[k,\infty)}(j)f_j\}_{j\in\mathbb{Z}}$. Thus, e.g.,
$$
\|\{f_j\}_{j\in\mathbb Z_+}\|_{LB_{pq}}
=\left[\sum_{j=0}^\infty\|f_j\|_{L^p}^q\right]^{\frac{1}{q}}.
$$

For any $P\in\mathscr{Q}$, we abbreviate $\widehat{P}:=P\times\{j_P,j_P+1,\ldots\}$ and $\widehat{P}_+:=P\times\{(j_P)_+,(j_P)_++1,\ldots\}$. Thus, e.g.,
$$
\|\{f_j\}_{j\in\mathbb Z}\|_{LB_{pq}(\widehat{P})}
=\|\{f_j\}_{j\in\mathbb Z}\|_{\ell^qL^p(\widehat{P})}
=\left[\sum_{j=j_P}^\infty\|f_j\|_{L^p(P)}^q\right]^{\frac{1}{q}}
$$
and $\|\{f_j\}_{j\in\mathbb Z}\|_{LF_{pq}(\widehat{P})}$,
$\|\{f_j\}_{j\in\mathbb Z}\|_{LB_{pq}(\widehat{P}_+)}$,
and $\|\{f_j\}_{j\in\mathbb Z}\|_{LF_{pq}(\widehat{P}_+)}$
have analogous expressions. When applying these norms to
sequences supported on nonnegative integers only, we note that
\begin{equation*}
\|\{f_j\}_{j\in\mathbb Z_+}\|_{LA_{pq}(\widehat{P})}
=\|\{f_j\}_{j\in\mathbb Z_+}\|_{LA_{pq}(\widehat{P}_+)}
=\|\{f_j\}_{j\in\mathbb Z}\|_{LA_{pq}(\widehat{P}_+)}.
\end{equation*}

We further define
\begin{equation}\label{LApq}
\|\{f_j\}_{j\in\mathbb Z}\|_{LA_{p,q}^\tau}
:=\sup_{P\in\mathscr{Q}}|P|^{-\tau}\|\{f_j\}_{j\in\mathbb Z}\|_{LA_{pq}(\widehat{P})}
\end{equation}
for both choices of $LA_{p,q}^\tau\in\{LB_{p,q}^\tau,LF_{p,q}^\tau\}$.
For sequences supported on nonnegative integers, there is no difference between $\widehat P$ and $\widehat P_+$ on the right, and hence
\begin{equation}\label{LApq in}
\|\{f_j\}_{j\in\mathbb Z_+}\|_{LA_{p,q}^\tau}
=\sup_{P\in\mathscr{Q}}|P|^{-\tau}\|\{f_j\}_{j\in\mathbb Z_+}\|_{LA_{pq}(\widehat{P}_+)}.
\end{equation}

\begin{remark}
If we replace the dyadic cube $P$
and the corresponding $j_P$ in \eqref{LApq},
respectively, by arbitrary cube $P$ and the corresponding $\lfloor-\log_2\ell(P)\rfloor$, we then obtain equivalent quasi-norms.
The same remark applies also to other spaces defined with
the help of the quasi-norms \eqref{LApq} below.
\end{remark}

The following lemma is \cite[Theorem 1]{fs71} (see also \cite[Lemma 3.12]{bhyy}).

\begin{lemma}\label{Fefferman Stein}
Let $p\in(1,\infty)$, $q\in(1,\infty]$,
and $\mathcal{M}$ be the maximal operator as in \eqref{maximal}.
Then there exists a positive constant $C$ such that,
for any sequence $\{f_j\}_{j\in\mathbb Z}$ of measurable functions on $\mathbb R^n$,
\begin{equation*}
\left\|\left\{\mathcal{M}\left(f_j\right)\right\}_{j\in\mathbb Z}\right\|_{LA_{pq}}
\leq C\left\|\left\{f_j\right\}_{j\in\mathbb Z}\right\|_{LA_{pq}}.
\end{equation*}
\end{lemma}

It is evident that Lemma \ref{Fefferman Stein} remains valid for any $J\subset\mathbb Z$ in place of $\mathbb Z$, simply by applying the lemma to the sequence $\{\mathbf 1_J(j)f_j\}_{j\in\mathbb Z}$.

The remainder of this section is organized as follows.
In Subsection \ref{phi Transform of A},
we introduce the inhomogeneous averaging matrix-weighted Besov-type
and Triebel--Lizorkin-type spaces $A^{s,\tau}_{p,q}(\mathbb A)$
and corresponding sequence spaces $a^{s,\tau}_{p,q}(\mathbb A)$
and then establish the $\varphi$-transform characterization of $A^{s,\tau}_{p,q}(\mathbb A)$.
In Subsection \ref{phi-transform},
we introduce the pointwise matrix-weighted spaces
$A^{s,\tau}_{p,q}(W)$ and $a^{s,\tau}_{p,q}(W)$
and obtain their coincidence with the averaging spaces for $W\in A_{p,\infty}$.
In Subsection \ref{reproducing},
we establish a version of the reproducing formula with generic sampling.
In Subsection \ref{homog space}, we make some connections with the homogeneous versions of these spaces.

\subsection{Averaging Spaces $A^{s,\tau}_{p,q}(\mathbb{A})$ and $a^{s,\tau}_{p,q}(\mathbb A)$}
\label{phi Transform of A}

We first introduce the inhomogeneous averaging matrix-weighted
Besov-type and Triebel--Lizorkin-type spaces
and corresponding sequence spaces as follows.
The unweighted versions of these spaces (see, for instance, \cite[Definitions 2.1 and 2.2]{ysy10}) will be a natural special case, as we will shortly observe.

\begin{definition}
Let $s\in\mathbb{R}$, $\tau\in[0,\infty)$, $p\in(0,\infty)$, and $q\in(0,\infty]$.
Let $\Phi\in\mathcal{S}$ satisfy \eqref{19} and
$\varphi\in\mathcal{S}$ satisfy \eqref{20}.
Let $\mathbb{A}:=\{A_Q\}_{Q\in\mathscr{Q}_+}$ be a sequence of positive definite matrices.
The \emph{inhomogeneous averaging matrix-weighted Besov-type space}
$B^{s,\tau}_{p,q}(\mathbb{A},\Phi,\varphi)$
and the \emph{inhomogeneous averaging matrix-weighted Triebel--Lizorkin-type space}
$F^{s,\tau}_{p,q}(\mathbb{A},\Phi,\varphi)$
are defined by setting
$$
A^{s,\tau}_{p,q}(\mathbb{A},\Phi,\varphi)
:=\left\{\vec{f}\in(\mathcal{S}')^m:\
\left\|\vec{f}\right\|_{A^{s,\tau}_{p,q}(\mathbb{A},\Phi,\varphi)}<\infty\right\},
$$
where, for any $\vec{f}\in(\mathcal{S}')^m$
$$
\left\|\vec{f}\right\|_{A^{s,\tau}_{p,q}(\mathbb{A},\Phi,\varphi)}
:=\left\|\left\{2^{js}\left|A_j\left(\varphi_j*\vec f\right)
\right|\right\}_{j\in\mathbb Z_+}\right\|_{LA_{p,q}^\tau}
$$
with $\varphi_0$ replaced by $\Phi$, the norm $\|\cdot\|_{LA_{p,q}^\tau}$ is defined as in \eqref{LApq in},
and, for any $j\in\mathbb{Z}_+$
\begin{equation}\label{Aj}
A_j:=\sum_{Q\in\mathscr{Q}_j}A_Q\mathbf{1}_Q.
\end{equation}
The unweighted spaces $A^{s,\tau}_{p,q}(\Phi,\varphi)$ are obtained by taking $A_Q\equiv I_m$.
\end{definition}

For any $Q\in\mathscr{Q}$,
let $\widetilde{\mathbf{1}}_Q:=|Q|^{-\frac12}\mathbf{1}_Q$.

\begin{definition}
Let $s\in\mathbb{R}$, $\tau\in[0,\infty)$, $p\in(0,\infty)$, and $q\in(0,\infty]$.
Let $\mathbb{A}:=\{A_Q\}_{Q\in\mathscr{Q}_+}$ be a sequence of positive definite matrices.
The \emph{inhomogeneous averaging matrix-weighted Besov-type sequence space} $b^{s,\tau}_{p,q}(\mathbb{A})$
and the \emph{inhomogeneous averaging matrix-weighted Triebel--Lizorkin-type sequence space} $f^{s,\tau}_{p,q}(\mathbb{A})$
are defined to be the sets of all sequences
$\vec t:=\{\vec t_Q\}_{Q\in\mathscr{Q}_+}\subset\mathbb{C}^m$ such that
$$
\left\|\vec t\right\|_{a^{s,\tau}_{p,q}(\mathbb A)}
:=\left\|\left\{2^{js}\left|A_j\vec t_j\right|\right\}_{j\in\mathbb Z_+}\right\|_{LA_{p,q}^\tau}
<\infty,
$$
where $A_j$ and $\|\cdot\|_{LA_{p,q}^\tau}$ are the same as, respectively,
in \eqref{Aj} and \eqref{LApq in} and, for any $j\in\mathbb Z_+$,
\begin{align}\label{vec tj}
\vec t_j:=\sum_{Q\in\mathscr{Q}_j}\vec{t}_Q\widetilde{\mathbf{1}}_Q.
\end{align}
The unweighted spaces $a^{s,\tau}_{p,q}$ are obtained by taking $A_Q\equiv I_m$.
\end{definition}

Here and below, it is understood that we make a consistent choice of the symbols $A\in\{B,F\}$ and $a\in\{b,f\}$, i.e., either $(A,a)=(B,b)$ or $(A,a)=(F,f)$, throughout the entire statement.

Recall that the \emph{$\varphi$-transform} is defined to be the map taking each
$\vec f\in(\mathcal{S}')^m$ to the sequence
$S_\varphi\vec{f}:=\{(S_\varphi\vec{f})_Q\}_{Q\in\mathscr{Q}_+}$, where
\begin{align*}
\left(S_\varphi\vec{f}\right)_Q:=
\begin{cases}
\left\langle\vec{f},\Phi_Q\right\rangle
&\text{if }Q\in\mathscr{Q}_+\text{ with }\ell(Q)=1,\\
\left\langle\vec{f},\varphi_Q\right\rangle
&\text{if }Q\in\mathscr{Q}_+\text{ with }\ell(Q)<1,
\end{cases}
\end{align*}
and the \emph{inverse $\varphi$-transform} is defined to be the map taking a sequence
$\vec t:=\{\vec t_Q\}_{Q\in\mathscr{Q}_+}\subset\mathbb{C}^m$ to
\begin{align*}
T_\psi\vec t:=\sum_{Q\in\mathscr{Q}_0}\vec t_Q\Psi_Q
+\sum_{j=1}^\infty\sum_{Q\in\mathscr{Q}_j}\vec t_Q\psi_Q
\end{align*}
in $(\mathcal{S}')^m$ (see, for instance, \cite[p.\,131]{fj90}).
For any complex-valued function $\varphi$ on $\mathbb R^n$
and for any $x\in\mathbb R^n$,
let $\widetilde{\varphi}(x):=\overline{\varphi(-x)}$.
The following theorem is the main result of this subsection.

\begin{theorem}\label{phi A}
Let $s\in\mathbb{R}$, $\tau\in[0,\infty)$, $p\in(0,\infty)$, and $q\in(0,\infty]$.
Let $\Phi,\Psi\in\mathcal{S}$ satisfy \eqref{19} and
$\varphi,\psi\in\mathcal{S}$ satisfy \eqref{20},
Let $\mathbb{A}:=\{A_Q\}_{Q\in\mathscr{Q}_+}$
be strongly doubling of order $(d_1,d_2;p)$ for some $d_1,d_2\in[0,\infty)$.
Then the operators
$S_\varphi:\ A^{s,\tau}_{p,q}(\mathbb{A},\widetilde{\Phi},\widetilde{\varphi})
\to a^{s,\tau}_{p,q}(\mathbb{A})$ and $T_\psi:\ a^{s,\tau}_{p,q}(\mathbb{A})
\to A^{s,\tau}_{p,q}(\mathbb{A},\Phi,\varphi)$
are bounded. Furthermore, if $\Phi$, $\Psi$, $\varphi$, and $\psi$ satisfy \eqref{21},
then $T_\psi\circ S_\varphi$ is the identity on
$A^{s,\tau}_{p,q}(\mathbb{A},\widetilde\Phi,\widetilde\varphi)$.
\end{theorem}

To show Theorem \ref{phi A}, we need several preparations.
For any sequence $t:=\{t_Q\}_{Q\in\mathscr{Q}}\subset\mathbb{C}$ and
$r,\lambda\in(0,\infty)$,
let $t_{r,\lambda}^*:=\{(t_{r,\lambda}^*)_Q\}_{Q\in\mathscr{Q}}$,
where, for any $Q\in\mathscr{Q}$,
$$
\left(t_{r,\lambda}^*\right)_Q
:=\left[\sum_{R\in\mathscr{Q},\,\ell(R)=\ell(Q)}
\frac{|t_R|^r}{\{1+[\ell(R)]^{-1}|x_R-x_Q|\}^\lambda}\right]^{\frac{1}{r}}.
$$
If $t$ is supported in $\mathscr Q_+$, then so is $t_{r,\lambda}^*$.
Thus, results for sequences indexed by $\mathscr Q_+$ only are
immediate corollaries of corresponding results for sequences
indexed by $\mathscr Q$ by considering the new sequence
$\{\mathbf 1_{\mathscr Q_+}(Q)t_Q\}_{Q\in\mathscr Q}$.
Then we have the following conclusion.

\begin{lemma}\label{4}
Let $s\in\mathbb{R}$, $\tau\in[0,\infty)$,
$p\in(0,\infty)$, $q\in(0,\infty]$, and $\lambda\in(n,\infty)$.
Let $\mathbb{A}:=\{A_Q\}_{Q\in\mathscr{Q}_+}$ be a sequence of positive definite matrices.
Then every $\vec{t}\in a^{s,\tau}_{p,q}(\mathbb{A})$ satisfies
$
\|\vec{t}\|_{a^{s,\tau}_{p,q}(\mathbb{A})}
\sim\|(\{|A_Q\vec t_Q|\}_{Q\in\mathscr{Q}_+}
)_{p\wedge q,\lambda}^*\|_{a^{s,\tau}_{p,q}},
$
where the positive equivalence constants are independent of $\vec t$.
\end{lemma}

\begin{proof}
Let $u:=\{u_Q\}_{Q\in\mathscr{Q}_+}$,
where $u_Q:=|A_Q\vec t_Q|$ for each $Q\in\mathscr{Q}_+$.
Then, by \cite[Lemma 2.8]{ysy10}, we find that
$
\|\vec{t}\|_{a^{s,\tau}_{p,q}(\mathbb{A})}
=\|u\|_{a^{s,\tau}_{p,q}}
\sim\|u_{p\wedge q,\lambda}^*\|_{a^{s,\tau}_{p,q}}.
$
This finishes the proof of Lemma \ref{4}.
\end{proof}

Let $\mathbb{A}:=\{A_Q\}_{Q\in\mathscr{Q}_+}$ be a sequence of positive definite matrices
and $\Phi,\varphi\in\mathcal{S}$. For any $\vec f\in(\mathcal{S}')^m$, let
\begin{equation}\label{sup}
\sup_{\mathbb{A},\varphi}\left(\vec f\right)
:=\left\{\sup_{\mathbb{A},\varphi,Q}\left(\vec f\right)\right\}_{Q\in\mathscr{Q}_+},
\end{equation}
where, for any $Q\in\mathscr{Q}_+$,
$$
\sup_{\mathbb{A},\varphi,Q}\left(\vec f\right)
:=|Q|^{\frac12}\sup_{y\in Q}
\left|A_Q\left(\varphi_{j_Q}*\vec f\right)(y)\right|
$$
with $\varphi_0$ replaced by $\Phi$.
For any $N\in\mathbb Z_+$ and $\vec f\in(\mathcal{S}')^m$, let
\begin{equation}\label{inf}
\inf_{\mathbb{A},\varphi,N}\left(\vec f\right)
:=\left\{\inf_{\mathbb{A},\varphi,Q,N}
\left(\vec f\right)\right\}_{Q\in\mathscr{Q}_+},
\end{equation}
where, for any $Q\in\mathscr{Q}_+$,
\begin{equation}\label{3.11x}
\inf_{\mathbb{A},\varphi,Q,N}\left(\vec f\right)
:=|Q|^{\frac12}\max\left\{\inf_{y\in \widetilde{Q}}
\left|A_{\widetilde{Q}}\left(\varphi_{j_Q}*\vec f\right)(y)\right|:\
\widetilde{Q}\in\mathscr{Q}_{j_Q+N},\ \widetilde{Q}\subset Q\right\}
\end{equation}
with $\varphi_0$ replaced by $\Phi$.
For any $E\subset\mathbb{R}^n$ and $\vec f:=(f_1,\ldots,f_m)\in (\mathcal{S}_\infty)^m$ or $(\mathcal{S}_\infty')^m$,
the notation $\operatorname{supp}\vec f\subset E$ means that $\operatorname{supp}f_i\subset E$ for every $i\in\{1,\ldots,m\}$.
For any $\vec f:=(f_1,\ldots,f_m)$ and $i\in\{1,\ldots,m\}$, let $(\vec f)_i:=f_i$.
Motivated by \cite[Lemma A.4]{fj90}, we obtain the following conclusion.

\begin{lemma}\label{inf and sup}
Let $j\in\mathbb Z_+$ and $\vec f\in(\mathcal{S}')^m$ satisfy
$\operatorname{supp}\widehat{\vec f}\subset\{\xi\in\mathbb R^n:\ |\xi|\leq2^{j+1}\}$.
Let $\mathbb{A}:=\{A_Q\}_{Q\in\mathscr{Q}_+}$ be strongly doubling of order $(d_1,d_2;p)$
for some $d_1,d_2\in[0,\infty)$.
Let $N\in\mathbb Z_+$ be sufficiently large.
For any $Q\in\mathscr Q_+$, let
$
a_Q:=|Q|^{\frac12}\sup_{y\in Q}|A_Q\vec f(y)|
$
and
$$
b_{Q,N}:=|Q|^{\frac12}\max\left\{\inf_{y\in \widetilde{Q}}
\left|A_{\widetilde{Q}}\vec f(y)\right|:\
\widetilde{Q}\in\mathscr{Q}_{j_Q+N},\ \widetilde{Q}\subset Q\right\}.
$$
Let $a:=\{a_Q\}_{Q\in\mathscr{Q}_+}$ and $b:=\{b_{Q,N}\}_{Q\in\mathscr{Q}_+}$.
Let $r\in(0,\infty)$ and $\lambda\in(n,\infty)$.
Then, for any $Q\in\mathscr{Q}_j$,
$
(a_{r,\lambda}^*)_Q
\sim(b_{r,\lambda}^*)_Q,
$
where the positive equivalence constants are independent of $\vec f$, $j$,
and $Q$, but may depend on other parameters including $N$.
\end{lemma}

\begin{proof}
We first prove that, for any $Q\in\mathscr{Q}_j$,
\begin{align}\label{b < a}
\left(b_{r,\lambda}^*\right)_Q
\lesssim\left(a_{r,\lambda}^*\right)_Q.
\end{align}
Since $\mathbb{A}$ is strongly doubling of order $(d_1,d_2;p)$, we deduce that,
for any $R\in\mathscr{Q}_j$, $\widetilde{R}\in\mathscr{Q}_{j+N}$
with $\widetilde{R}\subset R$, and $y\in\widetilde{R}$,
\begin{align*}
\left|A_{\widetilde{R}}\vec f(y)\right|
\leq\left\|A_{\widetilde{R}}A_R^{-1}\right\|\left|A_R\vec f(y)\right|
\lesssim2^{N\frac{d_1}{p}}|R|^{-\frac12}a_R
\end{align*}
and hence $b_{R,N}\lesssim a_R$. This further implies \eqref{b < a}.

Now, we show that, for any $Q\in\mathscr{Q}_j$,
\begin{align}\label{a < b}
\left(a_{r,\lambda}^*\right)_Q
\lesssim\left(b_{r,\lambda}^*\right)_Q.
\end{align}
To this end, we first consider the special case where $\vec f\in(\mathcal{S})^m$ satisfies
$$
\operatorname{supp}\widehat{\vec f}\subset\left\{\xi\in\mathbb R^n:\ |\xi|\leq 3\cdot 2^j\right\}.
$$
In this case, let $L\in(\frac{d_1+d_2}p+\frac nr,\infty)$.
From $\vec f\in(\mathcal S)^m$, we infer that, for any $R\in\mathscr{Q}_j$,
\begin{align*}
a_R
&\lesssim\|A_R\|\sup_{y\in R}\left|\vec f(y)\right|
\lesssim\left\|A_RA_{Q_{0,\mathbf{0}}}^{-1}\right\|
\left\|A_{Q_{0,\mathbf{0}}}\right\|\sup_{y\in R}\frac{1}{(1+|y|)^L}\\
&\lesssim2^{j\frac{d_1}p}(1+|x_R|)^{\frac{d_1+d_2}p}\frac{1}{(1+|x_R|)^L}
\sim_j(1+|x_R|)^{\frac{d_1+d_2}p-L}.
\end{align*}
This, together with Lemma \ref{253x}(ii), further implies that,
for any $Q\in\mathscr{Q}_j$,
\begin{align}\label{< infty}
\left(a_{r,\lambda}^*\right)_Q
\leq\left(\sum_{R\in\mathscr{Q}_j}a_R^r\right)^{\frac{1}{r}}
\lesssim_j\left[\sum_{R\in\mathscr{Q}_j}
(1+|x_R|)^{(\frac{d_1+d_2}p-L)r}\right]^{\frac{1}{r}}
<\infty.
\end{align}
Note that, in \eqref{< infty}, we only care about the qualitative finiteness,
and hence the implicit constants are exceptionally allowed to
depend on the parameter $j$ as well.

Let $d:=\{d_R\}_{R\in\mathscr{Q}_+}$,
where, for any $R\in\mathscr{Q}_+$,
\begin{align}\label{dR}
d_R:=|R|^{\frac12+\frac1n}\sum_{i=1}^n\sup_{\xi\in R}
\left|\nabla\left(A_R\vec f\right)_i(\xi)\right|.
\end{align}
By the mean-value theorem, we find that,
for any $R\in\mathscr{Q}_+$,
$\widetilde{R}\in\mathscr{Q}_{j+N}$ with $\widetilde{R}\subset R$,
and $x,y\in\widetilde{R}$,
\begin{align*}
\left|A_R\vec f(x)\right|
&\leq\left|A_R\vec f(y)\right|+\left|A_R\vec f(x)-A_R\vec f(y)\right|\\
&\leq\left\|A_RA_{\widetilde{R}}^{-1}\right\|\left|A_{\widetilde{R}}\vec f(y)\right|
+\left[\sum_{i=1}^n\sup_{\xi\in R}
\left|\nabla\left(A_R\vec f\right)_i(\xi)\right|^2\right]^{\frac12}|x-y|\\
&\lesssim2^{N\frac{d_2}{p}}\left|A_{\widetilde{R}}\vec f(y)\right|
+\sum_{i=1}^n\sup_{\xi\in R}
\left|\nabla\left(A_R\vec f\right)_i(\xi)\right|2^{-(j+N)}
\end{align*}
and hence
\begin{align*}
\sup_{x\in\widetilde{R}}\left|A_R\vec f(x)\right|
\lesssim2^{N\frac{d_2}{p}}\inf_{y\in\widetilde{R}}\left|A_{\widetilde{R}}\vec f(y)\right|
+2^{-N}|R|^{-\frac12}d_R.
\end{align*}
Thus, for any $R\in\mathscr{Q}_+$,
$
a_R\lesssim2^{N\frac{d_2}{p}}b_{R,N}+2^{-N}d_R,
$
which further implies that, for any $Q\in\mathscr{Q}_+$,
\begin{align}\label{aQ}
\left(a_{r,\lambda}^*\right)_Q
\lesssim2^{N\frac{d_2}{p}}\left(b_{r,\lambda}^*\right)_Q
+2^{-N}\left(d_{r,\lambda}^*\right)_Q.
\end{align}
Next, we estimate $(d_{r,\lambda}^*)_Q$.
From Lemma \ref{10x} with $y=\mathbf{0}$ and $\alpha=3$, we deduce that,
for any $j\in\mathbb Z_+$, $R\in\mathscr{Q}_j$, and $x\in\mathbb R^n$,
$
A_R\vec f(x)=\sum_{J\in\mathscr{Q}_j}2^{-jn}A_R\vec f(x_J)\gamma_j(x-x_J)
$
pointwise, where $\gamma$ is the same as in Lemma \ref{10x}.

Let $M\in(\frac{d_1+d_2}{p}+\frac{\lambda}{\min\{1,r\}},\infty)$.
Then, for any $R\in\mathscr{Q}_j$, $i\in\{1,\ldots,n\}$, and $\xi\in R$,
\begin{align*}
\left|\nabla\left(A_R\vec f\right)_i(\xi)\right|
&\leq\sum_{J\in\mathscr{Q}_j}2^{-jn}
\left|\left(A_R\vec f\right)_i(x_J)\right||\nabla\gamma_j(\xi-x_J)|\\
&\lesssim|R|^{-\frac1n}\sum_{J\in\mathscr{Q}_j}
\left|\left(A_R\vec f\right)_i(x_J)\right|\frac1{(1+2^j|\xi-x_J|)^M}\\
&\sim|R|^{-\frac1n}\sum_{J\in\mathscr{Q}_j}
\left|\left(A_R\vec f\right)_i(x_J)\right|\frac1{(1+2^j|x_R-x_J|)^M},
\end{align*}
which, combined with \eqref{dR} and Lemma \ref{sharp}, further implies that
\begin{align*}
d_R&\lesssim|R|^{\frac12}\sum_{J\in\mathscr{Q}_j}\sum_{i=1}^n
\left|\left(A_R\vec f\right)_i(x_J)\right|\frac1{(1+2^j|x_R-x_J|)^M}
\sim|R|^{\frac12}\sum_{J\in\mathscr{Q}_j}
\left|A_R\vec f(x_J)\right|\frac1{(1+2^j|x_R-x_J|)^M}\\
&\sim|R|^{\frac12}\sum_{J\in\mathscr{Q}_j}
\left\|A_RA_J^{-1}\right\|\left|A_J\vec f(x_J)\right|\frac1{(1+2^j|x_R-x_J|)^M}
\lesssim\sum_{J\in\mathscr{Q}_j}\frac{a_J}{(1+2^j|x_R-x_J|)^{\widetilde{M}}},
\end{align*}
where $\widetilde{M}:=M-\frac{d_1+d_2}p\in(\frac{\lambda}{\min\{1,r\}},\infty)$.
Thus, for any $Q\in\mathscr{Q}_j$,
\begin{align*}
\left(d_{r,\lambda}^*\right)_Q
\lesssim\left\{\sum_{R\in\mathscr{Q}_j}\frac{1}{(1+2^j|x_R-x_Q|)^\lambda}
\left[\sum_{J\in\mathscr{Q}_j}\frac{a_J}{(1+2^j|x_R-x_J|)^{\widetilde{M}}}\right]^r\right\}^{\frac1r}.
\end{align*}
When $r\in(0,1]$, by Lemma \ref{lucas2},
we conclude that, for any $Q\in\mathscr{Q}_j$,
\begin{align*}
\left(d_{r,\lambda}^*\right)_Q
&\leq\left[\sum_{J\in\mathscr{Q}_j}a_J^r
\sum_{R\in\mathscr{Q}_j}\frac1{(1+2^j|x_R-x_Q|)^\lambda}
\frac1{(1+2^j|x_R-x_J|)^{\widetilde{M}r}}\right]^{\frac{1}{r}}\\
&\lesssim\left[\sum_{J\in\mathscr{Q}_j}
\frac{a_J^r}{(1+2^j|x_J-x_Q|)^\lambda}\right]^{\frac{1}{r}}
=\left(a_{r,\lambda}^*\right)_Q.
\end{align*}
When $r\in(1,\infty)$, from H\"older's inequality and Lemma \ref{253x}(ii),
we infer that, for any $Q\in\mathscr{Q}_j$,
\begin{align*}
\sum_{J\in\mathscr{Q}_j}\frac{a_J}{(1+2^j|x_R-x_J|)^{\widetilde{M}}}
&\leq\left[\sum_{J\in\mathscr{Q}_j}\frac{a_J^r}{(1+2^j|x_R-x_J|)^{\widetilde{M}}}\right]^{\frac1r}
\left[\sum_{J\in\mathscr{Q}_j}\frac1{(1+2^j|x_R-x_J|)^{\widetilde{M}}}\right]^{\frac1{r'}}\\
&\sim\left[\sum_{J\in\mathscr{Q}_j}\frac{a_J^r}{(1+2^j|x_R-x_J|)^{\widetilde{M}}}\right]^{\frac1r},
\end{align*}
which, together with Lemma \ref{lucas2}, further implies that
\begin{align*}
\left(d_{r,\lambda}^*\right)_Q
&\leq\left[\sum_{J\in\mathscr{Q}_j}a_J^r
\sum_{R\in\mathscr{Q}_j}\frac1{(1+2^j|x_R-x_Q|)^\lambda}
\frac1{(1+2^j|x_R-x_J|)^{\widetilde{M}}}\right]^{\frac{1}{r}}\\
&\sim\left[\sum_{J\in\mathscr{Q}_j}
\frac{a_J^r}{(1+2^j|x_J-x_Q|)^\lambda}\right]^{\frac{1}{r}}
=\left(a_{r,\lambda}^*\right)_Q.
\end{align*}
By these and \eqref{aQ}, we conclude that
there exists a positive constant $\widetilde C$ such that,
for any $Q\in\mathscr{Q}_+$,
\begin{align}\label{2023.6.6}
\left(a_{r,\lambda}^*\right)_Q
\leq \widetilde C2^{N\frac{d_2}{p}}\left(b_{r,\lambda}^*\right)_Q
+\widetilde C2^{-N}\left(a_{r,\lambda}^*\right)_Q.
\end{align}
Let $N\in\mathbb Z_+$ be sufficiently large such that $\widetilde C2^{-N}<2^{-1}$.
From \eqref{2023.6.6} and \eqref{< infty},
it follows that \eqref{a < b} holds in this special case.

Finally, we consider the general case where $\vec f\in(\mathcal{S}')^m$ satisfies
$\operatorname{supp}\widehat{\vec f}\subset\{\xi\in\mathbb R^n:\ |\xi|\leq2^{j+1}\}$
by applying a standard regularization argument.
Let $g\in\mathcal{S}$ satisfy $\operatorname{supp}\widehat{g}\subset B(\mathbf{0},1)$,
$\widehat{g}(\xi)\geq0$ for any $\xi\in\mathbb R^n$, and $g(\mathbf{0})=1$.
These, combined with the inverse Fourier transform formula, further imply that,
for any $x\in\mathbb R^n$,
\begin{align}\label{< 1}
|g(x)|
\leq\int_{\mathbb{R}^n}\left|\widehat{g}(\xi)\right|\,d\xi
=g(\mathbf{0})
=1.
\end{align}
For any $\delta\in(0,\infty)$, let $\vec f_\delta(\cdot):=\vec f(\cdot)g(\delta\cdot)$.
Then
\begin{align*}
\operatorname{supp}\widehat{\vec f_\delta}
\subset\operatorname{supp}\widehat{\vec f}+\operatorname{supp}\widehat{g(\delta\cdot)}
\subset\left\{\xi\in\mathbb R^n:\ |\xi|\leq 3\cdot 2^j\right\} 
\end{align*}
as soon as $\delta<2^j$.

By \cite[Theorem 2.3.20]{g14c}, we find that $\vec f_\delta\in(\mathcal{S})^m$.
We define new sequences $a^\delta$ and $b^\delta$ like $a$ and $b$
but with $\vec f$ replaced by $\vec f_\delta$. From \eqref{< 1},
it follows that $b^\delta_{Q,N}\leq b_{Q,N}$.
Applying \eqref{a < b} to $a^\delta$ and $b^\delta$, we obtain, for any $Q\in\mathscr{Q}_+$,
\begin{align*}
\left(a_{r,\lambda}^*\right)_Q^\delta
:=\left[\sum_{R\in\mathscr{Q},\,\ell(R)=\ell(Q)}
\frac{[|R|^{\frac12}\sup_{y\in R}|A_R\vec f_\delta(y)|]^r}
{\{1+[\ell(R)]^{-1}|x_R-x_Q|\}^\lambda}\right]^{\frac{1}{r}}
\lesssim
\left(b_{r,\lambda}^*\right)_Q^\delta
\leq \left(b_{r,\lambda}^*\right)_Q.
\end{align*}
Notice that $\vec f_\delta\to\vec f$ as $\delta\to0$ uniformly on compact sets
and hence, for any $R\in\mathscr{Q}_j$,
\begin{align*}
\sup_{y\in R}\left|A_R\vec f(y)\right|
=\lim_{\delta\to0}\sup_{y\in R}\left|A_R\vec f_\delta(y)\right|.
\end{align*}
Then letting $\delta\to0$ and using Fatou's lemma,
we conclude that, for any $Q\in\mathscr{Q}_+$,
\begin{align*}
\left(a_{r,\lambda}^*\right)_Q
\leq\liminf_{\delta\to0}\left(a_{r,\lambda}^*\right)_Q^\delta
\lesssim\left(b_{r,\lambda}^*\right)_Q.
\end{align*}
This finishes the proof of \eqref{a < b} in the general case
and hence Lemma \ref{inf and sup}.
\end{proof}

Now, we give the relations among $\vec f$,
$\sup_{\mathbb{A},\varphi}(\vec f)$,
and $\inf_{\mathbb{A},\varphi,\gamma}(\vec f)$.

\begin{lemma}\label{56}
Let $s\in\mathbb{R}$, $\tau\in[0,\infty)$, $p\in(0,\infty)$, and $q\in(0,\infty]$.
Let $\Phi\in\mathcal{S}$ satisfy \eqref{19}
and $\varphi\in\mathcal{S}$ satisfy \eqref{20}.
Let $\mathbb{A}:=\{A_Q\}_{Q\in\mathscr{Q}_+}$
be strongly doubling of order $(d_1,d_2;p)$ for some $d_1,d_2\in[0,\infty)$.
Let $\gamma\in\mathbb Z_+$ be sufficiently large.
Then, for any $\vec f\in(\mathcal{S}')^m$,
$$
\left\|\vec f\right\|_{A^{s,\tau}_{p,q}(\mathbb{A},\Phi,\varphi)}
\sim\left\|\sup_{\mathbb{A},\varphi}\left(\vec f\right)\right\|_{a^{s,\tau}_{p,q}}
\sim\left\|\inf_{\mathbb{A},\varphi,\gamma}\left(\vec f\right)\right\|_{a^{s,\tau}_{p,q}},
$$
where $\sup_{\mathbb{A},\varphi}(\vec f)$
and $\inf_{\mathbb{A},\varphi,\gamma}(\vec f)$
are the same as, respectively, in \eqref{sup} and \eqref{inf}
and the positive equivalence constants are independent of $\vec f$.
\end{lemma}

\begin{proof}
The first ``$\sim$'' is proved in \cite[Lemma 3.17]{bhyy}
for the homogeneous version of these spaces
(i.e., with $\{\varphi_j*\vec f\}_{j\in\mathbb Z}$ in place of
$\{\Phi*\vec f\}\cup\{\varphi_j*\vec f\}_{j=1}^\infty$)
and under the assumption that $\{A_Q\}_{Q\in\mathscr Q}$
are the reducing operators of order $p$ for some $W\in A_p$.
However, an inspection of the proof shows that:
\begin{itemize}
\item only the Fourier support property
$\operatorname{supp}\widehat\varphi_j\subset\{\xi:\ |\xi|\leq 2^{j+1}\}$
of the functions $\varphi_j\in\mathcal S$ is used, and this is also
satisfied by $\Phi$ in place of $\varphi_0$
(and certainly by $0$ in place of $\varphi_j$ for $j<0$);
\item only the weak doubling property of the sequence $\mathbb A$ is used,
and this is certainly satisfied under the assumed strong doubling property.
\end{itemize}
Thus, the same argument also proves the first ``$\sim$'' of the present lemma.

To show the second ``$\sim$'', let $\vec f\in(\mathcal{S}')^m$ and $\lambda\in(n,\infty)$ be fixed.
By Lemma \ref{inf and sup} with $\vec f$ replaced by $\varphi_{j_Q}*\vec f$,
we find that, for any $Q\in\mathscr{Q}_+$,
$
([\sup_{\mathbb{A},\varphi}(\vec f)]_{p\wedge q, \lambda}^*)_Q
\sim([\inf_{\mathbb{A},\varphi,\gamma}(\vec f)]_{p\wedge q, \lambda}^*)_Q.
$
This, together with Lemma \ref{4}, further implies that
\begin{align*}
\left\|\sup_{\mathbb{A},\varphi}\left(\vec f\right)\right\|_{a^{s,\tau}_{p,q}}
\sim\left\|\left[\sup_{\mathbb{A},\varphi}\left(\vec f\right)
\right]_{p\wedge q,\lambda}^*\right\|_{a^{s,\tau}_{p,q}}
\sim\left\|\left[\inf_{\mathbb{A},\varphi,\gamma}\left(\vec f\right)
\right]_{p\wedge q,\lambda}^*\right\|_{a^{s,\tau}_{p,q}}
\sim\left\|\inf_{\mathbb{A},\varphi,\gamma}\left(\vec f\right)\right\|_{a^{s,\tau}_{p,q}},
\end{align*}
which completes the proof of Lemma \ref{56}.
\end{proof}

The following lemma proves that $T_\psi\vec t$ is well defined
for any $\vec t\in a^{s,\tau}_{p,q}(\mathbb{A})$.

\begin{lemma}\label{well defined}
Let $s\in\mathbb{R}$, $\tau\in[0,\infty)$, $p\in(0,\infty)$, and $q\in(0,\infty]$.
Let $\mathbb{A}:=\{A_Q\}_{Q\in\mathscr{Q}_+}$
be strongly doubling of order $(d_1,d_2;p)$ for some $d_1,d_2\in[0,\infty)$.
Then, for any $\Psi\in\mathcal{S}$, $\psi\in\mathcal{S}_\infty$,
and $\vec t:=\{\vec t_Q\}_{Q\in\mathscr{Q}_+}\in a^{s,\tau}_{p,q}(\mathbb{A})$,
the series
$
\sum_{Q\in\mathscr{Q}_0}\vec t_Q\Psi_Q
+\sum_{j=1}^\infty\sum_{Q\in\mathscr{Q}_j}\vec t_Q\psi_Q
$
converges in $(\mathcal{S}')^m$.
Moreover, if $M\in\mathbb{Z}_+$ satisfies
\begin{align}\label{239}
M>\max\left\{\frac{d_1+d_2}{p},\frac{n+d_2}{p}-(s+n\tau)\right\},
\end{align}
then there exists a positive constant $C$, such that,
for any $\Psi,\phi\in\mathcal{S}$, $\psi\in\mathcal{S}_\infty$,
and $\vec t\in a^{s,\tau}_{p,q}(\mathbb{A})$,
\begin{align}\label{well defined equ}
\sum_{Q\in\mathscr{Q}_0}\left|\vec{t}_Q\right||\langle\Psi_Q,\phi\rangle|
+\sum_{j=1}^\infty\sum_{Q\in\mathscr{Q}_j}\left|\vec{t}_Q\right||\langle\psi_Q,\phi\rangle|
\leq C\left\|\vec t\right\|_{a^{s,\tau}_{p,q}(\mathbb{A})}
(\|\Psi\|_{S_{M+1}}+\|\psi\|_{S_{M+1}})\|\phi\|_{S_{M+1}},
\end{align}
where $\|\cdot\|_{S_M}$ is the same as in \eqref{SM}.
\end{lemma}

\begin{proof}
The proof is similar to the homogeneous version given in \cite[Lemma 3.33]{bhyy},
but the modification requires some care because all test functions appearing
in \cite[Lemma 3.33]{bhyy} are in the smaller class $\mathcal S_\infty$
with additional cancellation that we do not have in the case at hand.
For this reason, we revisit the argument to implement the necessary changes.
First note that
\begin{equation*}
\left|\vec t_Q\right|
\leq\left\|A_Q^{-1}\right\| \left|A_Q\vec t_Q\right|
\leq\left\|A_{Q_{0,\mathbf 0}}^{-1}\right\|
\left\|A_{Q_{0,\mathbf 0}} A_Q^{-1}\right\|
|Q|^{\frac sn+\frac12-\frac1p+\tau}
\left\|\vec t\right\|_{a^{s,\tau}_{p,q}(\mathbb A)}.
\end{equation*}
By the assumed strong doubling property, we find that
\begin{equation*}
\left\|A_{Q_{0,\mathbf 0}} A_Q^{-1}\right\|
\lesssim[\ell(Q)]^{-\frac{d_2}p}(1+|x_Q|)^{\frac{d_1+d_2}p},
\end{equation*}
noting that $\ell(Q)\leq\ell(Q_{0,\mathbf 0})= 1$ to simplify the right-hand side in the strong doubling estimate.

Since $\psi\in\mathcal S_\infty$ satisfies any number of the cancellation conditions assumed in Corollary \ref{85x}, the said corollary provides us with the bound
\begin{equation*}
|\langle\psi_Q,\phi\rangle|
\lesssim\|\psi\|_{S_{M+1}}\|\phi\|_{S_{M+1}}
[\ell(Q)]^{M+\frac n2}(1+|x_Q|)^{-n-M},
\end{equation*}
where we can take $M$ as large as we like.  For $|\langle\Psi_Q,\phi\rangle|$, we apply the following direct computation instead:
By the definition of $\|\cdot\|_{S_{M+1}}$ and Lemma \ref{lucas},
we conclude that, for any $Q\in\mathscr{Q}$ with $\ell(Q)=1$,
\begin{align*}
|\langle\Psi_Q,\phi\rangle|
&\leq\int_{\mathbb R^n}|\Psi(x-x_Q)\phi(x)|\,dx\\
&\leq\|\Psi\|_{S_{M+1}}\|\phi\|_{S_{M+1}}
\int_{\mathbb R^n}\frac{1}{(1+|x-x_Q|)^{n+M+1}}\frac{1}{(1+|x|)^{n+M+1}}\,dx\\
&\lesssim\|\Psi\|_{S_{M+1}}\|\phi\|_{S_{M+1}}\frac{1}{(1+|x_Q|)^{n+M+1}}.
\end{align*}
Substituting these estimates into the left-hand side of \eqref{well defined equ}, we obtain
\begin{align*}
&\sum_{Q\in\mathscr{Q}_0}\left|\vec{t}_Q\right||\langle\Psi_Q,\phi\rangle|
+\sum_{j=1}^\infty\sum_{Q\in\mathscr{Q}_j} \left|\vec{t}_Q\right||\langle\psi_Q,\phi\rangle|\\
&\quad\lesssim\left\|\vec t\right\|_{a^{s,\tau}_{p,q}(\mathbb A)}
\|\phi\|_{S_{M+1}}\left[\|\Psi\|_{S_{M+1}}
\sum_{Q\in\mathscr{Q}_0}(1+|x_Q|)^{\frac{d_1+d_2}p-n-M-1}\right. \\
&\qquad+\left.\|\psi\|_{S_{M+1}}\sum_{j=1}^\infty 2^{-j(s+\frac{n}{2}-\frac{n}{p}+n\tau)}2^{j\frac{d_2}{p}}2^{-j(M+\frac{n}{2})}
\sum_{Q\in\mathscr{Q}_j}(1+|x_Q|)^{\frac{d_1+d_2}p-n-M}\right].
\end{align*}
The sum over $Q\in\mathscr Q_0$ converges under the assumption
$M>\frac{d_1+d_2}p$. For any $j\in\mathbb N$, the sum over $Q\in\mathscr Q_j$
has order $2^{jn}$, and then the sum over $j\in\mathbb N$ converges
under the assumption that $M>\frac{n+d_2}p-(s+n\tau)$.
This finishes the proof of Lemma \ref{well defined}.
\end{proof}

Next, we can show Theorem \ref{phi A}.

\begin{proof}[Proof of Theorem \ref{phi A}]
The proof is similar to that of \cite[Theorem 3.29]{bhyy}, which is formulated for the $W$-weighted spaces, but proceeds via the $\mathbb A$-weighted spaces in its intermediate steps that we are able to reuse here. Recall that the theorem has three claims: the boundedness of $S_\varphi$, the boundedness of $T_\psi$, and the claim that $T_\psi\circ S_\varphi$ is the identity.

We first prove the boundedness of
$S_\varphi:\ A^{s,\tau}_{p,q}(\mathbb A,\widetilde\Phi,\widetilde{\varphi})\to a^{s,\tau}_{p,q}(\mathbb A)$.
Recalling the notation \eqref{sup},
for any $\vec f\in A^{s,\tau}_{p,q}(\mathbb A,\widetilde\Phi,\widetilde{\varphi})$ and $Q\in\mathscr{Q}_+$, we have
\begin{align*}
\left|A_Q\left(S_{\varphi}\vec f\right)_Q\right|
\leq\left|A_Q\left\langle\vec f,\varphi_Q\right\rangle\right|
=|Q|^{\frac12}\left|A_Q\left(\widetilde{\varphi}_{j_Q}*\vec f\right)(x_Q)\right|
\leq\sup_{\mathbb{A},\widetilde{\varphi},Q}\left(\vec f\right),
\end{align*}
where $\varphi_Q$ and $\varphi_{j_Q}$ are replaced by $\Phi_Q$ and $\Phi$ if $j_Q=0$.
Using this and Lemma \ref{56}, we obtain
$$
\left\|S_{\varphi}\vec f\right\|_{a^{s,\tau}_{p,q}(\mathbb{A})}
\leq\left\|\sup_{\mathbb{A},\widetilde{\varphi}}\left(\vec f\right)\right\|_{a^{s,\tau}_{p,q}}
\sim\left\|\vec f\right\|_{A^{s,\tau}_{p,q}(\mathbb A,\widetilde{\varphi})}.
$$
This finishes the proof of the boundedness of $ S_{\varphi}$.

We turn to the boundedness of $T_\psi:\ a^{s,\tau}_{p,q}(\mathbb A)\to A^{s,\tau}_{p,q}(\mathbb A,\Phi,\varphi)$. Here, we only indicate the necessary changes compared to the proof of \cite[Theorem 3.29]{bhyy}. As in that proof, estimating $\|T_\psi\vec t\|_{A^{s,\tau}_{p,q}(\mathbb A,\Phi,\varphi)}$ begins with, for any $j\in\mathbb Z_+$ and $x\in\mathbb R^n$,
\begin{align*}
\left|A_Q\left[\varphi_j*\left(T_\psi\vec t\right)\right](x)\right|
&=\left|\sum_{i=j-1}^{j+1}\sum_{R\in\mathscr{Q}_i}
A_Q\vec t_R\left(\varphi_j*\psi_R\right)(x)\right| \\
&\leq\sum_{i=j-1}^{j+1}\sum_{R\in\mathscr{Q}_i}
\left\|A_QA_R^{-1}\right\|\left|A_R\vec t_R\right|
\left|\left(\varphi_j*\psi_R\right)(x)\right|,
\end{align*}
where $\varphi_0$ should be replaced by $\Phi$
and, similarly, $\psi_R$ by $\Psi_R$ if $\ell(R)=1$ and by $0$ if $\ell(R)>1$.

To estimate $\|A_QA_R^{-1}\|$, the argument of \cite[Theorem 3.29]{bhyy}
only uses the strong doubling property of $\mathbb A$,
which is assumed now as well. For the factor $|(\varphi_j*\psi_R)(x)|$,
the proof of \cite[Theorem 3.29]{bhyy} uses an estimate,
where a cancellation condition on both $\varphi$ and $\psi$ is assumed.
However, Corollary \ref{85x} shows that it is enough to have
a cancellation condition on the function related to the smaller length scale,
and this covers all cases but the one where both $j=0$ and $\ell(R)=1$.
In this remaining case, from the definition of $\|\cdot\|_{S_M}$
and Lemma \ref{lucas}, it follows that,
for any $R\in\mathscr{Q}_0$ and $x\in\mathbb{R}^n$,
\begin{align*}
\left|\left(\Phi*\Psi_R\right)(x)\right|
&=\left|\int_{\mathbb{R}^n}\Phi(x-y)\Psi_R(y)\,dy\right|
\leq\int_{\mathbb R^n}
\left|\Phi(x-y)\Psi\left(y-x_R\right)\right|\,dy\\
&\leq\|\Phi\|_{S_M}\|\Psi\|_{S_M}
\int_{\mathbb R^n}\frac{1}{(1+|x-y|)^{n+M}}
\frac{1}{(1+|y-x_R|)^{n+M}}\,dy\\
&\lesssim\frac{1}{(1+|x-x_R|)^{n+M}}
\sim|R|^{-\frac12}\frac{1}{\{1+[\ell(R)]^{-1}|x-x_R|\}^{n+M}}.
\end{align*}
Using this estimate and Corollary \ref{85x} in place of the bounds
for $|(\varphi_j*\psi_R)(x)|$ in the proof of \cite[Theorem 3.30]{bhyy},
the remainder of this argument can be repeated to prove that $T_\psi$ is bounded.

In addition, if $\Phi$, $\Psi$, $\varphi$, and $\psi$ satisfy \eqref{21},
then, by Lemma \ref{7}, we find that $T_\psi\circ S_\varphi$
is the identity on $A^{s,\tau}_{p,q}(W,\widetilde\Phi,\widetilde\varphi)$.
This finishes the proof of Theorem \ref{phi A}.
\end{proof}

Applying Theorem \ref{phi A}, we obtain the following proposition
which shows that $A^{s,\tau}_{p,q}(\mathbb{A},\Phi,\varphi)$
is independent of the choice of $\Phi$ and $\varphi$.

\begin{proposition}\label{38}
Let $s\in\mathbb{R}$, $\tau\in[0,\infty)$, $p\in(0,\infty)$, and $q\in(0,\infty]$.
Let $\Phi\in\mathcal{S}$ satisfy \eqref{19} and $\varphi\in\mathcal{S}$ satisfy \eqref{20}.
Let $\mathbb{A}:=\{A_Q\}_{Q\in\mathscr{Q}_+}$
be strongly doubling of order $(d_1,d_2;p)$ for some $d_1,d_2\in[0,\infty)$.
Then $A^{s,\tau}_{p,q}(\mathbb{A},\Phi,\varphi)$
is independent of the choice of $\Phi$ and $\varphi$.
\end{proposition}

\begin{proof}
Let $\Phi^{(2)},\Psi^{(2)}\in\mathcal{S}$ satisfy \eqref{19},
$\varphi^{(2)},\psi^{(2)}\in\mathcal{S}$ satisfy \eqref{20},
and $\Phi^{(2)}$, $\Psi^{(2)}$, $\varphi^{(2)}$, and $\psi^{(2)}$ satisfy \eqref{21}.
Then, from Lemma \ref{7} and Theorem \ref{phi A}, we deduce that,
for any $\vec{f}\in A^{s,\tau}_{p,q}(\mathbb{A},\Phi^{(2)},\varphi^{(2)})$,
\begin{align*}
\left\|\vec f\right\|_{A^{s,\tau}_{p,q}(\mathbb{A},\Phi,\varphi)}
=\left\|\left(T_{\widetilde{\Psi^{(2)}},\widetilde{\psi^{(2)}}}
\circ S_{\widetilde{\Phi^{(2)}},\widetilde{\varphi^{(2)}}}\right)
\left(\vec f\right)\right\|_{A^{s,\tau}_{p,q}(\mathbb{A},\Phi,\varphi)}
\lesssim\left\|S_{\widetilde{\Phi^{(2)}},\widetilde{\varphi^{(2)}}}\vec f
\right\|_{a^{s,\tau}_{p,q}(\mathbb{A})}
\lesssim\left\|\vec f\right\|_{A^{s,\tau}_{p,q}(\mathbb{A},\Phi^{(2)},\varphi^{(2)})}.
\end{align*}
By symmetry, we also obtain the reverse inequality.
This finishes the proof of Proposition \ref{38}.
\end{proof}

Based on Proposition \ref{38}, in what follows,
we denote $A^{s,\tau}_{p,q}(\mathbb{A},\Phi,\varphi)$
simply by $A^{s,\tau}_{p,q}(\mathbb{A})$.
By an argument similar to that used in the proof of \cite[Corollary 3.14]{b07},
we obtain the following proposition.
For the convenience of the reader, we give the details of its proof here.

\begin{proposition}\label{174}
Let $s\in\mathbb{R}$, $\tau\in[0,\infty)$, $p\in(0,\infty)$, and $q\in(0,\infty]$.
Let $\mathbb{A}:=\{A_Q\}_{Q\in\mathscr{Q}_+}$
be strongly doubling of order $(d_1,d_2;p)$ for some $d_1,d_2\in[0,\infty)$.
Then $A^{s,\tau}_{p,q}(\mathbb{A})\subset(\mathcal{S}')^m$.
Moreover, if $M\in\mathbb{Z}_+$ satisfies \eqref{239},
then there exists a positive constant $C$ such that,
for any $\vec f\in A^{s,\tau}_{p,q}(\mathbb{A})$ and $\phi\in\mathcal{S}$,
$$
\left|\left\langle\vec f,\phi\right\rangle\right|
\leq C\left\|\vec f\right\|_{A^{s,\tau}_{p,q}(\mathbb{A})}\|\phi\|_{S_{M+1}},
$$
where $\|\cdot\|_{S_{M+1}}$ is the same as in \eqref{SM}.
\end{proposition}

\begin{proof}
Let $\Phi,\Psi\in\mathcal{S}$ satisfy \eqref{19},
$\varphi,\psi\in\mathcal{S}$ satisfy \eqref{20},
and $\Phi$, $\Psi$, $\varphi$, and $\psi$ satisfy \eqref{21}.
By Lemmas \ref{7} and \ref{well defined} and Theorem \ref{phi A}, we find that,
for any $\vec f\in A^{s,\tau}_{p,q}(W)$ and $\phi\in\mathcal{S}$,
\begin{align*}
\left|\left\langle\vec f,\phi\right\rangle\right|
&=\left|\left\langle\left(T_\psi\circ S_\varphi\right)\vec f,
\phi\right\rangle\right|
\leq\sum_{Q\in\mathscr{Q}_+}\left|\left(S_\varphi\vec f\right)_Q
\right||\langle\psi_Q,\phi\rangle|\\
&\lesssim\left\|S_\varphi\vec f\right\|_{a^{s,\tau}_{p,q}(\mathbb{A})}\|\phi\|_{S_{M+1}}
\lesssim\left\|\vec f\right\|_{A^{s,\tau}_{p,q}(\mathbb{A})}\|\phi\|_{S_{M+1}}.
\end{align*}
This finishes the proof of Proposition \ref{174}.
\end{proof}

Applying Proposition \ref{174} and
an argument similar to that used in the proof of \cite[Proposition 2.3.1]{g14},
we obtain the following conclusion; we omit the details.

\begin{proposition}\label{256}
Let $s\in\mathbb{R}$, $\tau\in[0,\infty)$,
$p\in(0,\infty)$, and $q\in(0,\infty]$.
Let $\mathbb{A}:=\{A_Q\}_{Q\in\mathscr{Q}_+}$
be strongly doubling of order $(d_1,d_2;p)$ for some $d_1,d_2\in[0,\infty)$.
Then $A^{s,\tau}_{p,q}(\mathbb{A})$ is a complete quasi-normed space.
\end{proposition}

%

\subsection{Pointwise Weighted Spaces $A^{s,\tau}_{p,q}(W)$ and $a^{s,\tau}_{p,q}(W)$}
\label{phi-transform}

We introduce the matrix-weighted Besov-type and Triebel--Lizorkin-type spaces as follows.

\begin{definition}
Let $s\in\mathbb{R}$, $\tau\in[0,\infty)$, $p\in(0,\infty)$, $q\in(0,\infty]$, and $W\in A_{p,\infty}$.
Let $\Phi\in\mathcal{S}$ satisfy \eqref{19} and
$\varphi\in\mathcal{S}$ satisfy \eqref{20}.
The \emph{inhomogeneous matrix-weighted Besov-type space}
$B^{s,\tau}_{p,q}(W,\Phi,\varphi)$
and the \emph{inhomogeneous matrix-weighted Triebel--Lizorkin-type space}
$F^{s,\tau}_{p,q}(W,\Phi,\varphi)$
are defined by setting
$$
A^{s,\tau}_{p,q}(W,\Phi,\varphi)
:=\left\{\vec{f}\in(\mathcal{S}')^m:\
\left\|\vec{f}\right\|_{A^{s,\tau}_{p,q}(W,\Phi,\varphi)}<\infty\right\},
$$
where, for any $\vec{f}\in(\mathcal{S}')^m$,
$$
\left\|\vec{f}\right\|_{A^{s,\tau}_{p,q}(W,\Phi,\varphi)}
:=\left\|\left\{2^{js}\left|W^{\frac{1}{p}}\left(\varphi_j*\vec f\right)
\right|\right\}_{j\in\mathbb Z_+}\right\|_{LA_{p,q}^\tau}
$$
with $\varphi_0$ replaced by $\Phi$
and $\|\cdot\|_{LA_{p,q}^\tau}$ in \eqref{LApq in}.
\end{definition}

\begin{lemma}\label{66}
Let $s\in\mathbb{R}$, $\tau\in[0,\infty)$,
$p\in(0,\infty)$, $q\in(0,\infty]$, and $\delta\in(0,1)$.
Suppose that, for any $Q\in\mathscr{Q}_+$, $E_Q\subset Q$ is a measurable set
with $|E_Q|\geq\delta|Q|$.
Then, for any sequence $t:=\{t_Q\}_{Q\in\mathscr{Q}_+}\subset\mathbb{C}$,
\begin{align*}
\|t\|_{a^{s,\tau}_{p,q}}
\sim\left\|\left\{2^{j(s+\frac{n}{2})}\sum_{Q\in\mathscr{Q}_j}t_Q\mathbf{1}_{E_Q}
\right\}_{j\in\mathbb Z_+}\right\|_{LA_{p,q}^\tau},
\end{align*}
where the positive equivalence constants are independent of $t$.
\end{lemma}

\begin{proof}
When $a^{s,\tau}_{p,q}=b^{s,\tau}_{p,q}$ and $LA_{p,q}^\tau=LB_{p,q}^\tau$,
the present lemma can be proved by some simple computations.
When $a^{s,\tau}_{p,q}=f^{s,\tau}_{p,q}$ and $LA_{p,q}^\tau=LF_{p,q}^\tau$,
the present lemma is the inhomogeneous part of \cite[Proposition 2.4]{yyz13}.
This finishes the proof of Lemma \ref{66}.
\end{proof}

Next, we can establish the relations between
$\|\vec f\|_{A^{s,\tau}_{p,q}(W,\Phi,\varphi)}$
and $\|\vec f\|_{A^{s,\tau}_{p,q}(\mathbb{A},\Phi,\varphi)}$.

\begin{theorem}\label{69}
Let $s\in\mathbb{R}$, $\tau\in[0,\infty)$, $p\in(0,\infty)$, and $q\in(0,\infty]$.
Let $\Phi\in\mathcal{S}$ satisfy \eqref{19} and
$\varphi\in\mathcal{S}$ satisfy \eqref{20}.
Let $W\in A_{p,\infty}$ and $\mathbb{A}:=\{A_Q\}_{Q\in\mathscr{Q}_+}$
be a sequence of reducing operators of order $p$ for $W$.
Then $A^{s,\tau}_{p,q}(W,\Phi,\varphi)=A^{s,\tau}_{p,q}(\mathbb{A},\Phi,\varphi)$.
Moreover, for any $\vec f\in(\mathcal{S}')^m$,
$
\|\vec f\|_{A^{s,\tau}_{p,q}(W,\Phi,\varphi)}
\sim\|\vec f\|_{A^{s,\tau}_{p,q}(\mathbb{A},\Phi,\varphi)},
$
where the positive equivalence constants are independent of $\vec f$.
\end{theorem}

\begin{proof}
The proof of ``$\lesssim$'' is similar to that of \cite[Lemma 3.24]{bhyy},  but we provide the details for completeness: We have
\begin{equation*}
\left\|\vec f\right\|_{A^{s,\tau}_{p,q}(W,\Phi,\varphi)}
=\left\|\{g_j\}_{j=0}^\infty\right\|_{LA^\tau_{p,q}},
\quad\text{where}\quad
g_j:=2^{js}\left|W^{\frac1p}\left(\varphi_j*\vec f\right)\right|,
\, \forall\,j\in\mathbb Z_+
\end{equation*}
and $\varphi_0:=\Phi$, as usual. We can then estimate, for any $j\in\mathbb Z_+$,
\begin{equation*}
g_j
\leq 2^{js}\sum_{Q\in\mathscr Q_j}\left\|W^{\frac1p}A_Q^{-1}\right\|
\left|A_Q\left(\varphi_{j_Q}*\vec f\right)\right| \mathbf 1_Q
\leq 2^{js} \sum_{Q\in\mathscr Q_j}\left\|W^{\frac1p}A_Q^{-1}\right\| \, \sup_{\mathbb A,\varphi,Q}\left(\vec f\right)\widetilde{\mathbf 1}_Q
\end{equation*}
Thus,
\begin{align*}
\left\|\{g_j\}_{j=0}^\infty\right\|_{LA^\tau_{p,q}}
&:=\sup_{P\in\mathscr Q_+} \frac1{|P|^\tau}
\left\|\{\mathbf{1}_P g_j\}_{j=j_P}^\infty\right\|_{LA_{p,q}} \\
&\leq\sup_{P\in\mathscr Q_+} \frac1{|P|^\tau}
\left\|\left\{2^{js} \sum_{Q\in\mathscr Q_j,\,Q\subset P}
\left\|W^{\frac1p}A_Q^{-1}\right\|
\sup_{\mathbb A,\varphi,Q}\left(\vec f\right) \widetilde{\mathbf 1}_Q
\right\}_{j= j_P}^\infty\right\|_{LA_{p,q}}.
\end{align*}
In the Triebel--Lizorkin-type case, an application of Corollary \ref{46x} now gives
\begin{align*}
\left\|\{g_j\}_{j=0}^\infty\right\|_{LF^\tau_{p,q}}
&\lesssim \sup_{P\in\mathscr Q_+} \frac1{|P|^\tau}
\left\|\left\{2^{js} \sum_{Q\in\mathscr Q_j,\,Q\subset P}
\sup_{\mathbb A,\varphi,Q}\left(\vec f\right) \, \widetilde{\mathbf 1}_Q
\right\}_{j= j_P}^\infty\right\|_{LF_{p,q}} \\
&=\left\|\left\{2^{js} \sum_{Q\in\mathscr Q_j }
\sup_{\mathbb A,\varphi,Q}\left(\vec f\right) \, \widetilde{\mathbf 1}_Q
\right\}_{j=0}^\infty\right\|_{LF_{p,q}^\tau}
=\left\|\sup_{\mathbb A,\varphi}\left(\vec f\right)\right\|_{f^{s,\tau}_{p,q}}.
\end{align*}
The analogous conclusion in the Besov case is similar but easier and we omit the details.

Combining these estimates with Lemma \ref{56}, we have thus proved that,
for any $\vec f\in(\mathcal{S}')^m$,
\begin{align}\label{69part1}
\left\|\vec f\right\|_{A^{s,\tau}_{p,q}(W,\Phi,\varphi)}
\lesssim\left\|\sup_{\mathbb{A},\varphi}
\left(\vec f\right)\right\|_{a^{s,\tau}_{p,q}}
\sim\left\|\vec f\right\|_{A^{s,\tau}_{p,q}(\mathbb{A},\Phi,\varphi)}
\sim \left\|\inf_{\mathbb{A},\varphi,\gamma}\left(\vec f\right)\right\|_{a^{s,\tau}_{p,q}},
\end{align}
where $\gamma\in\mathbb Z_+$ is the same as in Lemma \ref{56}.

We then turn to the proof of the reverse estimate,
starting with the quantity on the extreme right-hand side of the previous chain.
By the definition \eqref{3.11x} with $N$ replaced by $\gamma$,
for each $Q\in\mathscr{Q}_+$, there exists
$\widetilde{Q}\in\mathscr{Q}_{j_Q+\gamma}$ such that $\widetilde{Q}\subset Q$ and
\begin{align}\label{inf equ}
|Q|^{-\frac12}\inf_{\mathbb{A},\varphi,Q,\gamma}\left(\vec f\right)
&=\inf_{y\in \widetilde{Q}}
\left|A_{\widetilde{Q}}\left(\varphi_{j_Q}*\vec f\right)(y)\right|
\leq\inf_{y\in \widetilde{Q}}
\left\|A_{\widetilde{Q}}W^{-\frac1p}(y)\right\|
\left|W^{\frac1p}(y)\left(\varphi_{j_Q}*\vec f\right)(y)\right|.
\end{align}
For each $Q\in\mathscr{Q}_+$, let
$
E_Q:=\{y\in\widetilde{Q}:\
\|A_{\widetilde{Q}}W^{-\frac1p}(y)\|^p<(C[W]_{A_{p,\infty}})^2\},
$
where $C$ is the same as in \eqref{eq 8 var} of Lemma \ref{8 prepare}.
Then, from \eqref{inf equ} and the definition of $E_Q$, we infer that,
for each $Q\in\mathscr{Q}_+$,
\begin{align*}
|Q|^{-\frac12}\inf_{\mathbb{A},\varphi,Q,\gamma}\left(\vec f\right)
\lesssim\inf_{y\in E_Q}\left|W^{\frac1p}(y)\left(\varphi_{j_Q}*\vec f\right)(y)\right|.
\end{align*}
Moreover, by Lemma \ref{8 prepare}, we conclude that, for each $Q\in\mathscr{Q}_+$, the subset $E_Q$ satisfies
$|E_Q|>\frac12|\widetilde{Q}|=2^{-\gamma n-1}|Q|$.
Then Lemma \ref{66} further implies that
\begin{align*}
\left\|\inf_{\mathbb{A},\varphi,\gamma}\left(\vec f\right)\right\|_{a^{s,\tau}_{p,q}}
&\sim\left\|\left\{2^{j(s+\frac{n}{2})}\sum_{Q\in\mathscr{Q}_j}
\inf_{\mathbb{A},\varphi,Q,\gamma}\left(\vec f\right)\mathbf{1}_{E_Q}
\right\}_{j\in\mathbb Z_+}\right\|_{LA_{p,q}^\tau}\\
&\lesssim\left\|\left\{2^{js}\sum_{Q\in\mathscr{Q}_j}
\left|W^{\frac1p}\left(\varphi_j*\vec f\right)\right|\mathbf{1}_{E_Q}
\right\}_{j\in\mathbb Z_+}\right\|_{LA_{p,q}^\tau}
\leq\left\|\vec f\right\|_{A^{s,\tau}_{p,q}(W,\Phi,\varphi)}.
\end{align*}
In combination with \eqref{69part1}, this completes the proof of Theorem \ref{69}.
\end{proof}

Next, we introduce matrix-weighted
Besov-type and Triebel--Lizorkin-type sequence spaces.

\begin{definition}
Let $s\in\mathbb{R}$, $\tau\in[0,\infty)$, $p\in(0,\infty)$, $q\in(0,\infty]$, and $W\in A_{p,\infty}$.
The \emph{inhomogeneous matrix-weighted Besov-type sequence space} $b^{s,\tau}_{p,q}(W)$
and the \emph{inhomogeneous matrix-weighted Triebel--Lizorkin-type sequence space} $f^{s,\tau}_{p,q}(W)$
are defined to be the sets of all sequences
$\vec t:=\{\vec t_Q\}_{Q\in\mathscr{Q}_+}\subset\mathbb{C}^m$ such that
$$
\left\|\vec t\right\|_{a^{s,\tau}_{p,q}(W)}
:=\left\|\left\{2^{js}\left|W^{\frac{1}{p}}\vec t_j\right|\right\}_{j\in\mathbb Z_+}\right\|_{LA_{p,q}^\tau}
<\infty,
$$
where $\vec t_j$ for any $j\in\mathbb Z_+$
and $\|\cdot\|_{LA_{p,q}^\tau}$ are the same as, respectively,
in \eqref{vec tj} and \eqref{LApq in}.
\end{definition}

Now, we establish the relations between
$a^{s,\tau}_{p,q}(W)$ and $a^{s,\tau}_{p,q}(\mathbb{A})$.

\begin{theorem}\label{37}
Let $s\in\mathbb{R}$, $\tau\in[0,\infty)$,
$p\in(0,\infty)$, and $q\in(0,\infty]$.
Let $W\in A_{p,\infty}$ and $\mathbb{A}:=\{A_Q\}_{Q\in\mathscr{Q}_+}$ be a sequence of
reducing operators of order $p$ for $W$.
Then $a^{s,\tau}_{p,q}(W)=a^{s,\tau}_{p,q}(\mathbb{A})$.
Moreover, for any $\vec t:=\{\vec t_Q\}_{Q\in\mathscr{Q}_+}\subset\mathbb{C}^m$,
$
\|\vec{t}\|_{a^{s,\tau}_{p,q}(W)}
\sim\|\vec{t}\|_{a^{s,\tau}_{p,q}(\mathbb{A})}
$,
where the positive equivalence constants are independent of $\vec t$.
\end{theorem}

\begin{proof}
Using \eqref{equ_reduce}, we obtain, for any
$\vec t:=\{\vec t_Q\}_{Q\in\mathscr{Q}_+}\subset\mathbb{C}^m$,
$
\|\vec{t}\|_{b^{s,\tau}_{p,q}(W)}
\sim\|\vec{t}\|_{b^{s,\tau}_{p,q}(\mathbb{A})}
$.
From Corollary \ref{46x} in the key estimate $\lesssim$ below, it follows that
\begin{align*}
\left\|\vec{t}\right\|_{f^{s,\tau}_{p,q}(W)}
&=\sup_{P\in\mathscr Q} \frac1{|P|^\tau}
\left\|\,\left[\sum_{Q\subset P,\,j_Q\geq 0}
\left(2^{j_Q s}\left|W^{\frac{1}{p}}\vec t_Q\right|\right)^q
\widetilde{\mathbf 1}_Q \right]^{\frac1q}\right\|_{L^p} \\
&\leq \sup_{P\in\mathscr Q} \frac1{|P|^\tau}
\left\|\,\left[\sum_{Q\subset P,\,j_Q\geq 0}
\left\|W^{\frac{1}{p}}A_Q^{-1}\right\|^q
\left(2^{j_Qs}\left|A_Q\vec t_Q\right|\right)^q \widetilde{\mathbf 1}_Q \right]^{\frac1q}\right\|_{L^p} \\
&\lesssim\sup_{P\in\mathscr Q} \frac1{|P|^\tau}
\left\|\,\left[\sum_{Q\subset P,\,j_Q\geq 0}
\left(2^{j_Qs}\left|A_Q\vec t_Q\right|\right)^q
\widetilde{\mathbf 1}_Q \right]^{\frac1q}\right\|_{L^p}
=\left\|\vec{t}\right\|_{f^{s,\tau}_{p,q}(\mathbb A)}.
\end{align*}

It remains to show that, for any
$\vec t:=\{\vec t_Q\}_{Q\in\mathscr{Q}_+}\subset\mathbb{C}^m$,
\begin{equation}\label{43}
\left\|\vec{t}\right\|_{f^{s,\tau}_{p,q}(\mathbb{A})}
\lesssim\left\|\vec{t}\right\|_{f^{s,\tau}_{p,q}(W)}.
\end{equation}
For each $Q\in\mathscr{Q}_+$, let
$
E_Q:=\{y\in Q:\ \|A_QW^{-\frac{1}{p}}(y)\|
\leq (C[W]_{A_{p,\infty}})^2\},
$
where $C$ is as in Lemma \ref{8 prepare}.
The said lemma proves that $|E_Q|\geq\frac12|Q|$.
This, combined with both Lemma \ref{66} with $\delta$ replaced by $\frac12$
and the definition of $E_Q$, further implies that
\begin{align*}
\left\|\vec{t}\right\|_{f^{s,\tau}_{p,q}(\mathbb{A})}
&=\left\|\left\{\left|A_Q\vec t_Q\right|
\right\}_{Q\in\mathscr{Q}_+}\right\|_{f^{s,\tau}_{p,q}}
\sim\left\|\left\{2^{js}\sum_{Q\in\mathscr{Q}_j}
\left|A_Q\vec t_Q\right|\widetilde{\mathbf{1}}_{E_Q}\right\}_{j\in\mathbb Z_+}
\right\|_{LF_{p,q}^\tau}\\
&\lesssim\left\|\left\{2^{js}\sum_{Q\in\mathscr{Q}_j}
\left|W^{\frac{1}{p}}\vec t_Q\right|\widetilde{\mathbf{1}}_{E_Q}\right\}_{j\in\mathbb Z_+}
\right\|_{LF_{p,q}^\tau}
\leq\left\|\left\{2^{js}\left|W^{\frac{1}{p}}\vec t_j\right|\right\}_{j\in\mathbb Z_+}
\right\|_{LF_{p,q}^\tau}
=\left\|\vec{t}\right\|_{f^{s,\tau}_{p,q}(W)},
\end{align*}
where $\vec t_j$ for any $j\in\mathbb Z_+$ is the same as in \eqref{vec tj}.
This finishes the proof of \eqref{43} and hence Theorem \ref{37}.
\end{proof}

We recall the homogeneous matrix-weighted Besov-type and Triebel--Lizorkin-type sequence spaces, making a slight extension of \cite[Definition 3.24]{bhyy}, where only $W\in A_p$ were considered.

\begin{definition}
Let $s\in\mathbb{R}$, $\tau\in[0,\infty)$, $p\in(0,\infty)$, $q\in(0,\infty]$, and $W\in A_{p,\infty}$.
The \emph{homogeneous matrix-weighted Besov-type sequence space} $\dot b^{s,\tau}_{p,q}(W)$
and the \emph{homogeneous matrix-weighted Triebel--Lizorkin-type sequence space} $\dot f^{s,\tau}_{p,q}(W)$
are defined to be the sets of all sequences
$\vec t:=\{\vec t_Q\}_{Q\in\mathscr{Q}}\subset\mathbb{C}^m$ such that
$$
\left\|\vec t\right\|_{\dot a^{s,\tau}_{p,q}(W)}
:=\left\|\left\{2^{js}\left|W^{\frac{1}{p}}\vec t_j\right|\right\}_{j\in\mathbb Z}\right\|_{LA_{p,q}^\tau}
<\infty,
$$
where $\vec t_j$ and $\|\cdot\|_{LA_{p,q}^\tau}$ are the same as, respectively,
in \eqref{vec tj} and \eqref{LApq}.
\end{definition}

Next, we recall homogeneous averaging matrix-weighted
Besov-type and Triebel--Lizorkin-type sequence spaces
(see \cite[Definition 3.26]{bhyy}).

\begin{definition}
Let $s\in\mathbb{R}$, $\tau\in[0,\infty)$, $p\in(0,\infty)$, and $q\in(0,\infty]$.
Let $W\in A_{p,\infty}$ and $\mathbb{A}:=\{A_Q\}_{Q\in\mathscr{Q}}$ be a sequence of
reducing operators of order $p$ for $W$.
The \emph{homogeneous averaging matrix-weighted Besov-type sequence space} $\dot b^{s,\tau}_{p,q}(\mathbb{A})$
and the \emph{homogeneous averaging matrix-weighted Triebel--Lizorkin-type sequence space} $\dot f^{s,\tau}_{p,q}(\mathbb{A})$
are defined to be the sets of all sequences
$\vec t:=\{\vec t_Q\}_{Q\in\mathscr{Q}}\subset\mathbb{C}^m$ such that
$$
\left\|\vec t\right\|_{\dot a^{s,\tau}_{p,q}(\mathbb{A})}
:=\left\|\left\{2^{js}\left|A_j\vec t_j\right|\right\}_{j\in\mathbb Z}\right\|_{LA_{p,q}^\tau}
<\infty,
$$
where $A_j$ and $\vec t_j$ for any $j\in\mathbb Z$
and where $\|\cdot\|_{LA_{p,q}^\tau}$
are the same as, respectively, in \eqref{Aj}, \eqref{vec tj}, and \eqref{LApq}.
\end{definition}

Repeating the proof of Theorem \ref{37} with Lemma \ref{66}
replaced by \cite[Proposition 2.4]{yyz13},
we obtain the following homogeneous version of Theorem \ref{37}.

\begin{theorem}\label{37 ho}
Let $s\in\mathbb{R}$, $\tau\in[0,\infty)$,
$p\in(0,\infty)$, and $q\in(0,\infty]$.
Let $W\in A_{p,\infty}$ and $\mathbb{A}:=\{A_Q\}_{Q\in\mathscr{Q}}$ be a sequence of
reducing operators of order $p$ for $W$.
Then $\dot a^{s,\tau}_{p,q}(W)=\dot a^{s,\tau}_{p,q}(\mathbb{A})$.
Moreover, for any $\vec t:=\{\vec t_Q\}_{Q\in\mathscr{Q}}\subset\mathbb{C}^m$,
$
\|\vec{t}\|_{\dot a^{s,\tau}_{p,q}(W)}
\sim\|\vec{t}\|_{\dot a^{s,\tau}_{p,q}(\mathbb{A})}
$,
where the positive equivalence constants are independent of $\vec t$.
\end{theorem}

\begin{remark}
Volberg obtained a result similar to Theorem \ref{37} 
(see \cite[``Remarks" on pp.\,454--455 and ``Remark and question" on p.\,460]{v97})
in the special case where $\dot a^{s,\tau}_{p,q}=\dot f^0_{p,2}$
(which means $a=f$, $s=0=\tau$, and $q=2$), i.e.,
\begin{align}\label{Volberg}
\left\|\vec{t}\right\|_{\dot f^0_{p,2}(W)}
\sim\left\|\vec{t}\right\|_{\dot f^0_{p,2}(\mathbb{A})}.
\end{align}
Volberg observed that this is true for any $p\in(1,\infty)$ and $W\in A_p$,
but also for $p=2$ and any matrix weight $W$.
He also stated that $W\in A_{p,\infty}$ is a sufficient condition
for \eqref{Volberg} for any $p\in[2,\infty)$. Then, in 
\cite[``Remarks2)" on pp.\,454--455]{v97}), Volberg asked
what about $p<2$. Our Theorem \ref{37 ho} even in this special case
where $\dot a^{s,\tau}_{p,q}=\dot f^0_{p,2}$ recovers this result 
and extends it to all $p\in(0,\infty)$ and hence completely answers 
this question of Volberg.

Volberg in \cite[``Remark and question" on p.\,460]{v97} 
also raised the question about a precise criterion on $W$ and $\mathbb{A}$
for \eqref{Volberg} to hold. This problem remains open.
\end{remark}

Finally, applying Theorems \ref{69}, \ref{37}, and \ref{phi A}
and Propositions \ref{38} through \ref{256},
we obtain the following conclusions; we omit the details.

\begin{theorem}\label{phi W}
Let $s\in\mathbb{R}$, $\tau\in[0,\infty)$, $p\in(0,\infty)$, $q\in(0,\infty]$, and $W\in A_{p,\infty}$.
Let $\Phi,\Psi\in\mathcal{S}$ satisfy \eqref{19} and
$\varphi,\psi\in\mathcal{S}$ satisfy \eqref{20}.
Then the operators
$S_\varphi:\ A^{s,\tau}_{p,q}(W,\widetilde{\Phi},\widetilde{\varphi})
\to a^{s,\tau}_{p,q}(W)$ and $T_\psi:\ a^{s,\tau}_{p,q}(W)
\to A^{s,\tau}_{p,q}(W,\Phi,\varphi)$
are bounded. Furthermore, if $\Phi$, $\Psi$, $\varphi$, and $\psi$ satisfy \eqref{21},
then $T_\psi\circ S_\varphi$ is the identity on
$A^{s,\tau}_{p,q}(W,\widetilde\Phi,\widetilde\varphi)$.
\end{theorem}

\begin{proposition}\label{independent W}
Let $s\in\mathbb{R}$, $\tau\in[0,\infty)$, $p\in(0,\infty)$, $q\in(0,\infty]$, and $W\in A_{p,\infty}$.
Let $\Phi\in\mathcal{S}$ satisfy \eqref{19} and
$\varphi\in\mathcal{S}$ satisfy \eqref{20}.
Then $A^{s,\tau}_{p,q}(W,\Phi,\varphi)$
is independent of the choice of $\Phi$ and $\varphi$.
\end{proposition}

Based on Proposition \ref{independent W}, in what follows,
we denote $A^{s,\tau}_{p,q}(W,\Phi,\varphi)$
simply by $A^{s,\tau}_{p,q}(W)$.

\begin{proposition}
Let $s\in\mathbb{R}$, $\tau\in[0,\infty)$,
$p\in(0,\infty)$, $q\in(0,\infty]$, and $W\in A_{p,\infty}$.
Then $A^{s,\tau}_{p,q}(W)\subset(\mathcal{S}')^m$.
Moreover, if $M\in\mathbb{Z}_+$ satisfies \eqref{239},
then there exists a positive constant $C$ such that,
for any $\vec f\in A^{s,\tau}_{p,q}(W)$ and $\phi\in\mathcal{S}$,
$
|\langle\vec f,\phi\rangle|
\leq C\|\vec f\|_{A^{s,\tau}_{p,q}(W)}\|\phi\|_{S_{M+1}},
$
where $\|\cdot\|_{S_{M+1}}$ is the same as in \eqref{SM}.
\end{proposition}

\begin{proposition}
Let $s\in\mathbb{R}$, $\tau\in[0,\infty)$,
$p\in(0,\infty)$, $q\in(0,\infty]$, and $W\in A_{p,\infty}$.
Then $A^{s,\tau}_{p,q}(W)$ is a complete quasi-normed space.
\end{proposition}


\subsection{Reproducing Formula with Generic Sampling}\label{reproducing}

By this point, we have already seen several applications of
the sampling formula \eqref{10x1}, which allows one to reproduce
a band-limited function like $\varphi_j*\vec f$ from the uniform sample of
its values at the rectangular lattice of points $x_R+y$,
where $R\in\mathscr Q_j$ and $y\in \mathbb R^n$ is arbitrary but
independent of $R$ and $j$. It is sometimes convenient to be able to pick
the sampling points entirely arbitrarily for each cube $R$.
Such Calder\'on-type reproducing formulae seem to go back to
\cite[Theorem 2.35]{han00}. (This work appears to be rather older than
its publication date, judging e.g. by the fact that a main result of \cite{han00}
is already quoted as \cite[Theorem 2.1]{han98} in the earlier-published paper
\cite{han98} of the same author.) Formulae of this type are frequently exploited,
e.g., in the development of the theory of function spaces
in various multi-parameter settings; see \cite[Theorem 2.5]{hllw},
\cite[Theorem 1.8]{hl08}, and \cite[Theorem 1.3]{ruan},
for a sample of extensions.

Motivated by these developments, we also provide a version of the reproducing formula with generic sampling that is suitable for the present considerations. In particular, in order to be compatible with the multiplication by a (matrix) weight, we require pointwise identities rather than, say, distributional convergence. We will illustrate the usefulness of the reproducing formula by giving another proof of Theorem \ref{69}.

We begin with introducing a scale of auxiliary function spaces convenient for the present purposes. These have some resemblance with the test function spaces introduced in \cite[Definition 1.14]{han94} and used in \cite{han00}, but they are not the same. Our polynomial growth condition is, in a sense, ``dual'' to the polynomial decay in \cite[Definition 1.14]{han94}, whereas the additional smoothness and cancellation conditions of  \cite[Definition 1.14]{han94} are not present in our treatment.

\begin{definition}
Let $j\in\mathbb{Z}$ and $L\in[0,\infty)$.
The \emph{polynomial growth at infinity space $G_{j,L}(\mathbb{R}^n)$} is defined by setting
$
G_{j,L}(\mathbb{R}^n):=\{f:\ \mathbb{R}^n\to\mathbb C:\
\|f\|_{G_{j,L}(\mathbb{R}^n)}<\infty\},
$
where, for any complex-valued function $f$ on $\mathbb{R}^n$,
\begin{align*}
\|f\|_{G_{j,L}(\mathbb{R}^n)}
:=\sup_{x\in\mathbb R^n}\left(1+2^j|x|\right)^{-L}|f(x)|.
\end{align*}
\end{definition}

In what follows, we denote $G_{j,L}(\mathbb{R}^n)$ simply by $G_{j,L}$.
It is easy to prove that $\{G_{j,L}\}_{j\in\mathbb Z,L\in[0,\infty)}$ are Banach spaces.
Let $\vec f\in (\mathcal S')^m$ and $\varphi\in\mathcal S$ satisfy \eqref{20}.
By the Paley--Wiener theorem, we find that, for any $j\in\mathbb Z_+$,
there exists $L\in[0,\infty)$ such that $\varphi_j*\vec f\in(G_{j,L})^m$.
The following lemma shows that, if we further assume
$\vec f\in A^{s,\tau}_{p,q}(W,\Phi,\varphi)$,
then we can find a uniform $L$ that is independent of~$j$.

\begin{lemma}\label{fjBd}
Let $s\in\mathbb R$, $\tau\in[0,\infty)$, $p\in(0,\infty)$, $q\in(0,\infty]$, and $W\in A_{p,\infty}$.
Let $d_1\in[\![d_{p,\infty}^{\mathrm{lower}}(W),n)$ and
$d_2\in[\![d_{p,\infty}^{\mathrm{upper}}(W),\infty)$.
Let $\Phi\in\mathcal S$ satisfy \eqref{19} and $\varphi\in\mathcal S$ satisfy \eqref{20}.
Then, for any $\vec f\in A^{s,\tau}_{p,q}(W,\Phi,\varphi)$ and $j\in\mathbb Z_+$,
$\varphi_j*\vec f\in(G_{j,L})^m$,
where $\varphi_0$ is replaced by $\Phi$ and $L=\frac{d_1+d_2}p$.
\end{lemma}

\begin{proof}
Let $\vec f\in A^{s,\tau}_{p,q}(W,\Phi,\varphi)$ and $j\in\mathbb Z_+$ be fixed
and let $\vec f_j:=\varphi_j*\vec f$.
By \eqref{19} and \eqref{20}, we find that
$
\operatorname{supp}\widehat{\vec f_j}
\subset\{\xi\in\mathbb R^n:\ |\xi|\leq 2^{j+1}\}.
$
Then Lemma \ref{10x} guarantees that, for any $x,y\in\mathbb R^n$,
\begin{equation}\label{Shannon}
\vec f_j(x)=\sum_{R\in\mathscr Q_j}2^{-jn}\gamma_j(x-x_R-y)\vec f_j(x_R+y)
\end{equation}
pointwise, where $\gamma$ is the same as in Lemma \ref{10x}.
Let $r\in(0,1)$ and $M\in(0,\infty)$, where $M$ is determined later.
Thus, for any $x,y\in\mathbb R^n$,
\begin{align*}
\left|\vec f_j(x)\right|^r
&\leq\sum_{R\in\mathscr Q_j}\left|2^{-jn}\gamma_j(x-x_R-y)\right|^r \left|\vec f_j(x_R+y)\right|^r \\
&\lesssim\sum_{R\in\mathscr Q_j}\left(1+2^j|x-x_R-y|\right)^{-Mr} \left|\vec f_j(x_R+y)\right|^r.
\end{align*}
Integrating both sides over $y\in Q_{j,\mathbf{0}}$
and noticing that $1+2^j|x-x_R-y|\sim 1+2^j|x-x_R|$ for such $y$, we obtain
\begin{align*}
\left|\vec f_j(x)\right|^r
&\lesssim\sum_{R\in\mathscr Q_j}\left(1+2^j|x-x_R|\right)^{-Mr}
\fint_R\left|\vec f_j(z)\right|^r\, dz \\
&\leq\left\|A_{Q_{j,\mathbf{0}}}^{-1}\right\|^r
\sum_{R\in\mathscr Q_j}\left(1+2^j|x-x_R|\right)^{-Mr}
\left\|A_{Q_{j,\mathbf{0}}}A_R^{-1}\right\|^r
\fint_R\left\|A_R W^{-\frac1p}(z)\right\|^r
\left|W^{\frac1p}(z)\vec f_j(z)\right|^r\, dz.
\end{align*}
Let $u\in(0,1)$ be the same as in Proposition \ref{compare}(\ref{WAsL}).
Observe that, if \eqref{WAs} holds for some $u\in(0,\infty)$,
it clearly also holds for any smaller $u$.
Thus, without loss of generality, we may always assume $u\in(0,1)$.
Let then $r\in(0,\infty)$ be chosen such that $\frac 1r=\frac 1u+\frac 1p$;
then $\frac 1r>\frac 1u$ and hence $r<u<1$ as required.
We then apply H\"older's inequality to the above integral to find that
\begin{align*}
&\fint_R\left\|A_R W^{-\frac1p}(z)\right\|^r\left|W^{\frac1p}(z)\vec f_j(z)\right|^r\, dz\\
&\quad\leq\left[\fint_R\left\|A_R W^{-\frac1p}(z)\right\|^u\,dz\right]^{\frac ru}
\left[\fint_R\left|W^{\frac1p}(z)\vec f_j(z)\right|^p\,dz\right]^{\frac rp}
=:\mathrm{I}\times \mathrm{II}.
\end{align*}
Here
\begin{equation*}
\mathrm{I}\lesssim 1
\quad\text{and}\quad
\mathrm{II}\leq\left[\left\|\vec f\right\|_{A^{s,\tau}_{p,q}(W,\Phi,\varphi)}
|R|^{\tau-\frac1p+\frac sn}\right]^r
\end{equation*}
by Lemma \ref{WAs} and the definition of the norm on the right-hand side, respectively.
Since $|R|=2^{-jn}$, we deduce that, for any $x\in\mathbb R^n$,
\begin{equation*}
\left|\vec f_j(x)\right|^r
\lesssim \left\|\vec f\right\|_{A^{s,\tau}_{p,q}(W,\Phi,\varphi)}^r
\sum_{R\in\mathscr Q_j}\left(1+2^j|x-x_R|\right)^{-Mr}
\left\|A_{Q_{j,\mathbf{0}}}A_R^{-1}\right\|^r.
\end{equation*}
Using Lemma \ref{sharp}(ii), we obtain
\begin{equation*}
\left\|A_{Q_{j,\mathbf{0}}}A_R^{-1}\right\|
\lesssim\left(1+2^j|x_R|\right)^L
\lesssim\left(1+2^j|x-x_R|\right)^L\left(1+2^j|x|\right)^L.
\end{equation*}
If $M>L+\frac nr$, then,
from Lemma \ref{253x}(ii), we infer that, for any $x\in\mathbb R^n$,
\begin{align*}
\left|\vec f_j(x)\right|^r
\lesssim \left\|\vec f\right\|_{A^{s,\tau}_{p,q}(W,\Phi,\varphi)}^r
\sum_{R\in\mathscr Q_j}\left(1+2^j|x-x_R|\right)^{(L-M)r}
\left(1+2^j|x|\right)^{Lr}
\lesssim \left\|\vec f\right\|_{A^{s,\tau}_{p,q}(W,\Phi,\varphi)}^r
\left(1+2^j|x|\right)^{Lr}
\end{align*}
and hence
\begin{align*}
\left(1+2^j|x|\right)^{-L}\left|\vec f_j(x)\right|
\lesssim\left\|\vec f\right\|_{A^{s,\tau}_{p,q}(W,\Phi,\varphi)},
\end{align*}
which further implies that $\vec f_j\in(G_{j,L})^m$.
This finishes the proof of Lemma \ref{fjBd}.
\end{proof}

When $m=1$ and $W\equiv1$, $A^{s,\tau}_{p,q}(W,\Phi,\varphi)$
reduces to $A^{s,\tau}_{p,q}$ and we can let $d_1=d_2=0$.
In this case, Lemma \ref{fjBd} give us the following interesting conclusion;
we omit the details.

\begin{corollary}
Let $s\in\mathbb R$, $\tau\in[0,\infty)$, $p\in(0,\infty)$, and $q\in(0,\infty]$.
Let $\Phi\in\mathcal S$ satisfy \eqref{19} and $\varphi\in\mathcal S$ satisfy \eqref{20}.
Then, for any $f\in A^{s,\tau}_{p,q}$ and $j\in\mathbb Z_+$,
$\varphi_j*f$ is a bounded function,
where $\varphi_0$ is replaced by $\Phi$.
\end{corollary}

The following is our version of the Calder\'on-type reproducing formulae with generic sampling points $y_R$.

\begin{proposition}\label{reproL}
Let $s\in\mathbb R$, $\tau\in[0,\infty)$, $p\in(0,\infty)$, $q\in(0,\infty]$, and $W\in A_{p,\infty}$.
Let $d_1\in[\![d_{p,\infty}^{\mathrm{lower}}(W),n)$ and
$d_2\in[\![d_{p,\infty}^{\mathrm{upper}}(W),\infty)$.
Let $\Phi\in\mathcal S$ satisfy \eqref{19} and $\varphi\in\mathcal S$ satisfy \eqref{20}.
Then, for any $M\in(n+\frac{d_1+d_2}p,\infty)$,
there exists $N\in\mathbb N$
such that, for any $j\in\mathbb Z_+$
and for an arbitrary choice of points $y_R$ in each of the respective cubes $R\in\mathscr Q_{j+N}$,
there exist functions $\psi_{j,R}$ indexed by these same cubes such that,
for any $\vec f\in A^{s,\tau}_{p,q}(W,\Phi,\varphi)$ and $x\in\mathbb R^n$,
$$
\varphi_j*\vec f(x)=\sum_{R\in\mathscr Q_{j+N}}|R|\,\psi_{j,R}(x)\,\varphi_j*\vec f(y_R)
$$
pointwise, where $\varphi_0$ is replaced by $\Phi$
and each $\psi_{j,R}$ satisfies that, for any $x\in\mathbb R^n$,
\begin{equation*}
|R|\left|\psi_{j,R}(x)\right|\leq C\left(1+2^j|x-x_R|\right)^{-M}.
\end{equation*}
Here $C$ is a positive constant independent of $j$, $R$, and the points $\{y_R\}_{R\in\mathscr Q_{j+N}}$.
\end{proposition}

\begin{proof}
Let $M\in(n+\frac{d_1+d_2}p,\infty)$, $\vec f\in A^{s,\tau}_{p,q}(W,\Phi,\varphi)$,
and $j\in\mathbb Z_+$ be fixed,
and let $\vec f_j:=\varphi_j*\vec f$.
In analogy with \eqref{Shannon}, we also have the continuous version,
for any $x\in\mathbb R^n$,
\begin{align}\label{continue}
\vec f_j(x)=\gamma_j*\vec f_j(x),
\end{align}
where $\gamma$ is the same as in Lemma \ref{10x}.
Indeed, this follows by integrating \eqref{Shannon} over $y\in Q_{j,\mathbf{0}}$,
but it can also be verified more directly by taking the Fourier transform of both sides.
Let $N\in\mathbb N$ be determined later.
Then, from \eqref{continue}, we deduce that, for any $x\in\mathbb R^n$,
\begin{align}\label{repro1}
\vec f_j(x)
&=\int_{\mathbb R^n}\gamma_j(x-y)\vec f_j(y)\,dy
=\sum_{R\in\mathscr Q_{j+N}}\int_R\gamma_j(x-y)\vec f_j(y)\,dy\notag\\
&=\sum_{R\in\mathscr Q_{j+N}}\int_R\gamma_j(x-y)\,dy\,\vec f_j(y_R)
+\sum_{R\in\mathscr Q_{j+N}}
\int_R\gamma_j(x-y)\left[\vec f_j(y)-\vec f_j(y_R)\right]\,dy\notag\\
&=:\sum_{R\in\mathscr Q_{j+N}}\int_R\gamma_j(x-y)\,dy\,\vec f_j(y_R)
+J(x),
\end{align}
where each $y_R$ is an arbitrary point in the respective cube $R$.
By \eqref{continue} again, we conclude that, for any $x\in\mathbb R^n$,
\begin{align*}
J(x)
&=\sum_{R\in\mathscr Q_{j+N}}\int_R\gamma_j(x-y)\int_{\mathbb R^n}
[\gamma_j(y-z)-\gamma_j(y_R-z)]\vec f_j(z)\,dz\,dy \\
&=\int_{\mathbb R^n}\sum_{R\in\mathscr Q_{j+N}}
\int_R\gamma_j(x-y)[\gamma_j(y-z)-\gamma_j(y_R-z)]\,dy \vec f_j(z)\,dz \\
&=:\int_{\mathbb R^n}\rho_j(x,z)\vec f_j(z)\,dz
=:\mathcal{R}_j\left(\vec f_j\right)(x).
\end{align*}
Thus, for any $x\in\mathbb R^n$, we can rewrite \eqref{repro1} as
\begin{equation}\label{repro2}
\vec f_j(x)=\sum_{R\in\mathscr Q_{j+N}}\int_R\gamma_j(x-y)\,dy\,\vec f_j(y_R)
+\mathcal R_j\left(\vec f_j\right)(x).
\end{equation}

Next, we focus on the properties of $\mathcal R_j(\vec f_j)(x)$.
To this end, we need to study its kernel $\rho_j(x,z)$. Notice that,
for some $\widetilde{y}$ on the line segment connecting $y$ and $y_R$, we have
\begin{align*}
\left|\gamma_j(y-z)-\gamma_j(y_R-z)\right|
&=\left|(y-y_R)\cdot\nabla\gamma_j(\widetilde{y}-z)\right|
\lesssim |y-y_R| 2^{j(n+1)}\left(1+2^j|y'-z|\right)^{-M} \\
&\lesssim2^{-j-N} 2^{j(n+1)}\left(1+2^j|y-z|\right)^{-M}
= 2^{-N} 2^{jn}\left(1+2^j|y-z|\right)^{-M}.
\end{align*}
From this and Lemma \ref{lucas}, we infer that,
for any $x,z\in\mathbb R^n$,
\begin{align}\label{7.15}
|\rho_j(x,z)|
&\lesssim \sum_{R\in\mathscr Q_j}\int_R 2^{jn}\left(1+2^j|x-y|\right)^{-M}
2^{-N}2^{jn}\left(1+2^j|y-z|\right)^{-M}\,dy\notag\\
&=2^{-N}\int_{\mathbb R^n}2^{jn}\left(1+2^j|x-y|\right)^{-M}
2^{jn}\left(1+2^j|y-z|\right)^{-M}\,dy\notag\\
&\lesssim 2^{-N} 2^{jn}\left(1+2^j|x-z|\right)^{-M}.
\end{align}
Let $L:=\frac{d_1+d_2}p$.
We now prove $\mathcal R_j:\ G_{j,L}\to G_{j,L}$ is a bounded operator.
By the definition of $\|g\|_{G_{j,L}}$, we conclude that,
for any $g\in G_{j,L}$ and $x\in\mathbb R^n$,
\begin{align*}
\left|\mathcal R_j(g)(x)\right|
&\lesssim\int_{\mathbb R^n}\left|\rho_j(x,z)\right| |g(z)|\,dz
\lesssim \int_{\mathbb R^n} 2^{-N} 2^{jn}\left(1+2^j|x-z|\right)^{-M}
\left(1+2^j|z|\right)^L\|g\|_{G_{j,L}}\,dz \\
&\lesssim 2^{-N}\|g\|_{G_{j,L}}\int_{\mathbb R^n}
2^{jn}\left(1+2^j|x-z|\right)^{-M+L}\,dz\,\left(1+2^j|x|\right)^L.
\end{align*}
From $M>n+L$ and Lemma \ref{253x}(i),
we deduce that the integral converges, and hence
\begin{equation*}
\|\mathcal R_j g\|_{G_{j,L}}\lesssim 2^{-N}\|g\|_{G_{j,L}},
\end{equation*}
where the implicit positive constant may depends on $L$, $M$, and $n$ but,
importantly, they are independent of $j$
(as well as on the choice of the points $y_R$,
but this is perhaps fairly obvious,
because these points already disappeared from the computation a while ago).
Thus, we may choose $N$, depending only on these parameters,
so that, say, $\|\mathcal R_j g\|_{G_{j,L}}\leq \|g\|_{G_{j,L}}/2$.
This, together with the fact that $G_{j,L}$ is a Banach space, further implies that
$I-\mathcal R_j$ is invertible with inverse $\sum_{k=0}^\infty \mathcal R_j^k$.
From this and \eqref{repro2}, it follows that, for any $x\in\mathbb R^n$,
\begin{align*}
\vec f_j(x)
&=(I-\mathcal R_j)^{-1}\left(\sum_{R\in\mathscr Q_{j+N}}\int_R\gamma_j(\cdot-y)\,dy\right)(x)\vec f_j(y_R) \\
&=\sum_{R\in\mathscr Q_{j+N}}|R|\fint_R \sum_{k=0}^\infty\mathcal R_j^k \left(\gamma_j(\cdot-y)\right)(x)\,dy\vec f_j(y_R)
=:\sum_{R\in\mathscr Q_{j+N}}|R|\psi_{j,R}(x)\vec f_j(y_R).
\end{align*}

It remains to estimate the function $\psi_{j,R}(x)$ above.
By \eqref{7.15} and Lemma \ref{lucas},
we find that there exists a positive constant $C$ such that,
for any $x,y\in\mathbb R^n$,
\begin{align*}
\left|\mathcal R_j\left(\gamma_j(\cdot-y)\right)(x)\right|
\leq\int_{\mathbb R^n}\left|\rho_j(x,z)\right|\left|\gamma_j(z-y)\right|\,dz
\leq C2^{jn}\left(1+2^j|x-z|\right)^{-M}.
\end{align*}
We show by induction that, for any $x,y\in\mathbb R^n$,
\begin{equation*}
\left|\mathcal R_j^k \left(\gamma_j(\cdot-y)\right)(x)\right|
\leq C\left(C^2 2^{-N}\right)^k 2^{jn}\left(1+2^j|x-y|\right)^{-M}.
\end{equation*}
This is clear for $k=0$ and, if it holds for some $k\in\mathbb Z_+$,
then, from Lemma \ref{lucas} again, we infer that, for any $x,y\in\mathbb R^n$,
\begin{align*}
\left|\mathcal R_j^{k+1}\left(\gamma_j(\cdot-y)\right)(x)\right|
&\leq\int_{\mathbb R^n}C2^{-N}2^{jn}\left(1+2^j|x-z|\right)^{-M}
\left|\mathcal R_j^{k}\left(\gamma_j(\cdot-y)\right)(z)\right|\,dz \\
&\leq\int_{\mathbb R^n}C2^{-N}2^{jn}\left(1+2^j|x-z|\right)^{-M}
C\left(C^2 2^{-N}\right)^k 2^{jn}\left(1+2^j|y-z|\right)^{-M}\,dz \\
&\leq C\left(C^2 2^{-N}\right)^{k+1} 2^{jn}\left(1+2^j|x-y|\right)^{-M},
\end{align*}
which completes the induction. Thus, if $N$ is sufficiently large,
then, we obtain, for any $y\in R\in\mathscr Q_{j+N}$,
\begin{align*}
\sum_{k=0}^\infty \left|\mathcal R_j^k\left(\gamma_j(\cdot-y)\right)(x)\right|
&\lesssim\sum_{k=0}^\infty 2^{-k} 2^{jn}\left(1+2^j|x-y|\right)^{-M}
\sim 2^{jn}\left(1+2^j|x-x_R|\right)^{-M}.
\end{align*}
Hence, $|\psi_{j,R}(x)|$ has this same upper bound and finally
\begin{equation*}
|R|\left|\psi_{j,R}(x)\right|\lesssim\left(1+2^j|x-x_R|\right)^{-M}.
\end{equation*}
This finishes the proof of Proposition \ref{reproL}.
\end{proof}

Next, we can give another proof of Theorem \ref{69}.

\begin{proof}[Second Proof of Theorem \ref{69}]
For the proof of the estimate
\begin{align*}
\left\|\vec f\right\|_{A^{s,\tau}_{p,q}(W,\Phi,\varphi)}
\lesssim\left\|\vec f\right\|_{A^{s,\tau}_{p,q}(\mathbb{A},\Phi,\varphi)}
\end{align*}
we have nothing new to add here, and the reader is referred to the first proof given right after the statement of Theorem \ref{69}. We now provide a new proof of the reverse inequality based on Proposition \ref{reproL}.

Let $d_1\in[\![d_{p,\infty}^{\mathrm{lower}}(W),n)$ and
$d_2\in[\![d_{p,\infty}^{\mathrm{upper}}(W),\infty)$.
Let $r\in(0,1)$, $L:=\frac{d_1+d_2}p$, and $M\in(\frac nr+L,\infty)$.
By Proposition \ref{reproL}, we find that there exists $N\in\mathbb{N}$ such that,
for any $j\in\mathbb Z$
and for arbitrary points $y_R$ in the respective cubes $R\in\mathscr Q_{j+N}$,
there are functions $\psi_{j,R}$ such that, for any $x\in\mathbb R^n$,
\begin{align*}
\varphi_j*\vec f(x)=\sum_{R\in\mathscr Q_{j+N}}
|R|\psi_{j,R}(x)\varphi_j*\vec f(y_R)
\end{align*}
pointwise, where
\begin{equation*}
|R|\left|\psi_{j,R}(x)\right|
\leq C\left(1+2^j|x-x_R|\right)^{-M}.
\end{equation*}
Here $C$ is a positive constant independent of $j$ and $\{y_R\}_{R\in\mathscr Q_{j+N}}$.
Then, for any $x\in Q\in\mathscr Q_j$,
\begin{align*}
\left|A_Q\varphi_j*\vec f(x)\right|^r
&\leq\sum_{R\in\mathscr{Q}_{j+N}}|R|^r\left|\psi_{j,R}(x)\right|^r
\left|A_Q\varphi_j*\vec f(y_R)\right|^r\\
&\lesssim\sum_{R\in\mathscr{Q}_{j+N}}\left(1+2^j|x-x_R|\right)^{-Mr}
\left|A_Q\varphi_j*\vec f(y_R)\right|^r\\
&\lesssim\sum_{R\in\mathscr{Q}_{j+N}}
\left(1+2^j|x-x_R|\right)^{-Mr}\left\|A_QA_R^{-1}\right\|^r
\left|A_R\varphi_j*\vec f(y_R)\right|^r.
\end{align*}
From Lemma \ref{sharp}(i), it follows that,
for any $Q\in\mathscr{Q}_j$ and $R\in\mathscr{Q}_{j+N}$,
\begin{align*}
\left\|A_QA_R^{-1}\right\|
\lesssim2^{Nd_2}\left(1+2^j|x_Q-x_R|\right)^L
\sim\left(1+2^j|x-x_R|\right)^L.
\end{align*}
Thus, for any $r\in(0,1]$ and $x\in Q\in\mathscr Q_j$,
\begin{align}\label{integral}
\left|A_Q\varphi_j*\vec f(x)\right|^r
\lesssim\sum_{R\in\mathscr{Q}_{j+N}}\left(1+2^j|x-x_R|\right)^{(L-M)r}
\left|A_R\varphi_j*\vec f(y_R)\right|^r.
\end{align}
For any $R\in\mathscr{Q}_j$, let
$
E_R:=\{y\in R:\
\|A_RW^{-\frac1p}(y)\|^p<(C[W]_{A_{p,\infty}})^2\},
$
where $C$ is the same as in \eqref{eq 8 var} of Lemma \ref{8 prepare}.
Then, from Lemma \ref{8 prepare}, we deduce that,
for any $R\in\mathscr{Q}_+$, $|E_R|>\frac12|R|$.
Therefore, by integrating \eqref{integral} over $y_R\in E_R$ for each $R\in\mathscr{Q}_j$,
we find that, for any $r\in(0,1]$ and $x\in Q\in\mathscr Q_j$,
\begin{align*}
\left|A_Q\varphi_j*\vec f(x)\right|^r
&\lesssim\sum_{R\in\mathscr{Q}_{j+N}}\left(1+2^j|x-x_R|\right)^{(L-M)r}
\fint_{E_R}\left|A_R\varphi_j*\vec f(y)\right|^r\,dy\\
&\lesssim\sum_{R\in\mathscr{Q}_{j+N}}2^{jn}
\int_{E_R}\left(1+2^j|x-y|\right)^{(L-M)r}
\left|W^{\frac1p}(y)\varphi_j*\vec f(y)\right|^r\,dy\\
&\leq2^{jn}\int_{\mathbb R^n}\left(1+2^j|x-y|\right)^{(L-M)r}
\left|W^{\frac1p}(y)\varphi_j*\vec f(y)\right|^r\,dy,
\end{align*}
where $(L-M)r<-n$.

By \cite[Lemma 3.13]{bhyy}, if two sequences of functions
\begin{equation*}
g_j(x):=\sum_{Q\in\mathscr Q_j}\left|A_Q\varphi_j*\vec f(x)\right|,\quad
h_j(y):=\left|W^{\frac1p}(y)\varphi_j*\vec f(y)\right|,\
\forall\,j\in\mathbb Z_+
\end{equation*}
satisfy the pointwise bound above, i.e.,
for any $j\in\mathbb Z_+$ and $x\in\mathbb R^n$,
\begin{equation*}
|g_j(x)|^r
\lesssim 2^{jn} \int_{\mathbb R^n}\left(1+2^j|x-y|\right)^{-K}|h_j(y)|^r\,dy
\end{equation*}
with $K>n$ and $r\in(0,\min\{p,q\})$, then
\begin{equation*}
\left\|\vec f\right\|_{A^{s,\tau}_{p,q}(\mathbb A,\Phi,\varphi)}
=\left\|\left\{2^{js}g_j\right\}_{j=0}^\infty\right\|_{LA^\tau_{p,q}}\lesssim \left\|\left\{2^{js}h_j\right\}_{j=0}^\infty\right\|_{LA^\tau_{p,q}}
=\left\|\vec f\right\|_{A^{s,\tau}_{p,q}(W,\Phi,\varphi)}.
\end{equation*}
Since we are free to choose $r$ as small as we like, we obtain this estimate, which completes the second proof of  Theorem \ref{69}.
\end{proof}

\subsection{Relation to Homogeneous Spaces}\label{homog space}

In this section, we show that the inhomogeneous Besov-type and Triebel--Lizorkin-type spaces considered in the previous sections can, for certain parameters, also be characterised in terms of their homogeneous counterparts from \cite{bhyy}, intersected with certain modified Lebesgue spaces. For this result, we will need to assume positive smoothness $s\in(0,\infty)$, integrability in the Banach range $p\in[1,\infty)$, as well as weights of class $A_p$ (instead of the more general $A_{p,\infty}$).

First, we recall the definition of the homogeneous Besov-type and Triebel--Lizorkin-type spaces, where we still allow the full parameter range as in, for instance, \cite[Definition 3.5]{bhyy}.

\begin{definition}
Let $s\in\mathbb{R}$, $\tau\in[0,\infty)$, $p\in(0,\infty)$, $q\in(0,\infty]$, and $W\in A_p$.
Let $\varphi\in\mathcal{S}$ satisfy \eqref{20}.
The \emph{homogeneous matrix-weighted Besov-type space} $\dot B^{s,\tau}_{p,q}(W)$
and the \emph{homogeneous matrix-weighted Triebel--Lizorkin-type space} $\dot F^{s,\tau}_{p,q}(W)$
are defined by setting
$$
\dot A^{s,\tau}_{p,q}(W)
:=\left\{\vec{f}\in(\mathcal{S}_\infty')^m:\
\left\|\vec{f}\right\|_{\dot A^{s,\tau}_{p,q}(W)}<\infty\right\},
$$
where, for any $\vec{f}\in(\mathcal{S}_\infty')^m$,
$$
\left\|\vec{f}\right\|_{\dot A^{s,\tau}_{p,q}(W)}
:=\left\|\left\{2^{js}\left|W^{\frac{1}{p}}\left(\varphi_j*\vec f\right)
\right|\right\}_{j\in\mathbb Z}\right\|_{LA_{p,q}^\tau}
$$
with $\|\cdot\|_{LA_{p,q}^\tau}$ the same as in \eqref{LApq}.
\end{definition}

\begin{definition}
Let $p\in(0,\infty)$, $\tau\in[0,\infty)$, $j\in\mathbb Z$, and $W$ be a matrix weight.
The \emph{matrix-weighted modified Lebesgue space} $L^p_\tau(W,\mathbb{R}^n)$
is defined to be the set of all measurable vector-valued functions
$\vec f:\ \mathbb{R}^n\to\mathbb{C}^m$ such that
$$
\left\|\vec{f}\right\|_{L^p_\tau(W,\mathbb{R}^n)}
:=\sup_{P\in\mathscr{Q},\,|P|\geq1}
\frac{1}{|P|^\tau}\left\|\vec f\right\|_{L^p(W,P)}<\infty.
$$
\end{definition}

In what follows, we denote $L^p_\tau(W,\mathbb{R}^n)$ simply by $L^p_\tau(W)$.
Now, we establish the relations between $\dot A_{p,q}^{s,\tau}(W)$ and $A_{p,q}^{s,\tau}(W)$.

\begin{theorem}\label{dot A and A}
Let $s\in(0,\infty)$, $\tau\in[0,\infty)$,
$p\in[1,\infty)$, $q\in(0,\infty]$, and $W\in A_p$.
\begin{enumerate}[{\rm(i)}]
\item For any $\vec f\in\dot A_{p,q}^{s,\tau}(W)\cap L^p_\tau(W)$,
one has $\vec f\in A_{p,q}^{s,\tau}(W)$ and
\begin{align*}
\left\|\vec f\right\|_{A_{p,q}^{s,\tau}(W)}
\leq C\left[\left\|\vec f\right\|_{L^p_\tau(W)}
+\left\|\vec f\right\|_{\dot A_{p,q}^{s,\tau}(W)}\right],
\end{align*}
where $C$ is a positive constant independent of $\vec f$.
\item For any $\vec f\in A_{p,q}^{s,\tau}(W)$,
there exists $\vec g\in L^p_\tau(W)\cap\dot A_{p,q}^{s,\tau}(W)$
such that $\vec f=\vec g$ in $(\mathcal{S}')^m$ and
\begin{align*}
\left\|\vec g\right\|_{L^p_\tau(W)}+\left\|\vec g\right\|_{\dot A_{p,q}^{s,\tau}(W)}
\leq C\left\|\vec f\right\|_{A_{p,q}^{s,\tau}(W)},
\end{align*}
where $C$ is a positive constant independent of $\vec f$.
\end{enumerate}
\end{theorem}

We point out that all the results in \cite{bhyy,bhyyp2,bhyyp3}
on homogeneous case can be generalized to $A_{p,\infty}$ matrix weights and,
even so, the present method used to prove Theorem \ref{dot A and A}
can not generalize this theorem to $A_{p,\infty}$ matrix weights.

To show Theorem \ref{dot A and A}, we need the following conclusion, which
[in particular, \eqref{key2} and \eqref{key3}] strongly depends on $W\in A_p$.

\begin{lemma}\label{2023.4.7}
Let $\tau\in[0,\infty)$, $p\in[1,\infty)$, $W\in A_p$, and $\varphi\in\mathcal{S}$.
Then there exists a positive constant $C$ such that,
for any $\vec f\in L^p_\tau(W)$, $P\in\mathscr{Q}$ with $|P|\geq1$,
and $j\in\mathbb Z$ with $j\geq j_P$,
\begin{align}\label{equ 47}
\frac{1}{|P|^\tau}\left\|\varphi_j*\vec f\right\|_{L^p(W,P)}
\leq C\left\|\vec f\right\|_{L^p_\tau(W)}.
\end{align}
\end{lemma}

\begin{proof}
Let $\vec f\in L^p_\tau(W)$, $P\in\mathscr{Q}$ satisfy $|P|\geq1$,
and $j\in\mathbb Z$ satisfy $j\geq j_P$.
To prove \eqref{equ 47}, we consider the following two cases on $p$.

\emph{Case 1)} $p=1$.
In this case, by Tonelli's theorem and $\varphi\in\mathcal S$, we find that
\begin{align}\label{key}
\left\|\varphi_j*\vec f\right\|_{L^1(W,P)}
&=\int_P\left|W(x)\varphi_j*\vec f(x)\right|\,dx
\leq\int_P\int_{\mathbb{R}^n}|\varphi_j(x-y)|\left|W(x)\vec f(y)\right|\,dy\,dx\notag\\
&=\int_{\mathbb{R}^n}\int_P|\varphi_j(x-y)|\left|W(x)\vec f(y)\right|\,dx\,dy\notag\\
&\leq\int_{\mathbb{R}^n}\left[\int_P|\varphi_j(x-y)|
\left\|W(x)W^{-1}(y)\right\|\,dx\right]
\left|W(y)\vec f(y)\right|\,dy\notag\\
&=\int_{3P}I_P(y)\left|W(y)\vec f(y)\right|\,dy
+\sum_{k\in\mathbb Z^n,\,\|k\|_\infty\geq1}\int_{P+k\ell(P)}\cdots,
\end{align}
where, for any $y\in\mathbb{R}^n$ and any cube $P\subset\mathbb R^n$,
\begin{align}\label{3.32x}
I_P(y):=\int_P|\varphi_j(x-y)|\left\|W(x)W^{-1}(y)\right\|\,dx.
\end{align}
Let $I_{\mathbb R^n}$ be defined as in
\eqref{3.32x} with $P$ replaced by $\mathbb R^n$.
Clearly $I_P\leq I_{\mathbb R^n}$,
and it has been verified in the proof of \cite[Lemma 4.4]{fr21}
(see in particular the last display of that proof) that
\begin{align}\label{key2}
\|I_{\mathbb R^n}\|_{L^\infty}\lesssim1.
\end{align}
Let $\mathbb{A}:=\{A_Q\}_{Q\in\mathscr{Q}}$ be a sequence of reducing operators of order $p=1$ for $W$.
Since
\begin{equation*}
|\varphi_j(x-y)|\lesssim \frac{2^{jn}}{(1+2^j|x-y|)^M},
\end{equation*}
where we take $M\in(2n,\infty)$,
from Lemmas \ref{33} and \ref{reduceM}, we infer that,
for any $k\in\mathbb Z^n$ with $\|k\|_\infty\geq1$ and for any $y\in P+k\ell(P)$,
\begin{align*}
I_P(y)
\sim\frac{2^{(j-j_P)(n-M)}}{|k|^M}\fint_P\left\|W(x)W^{-1}(y)\right\|\,dx
\lesssim\frac{1}{|k|^M}\left\|A_PW^{-1}(y)\right\|,
\end{align*}
which, together with \cite[Corollary 2.18 and Lemma 2.29]{bhyy}, further implies that
\begin{align*}
\|I_P\|_{L^\infty(P+k\ell(P))}
\lesssim\frac{1}{|k|^M}\mathop{\operatorname{ess\,sup}}_{y\in P+k\ell(P)}
\left\|A_PW^{-1}(y)\right\|
\sim\frac{1}{|k|^M}\left\|A_PA_{P+k\ell(P)}^{-1}\right\|
\lesssim\frac{1}{|k|^M}(1+|k|)^n
\sim\frac{1}{|k|^{M-n}}.
\end{align*}
By this, \eqref{key}, and \eqref{key2}, we conclude that
\begin{align*}
\left\|\varphi_j*\vec f\right\|_{L^1(W,P)}
&\lesssim\int_{3P}\left|W(y)\vec f(y)\right|\,dy
+\sum_{k\in\mathbb Z^n,\,\|k\|_\infty\geq1}\frac{1}{|k|^{M-n}}
\int_{P+k\ell(P)}\left|W(y)\vec f(y)\right|\,dy\\
&\lesssim|P|^\tau\left\|\vec f\right\|_{L^p_\tau(W)}
+\sum_{k\in\mathbb Z^n,\,\|k\|_\infty\geq1}\frac{1}{|k|^{M-n}}
|P|^\tau\left\|\vec f\right\|_{L^p_\tau(W)}
\sim|P|^\tau\left\|\vec f\right\|_{L^p_\tau(W)}.
\end{align*}
This finishes the proof of \eqref{equ 47} in this case.

\emph{Case 2)} $p\in(1,\infty)$. In this case,
let $\frac{1}{p}+\frac{1}{p'}=1$ and $M\in(2n,\infty)$.
From $\varphi\in\mathcal S$, it follows that, for any $x\in P$,
\begin{align*}
\left|W^{\frac{1}{p}}(x)\varphi_j*\vec f(x)\right|
&\leq\int_{\mathbb{R}^n}|\varphi_j(x-y)|\left|W^{\frac{1}{p}}(x)\vec f(y)\right|\,dy
\lesssim\int_{\mathbb{R}^n}\frac{2^{jn}}{(1+2^j|x-y|)^M}
\left|W^{\frac{1}{p}}(x)\vec f(y)\right|\,dy\\
&=\int_{3P}\frac{2^{jn}}{(1+2^j|x-y|)^M}
\left|W^{\frac{1}{p}}(x)\vec f(y)\right|\,dy
+\sum_{k\in\mathbb Z^n,\,\|k\|_\infty\geq1}\int_{P+k\ell(P)}\cdots\\
&=:J_1(x)+J_2(x).
\end{align*}
For any $x\in\mathbb R^n$, let
$$
\mathcal{M}_{W,p}\left(\vec{f}\right)(x)
:=\sup_{t\in(0,\infty)}\fint_{B(x,t)}
\left|W^{\frac{1}{p}}(x)W^{-\frac{1}{p}}(y)\vec{f}(y)\right|\,dy.
$$
Using the definition of $\mathcal{M}_{W,p}$, we obtain, for any $x\in P$,
\begin{align*}
J_1(x)
&=\int_{B(x,2^{-j})}\frac{2^{jn}}{(1+2^j|x-y|)^M}
\left|W^{\frac{1}{p}}(x)\vec f(y)\right|\mathbf{1}_{3P}(y)\,dy
+\sum_{l=1}^\infty\int_{B(x,2^{l-j})\setminus B(x,2^{l-j-1})}\cdots\\
&\lesssim\fint_{B(x,2^{-j})}
\left|W^{\frac{1}{p}}(x)W^{-\frac{1}{p}}(y)
W^{\frac{1}{p}}(y)\vec f(y)\mathbf{1}_{3P}(y)\right|\,dy
+\sum_{l=1}^\infty2^{(n-M)l}\fint_{B(x,2^{l-j})}\cdots\\
&\lesssim\sum_{l=0}^\infty2^{(n-M)l}
\mathcal{M}_{W,p}\left(W^{\frac{1}{p}}\vec f\mathbf{1}_{3P}\right)(x)
\sim \mathcal{M}_{W,p}\left(W^{\frac{1}{p}}\vec f\mathbf{1}_{3P}\right)(x).
\end{align*}
This, combined with \cite[Theorem 3.2]{g03}, further implies that
\begin{align}\label{key3}
\|J_1\|_{L^p(P)}
&\lesssim\left\|\mathcal{M}_{W,p}\left(W^{\frac{1}{p}}\vec f\mathbf{1}_{3P}\right)\right\|_{L^p(P)}
\leq\left\|\mathcal{M}_{W,p}\left(W^{\frac{1}{p}}\vec f\mathbf{1}_{3P}\right)\right\|_{L^p}\notag\\
&\lesssim\left\|\,\left|W^{\frac{1}{p}}\vec f\right|\,\right\|_{L^p(3P)}
\lesssim|P|^\tau\left\|\vec f\right\|_{L^p_\tau(W)}.
\end{align}
By Lemma \ref{33}, H\"older's inequality, and Lemma \ref{Ap dual},
we conclude that, for any $x\in P$,
\begin{align*}
J_2(x)
&\lesssim\sum_{k\in\mathbb Z^n,\,\|k\|_\infty\geq1}\frac{2^{(j-j_P)(n-M)}}{|k|^M}
\fint_{P+k\ell(P)}\left\|W^{\frac{1}{p}}(x)W^{-\frac{1}{p}}(y)\right\|
\left|W^{\frac{1}{p}}(y)\vec f(y)\right|\,dy\\
&\leq\sum_{k\in\mathbb Z^n,\,\|k\|_\infty\geq1}\frac{1}{|k|^M}
\left[\fint_{P+k\ell(P)}
\left\|W^{\frac{1}{p}}(x)W^{-\frac{1}{p}}(y)\right\|^{p'}\,dy\right]^{\frac{1}{p'}}
\left[\fint_{P+k\ell(P)}
\left|W^{\frac{1}{p}}(y)\vec f(y)\right|^p\,dy\right]^{\frac{1}{p}}\\
&\lesssim\sum_{k\in\mathbb Z^n,\,\|k\|_\infty\geq1}\frac{1}{|k|^M}
\left\|W^{\frac{1}{p}}(x)A_{P+k\ell(P)}^{-1}\right\|
|P|^{\tau-\frac{1}{p}}\left\|\vec f\right\|_{L^p_\tau(W)},
\end{align*}
which, together with Lemma \ref{reduceM} and \cite[Lemma 2.29]{bhyy}, further implies that
\begin{align*}
\|J_2\|_{L^p(P)}
&\lesssim\sum_{k\in\mathbb Z^n,\,\|k\|_\infty\geq1}\frac{1}{|k|^M}
|P|^{\tau}\left\|\vec f\right\|_{L^p_\tau(W)}
\left[\fint_P\left\|W^{\frac{1}{p}}(x)A_{P+k\ell(P)}^{-1}\right\|^p\,dx\right]^{\frac{1}{p}}\\
&\sim\sum_{k\in\mathbb Z^n,\,\|k\|_\infty\geq1}\frac{1}{|k|^M}
|P|^{\tau}\left\|\vec f\right\|_{L^p_\tau(W)}\left\|A_PA_{P+k\ell(P)}^{-1}\right\|
\lesssim\sum_{k\in\mathbb Z^n,\,\|k\|_\infty\geq1}\frac{1}{|k|^M}
|P|^{\tau}\left\|\vec f\right\|_{L^p_\tau(W)}\left(1+|k|\right)^n\\
&\sim\sum_{k\in\mathbb Z^n,\,\|k\|_\infty\geq1}\frac{1}{|k|^{M-n}}
|P|^{\tau}\left\|\vec f\right\|_{L^p_\tau(W)}
\sim|P|^{\tau}\left\|\vec f\right\|_{L^p_\tau(W)}.
\end{align*}
Therefore,
\begin{align*}
\frac{1}{|P|^\tau}\left\|\varphi_j*\vec f\right\|_{L^p(W,P)}
\lesssim\frac{1}{|P|^\tau}\left[\|J_1\|_{L^p(P)}+\|J_2\|_{L^p(P)}\right]
\lesssim\left\|\vec f\right\|_{L^p_\tau(W)}.
\end{align*}
This finishes the proof of \eqref{equ 47} in this case and hence Lemma \ref{2023.4.7}.
\end{proof}

Next, we can show Theorem \ref{dot A and A}.

\begin{proof}[Proof of Theorem \ref{dot A and A}]
By similarity, we only consider $F_{p,q}^{s,\tau}(W)$.
Let $\Phi\in\mathcal{S}$ satisfy \eqref{19} and
$\varphi\in\mathcal{S}$ satisfy \eqref{20}.

We first prove (i).
From the definition of $\|\cdot\|_{\dot F_{p,q}^{s,\tau}(W)}$, we deduce that,
for any $P\in\mathscr{Q}$ with $|P|<1$,
\begin{align*}
\frac{1}{|P|^\tau}\left\|\left\{2^{js}\left|W^{\frac{1}{p}}\varphi_j*\vec f\right|
\right\}_{j\geq j_P}\right\|_{LF_{pq}(\widehat P_+)}
\leq\left\|\vec f\right\|_{\dot F_{p,q}^{s,\tau}(W)}.
\end{align*}
Using Lemma \ref{2023.4.7} and the definition of $\|\cdot\|_{\dot F_{p,q}^{s,\tau}(W)}$ again,
we obtain, for any $P\in\mathscr{Q}$ with $|P|\geq1$,
\begin{align*}
&\frac{1}{|P|^\tau}\left\|\left\{2^{js}\left|W^{\frac{1}{p}}\varphi_j*\vec f\right|
\right\}_{j\in\mathbb Z_+}\right\|_{LF_{pq}(\widehat P_+)}\\
&\quad\sim\frac{1}{|P|^\tau}\left\|\Phi*\vec f\right\|_{L^p(W,P)}
+\frac{1}{|P|^\tau}\left\|\left\{2^{js}\left|W^{\frac{1}{p}}\varphi_j*\vec f\right|
\right\}_{j\geq1}\right\|_{LF_{pq}(\widehat P_+)}
\lesssim\left\|\vec f\right\|_{L^p_\tau(W)}
+\left\|\vec f\right\|_{\dot F_{p,q}^{s,\tau}(W)}.
\end{align*}
These finish the proof of (i).

Now, we show (ii). Based on Proposition \ref{independent W},
we further assume that $\Phi$ and $\varphi$ satisfy that, for any $\xi\in\mathbb{R}^n$,
$
\widehat{\Phi}(\xi)
+\sum_{j=1}^\infty\widehat{\varphi}_j(\xi)=1
$
and hence
$
\vec f=\Phi*\vec f+\sum_{j=1}^\infty\varphi_j*\vec f=:\vec g
$
in $(\mathcal{S}')^m$ (see, for instance, \cite[p.\,43]{fjw91}).
We first prove
\begin{align}\label{estimate1}
\left\|\vec g\right\|_{L^p_\tau(W)}\lesssim\left\|\vec f\right\|_{F_{p,q}^{s,\tau}(W)}.
\end{align}
By the definition of $\vec g$, we conclude that, for any $x\in\mathbb{R}^n$,
\begin{align*}
\left|W^{\frac{1}{p}}(x)\vec g(x)\right|
\leq\left|W^{\frac{1}{p}}(x)\Phi*\vec f(x)\right|
+\sum_{j=1}^\infty\left|W^{\frac{1}{p}}(x)\varphi_j*\vec f(x)\right|
=:\left|W^{\frac{1}{p}}(x)\Phi*\vec f(x)\right|+I(x).
\end{align*}
When $q\in(0,1]$, using the elementary inequality
\begin{align}\label{b1}
\left(\sum_i |t_i|\right)^\alpha\leq\sum_i |t_i|^\alpha
\end{align}
for any $\{t_i\}_i\subset\mathbb{C}$ and $\alpha\in(0,1]$,
we obtain, for any $x\in\mathbb{R}^n$,
\begin{align*}
I(x)
\leq\left[\sum_{j=1}^\infty\left|W^{\frac{1}{p}}(x)
\varphi_j*\vec f(x)\right|^q\right]^{\frac{1}{q}}
\leq\left[\sum_{j=1}^\infty2^{jsq}\left|W^{\frac{1}{p}}(x)\varphi_j*\vec f(x)\right|^q\right]^{\frac{1}{q}}
\end{align*}
or, when $q\in(1,\infty]$, from H\"older's inequality,
it follows that, for any $x\in\mathbb{R}^n$,
\begin{align*}
I(x)
&=\sum_{j=1}^\infty2^{-js}2^{js}\left|W^{\frac{1}{p}}(x)\varphi_j*\vec f(x)\right|\\
&\leq\left[\sum_{j=1}^\infty2^{-jsq'}\right]^{\frac{1}{q'}}
\left[\sum_{j=1}^\infty2^{jsq}\left|W^{\frac{1}{p}}(x)\varphi_j*\vec f(x)\right|^q\right]^{\frac{1}{q}}
\sim\left[\sum_{j=1}^\infty2^{jsq}\left|W^{\frac{1}{p}}(x)\varphi_j*\vec f(x)\right|^q\right]^{\frac{1}{q}}.
\end{align*}
Thus, for any $x\in\mathbb{R}^n$,
\begin{align*}
\left|W^{\frac{1}{p}}(x)\vec g(x)\right|
\lesssim\left|W^{\frac{1}{p}}(x)\Phi*\vec f(x)\right|
+\left[\sum_{j=1}^\infty2^{jsq}\left|W^{\frac{1}{p}}(x)\varphi_j*\vec f(x)\right|^q\right]^{\frac{1}{q}}
\end{align*}
and hence
\begin{align*}
\left\|\vec g\right\|_{L^p_\tau(W)}
\lesssim\sup_{P\in\mathscr{Q},\,|P|\geq1}\frac{1}{|P|^\tau}\left\|
\left[\left|W^{\frac{1}{p}}(x)\Phi*\vec f(x)\right|^q
+\sum_{j=1}^\infty2^{jsq}\left|W^{\frac{1}{p}}(x)\varphi_j*\vec f(x)\right|^q\right]^{\frac{1}{q}}\right\|_{L^p(P)}
\lesssim\left\|\vec f\right\|_{F_{p,q}^{s,\tau}(W)}.
\end{align*}
This finishes the proof of \eqref{estimate1}.

Next, we show that
\begin{align}\label{estimate2}
\left\|\vec g\right\|_{\dot F_{p,q}^{s,\tau}(W)}
\lesssim\left\|\vec f\right\|_{F_{p,q}^{s,\tau}(W)}.
\end{align}
By the definition of $\|\cdot\|_{F_{p,q}^{s,\tau}(W)}$, we find that,
for any $P\in\mathscr{Q}$ with $|P|<1$,
\begin{align}\label{estimate2 equ}
\frac{1}{|P|^\tau}\left\|\left\{2^{js}\left|W^{\frac{1}{p}}\varphi_j*\vec g
\right|\right\}_{j\in\mathbb Z}\right\|_{LF_{p,q}(\widehat{P})}
\leq\left\|\vec g\right\|_{F_{p,q}^{s,\tau}(W)}
=\left\|\vec f\right\|_{F_{p,q}^{s,\tau}(W)}
\end{align}
and, for any $P\in\mathscr{Q}$ with $|P|\geq1$,
\begin{align*}
&\frac{1}{|P|^\tau}\left\|\left\{2^{js}\left|W^{\frac{1}{p}}\varphi_j*\vec g
\right|\right\}_{j\in\mathbb Z}\right\|_{LF_{p,q}(\widehat{P})}\\
&\quad\sim\frac{1}{|P|^\tau}\left\|\,\left[\sum_{j=j_P}^0
2^{jsq}\left|W^{\frac{1}{p}}\varphi_j*\vec g
\right|^q\right]^{\frac{1}{q}}\right\|_{L^p(P)}
+\frac{1}{|P|^\tau}\left\|\,\left[\sum_{j=1}^\infty
2^{jsq}\left|W^{\frac{1}{p}}\varphi_j*\vec g
\right|^q\right]^{\frac{1}{q}}\right\|_{L^p(P)}\\
&\quad\leq\frac{1}{|P|^\tau}\left\|\,\left[\sum_{j=j_P}^0
2^{jsq}\left|W^{\frac{1}{p}}\varphi_j*\vec g
\right|^q\right]^{\frac{1}{q}}\right\|_{L^p(P)}
+\left\|\vec g\right\|_{F_{p,q}^{s,\tau}(W)}
=:J(P)+\left\|\vec g\right\|_{F_{p,q}^{s,\tau}(W)}.
\end{align*}
For any $P\in\mathscr{Q}$ with $|P|\geq1$, when $q\in[1,\infty]$,
from the elementary inequality \eqref{b1},
the triangle inequality, and Lemma \ref{2023.4.7}, we infer that
\begin{align*}
J(P)&\leq\frac{1}{|P|^\tau}\left\|\sum_{j=j_P}^0
2^{js}\left|W^{\frac{1}{p}}\varphi_j*\vec g
\right|\,\right\|_{L^p(P)}
\leq\frac{1}{|P|^\tau}\sum_{j=j_P}^02^{js}
\left\|\varphi_j*\vec g\right\|_{L^p(W,P)}\\
&\lesssim\sum_{j=j_P}^02^{js}\left\|\vec g\right\|_{L^p_\tau(W)}
\sim\left\|\vec g\right\|_{L^p_\tau(W)}
\end{align*}
or, when $q\in(0,1)$, using H\"older's inequality, we obtain
\begin{align*}
J(P)
&\leq\frac{1}{|P|^\tau}\left\|\,\left[\sum_{j=j_P}^0
\left(2^{j\frac{sq}{2}}\right)^{(\frac{1}{q})'}\right]^{\frac{q}{(\frac{1}{q})'}}
\sum_{j=j_P}^02^{j\frac{s}{2}}\left|W^{\frac{1}{p}}\varphi_j*\vec g
\right|\,\right\|_{L^p(P)}\\
&\sim\frac{1}{|P|^\tau}\left\|\sum_{j=j_P}^02^{j\frac{s}{2}}\left|W^{\frac{1}{p}}\varphi_j*\vec g
\right|\,\right\|_{L^p(P)}
\leq\frac{1}{|P|^\tau}\sum_{j=j_P}^02^{j\frac{s}{2}}
\left\|\varphi_j*\vec g\right\|_{L^p(W,P)}\\
&\lesssim\sum_{j=j_P}^02^{j\frac{s}{2}}\left\|\vec g\right\|_{L^p_\tau(W)}
\sim\left\|\vec g\right\|_{L^p_\tau(W)}.
\end{align*}
Thus, for any $P\in\mathscr{Q}$ with $|P|\geq1$,
\begin{align*}
\frac{1}{|P|^\tau}\left\|\left\{2^{js}\left|W^{\frac{1}{p}}\varphi_j*\vec g
\right|\right\}_{j\in\mathbb Z}\right\|_{LF_{p,q}(\widehat{P})}
\lesssim\left\|\vec g\right\|_{L^p_\tau(W)}
+\left\|\vec g\right\|_{F_{p,q}^{s,\tau}(W)}
\lesssim\left\|\vec f\right\|_{F_{p,q}^{s,\tau}(W)}.
\end{align*}
This, combined with \eqref{estimate2 equ}, further implies \eqref{estimate2},
which completes the proof of (ii) and hence Theorem \ref{dot A and A}.
\end{proof}

\section{Almost Diagonal Operators}\label{Almost Diagonal Operators}

In this section, we focus on the boundedness of almost diagonal operators on
matrix-weighted Besov-type and Triebel--Lizorkin-type sequence spaces.
The result on inhomogeneous spaces is an immediate corollary of
the corresponding result on homogeneous spaces,
so we first consider the latter.
The structure of this section is organized as follows.
In Subsection \ref{Homogeneous},
we recall the homogeneous matrix-weighted Besov-type and Triebel--Lizorkin-type
sequence spaces $\dot a^{s,\tau}_{p,q}(W)$
and then establish the boundedness of almost diagonal operators on these spaces.
The result is sharp for certain ranges of parameters.
Subsection \ref{sec:BH} is devoted to comparing the results of Subsection \ref{Homogeneous}
with related results obtained by Bownik and Ho \cite{bh06}.
In Subsection \ref{Inhomogeneous},
as an application of the boundedness of almost diagonal operators on $\dot a^{s,\tau}_{p,q}(W)$,
we obtain the boundedness of almost diagonal operators on $a^{s,\tau}_{p,q}(W)$.

\subsection{Homogeneous Sequence Spaces}
\label{Homogeneous}

We begin with some definitions. Let $B:=\{b_{Q, P}\}_{Q, P\in\mathscr{Q}}\subset\mathbb{C}$.
For any sequence $\vec t:=\{\vec t_R\}_{R\in\mathscr{Q}}\subset\mathbb{C}^m$,
we define $B\vec t:=\{(B\vec t)_Q\}_{Q\in\mathscr{Q}}$ by setting,
for any $Q\in\mathscr{Q}$,
$$
\left(B\vec t\right)_Q:=\sum_{R\in\mathscr{Q}}b_{Q,R}\vec t_R
$$
if the above summation is absolutely convergent.
Now, we recall the concept of almost diagonal operators. We formulate the definition following \cite[Definition 3.1]{bhyy2}, but the basic idea of this notion goes back to \cite[(3.1)]{fj90}.

\begin{definition}\label{def:AD}
Let $D,E,F\in\mathbb{R}$. The special infinite matrix
$B^{DEF}:=\{b_{Q,R}^{DEF}\}_{Q,R\in\mathscr{Q}}\subset\mathbb{C}$
is defined by setting, for any $Q,R\in\mathscr{Q}$,
\begin{equation}\label{bDEF}
b_{Q,R}^{DEF}
:=\left[1+\frac{|x_Q-x_R|}{\ell(Q)\vee\ell(R)}\right]^{-D}
\left\{\begin{aligned}
&\left[\frac{\ell(Q)}{\ell(R)}\right]^E&&\text{if }\ell(Q)\leq\ell(R),\\
&\left[\frac{\ell(R)}{\ell(Q)}\right]^F&&\text{if }\ell(R)<\ell(Q).
\end{aligned}\right.
\end{equation}
An infinite matrix $B:=\{b_{Q,R}\}_{Q,R\in\mathscr{Q}}\subset\mathbb{C}$
is said to be \emph{$(D,E,F)$-almost diagonal} if there exists a positive constant $C$ such that $|b_{Q,R}|\leq C b_{Q,R}^{DEF}$.
\end{definition}

\begin{remark}
\begin{enumerate}[\rm(i)]
\item If $E+F>0$ (which is always the case in all situations of interest to us),
then the second factor on the right-hand side of \eqref{bDEF} could be equivalently written as
\begin{equation*}
\min\left\{\left[\frac{\ell(Q)}{\ell(R)}\right]^E,\
\left[\frac{\ell(R)}{\ell(Q)}\right]^F\right\}.
\end{equation*}

\item It is immediate from Definition \ref{def:AD} that
the matrix $B^{DEF}$ itself is $(D,E,F)$-almost diagonal.
We use this matrix in several counterexamples that
demonstrate the sharpness of almost diagonality assumptions in some results below.
\end{enumerate}
\end{remark}

The following notion was introduced in \cite[Definition 4.20]{bhyy}.

\begin{definition}\label{d4.4}
A function or a sequence space of Besov-type or Triebel--Lizorkin-type,
with parameters $(s,\tau,p,q)$, is said to be
\begin{enumerate}[\rm(i)]
\item \emph{supercritical} if $\tau>\frac{1}{p}$ or $(\tau,q)=(\frac{1}{p},\infty)$,
\item \emph{critical} if $\tau=\frac1p$ and $q<\infty$ and the space is of Triebel--Lizorkin-type,
\item \emph{subcritical} if $\tau<\frac1p$, or if $\tau=\frac1p$ and $q<\infty$ and the space is not of Triebel--Lizorkin-type.
\end{enumerate}
\end{definition}

\begin{definition}
In the context of any version (including in particular the weighted ones)
of the Besov-type and Triebel--Lizorkin-type spaces
$\dot A^{s,\tau}_{p,q}$ (resp. $A^{s,\tau}_{p,q}$)
or corresponding sequence spaces
$\dot a^{s,\tau}_{p,q}$ (resp. $a^{s,\tau}_{p,q}$),
the \emph{symbol $J$} always carries the following meaning:
\begin{equation}\label{J}
J:=\left\{\begin{aligned}
&\frac{n}{\min\{1,p\}}&&\text{if a Besov-type space is dealt},\\
&\frac{n}{\min\{1,p,q\}}&&\text{if a Triebel--Lizorkin-type space is dealt}.
\end{aligned}\right.
\end{equation}
\end{definition}

The following theorem is the main result of this section.

\begin{theorem}\label{ad BF2}
Let $s\in\mathbb R$, $\tau\in[0,\infty)$, $p\in(0,\infty)$, $q\in(0,\infty]$, and $W\in A_{p,\infty}$.
Let
\begin{equation}\label{tauJ2}
\begin{aligned}
J_{\tau} &:=
\begin{cases}
n&\text{in the supercritical case},\\
\dfrac{n}{\min\{1,q\}}&\text{in the critical case},\\
J&\text{in the subcritical case},
\end{cases}\\
\widehat\tau&:=\left[\tau-\frac{1}{p}
+\frac{d_{p,\infty}^{\mathrm{lower}}(W)}{np}\right]_+,\\
\widetilde{J} &:=J_{\tau}+\left[\left(n\widehat\tau\right)
\wedge\frac{d_{p,\infty}^{\mathrm{lower}}(W)}{p}\right]
+\frac{d_{p,\infty}^{\mathrm{upper}}(W)}{p},\quad\text{and}\\
\widetilde{s}&:=s+n\widehat\tau.
\end{aligned}
\end{equation}
If $B$ is $(D,E,F)$-almost diagonal, then $B$ is bounded on $\dot a^{s,\tau}_{p,q}(W)$
whenever
\begin{align}\label{ad new2}
D>\widetilde{J},\quad
E>\frac{n}{2}+\widetilde{s},\quad\text{ and }\quad
F &>\widetilde{J}-\frac{n}{2}-\widetilde{s}\notag \\
&=J_\tau+\frac{d_{p,\infty}^{\mathrm{upper}}(W)}{p}-\frac{n}{2}-s-n\left(\tau-\frac1p\right)_+.
\end{align}
\end{theorem}

\begin{remark}
Theorem \ref{ad BF2} should be compared with \cite[Theorem 9.1]{bhyyp2}, which represents the state of art under the stronger assumption of $A_p$ weights; see \cite[Remarks 7.1, 8.4, and 9.2]{bhyyp2} for the relation of \cite[Theorem 9.1]{bhyyp2} to other existing results of this type.

The key novelty, of course, is that we now deal with the larger weight class
$A_{p,\infty}\supsetneqq A_p$. The definition of $J_\tau$ in \eqref{tauJ2}
is the same as the quantity with the same symbol in \cite{bhyyp2}.
In the definition of $\widehat\tau$, we now use the
$A_{p,\infty}$-lower dimension $d_{p,\infty}^{\mathrm{lower}}(W)$
from Definition \ref{AinftyDim}, in place of
the $A_p$-dimension \cite[Definition 2.5]{bhyyp2}.
In analogy with \eqref{ApLower}, let us define
\begin{equation*}
d_p(W):=\inf\{ d\in [0,n): W\text{ has $A_p$-dimension }d\}.
\end{equation*}
One can then show that, for all matrix weights $W$,
\begin{equation*}
d_{p,\infty}^{\mathrm{lower}}(W)\leq d_p(W),
\end{equation*}
Hence the assumptions of Theorem \ref{ad BF2} featuring $d_{p,\infty}^{\mathrm{lower}}(W)$ are at most as restrictive as the corresponding conditions of \cite[Theorem 9.1]{bhyyp2}. (Note that, in the condition on $F$, the contributions of $d_{p,\infty}^{\mathrm{lower}}(W)$ actually cancel out.)

However, there is still a price to pay for considering the larger class of $A_{p,\infty}$ weights, and this is the presence of the term $d_{p,\infty}^{\mathrm{lower}}(W)/p$ in $\widetilde J$, and hence in the restrictions on the parameters $D$ and $F$; there is no counterpart of this term in \cite[Theorem 9.1]{bhyyp2}.

In the case of $W\in A_p$, one can prove the estimate
\begin{equation}\label{upperDim<}
d_{p,\infty}^{\mathrm{upper}}(W) \leq (p-1)_+ d_{p'}\left(W^{-\frac{1}{p-1}}\right),
\quad\text{thus}\quad
\frac{d_{p,\infty}^{\mathrm{upper}}(W)}{p}\leq \frac{d_{p'}(W^{-\frac{1}{p-1}})}{p'},
\end{equation}
featuring the $A_{p'}$-dimension $d_{p'}(W^{-\frac{1}{p-1}})$ of the dual weight (defined to be $d_{p'}(W^{-\frac{1}{p-1}}):=0$ for $W\in A_p$ with $p\in(0,1]$; since the factor $(p-1)_+$ vanishes in this case anyway, this definition is essentially a formality; cf. \cite[Definition 2.30]{bhyy}).

While this gives a control of the new term $d_{p,\infty}^{\mathrm{upper}}(W)/p$, this term can still be positive, and hence Theorem \ref{ad BF2} fails to fully recover \cite[Theorem 9.1]{bhyyp2}. Thus, if $p\in(1,\infty)$, then \cite[Theorem 9.1]{bhyyp2} provides sharper conclusions for $W\in A_p$, and one should only use the new Theorem \ref{ad BF2} for weights $W\in A_{p,\infty}\setminus A_p$. On the other hand, if $p\in(0,1]$ and $W\in A_p$, then \eqref{upperDim<} shows that $d_{p,\infty}^{\mathrm{upper}}(W)=0$, and the conclusions of Theorem \ref{ad BF2} are at least as strong as those of \cite[Theorem 9.1]{bhyyp2}. Actually, we will show the necessity of this term for $p\in(0,1]$, even in cases when it is genuinely present; see Lemmas \ref{ad Besov sharp} (Besov case) and \ref{ad Triebel sharp} (Triebel--Lizorkin case).

On the technical level, the source of the new extra term $d_{p,\infty}^{\mathrm{upper}}(W)/p$ can be traced to Lemmas \ref{ad B1} (Besov case) and \ref{ad F1} (Triebel--Lizorkin case), where the obtained upper bounds exhibit growth as a function of certain parameters, in contrast to their analogues with uniform estimates in \cite[Lemmas 5.1 and 6.5]{bhyyp2}.
\end{remark}

To show Theorem \ref{ad BF2}, we need several technical lemmas.
The following two lemmas, where the precise conditions on the matrix weight play no role, can be simply quoted from \cite[Lemma 4.3 and 4.4]{bhyyp2}; our work here then starts with the estimation of the right-hand side of \eqref{ABt} under the $A_{p,\infty}$ assumptions of the present paper. (Note that Lemma \ref{s=0 enough} is formulated in \cite[Lemmas 4.3]{bhyyp2} under the assumption that $W\in A_p$; however, an inspection of the short proof shows that this condition plays no role in the argument.)

\begin{lemma}\label{s=0 enough}
Let $s\in\mathbb{R}$, $\tau\in[0,\infty)$,
$p\in(0,\infty)$, $q\in(0,\infty]$, and $W$ be a matrix weight.
Let $D,E,F\in\mathbb R$.
Let $B:=\{b_{Q,R}\}_{Q,R\in\mathscr{Q}}\subset\mathbb{C}$
and $\widetilde B:=\{\widetilde b_{Q,R}\}_{Q,R\in\mathscr{Q}}$,
where $\widetilde{b}_{Q,R}:=[\ell(R)/\ell(Q)]^sb_{Q,R}$
for any $Q,R\in\mathscr{Q}$. Then
\begin{enumerate}[\rm(i)]
\item$B$ is $(D,E,F)$-almost diagonal if and only if
$\widetilde B$ is $(D,E-s,F+s)$-almost diagonal;
\item$B$ is bounded on $\dot a^{s,\tau}_{p,q}(W)$ if and only if
$\widetilde B$ is bounded on $\dot a^{0,\tau}_{p,q}(W)$.
\end{enumerate}
\end{lemma}

\begin{lemma}\label{ad prelim}
Let $p\in(0,\infty)$, $q\in(0,\infty]$,
and $B$ be $(D,E,F)$-almost diagonal for some $D,E,F\in\mathbb R$.
Then there exists a positive constant $C$ such that,
for any $\vec{t}=\{\vec{t}_Q\}_{Q\in\mathscr Q}\subset\mathbb C^m$
satisfying that $B\vec t$ is well defined, any $a\in(0,1]$,
and any sequence $\{H_j:\ \mathbb{R}^n\to M_m(\mathbb{C})\}_{j\in\mathbb{Z}}$
of measurable matrix-valued functions,
\begin{align}\label{ABt}
\left\|\left\{H_j\left(B\vec{t}\right)_j\right\}_{j\in\mathbb Z}\right\|_{LA_{p,q}}^r
&\leq C\sum_{k\in\mathbb{Z}}\sum_{l=0}^\infty
\Bigg[2^{-(E-\frac{n}{2})k_-}2^{-k_+(F+\frac{n}{2}-\frac{n}{a})}2^{-(D-\frac{n}{a})l}\notag\\
&\quad\times\left.\left\|\left\{
\left[\fint_{B(\cdot,2^{l+k_+-i})}\left| H_{i-k}(\cdot)\vec{t}_{i}(y)\right|^a
\,dy\right]^{\frac{1}{a}}\right\}_{i\in\mathbb Z}\right\|_{LA_{p,q}}\right]^r,
\end{align}
where $r:=p\wedge q\wedge1$,
$\vec{t}_i$ and $(B\vec{t})_j$ are the same as in \eqref{vec tj},
and the norms are taken in either of
$$
LA_{p,q}\in\{L^p\ell^q,\ell^qL^p\}.
$$
\end{lemma}

\begin{remark}\label{ad conv}
By \cite[Remark 4.5]{bhyyp2}, we obtain a slightly stronger estimate,
with the left-hand side of \eqref{ABt} replaced by
\begin{equation*}
\left\|\left\{\sum_{Q\in\mathscr{Q}_j}\widetilde{\mathbf{1}}_Q(\cdot)
\sum_{R\in\mathscr{Q}}\left|H_j(\cdot)b_{Q,R}\vec{t}_R\right|
\right\}_{j\in\mathbb Z}\right\|_{LA_{p,q}}^r.
\end{equation*}
In particular, when $H_j=W^{\frac{1}{p}}$ for any $j\in\mathbb Z$ and $W\in A_{p,\infty}$,
from \cite[Remark 4.5]{bhyyp2} again, we deduce that,
if we have the finiteness of the right-hand side of \eqref{ABt},
then the series defining $(B\vec{t})_Q$ is absolutely convergent,
in other word, for any $Q\in\mathscr{Q}$,
$\sum_{R\in\mathscr{Q}}|b_{Q,R}\vec{t}_R|<\infty$.
\end{remark}

Applying an argument similar to that used in the proof of Theorem \ref{37},
we obtain the following conclusion.

\begin{theorem}\label{37x}
Let $s\in\mathbb{R}$, $\tau\in[0,\infty)$,
$p\in(0,\infty)$, and $q\in(0,\infty]$.
Let $W\in A_{p,\infty}$ and $\mathbb{A}:=\{A_Q\}_{Q\in\mathscr{Q}}$ be a sequence of
reducing operators of order $p$ for $W$.
Then $\dot a^{s,\tau}_{p,q}(W)=\dot a^{s,\tau}_{p,q}(\mathbb{A})$.
Moreover, for any $\vec t:=\{\vec t_Q\}_{Q\in\mathscr{Q}}\subset\mathbb{C}^m$,
$
\|\vec{t}\|_{\dot a^{s,\tau}_{p,q}(W)}
\sim\|\vec{t}\|_{\dot a^{s,\tau}_{p,q}(\mathbb{A})}
$,
where the positive equivalence constants are independent of $\vec t$.
\end{theorem}

Now, we estimate the quantities on the right-hand side of
the conclusion of Lemma \ref{ad prelim}
separately for Besov and Triebel--Lizorkin spaces,
beginning with the easier case.

\begin{lemma}\label{ad B1}
Let $p\in(0,\infty)$, $q\in(0,\infty]$, $W\in A_{p,\infty}$
have $A_{p,\infty}$-upper dimension $d_2\in[0,\infty)$, and $a:=p\wedge1$.
Then there exists a positive constant $C$ such that,
for any $k\in\mathbb Z$, $l\in\mathbb Z_+$, and $\vec{t}\in\dot b^0_{p,q}(W)$,
$$
\left\|\left\{\left[\fint_{B(\cdot,2^{l+k_+-i})}
\left|W^{\frac{1}{p}}(\cdot)\vec{t}_i(y)\right|^a\,dy\right]^{\frac{1}{a}}
\right\}_{i\in\mathbb Z}\right\|_{\ell^qL^p}
\leq C2^{(l+k_+)\frac{d_2}{p}}\left\|\vec{t}\right\|_{\dot b^0_{p,q}(W)},
$$
where $\vec t_i$ for any $i\in\mathbb Z$ is the same as in \eqref{vec tj}.
\end{lemma}

\begin{proof}
Let $k,i\in\mathbb Z$, $l\in\mathbb Z_+$, and $\vec{t}\in\dot b^0_{p,q}(W)$ be fixed.
Let $j:=i-l-k_+$.
Then, by H\"older's inequality,
the fact that $B(x,2^{-j})\subset 3Q$ for any $x\in Q\in\mathscr{Q}_j$,
Tonelli's theorem, and \eqref{equ_reduce}, we find that
\begin{align}\label{161}
&\left\|\,\left[\fint_{B(\cdot,2^{-j})}
\left|W^{\frac{1}{p}}(\cdot)\vec{t}_i(y)\right|^a\,dy\right]^{\frac{1}{a}}\right\|_{L^p}^p\notag\\
&\quad\leq\sum_{Q\in\mathscr{Q}_j}
\int_Q\fint_{B(x,2^{-j})}
\left|W^{\frac{1}{p}}(x)\vec{t}_i(y)\right|^p\,dy\,dx
\lesssim\sum_{Q\in\mathscr{Q}_j}
\int_Q\fint_{3Q}\left|W^{\frac{1}{p}}(x)\vec{t}_i(y)\right|^p\,dy\,dx\notag\\
&\quad\leq\sum_{Q\in\mathscr{Q}_j}
\int_{3Q}\fint_{3Q}\left|W^{\frac{1}{p}}(x)\vec{t}_i(y)\right|^p\,dx\,dy
\sim\sum_{Q\in\mathscr{Q}_j}
\int_{3Q}\left|A_{3Q}\vec{t}_i(y)\right|^p\,dy.
\end{align}
From Lemma \ref{sharp}(iii), it follows that,
for any $Q\in\mathscr{Q}_j$,
\begin{align*}
\int_{3Q}\left|A_{3Q}\vec{t}_i(y)\right|^p\,dy
&\lesssim\sum_{R\in\mathscr{Q}_i,\,R\subset3Q}\left\|A_{3Q}A_R^{-1}\right\|^p
\int_R\left|A_R\vec{t}_i(y)\right|^p\,dy\\
&\lesssim\sum_{R\in\mathscr{Q}_i,\,R\subset3Q}2^{(l+k_+)d_2}
\int_R\left|A_R\vec{t}_i(y)\right|^p\,dy
\sim2^{(l+k_+)d_2}\int_{3Q}\left|A_i(y)\vec{t}_i(y)\right|^p\,dy.
\end{align*}
This, together with \eqref{161}, further implies that
\begin{align*}
\left\|\,\left[\fint_{B(\cdot,2^{-j})}
\left|W^{\frac{1}{p}}(\cdot)\vec{t}_i(y)\right|^a\,dy\right]^{\frac{1}{a}}\right\|_{L^p}^p
\leq2^{(l+k_+)d_2}\sum_{Q\in\mathscr{Q}_j}
\int_{3Q}\left|A_i(y)\vec{t}_i(y)\right|^p\,dy
\sim2^{(l+k_+)d_2}\left\|A_i\vec{t}_i\right\|_{L^p}^p
\end{align*}
using the bounded overlap of the expanded cubes $3Q$ in the last step,
and hence
\begin{align*}
\left\|\left\{\left[\fint_{B(\cdot,2^{l+k_+-i})}
\left|W^{\frac{1}{p}}(\cdot)\vec{t}_i(y)\right|^a\,dy\right]^{\frac{1}{a}}
\right\}_{i\in\mathbb Z}\right\|_{\ell^qL^p}
\lesssim2^{(l+k_+)\frac{d_2}{p}}\left\|\vec{t}\right\|_{\dot b^0_{p,q}(\mathbb{A})}
\sim2^{(l+k_+)\frac{d_2}{p}}\left\|\vec{t}\right\|_{\dot b^0_{p,q}(W)},
\end{align*}
where in the last step we used Theorem \ref{37x}.
This finishes the proof of Lemma \ref{ad B1}.
\end{proof}

\begin{theorem}\label{ad Besov}
Let $s\in\mathbb R$, $p\in(0,\infty)$, $q\in(0,\infty]$, and $W\in A_{p,\infty}$.
Let $B$ be $(D,E,F)$-almost diagonal with parameters
\begin{align*}
D>J+\frac{d_{p,\infty}^{\mathrm{upper}}(W)}{p},\quad
E>\frac{n}{2}+s,\quad\text{and}\quad
F>J-\frac{n}{2}-s+\frac{d_{p,\infty}^{\mathrm{upper}}(W)}{p},
\end{align*}
where $J:=\frac{n}{\min\{1,p\}}$.
Then $B$ is bounded on $\dot b^s_{p,q}(W)$.
\end{theorem}

\begin{proof}
By the definition of $d_{p,\infty}^{\mathrm{upper}}(W)$, we find that
there exists $d_2\in[d_{p,\infty}^{\mathrm{upper}}(W),\infty)$ such that
$W$ has $A_{p,\infty}$-upper dimension $d_2$ and
\begin{equation*}
D>J+\frac{d_2}{p},\quad
E>\frac{n}{2}+s,\quad\text{and}\quad
F>J-\frac{n}{2}-s+\frac{d_2}{p}.
\end{equation*}
Applying this and an argument similar to that used in the proof of \cite[Theorem 5.2]{bhyyp2}
with \cite[Lemma 5.1]{bhyyp2} replaced by Lemma \ref{ad B1},
we obtain $B$ is bounded on $\dot b^s_{p,q}(W)$.
This finishes the proof of Theorem \ref{ad Besov}.
\end{proof}

Compared with the sharp almost diagonal conditions of the unweighted spaces
$\dot b^s_{p,q}$ (\cite[Theorem 5.6]{bhyyp2}),
our conditions have an extra non-negative term $d_{p,\infty}^{\mathrm{upper}}(W)/p$.
This term is \emph{sharp} when $p\in(0,1]$ and hence irremovable.

\begin{lemma}\label{ad Besov sharp}
Let $s\in\mathbb R$, $p\in(0,\infty)$, $q\in(0,\infty]$,
$d_2\in[0,\infty)$, and $D,E,F\in\mathbb{R}$.
Suppose that every $(D,E,F)$-almost diagonal matrix $B$ is bounded
on $\dot b^s_{p,q}(W)$ whenever $W\in A_{p,\infty}$
satisfies $d_{p,\infty}^{\mathrm{upper}}(W)=d_2$.
Then $D>\frac{n}{p}+\frac{d_2}{p}$
and $F>\frac{n}{p}-\frac{n}{2}-s+\frac{d_2}{p}$.
\end{lemma}

\begin{proof}
We consider $W(x):=|x|^{d_2}I_m$ for any $x\in\mathbb R^n$.
By \cite[Lemmas 7.2(ii) and 7.5]{bhyy2}, we find that
$W\in A_{p,\infty}$ satisfies $d_{p,\infty}^{\mathrm{upper}}(W)=d_2$.
For a scalar weight, it is easy to see that an instance of the reducing operator of order $p$ is given by $A_Q=\langle W\rangle_Q^{\frac1p}$.
An elementary computation (or \cite[Corollary 2.42]{bhyy}) shows that
\begin{equation*}
\langle W\rangle_Q:=\fint_Q|x|^{d_2}\,dx
\sim [|c_Q|+\ell(Q)]^{d_2}
\sim 2^{-jd_2}(1+|k|)^{d_2}
\end{equation*}
for any $Q:=Q_{j,k}\in\mathscr Q$. Hence, without loss of generality, we may take $A_Q:=2^{-j\frac{d_2}p}(1+|k|)^{\frac{d_2}p}I_m$ for the  sequence of reducing operators of order $p$ for $W$.
Moreover, we consider the $(D,E,F)$-almost diagonal matrix
$B^{DEF}:=\{b_{Q,R}^{DEF}\}_{Q,R\in\mathscr{Q}}$ from \eqref{bDEF}.
By Theorem \ref{37x}, we obtain $B$ is bounded on $\dot b^s_{p,q}(\mathbb A)$.

Next, we prove
\begin{align}\label{estimate D}
D>\frac{n}{p}+\frac{d_2}{p}.
\end{align}
Let $\vec e\in\mathbb{C}^m$ satisfy $|\vec e|=1$
and $\vec t:=\{\vec t_Q\}_{Q\in\mathscr{Q}}$, where, for any $Q\in\mathscr{Q}$,
\begin{equation}\label{exa1}
\vec t_Q:=\begin{cases}
\vec e&\text{if }Q=Q_{0,\mathbf{0}},\\
\vec{\mathbf{0}}&\text{otherwise}.
\end{cases}
\end{equation}
Then $\|\vec t\|_{\dot b_{p,q}^s(\mathbb{A})}\sim1$.
By this and the assumption that $B$ is bounded on $\dot b^s_{p,q}(\mathbb{A})$, we find that
\begin{align}\label{est1}
\infty>\left\|B\vec t\right\|_{\dot b_{p,q}^s(\mathbb{A})}^p
\geq\sum_{Q\in\mathscr{Q}_0}
\left|A_Q\left(B\vec t\right)_Q\right|^p=\sum_{Q\in\mathscr{Q}_0}
\left(\left|A_Q\vec e\right|b_{Q,Q_{0,\mathbf{0}}}\right)^p
=\sum_{Q\in\mathscr{Q}_0}
\left(1+|x_Q|\right)^{-(Dp-d_2)}
\end{align}
and hence \eqref{estimate D} holds.

Now, we show
\begin{align}\label{estimate F}
F>\frac{n}{p}-\frac{n}{2}-s+\frac{d_2}{p}
\end{align}
by considering the following two cases on $q$.

\emph{Case 1)} $q\in(0,\infty)$.
In this case, let $\vec t$ be the same as in \eqref{exa1}.
Then $\|\vec t\|_{\dot b_{p,q}^s(\mathbb{A})}\sim1$.
This, combined with the boundedness of $B$ on $\dot b^s_{p,q}(\mathbb{A})$,
further implies that
\begin{align}\label{est3}
\infty
&>\left\|B\vec t\right\|_{\dot b_{p,q}^s(\mathbb{A})}^q
=\sum_{j\in\mathbb{Z}}2^{jsq}
\left[\sum_{Q\in\mathscr{Q}_j}|Q|^{1-\frac{p}{2}}
\left|A_Q\left(B\vec t\right)_Q\right|^p\right]^{\frac{q}{p}}\notag\\
&=\sum_{j\in\mathbb{Z}}2^{jsq}
\left[\sum_{Q\in\mathscr{Q}_j}|Q|^{1-\frac{p}{2}}
\left(\left|A_Q\vec e\right|b_{Q,Q_{0,\mathbf{0}}}\right)^p\right]^{\frac{q}{p}}\notag\\
&>\sum_{j=-\infty}^{-1}2^{jsq}
\left\{\sum_{Q\in\mathscr{Q}_j}|Q|^{1-\frac{p}{2}}2^{-jd_2}
\left(1+2^j|x_Q|\right)^{-Dp+d_2}[\ell(Q)]^{-Fp}\right\}^{\frac{q}{p}}\notag\\
&=\sum_{j=-\infty}^{-1}\left[2^{jp(F+\frac{n}{2}+s-\frac{n}{p}-\frac{d_2}{p})}
\sum_{k\in\mathbb{Z}^n}(1+|k|)^{-Dp+d_2}\right]^{\frac{q}{p}}.
\end{align}
From this, together with $D>\frac{n}{p}+\frac{d_2}{p}$ and Lemma \ref{253x}(ii),
it further follows that
$$
\infty>\sum_{j=-\infty}^{-1}2^{jq(F+\frac{n}{2}+s-\frac{n}{p}-\frac{d_2}{p})}
$$
and hence \eqref{estimate F} holds in this case.

\emph{Case 2)} $q=\infty$.
In this case, let $\vec t:=\{\vec t_Q\}_{Q\in\mathscr{Q}}$,
where, for any $Q\in\mathscr{Q}$,
$$
\vec t_{Q}:=\begin{cases}
|Q|^{\frac{s}{n}+\frac12-\frac{1}{p}}A_Q^{-1}\vec e&\text{if }x_Q=\mathbf{0},\\
\vec{\mathbf{0}}&\text{otherwise}.
\end{cases}
$$
Then $\|\vec t\|_{\dot b_{p,\infty}^s(\mathbb A)}=1$.
By the assumption that $(B\vec t)_{Q_{0,\mathbf{0}}}$ makes sense again, we conclude that
$$
\infty
>\sum_{R\in\mathscr{Q}}\left|b_{Q_{0,\mathbf{0}},R}\vec t_R\right|
>\sum_{i=0}^\infty2^{-i(F+s+\frac{n}{2}-\frac{n}{p}-\frac{d_2}{p})},
$$
and hence \eqref{estimate F} holds in this case,
which completes the proof of Lemma \ref{ad Besov sharp}.
\end{proof}

We will not compare our results with scalar weighted
Besov-type and Triebel--Lizorkin-type sequence spaces (\cite{t13}),
which had a small mistake.
To be precise, in the proof of \cite[Theorem 3]{t13},
the assumption that $\varepsilon_0>\frac{n(\alpha_2-\alpha_1)}{p}$
is sufficient to establish the desired estimate of $A_0$.
However, this assumption is insufficient to guarantee the desired estimate of $A_1$,
as claimed by the author without a proof.
Indeed, let us consider the special case $\dot b^{s}_{p,q}(w)$
for some scalar weight $w\in A_\infty(\mathbb R^n)$.
Let $d_2\in(0,\infty)$ and $w(\cdot):=|\cdot|^{d_2}$.
By \cite[Lemma 7.5]{bhyy2}, we find that $w\in A_\infty(\mathbb R^n)$ and
$d_\infty^{\mathrm{lower}}(w)=0$ is an $A_\infty$-lower dimension of $w$
and $d_\infty^{\mathrm{upper}}(w)=d_2$ is an $A_\infty$-upper dimension of $w$.
These, combined with \cite[Lemma 6.7(iii)]{bhyy2}, further imply that,
for any cube $Q,R\subset\mathbb R^n$ with $Q\cap R\neq\emptyset$,
\begin{align*}
\frac{w(Q)}{|Q|}\frac{|R|}{w(R)}
\lesssim\max\left\{1,\ \left[\frac{\ell(Q)}{\ell(R)}\right]^{d_2}\right\}
\end{align*}
and hence, for any cubes $Q,R\subset\mathbb R^n$ with $R\subset Q$,
\begin{align}\label{alpha}
\frac{|Q|}{|R|}
\lesssim\frac{w(Q)}{w(R)}
\lesssim\left(\frac{|Q|}{|R|}\right)^{1+\frac{d_2}{n}}.
\end{align}
Let $B:=\{b_{Q,R}\}_{Q,R\in\mathscr{Q}}\subset\mathbb{C}$
be $(D,E,F)$-almost diagonal for some $D,E,F\in\mathbb R$.
From \cite[Theorem 3]{t13} and \eqref{alpha}, we deduce that
$B$ is bounded on $\dot b^{s}_{p,q}(w)$ if
\begin{align}\label{wrong}
D>J+\frac{d_2}{p},\quad
E>\frac{n}{2}+s+\frac{d_2}{2p},\quad\text{and}\quad
F>J-\frac{n}{2}-s+\frac{d_2}{2p},
\end{align}
where $J:=\frac{n}{\min\{1,p\}}$.
However, the proof of Lemma \ref{ad Besov sharp} proves that
\begin{align*}
F>\frac np-\frac{n}{2}-s+\frac{d_2}{p}
\end{align*}
is necessary, which contradicts \eqref{wrong} when $p\in(0,1]$.

Next, we consider a more complicated case, namely Triebel--Lizorkin spaces.
To this end, we need the following well-known fact
about shifted systems of dyadic cubes of $\mathbb R^n$
(see, for instance, \cite[Lemma 3.2.26]{hnvw}).

\begin{lemma}\label{shifted}
On $\mathbb R^n$, there exist $3^n$ shifted systems of dyadic cubes
$\mathscr{Q}^\alpha$, $\alpha\in\{1,\ldots,3^n\}$,
such that, for any cube $Q\subset\mathbb R^n$,
there exist $\alpha\in\{1,\ldots,3^n\}$ and a cube $S\in\mathscr{Q}^\alpha$
such that $Q\subset S$ and $\ell(S)\in(\frac32\ell(Q),3\ell(Q)]$.
\end{lemma}

\begin{remark}
For any $\gamma\in\{0,\frac13,\frac23\}^n$, let
$
\mathscr{Q}^{\gamma}:=\bigcup_{\gamma\in\{0,\frac13,\frac23\}^n,\,j\in\mathbb{Z}}\mathscr{Q}^{\gamma}_j,
$
where, for any $j\in\mathbb Z$,
$$
\mathscr{Q}^{\gamma}_j
:=\left\{2^{-j}\left([0,1)^n+k+(-1)^j\gamma\right):\ k\in\mathbb{Z}^n\right\}.
$$
Then $\{\mathscr{Q}^{\alpha}\}_{\alpha=1}^{3^n}$ in Lemma \ref{shifted}
can be chosen to be a rearrangement of
$\{\mathscr{Q}^{\gamma}\}_{\gamma\in\{0,\frac13,\frac23\}^n}$.
\end{remark}

\begin{lemma}\label{ad F1}
Let $p\in(0,\infty)$, $q\in(0,\infty]$, $W\in A_{p,\infty}$
have $A_{p,\infty}$-upper dimension $d_2\in[0,\infty)$, and $a\in[0,p\wedge q\wedge 1)$.
Then there exists a positive constant $C$ such that,
for any $k\in\mathbb Z$, $l\in\mathbb Z_+$, and $\vec{t}\in\dot f^0_{p,q}(W)$,
\begin{equation*}
\left\|\left\{\left[\fint_{B(\cdot,2^{l+k_+-i})}
\left|W^{\frac{1}{p}}(\cdot)\vec{t}_i(y)\right|^a\,dy\right]^{\frac{1}{a}}
\right\}_{i\in\mathbb Z}\right\|_{L^p\ell^q}
\leq C2^{(l+k_+)\frac{d_2}{p}}\left\|\vec{t}\right\|_{\dot f^0_{p,q}(W)},
\end{equation*}
where $\vec{t}_i$ for any $i\in\mathbb Z$ is the same as in \eqref{vec tj}.
\end{lemma}

\begin{proof}
For any $i\in\mathbb{Z}$, any matrix $M\in M_m(\mathbb{C})$,
and any measurable set $E\subset\mathbb{R}^n$ with $|E|\in(0,\infty)$, define
$$
F_i(M,E):=\left[\fint_E\left|M\vec{t}_i(y)\right|^a\,dy\right]^{\frac{1}{a}}.
$$
From the proof of \cite[Lemma 6.5]{bhyyp2}, we infer that,
for any $k\in\mathbb Z$, $l\in\mathbb Z_+$, and $\vec t\in\dot f_{p,q}^0(W)$,
\begin{align*}
\left\|\left\{F_i\left(W^{\frac{1}{p}}(\cdot),B\left(\cdot,2^{l+k_+-i}\right)\right)
\right\}_{i\in\mathbb Z}\right\|_{L^p\ell^q}
\lesssim\max_{\alpha\in\{1,\ldots,3^n\}}I_\alpha,
\end{align*}
where
\begin{align*}
I_\alpha:=\left\|\left\{\sum_{S\in\mathscr{Q}^\alpha_{j}}
\mathbf{1}_S(\cdot)\left\|W^{\frac{1}{p}}(\cdot)A_S^{-1}\right\|F_{j+l+k_++2}(A_S,S)
\right\}_{j\in\mathbb Z}\right\|_{L^p\ell^q}
\end{align*}
with $\{A_S\}_{S\in\bigcup_{\alpha=1}^{3^n}\mathscr{Q}^{\alpha}}$
being a family of reducing operators of order $p$ for $W$.

Let $\alpha\in\{1, \ldots, 3^n\}$ be fixed.
We observe that the quantity on the right-hand side of the previous computation
is exactly of the form considered in Corollary \ref{46x}
with the shifted dyadic system $\mathscr{Q}^\alpha$ in place of $\mathscr{Q}$
and with $E_j(f_j)$ replaced by
$
\sum_{S\in\mathscr{Q}^\alpha_{j}}\mathbf{1}_S(\cdot)F_{j+l+k_++2}(A_S,S)
$
for any $j\in\mathbb{Z}$.
The mentioned Corollary \ref{46x} thus further implies that
\begin{align}\label{ad F equ}
I_\alpha
&\lesssim\left\|\left\{
\sum_{S\in\mathscr{Q}^\alpha_{j}}\mathbf{1}_S(\cdot)F_{j+l+k_++2}(A_S,S)
\right\}_{j\in\mathbb Z}\right\|_{L^p\ell^q}\notag\\
&=\left\|\left\{
\sum_{S\in\mathscr{Q}^\alpha_{j}}\mathbf{1}_S(\cdot)
\fint_{S}\left|A_S\vec{t}_{j+l+k_++2}(y)\right|^a\,dy
\right\}_{j\in\mathbb Z}\right\|_{L^{\frac{p}{a}}\ell^{\frac{q}{a}}}^{\frac{1}{a}}.
\end{align}
For any $S\in\mathscr{Q}^\alpha_{j}$,
there exists $Q\in\mathscr{Q}_j$ such that $Q\cap S\neq\emptyset$ and hence $S\subset 3Q$.
This, together with Lemma \ref{sharp}(iii), further implies that,
for any $S\in\mathscr{Q}^\alpha_{j}$ and $x\in S$,
\begin{align*}
&\fint_{S}\left|A_S\vec{t}_{j+l+k_++2}(y)\right|^a\,dy\\
&\quad\lesssim\fint_{3Q}\left|A_S\vec{t}_{j+l+k_++2}(y)\right|^a\,dy\\
&\quad\leq\frac{1}{|3Q|}\sum_{R\in\mathscr{Q}_{j+l+k_++2},\,R\subset3Q}
\left\|A_SA_{3Q}^{-1}\right\|^a\left\|A_{3Q}A_R^{-1}\right\|^a
\int_R\left|A_R\vec{t}_{j+l+k_++2}(y)\right|^a\,dy\\
&\quad\lesssim2^{(l+k_+)\frac{ad_2}{p}}
\fint_{3Q}\left|A_{j+l+k_++2}(y)\vec{t}_{j+l+k_++2}(y)\right|^a\,dy
\leq2^{(l+k_+)\frac{ad_2}{p}}
\mathcal{M}\left(\left|A_{j+l+k_++2}\vec{t}_{j+l+k_++2}\right|^a\right)(x).
\end{align*}
By this, \eqref{ad F equ}, Lemma \ref{Fefferman Stein}, and Theorem \ref{37x},
we conclude that
\begin{align*}
I_\alpha
&\lesssim2^{(l+k_+)\frac{d_2}{p}}\left\|\left\{
\mathcal{M}\left(\left|A_{j+l+k_++2}\vec{t}_{j+l+k_++2}\right|^a\right)
\right\}_{j\in\mathbb Z}\right\|_{L^{\frac{p}{a}}\ell^{\frac{q}{a}}}^{\frac{1}{a}}\\
&=2^{(l+k_+)\frac{d_2}{p}}\left\|\left\{
\mathcal{M}\left(\left|A_j\vec{t}_j\right|^a\right)
\right\}_{j\in\mathbb Z}\right\|_{L^{\frac{p}{a}}\ell^{\frac{q}{a}}}^{\frac{1}{a}}
\lesssim2^{(l+k_+)\frac{d_2}{p}}\left\|\left\{
\left|A_j\vec{t}_j\right|^a
\right\}_{j\in\mathbb Z}\right\|_{L^{\frac{p}{a}}\ell^{\frac{q}{a}}}^{\frac{1}{a}}\\
&=2^{(l+k_+)\frac{d_2}{p}}\left\|\left\{
\left|A_j\vec{t}_j\right|
\right\}_{j\in\mathbb Z}\right\|_{L^p\ell^q}
=2^{(l+k_+)\frac{d_2}{p}}\left\|\vec{t}\right\|_{\dot f^0_{p,q}(\mathbb{A})}
\sim2^{(l+k_+)\frac{d_2}{p}}\left\|\vec{t}\right\|_{\dot f^0_{p,q}(W)}.
\end{align*}
This finishes the proof of Lemma \ref{ad F1}.
\end{proof}

\begin{theorem}\label{ad F0}
Let $p\in(0,\infty)$, $q\in(0,\infty]$, $s\in\mathbb R$, and $W\in A_{p,\infty}$.
Suppose that $B$ is $(D,E,F)$-almost diagonal with parameters
\begin{equation*}
D>J+\frac{d_{p,\infty}^{\mathrm{upper}}(W)}{p},\quad
E>\frac{n}{2}+s,\quad\text{and}\quad
F>J-\frac{n}{2}-s+\frac{d_{p,\infty}^{\mathrm{upper}}(W)}{p},
\end{equation*}
where $J:=\frac{n}{\min\{1,p,q\}}$.
Then $B$ is bounded on $\dot f^s_{p,q}(W)$.
\end{theorem}

\begin{proof}
By the definition of $d_{p,\infty}^{\mathrm{upper}}(W)$, we find that
there exists $d_2\in[d_{p,\infty}^{\mathrm{upper}}(W),\infty)$ such that
$W$ has $A_{p,\infty}$-upper dimension $d_2$ and
\begin{equation*}
D>J+\frac{d_2}{p},\quad
E>\frac{n}{2}+s,\quad\text{and}\quad
F>J-\frac{n}{2}-s+\frac{d_2}{p}.
\end{equation*}
Applying this and an argument similar to that used in the proof of \cite[Theorem 6.6]{bhyyp2}
with \cite[Lemma 6.5]{bhyyp2} replaced by Lemma \ref{ad F1},
we conclude that $B$ is bounded on $\dot f^s_{p,q}(W)$.
This finishes the proof of Theorem \ref{ad F0}.
\end{proof}

Similarly to Theorem \ref{ad Besov},
Theorem \ref{ad F0} also includes a non-negative term $d_{p,\infty}^{\mathrm{upper}}(W)/p$.
This term is \emph{sharp} when $p\in(0,1]$ and $q\in[p,\infty]$ and hence irremovable.

\begin{lemma}\label{ad Triebel sharp}
Let $s\in\mathbb R$, $p\in(0,\infty)$, $q\in(0,\infty]$,
$d_2\in[0,\infty)$, and $D,E,F\in\mathbb{R}$.
Suppose that every $(D,E,F)$-almost diagonal matrix $B$ is bounded on $\dot f^s_{p,q}(W)$
whenever $W\in A_{p,\infty}$ satisfies $d_{p,\infty}^{\mathrm{upper}}(W)=d_2$.
Then $D>\frac{n}{p}+\frac{d_2}{p}$
and $F>\frac{n}{p}-\frac{n}{2}-s+\frac{d_2}{p}$.
\end{lemma}

\begin{proof}
Let $B,W,\mathbb A$ be the same as in the proof of Lemma \ref{ad Besov sharp}.
Then $B$ is bounded on $\dot f^s_{p,q}(\mathbb A)$.

We first estimate the range of $D$.
By some arguments similar to those used in the estimation of \eqref{est1}, we obtain
$$
\sum_{Q\in\mathscr{Q}_0}\frac{1}{(1+|x_Q|)^{Dp-d_2}}<\infty.
$$
This further implies that $D>\frac{n}{p}+\frac{d_2}{p}$.

Now, we estimate the range of $F$ by considering the following two cases on $q$.

\emph{Case 1)} $q\in(0,\infty)$.
In this case, we first recall a classical embedding.
Applying a basic calculation,
we find that, for any $\vec t:=\{\vec t_Q\}_{Q\in\mathscr{Q}}\subset\mathbb{C}^m$,
\begin{equation}\label{39y}
\left\|\vec t\right\|_{\dot b^{s}_{p,p\vee q}(\mathbb{A})}
\leq\left\|\vec t\right\|_{\dot f^{s}_{p,q}(\mathbb{A})}
\leq\left\|\vec t\right\|_{\dot b_{p,p\wedge q}^{s}(\mathbb{A})}.
\end{equation}
Let $\vec t$ be the same as in \eqref{exa1}.
Then $\|\vec t\|_{\dot f_{p,q}^s(\mathbb{A})}\sim1$. This, combined with
the boundedness of $B$ on $\dot f^s_{p,q}(\mathbb{A})$,
\eqref{39y}, and \eqref{est3} with $q$ replaced by $p\vee q$, further implies that
\begin{align*}
\infty>\left\|B\vec t\right\|_{\dot f_{p,q}^s(\mathbb{A})}^{p\vee q}
\geq\left\|B\vec t\right\|_{\dot b_{p,p\vee q}^s(\mathbb{A})}^{p\vee q}
>\sum_{j=-\infty}^{-1}\left\{2^{j(F+\frac{n}{2}+s-\frac{n}{p}-\frac{d_2}{p})p}
\sum_{k\in\mathbb{Z}^n}(1+|k|)^{-Dp}\right\}^{\frac{p\vee q}{p}}.
\end{align*}
From this, $D>\frac{n}{p}+\frac{d_2}{p}$, and Lemma \ref{253x}(ii), it follows that
$$
\infty>\sum_{j=-\infty}^{-1}2^{j(F+\frac{n}{2}+s-\frac{n}{p}-\frac{d_2}{p})(p\vee q)},
$$
and hence $F>\frac{n}{p}-\frac{n}{2}-s+\frac{d_2}{p}$.

\emph{Case 2)} $q=\infty$.
In this case, let $\vec t$ be the same as in \eqref{exa1}.
Then $\|\vec t\|_{\dot f_{p,\infty}^s(\mathbb A)}\sim1$.
This, together with the assumption that $B$ is bounded on
$\dot f^s_{p,\infty}(\mathbb A)$, further implies that
\begin{align*}
\infty
&>\left\|B\vec t\right\|_{\dot f_{p,\infty}^s(I_m)}^p
=\left\|\sup_{j\in\mathbb{Z}}2^{j(s+\frac{n}{2})}
\sup_{Q\in\mathscr{Q}_j}\left|A_Q\left(B\vec t\right)_Q\right|\mathbf{1}_Q\right\|_{L^p}^p\\
&\geq\int_{\mathbb{R}^n}\left[\sup_{j\in\mathbb{Z}}2^{j(s+\frac{n}{2})}
\left|A_{Q_{j,\mathbf 0}}\vec e\right||b_{Q_{j,\mathbf 0},Q_{0,\mathbf{0}}}|
\mathbf{1}_{Q_{j,\mathbf 0}}(x)\right]^p\,dx\\
&>\int_{\mathbb{R}^n}\left[\sup_{j\in\mathbb{Z},\,j\leq -1}2^{j(F+s+\frac{n}{2}-\frac{d_2}{p})}
\mathbf{1}_{Q_{j,\mathbf 0}}(x)\right]^p\,dx\\
&>\sum_{i=0}^\infty\int_{Q_{-(i+1),\mathbf{0}}\setminus Q_{-i,\mathbf{0}}}
\left[\sup_{j\in\mathbb{Z},\,j\leq -1}2^{j(F+s+\frac{n}{2}-\frac{d_2}{p})}
\mathbf{1}_{Q_{j,\mathbf 0}}(x)\right]^p\,dx\\
&\geq \sum_{i=0}^\infty|Q_{-(i+1),\mathbf{0}}\setminus Q_{-i,\mathbf{0}}|
\left[2^{-(i+1)(F+s+\frac{n}{2}-\frac{d_2}{p})}\right]^p
\sim\sum_{i=0}^\infty2^{-i[(F+s+\frac{n}{2}-\frac{d_2}{p})p-n]},
\end{align*}
and hence $F>\frac{n}{p}-\frac{n}{2}-s+\frac{d_2}{p}$.
This finishes the estimation of $F$ in this case
and hence finishes the proof of Lemma \ref{ad Triebel sharp}.
\end{proof}

Next, we show the following theorem,
which proves the subcritical part of Theorem \ref{ad BF2}.

\begin{theorem}\label{ad BF}
Let $s\in\mathbb R$, $\tau\in[0,\infty)$, $p\in(0,\infty)$, $q\in(0,\infty]$,
and $W\in A_{p,\infty}$. If $B$ is $(D,E,F)$-almost diagonal,
then $B$ is bounded on $\dot a^{s,\tau}_{p,q}(W)$ whenever
$$
D>J+n\widehat\tau+\frac{d_{p,\infty}^{\mathrm{upper}}(W)}{p},\quad
E>\frac{n}{2}+s+n\widehat\tau,\quad\text{and}\quad
F>J-\frac{n}{2}-s+\frac{d_{p,\infty}^{\mathrm{upper}}(W)}{p},
$$
where $J$ and $\widehat\tau$ are the same as, respectively, in \eqref{J} and \eqref{tauJ2}.
\end{theorem}

\begin{proof}
By the definitions of $d_{p,\infty}^{\mathrm{lower}}(W)$
and $d_{p,\infty}^{\mathrm{upper}}(W)$, we find that
there exist $d_1\in[d_{p,\infty}^{\mathrm{lower}}(W),\infty)$
and $d_2\in[d_{p,\infty}^{\mathrm{upper}}(W),\infty)$ such that
$W$ has both $A_{p,\infty}$-lower dimension $d_1$
and $A_{p,\infty}$-upper dimension $d_2$ and
\begin{align}\label{4.25}
\begin{cases}
\displaystyle D>J+\left(n\tau-\frac{n}{p}+\frac{d_1}{p}\right)_++\frac{d_2}{p},\\
\displaystyle E>\frac{n}{2}+s+\left(n\tau-\frac{n}{p}+\frac{d_1}{p}\right)_+,\\
\displaystyle F>J-\frac{n}{2}-s+\frac{d_2}{p}.
\end{cases}
\end{align}
From Lemma \ref{s=0 enough}, we deduce that,
to show the present theorem, it is enough to consider the case $s=0$.

Letting $P\in\mathscr{Q}$ be fixed, we need to estimate
$\|\{H_j(B\vec{t})_j\}_{j\in\mathbb Z}\|_{LA_{p,q}}$,
where, for any $j\in\mathbb{Z}$,
$$
H_j:=\mathbf{1}_P\mathbf{1}_{[j_P,\infty)}(j)W^{\frac{1}{p}}
\quad\text{and}\quad
LA_{p,q}:=
\begin{cases}
L^p\ell^q&\text{if }\dot a^{0,\tau}_{p,q}=\dot f^{0,\tau}_{p,q},\\
\ell^qL^p&\text{if }\dot a^{0,\tau}_{p,q}=\dot b^{0,\tau}_{p,q}.
\end{cases}
$$
Then, using Lemma \ref{ad prelim}, we obtain
\begin{align}\label{ABtx}
\left\|\left\{H_j\left(B\vec{t}\right)_j\right\}_{j\in\mathbb Z}\right\|_{LA_{p,q}}^r
&\lesssim\sum_{k\in\mathbb{Z}}\sum_{l=0}^\infty
\Bigg[2^{-(E-\frac{n}{2})k_-}2^{-k_+(F+\frac{n}{2}-\frac{n}{a})}2^{-(D-\frac{n}{a})l}\notag\\
&\qquad\times\left.\left\|\left\{\left[\fint_{B(\cdot,2^{l+k_+-i})}
\left|\mathbf{1}_P(\cdot)W^{\frac{1}{p}}(\cdot)\vec{t}_{i}(y)\right|^a\,dy
\right]^{\frac{1}{a}}\right\}_{i\geq j_P+k}\right\|_{LA_{p,q}}\right]^r,
\end{align}
where $r:=p\wedge q\wedge 1$ and we chose
\begin{equation*}
\begin{cases}
a\in(0,p\wedge q\wedge 1)&\text{if }\dot a^{0,\tau}_{p,q}=\dot f^{0,\tau}_{p,q},\\
a:=p\wedge 1&\text{if }\dot a^{0,\tau}_{p,q}=\dot b^{0,\tau}_{p,q}.
\end{cases}
\end{equation*}

Notice that, for any $k,i\in\mathbb Z$ and $l\in\mathbb Z_+$,
$2^{l+k_+-i}\leq\ell(P)=2^{-j_P}$ if and only if $i\geq j_P+k_+ +l$.
Then we further apply the $r$-triangle inequality
\eqref{b1} with $\alpha$ replaced by $r$
to the right-hand side of \eqref{ABtx} to obtain
\begin{align}\label{BF split}
&\left\|\left\{\left[\fint_{B(\cdot,2^{l+k_+-i})}
\left|\mathbf{1}_P(\cdot)W^{\frac{1}{p}}(\cdot)\vec{t}_{i}(y)\right|^a\,dy
\right]^{\frac{1}{a}}\right\}_{i\geq j_P+k}\right\|_{LA_{p,q}}^r\notag\\
&\quad\leq\left\|\left\{\left[\fint_{B(\cdot,2^{l+k_+-i})}
\left|\mathbf{1}_P(\cdot)W^{\frac{1}{p}}(\cdot)\vec{t}_{i}(y)\right|^a\,dy
\right]^{\frac{1}{a}}\right\}_{i\geq j_P+k_++l}\right\|_{LA_{p,q}}^r\notag\\
&\qquad+\sum_{i=j_P+k}^{j_P+k_++l-1}
\left\|\left[\fint_{B(\cdot,2^{l+k_+-i})}
\left|\mathbf{1}_P(\cdot)W^{\frac{1}{p}}(\cdot)\vec{t}_{i}(y)\right|^a\,dy
\right]^{\frac{1}{a}}\right\|_{L^p}^r\notag\\
&\quad=:\mathrm{I}^r+\sum_{i=j_P+k}^{j_P+k_++l-1}\mathrm{J}_i^r,
\end{align}
noticing that $L^p\ell^q$ and $\ell^qL^p$ norms reduce to just the $L^p$ norm
when applied to $\{f_i\}_{i\in\mathbb Z}$
with only one non-zero component $f_i$ for some $i\in\mathbb Z$.

We first estimate $\mathrm{I}$.
By construction, in the first term on the right-hand side of \eqref{BF split},
we have $2^{l+k_+-i}\leq\ell(P)$.
Hence, if $x\in P$ and $y\in B(x,2^{l+k_+-i})\subset B(x,\ell(P))$,
we obtain $y\in3P$. On the other hand,
it is clear that $j_P+k_++l\geq j_P$.
Thus, for any $\vec t\in\dot a^{0,\tau}_{p,q}(W)$, we have
\begin{align*}
\mathrm{I}
&\leq\left\|\left\{\left[\fint_{B(\cdot,2^{l+k_+-i})}
\left|W^{\frac{1}{p}}(\cdot)\left\{\mathbf{1}_{[j_P,\infty)}(i)
\mathbf{1}_{3P}(y)\vec{t}_{i}(y)\right\}\right|^a\,dy
\right]^{\frac{1}{a}}\right\}_{i\in\mathbb Z}\right\|_{LA_{p,q}}\\
&\lesssim 2^{(l+k_+)\frac{d_2}{p}}\left\|\left\{\mathbf{1}_{[j_P,\infty)}(i)\mathbf{1}_{3P}(\cdot)
W^{\frac{1}{p}}(\cdot)\vec{t}_{i}(\cdot)\right\}_{i\in\mathbb Z}\right\|_{LA_{p,q}}
\lesssim 2^{(l+k_+)\frac{d_2}{p}} |P|^{\tau}\left\|\vec{t}\right\|_{\dot a^{0,\tau}_{p,q}(W)},
\end{align*}
where, depending on whether we are in the Besov or the Triebel--Lizorkin type case,
we used Lemma \ref{ad B1} or \ref{ad F1} to
$\{\mathbf{1}_{[j_P,\infty)}(i)\mathbf{1}_{3P}(\cdot)\vec{t}_{i}(\cdot)\}_{i\in\mathbb Z}$
in place of $\{\vec{t}_i\}_{i\in\mathbb Z}$ and then the definition of $\dot a^{0,\tau}_{p,q}(W)$.

Now, we estimate $J_i$. Let $i\in\{j_P+k,\ldots,j_P+k_++l-1\}$ be fixed.
Notice that, when $x\in P$, one has
\begin{align}\label{5.25x}
B\left(x,2^{l+k_+-i}\right)\subset3\cdot2^{j_P+k_++l-i}P=:\widetilde{P}
\end{align}
and
\begin{align}\label{5.25y}
j_P+k_++l-i\in[1,k_-+l].
\end{align}
By \eqref{5.25x} and Lemma \ref{sharp}(iii), we find that
\begin{align*}
\mathrm{J}_i^p&\leq\left\|\,\left[\fint_{B(\cdot,2^{l+k_+-i})}
\left|\mathbf{1}_P(\cdot)W^{\frac{1}{p}}(\cdot)\vec{t}_{i}(y)\right|^p\,dy
\right]^{\frac{1}{p}}\right\|_{L^p}^p\lesssim\int_P\fint_{\widetilde{P}}
\left\|W^{\frac{1}{p}}(x)A_{\widetilde{P}}^{-1}\right\|^p
\left|A_{\widetilde{P}}\vec{t}_i(y)\right|^p\,dy\,dx\\
&\sim2^{(i-j_P-k_+-l)n}\left\|A_PA_{\widetilde{P}}^{-1}\right\|^p
\int_{\widetilde{P}}
\left|A_{\widetilde{P}}\vec{t}_i(y)\right|^p\,dy\\
&\lesssim2^{(i-j_P-k_+-l)(n-d_1)}\sum_{R\in\mathscr{Q}_i,\,R\subset\widetilde{P}}
\int_R\left\|A_{\widetilde{P}}A_R^{-1}\right\|^p\left|A_R\vec{t}_i(y)\right|^p\,dy \\
&\sim2^{(i-j_P-k_+-l)(n-d_1)}2^{(k_++l)d_2}
\int_{\widetilde{P}}\left|A_i\vec{t}_i(y)\right|^p\,dy.
\end{align*}
Using this, Theorem \ref{37x}, and \eqref{5.25y}, we conclude that
\begin{align*}
\mathrm{J}_i
&\lesssim2^{(i-j_P-k_+-l)(\frac{n}{p}-\frac{d_1}{p})}2^{(k_++l)\frac{d_2}{p}}
\left\|\mathbf{1}_{\widetilde P}A_i\vec{t}_i\right\|_{L^p}
\leq2^{(i-j_P-k_+-l)(\frac{n}{p}-\frac{d_1}{p})}2^{(k_++l)\frac{d_2}{p}}
\left|\widetilde P\right|^{\tau}
\left\|\vec{t}\right\|_{\dot a^{0,\tau}_{p,q}(\mathbb{A})}\\
&\sim2^{(j_P+k_++l-i)(n\tau-\frac{n}{p}+\frac{d_1}{p})}2^{(k_++l)\frac{d_2}{p}}
|P|^{\tau}\left\|\vec{t}\right\|_{\dot a^{0,\tau}_{p,q}(W)}
\leq2^{(k_-+l)(n\tau-\frac{n}{p}+\frac{d_1}{p})_+}2^{(k_++l)\frac{d_2}{p}}|P|^{\tau}
\left\|\vec{t}\right\|_{\dot a^{0,\tau}_{p,q}(W)}\\
&=2^{k_-(n\tau-\frac{n}{p}+\frac{d_1}{p})_+}
2^{l[(n\tau-\frac{n}{p}+\frac{d_1}{p})_++\frac{d_2}{p}]}
2^{k_+\frac{d_2}{p}}|P|^{\tau}
\left\|\vec{t}\right\|_{\dot a^{0,\tau}_{p,q}(W)}.
\end{align*}
We have this bound for each of the terms in the sum over $i$ in \eqref{BF split},
and the total number of these terms is $(j_P+k_++l)-(j_P+k)=k_-+l$.
Hence, we find that
\begin{align*}
&\left\|\left\{\left[\fint_{B(\cdot,2^{l+k_+-i})}
\left|\mathbf{1}_P(\cdot)W^{\frac{1}{p}}(\cdot)\vec{t}_{i}(y)\right|^a\,dy
\right]^{\frac{1}{a}}\right\}_{i\geq j_P+k}\right\|_{LA_{p,q}}^r\\
&\quad\lesssim(1+k_-+l)\left\{2^{k_-(n\tau-\frac{n}{p}+\frac{d_1}{p})_+}
2^{l[(n\tau-\frac{n}{p}+\frac{d_1}{p})_++\frac{d_2}{p}]}
2^{k_+\frac{d_2}{p}}|P|^{\tau}
\left\|\vec{t}\right\|_{\dot a^{0,\tau}_{p,q}(W)}\right\}^r.
\end{align*}
Substituting the above estimate into \eqref{ABtx},
dividing by the factor $|P|^{\tau}$,
and then taking the supremum over all $P\in\mathscr{Q}$, we obtain
\begin{align*}
\left\|B\vec{t}\right\|_{\dot a^{0,\tau}_{p,q}(W)}^r
&\lesssim\sum_{k\in\mathbb{Z}}\sum_{l=0}^\infty(1+k_-+l)
\Bigg\{2^{-k_-[E-\frac{n}{2}-(n\tau-\frac{n}{p}+\frac{d_1}{p})_+]}
2^{-k_+(F+\frac{n}{2}-\frac{n}{a}-\frac{d_2}{p})}\\
&\quad\times2^{-l[D-\frac{n}{a}-(n\tau-\frac{n}{p}+\frac{d_1}{p})_+-\frac{d_2}{p}]}
\left\|\vec{t}\right\|_{\dot a^{0,\tau}_{p,q}(W)}\Bigg\}^r,
\end{align*}
and next it is easy to read off the conditions for convergence,
which requires that each of the summation variables $k_-$, $k_+$, and $l$
needs to have a negative coefficient in the exponent.
(Once this condition is satisfied, the polynomial factor in front is irrelevant.)
For $k_-$, it is immediate that we need
$E>\frac{n}{2}+(n\tau-\frac{n}{p}+\frac{d_1}{p})_+$.
For $k_+$, recalling that $\frac{n}{a}-J$ is either zero
or can be made as close to zero as we like by an appropriate choice of $a$,
we find that we need $F>J-\frac{n}{2}+\frac{d_2}{p}$.
On the coefficient of $l$, for similar reasons, we need
$D>J+(n\tau-\frac{n}{p}+\frac{d_1}{p})_++\frac{d_2}{p}$.
These are precisely \eqref{4.25} when $s=0$, and hence
$
\|B\vec{t}\|_{\dot a^{0,\tau}_{p,q}(W)}
\lesssim\|\vec{t}\|_{\dot a^{0,\tau}_{p,q}(W)}.
$
This finishes the proof of Theorem \ref{ad BF}.
\end{proof}

The following theorem proves the critical and the supercritical parts of Theorem \ref{ad BF2}.

\begin{theorem}\label{ad iff cor}
Let $s\in\mathbb R$, $p\in(0,\infty)$, $\tau\in[\frac1p,\infty)$,  $q\in(0,\infty]$, $W\in A_{p,\infty}$, and
$$
J_\tau:=
\begin{cases}
n&\displaystyle\text{if }\tau>\frac1p,\\
\displaystyle\frac{n}{\min\{1,q\}}&\displaystyle\text{if }\tau=\frac1p.
\end{cases}
$$
If $\tau=\frac1p$ and $q\in(0,\infty)$, also assume that
$\dot a^{s,\tau}_{p,q}(W)=\dot f^{s,\frac1p}_{p,q}(W)$ is of Triebel--Lizorkin-type.
Let $D,E,F\in\mathbb R$ satisfy
$$
D>J_\tau+\frac{d_{p,\infty}^{\mathrm{lower}}(W)+d_{p,\infty}^{\mathrm{upper}}(W)}{p},\quad
E>\frac{n}{2}+s+n\left(\tau-\frac1p\right)+\frac{d_{p,\infty}^{\mathrm{lower}}(W)}{p},
$$
and
$$
F>J_\tau-\frac n2-s-n\left(\tau-\frac1p\right)+\frac{d_{p,\infty}^{\mathrm{upper}}(W)}p.
$$
Then all $(D,E,F)$-almost diagonal operators are bounded on $\dot a^{s,\tau}_{p,q}(W)$.
\end{theorem}

\begin{proof}
As in \cite[Theorem 8.1]{bhyyp2}, the key to the proof is that the space $\dot a^{s,\tau}_{p,q}(W)$ is equivalent to another space with more favourable parameters. To begin with, recall from Theorem \ref{37} that $\dot a^{s,\tau}_{p,q}(W)=\dot a^{s,\tau}_{p,q}(\mathbb A)$, where $\mathbb A=\{A_Q\}_{Q\in\mathscr Q}$ is a sequence of reducing operators of order $p$ for $W$.

For any $\tau=\frac1p$ and $q\in(0,\infty)$, denoting $u_Q:=|A_Q\vec t_Q|$, we obtain
\begin{equation}\label{fj90c57}
\left\|\vec t\right\|_{\dot f^{s,\frac1p}_{p,q}(\mathbb A)}
=\|u\|_{\dot f^{s,\frac1p}_{p,q}}
\sim \|u\|_{\dot f^{s,\frac1q}_{q,q}}
=: \|u\|_{\dot f^{s}_{\infty,q}}
=: \left\|\vec t\right\|_{\dot f^{s}_{\infty,q}(\mathbb A)},
\end{equation}
where ``$\sim$'' is a classical equivalence \cite[Corollary 5.7]{fj90}
for the {\em unweighted} spaces $\dot f^{s,\frac1p}_{p,q}$,
and the last two ``$=:$'' are just the definition of the norms of the space
$\dot f^s_{\infty,q}$ and the weighted space $\dot f^s_{\infty,q}(\mathbb A)$.
Note that no assumptions on $\mathbb A$ for \eqref{fj90c57}.
(This short argument was repeated from \cite[(4.4)]{bhyy} for convenience.)

For any $\tau\in(\frac1p,\infty)$ or $(\tau,q)=(\frac1p,\infty)$,
we apply \cite[Theorem 4.17]{bhyy}, which is conveniently formulated
for any sequence of positive definite matrices
$\mathbb A=\{A_Q\}_{Q\in\mathscr Q}$. Under these conditions,
the said theorem guarantees that
\begin{equation*}
\dot a^{s,\tau}_{p,q}(\mathbb A)=\dot f^{s+n(\tau-\frac1p)}_{\infty,\infty}(\mathbb A).
\end{equation*}
Hence, altogether we see that
\begin{equation}\label{critSpaces}
\dot a^{s,\tau}_{p,q}(W)=\dot a^{s,\tau}_{p,q}(\mathbb A)
=\begin{cases} \dot f^{s+n(\tau-\frac1p)}_{\infty,\infty}(\mathbb A) & \text{if } \tau\in(\frac1p,\infty)\text{ or }(\tau,q)=(\frac1p,\infty), \\
\dot f^{s}_{\infty,q}(\mathbb A) & \text{if } \tau=\frac1p\text{ and }\dot a=\dot f.\end{cases}
\end{equation}
(Note that the case $\dot a^{s,\tau}_{p,q}=\dot f^{s,\frac{1}{p}}_{p,\infty}$ is consistently covered by both cases.)

Therefore, it suffices to prove the boundedness on the spaces
on the right-hand side of \eqref{critSpaces}, and we hence turn to this task.

\subsubsection*{Case $\tau=\frac1p$ and $\dot a^{s,\tau}_{p,q}(W)=\dot f^s_{\infty,q}(\mathbb A)$:}

By the definitions of $d_{p,\infty}^{\mathrm{lower}}(W)$
and $d_{p,\infty}^{\mathrm{upper}}(W)$, we find that
there exist $d_1\in[d_{p,\infty}^{\mathrm{lower}}(W),\infty)$
and $d_2\in[d_{p,\infty}^{\mathrm{upper}}(W),\infty)$ such that
$W$ has both $A_{p,\infty}$-lower dimension $d_1$
and $A_{p,\infty}$-upper dimension $d_2$ and
\begin{align}\label{DEFcrit}
\begin{cases}
\displaystyle D>J_\tau+\frac{d_1+d_2}{p},\\
\displaystyle E>\frac{n}{2}+s+n\left(\tau-\frac1p\right)+\frac{d_1}{p},\\
\displaystyle F>J_\tau-\frac{n}{2}-s-n\left(\tau-\frac1p\right)+\frac{d_2}{p}.
\end{cases}
\end{align}

Lemma \ref{sharp} then guarantees that
\begin{equation}\label{82new}
\left\|A_Q A_R^{-1}\right\|^p
\leq C\max\left\{ \left[\frac{\ell(R)}{\ell(Q)}\right]^{d_1},\left[\frac{\ell(R)}{\ell(Q)}\right]^{d_2}\right\}
\left[1+\frac{|x_Q-x_R|}{\ell(Q)\vee\ell(R)}\right]^{d_1+d_2}.
\end{equation}

In the case under consideration, $J_\tau=\frac{n}{\min\{1,q\}}$, and \eqref{DEFcrit} reads as
\begin{equation}\label{DEFcrit1}
D>\frac{n}{\min\{1,q\}}+\frac{d_1+d_2}{p},\quad
E>\frac{n}{2}+s+\frac{d_1}{p},\quad
F>\frac{n}{\min\{1,q\}}-\frac{n}{2}-s-\frac{d_2}{p}.
\end{equation}

As in the proof of \cite[Lemma 8.3]{bhyyp2}, we consider
the upper and the lower triangular parts
\begin{equation*}
\left(B_0\vec t\right)_Q:=\sum_{R\in\mathscr Q,\ell(R)\geq\ell(Q)}b_{Q,R}\vec t_R,\quad
\left(B_1\vec t\right)_Q:=\sum_{R\in\mathscr Q,\ell(R)<\ell(Q)}b_{Q,R}\vec t_R
\end{equation*}
and denote $u_R:=|R|^{-(\frac{s}{n}+\frac12)}|A_R\vec t_R|$. From \eqref{82new} and $(D,E,F)$-almost diagonality, it follows that
\begin{align*}
|Q|^{-(\frac{s}{n}+\frac12)}\left|A_Q\left(B_0\vec t\right)_Q\right|
&\leq|Q|^{-(\frac{s}{n}+\frac12)}\sum_{R\in\mathscr Q,\,\ell(R)\geq\ell(Q)}
|b_{Q,R}|\left\|A_Q A_R^{-1}\right\|\left|A_R\vec t_R\right| \\
&\lesssim\sum_{R\in\mathscr Q,\,\ell(R)\geq\ell(Q)}
\left[\frac{\ell(Q)}{\ell(R)}\right]^{\widetilde E}
\left[1+\frac{|x_Q-x_R|}{\ell(R)}\right]^{-\widetilde D}u_R
\end{align*}
and
\begin{align*}
|Q|^{-(\frac{s}{n}+\frac12)}\left|A_Q\left(B_1\vec t\right)_Q\right|
&\leq|Q|^{-(\frac{s}{n}+\frac12)}\sum_{R\in\mathscr Q,\,\ell(R)<\ell(Q)}
|b_{Q,R}|\left\|A_Q A_R^{-1}\right\|\left|A_R\vec t_R\right| \\
&\lesssim\sum_{R\in\mathscr Q,\,\ell(R)<\ell(Q)}
\left[\frac{\ell(R)}{\ell(Q)}\right]^{\widetilde F}
\left[1+\frac{|x_Q-x_R|}{\ell(Q)}\right]^{-\widetilde D}u_R,
\end{align*}
where
\begin{equation*}
\widetilde D:=D-\frac{d_1+d_2}{p},\quad
\widetilde E:=E-s-\frac{n}{2}-\frac{d_1}{p},\quad
\widetilde F:=F+s+\frac{n}{2}-\frac{d_2}{p}.
\end{equation*}
The previous three displays are completely analogous to \cite[(8.4), (8.5), (8.7), and the unnumbered display preceding (8.7)]{bhyyp2}, except for different values of $\widetilde D$, $\widetilde E$, and $\widetilde F$. By inspection, the only properties of these values that are needed in \cite{bhyyp2} are those recorded in \cite[(8.5) and (8.7)]{bhyyp2}, namely that
\begin{equation*}
\widetilde D>\frac{n}{\min\{1,q\}},\quad
\widetilde E>0,\quad
\widetilde F>\frac{n}{\min\{1,q\}},
\end{equation*}
and these are clearly guaranteed by \eqref{DEFcrit1} in the present case.
Hence the remainder of the estimation may be simply borrowed from
\cite[Proof of Lemma 8.3]{bhyyp2}, which shows that
\begin{equation*}
\left\| B\vec t \right\|_{\dot f^s_{\infty,q}(\mathbb A)}
\lesssim
\left\| B_0\vec t \right\|_{\dot f^s_{\infty,q}(\mathbb A)}
+\left\| B_1\vec t \right\|_{\dot f^s_{\infty,q}(\mathbb A)}
\lesssim \left\| \vec t \right\|_{\dot f^s_{\infty,q}(\mathbb A)}.
\end{equation*}

\subsubsection*{Case $\tau\in(\frac1p,\infty)$ and $\dot a^{s,\tau}_{p,q}(W)=\dot f^{s+n(\tau-\frac1p)}_{\infty,\infty}(\mathbb A)$:}

Denoting $\widehat s:=s+n(\tau-\frac1p)$, from \eqref{critSpaces} we infer that
\begin{equation*}
\dot a^{s,\tau}_{p,q}(W)=\dot f^{\,\widehat s}_{\infty,\infty}(\mathbb A)=\dot f^{\,\widehat s,\frac1p}_{p,\infty}(W),
\end{equation*}
where the space on the right-hand side is in the scope of the previous case
of the proof that we already considered. (Note that the definition of
$J_\tau=n=\frac{n}{\min\{1,\infty\}}$ is consistent on both sides of
the identity above.) Applying that said previous case, we see that the conditions
\begin{equation*}
D>J_\tau+\frac{d_{p,\infty}^{\mathrm{lower}}(W)+d_{p,\infty}^{\mathrm{upper}}(W)}{p},\quad
E>\frac{n}{2}+\widehat s+\frac{d_{p,\infty}^{\mathrm{lower}}(W)}{p},\quad
F>J_\tau-\frac{n}{2}-\widehat{s}+\frac{d_{p,\infty}^{\mathrm{upper}}(W)}{p}
\end{equation*}
guarantee the boundedness of all $(D,E,F)$-almost diagonal operators on
$\dot f^{\,\widehat s,\frac1p}_{p,\infty}(W)$ and hence on
$\dot a^{s,\tau}_{p,q}(W)$. But, substituting the value of $\widehat s$,
these conditions are exactly those claimed in the present theorem,
and the proof is hence complete.
\end{proof}

\begin{proof}[Proof of Theorem \ref{ad BF2}]
Combining Theorems \ref{ad BF} and \ref{ad iff cor},
we obtain Theorem \ref{ad BF2}; we omit the details.
This finishes the proof of Theorem \ref{ad BF2}.
\end{proof}

\begin{remark}\label{ad conv 2}
In Theorem \ref{ad BF2}, similarly to Remark \ref{ad conv},
we have a slightly stronger estimate than the boundedness of $B$;
that is, there exists a positive constant $C$ such that,
for any $\vec t:=\{\vec t_R\}_{R\in\mathscr Q}\in\dot a^{s,\tau}_{p,q}(W)$,
$$
\left\|\left\{2^{js}\sum_{Q\in\mathscr{Q}_j}\widetilde{\mathbf{1}}_Q(\cdot)
\sum_{R\in\mathscr{Q}}\left|H_j(\cdot)b_{Q,R}\vec{t}_R\right|
\right\}_{j\in\mathbb Z}\right\|_{LA_{p,q}^{\tau}}
\leq C\left\|\vec t\right\|_{\dot a^{s,\tau}_{p,q}(W)},
$$
where $ LA_{p,q}^{\tau}\in\{LB_{p,q}^\tau,LF_{p,q}^\tau\}$ and,
for any $j\in\mathbb Z$, $H_j=W^{\frac{1}{p}}$ or,
for any $j\in\mathbb Z$, $H_j=A_j$ with $A_j$ in \eqref{Aj}.
\end{remark}

\subsection{Comparison with Results of Bownik and Ho}\label{sec:BH}

The aim of this subsection is to compare the results of Subsection \ref{Homogeneous}
to related results obtained by Bownik and Ho \cite{bh06}. While the scopes
of their results (in {\em scalar-weighted anisotropic} Triebel--Lizorkin spaces)
and ours (in {\em matrix-weighted isotropic} spaces) are not comparable,
we will make the comparison in their common intersection of scalar-weighed
isotropic Triebel--Lizorkin spaces.
In this context, we can compare Theorem \ref{ad BF2} with \cite[Theorem 4.1]{bh06}.
To this end, we first introduce some concepts.

\begin{definition}
Let $d\in\mathbb{R}$.
A weight $w$ on $\mathbb R^n$ is said to have \emph{$A_\infty$-lower dimension $d$}
if there exists a positive constant $C$ such that,
for any $\lambda\in[1,\infty)$ and any cube $Q\subset\mathbb{R}^n$,
\begin{align*}
\fint_Qw(x)\,dx
\exp\left(\fint_{\lambda Q}\log\left[w(x)^{-1}\right]\,dx\right)
\leq C\lambda^d.
\end{align*}
A weight $w$ is said to have \emph{$A_\infty$-upper dimension $d$}
if there exists a positive constant $C$ such that,
for any $\lambda\in[1,\infty)$ and any cube $Q\subset\mathbb{R}^n$,
\begin{align*}
\fint_{\lambda Q}w(x)\,dx
\exp\left(\fint_Q\log\left[w(x)^{-1}\right]\,dx\right)
\leq C\lambda^d.
\end{align*}
\end{definition}

\begin{definition}
Let $p\in(1,\infty)$ and $d\in\mathbb{R}$.
A weight $w$ on $\mathbb R^n$ is said to have \emph{$A_p$-dimension $d$}
if there exists a positive constant $C$ such that,
for any $\lambda\in[1,\infty)$ and any cube $Q\subset\mathbb{R}^n$,
\begin{align*}
\fint_Qw(x)\,dx
\left[\fint_{\lambda Q}[w(x)]^{-\frac{p'}{p}}\,dx\right]^{\frac{p}{p'}}
\leq C\lambda^d,
\end{align*}
where $\frac1p+\frac1{p'}=1$.
\end{definition}

By Lemma \ref{Ap dim prop} with $m=1$, we immediately obtain the following
conclusion; we omit the details.

\begin{proposition}
Let $w$ be a weight on $\mathbb R^n$.
\begin{enumerate}[\rm(i)]
\item\label{lower scalar} The possible forms of the set
$
\{d\in\mathbb R:\ w\text{ has }A_\infty\text{-lower dimension }d\}
$
are the empty set and all the intervals of the form $(a,\infty)$ and $[a,\infty)$, where $a\in[0,n)$.
\item\label{upper scalar} The possible forms of the set
$
\{d\in\mathbb R:\ w\text{ has }A_\infty\text{-upper dimension }d\}
$
are the empty set and all the intervals of the form $(b,\infty)$ and $[b,\infty)$, where $b\in[0,\infty)$.
\item
In both \eqref{lower scalar} and \eqref{upper scalar},
the empty set corresponds to $w\notin A_\infty$.
\end{enumerate}
\end{proposition}

For any $w\in A_\infty(\mathbb R^n)$, let
$$
d^{\mathrm{upper}}_\infty(w)
:= \inf\{d\in [0,\infty) : w \text{ has } A_\infty\text{-upper dimension } d\}.
$$

When $m=1$ and $\tau=0$, the space
$\dot f^{s,\tau}_{p,q}(W)$ reduces to $\dot f^s_{p,q}(w)$,
and the following conclusion is immediately inferred from Theorem \ref{ad BF2};
we omit the details.

\begin{theorem}\label{our}
Let $s\in\mathbb R$, $p\in(0,\infty)$, $q\in(0,\infty]$, and $w\in A_\infty(\mathbb R^n)$.
Suppose that $B$ is $(D,E,F)$-almost diagonal with parameters
\begin{equation}\label{ad new scalar}
D>J+\frac{d^{\mathrm{upper}}_\infty(w)}{p},\quad
E>\frac{n}{2}+s,\quad\text{and}\quad
F>J-\frac{n}{2}-s+\frac{d^{\mathrm{upper}}_\infty(w)}{p},
\end{equation}
where $J:=\frac{n}{\min\{1,p,q\}}$.
Then $B$ is bounded on $\dot f^s_{p,q}(w)$.
\end{theorem}

For any $w\in A_\infty(\mathbb R^n)$, define
$$
r_w:=\inf\left\{r\in(1,\infty):\ w\in A_r(\mathbb R^n)\right\}
$$
the \emph{critical index} of $w$.
The following theorem is a special case of \cite[Theorem 4.1]{bh06}.

\begin{theorem}\label{bownik}
Let $s\in\mathbb R$, $p\in(0,\infty)$, $q\in(0,\infty]$,
and $w\in A_\infty(\mathbb R^n)$.
Suppose that $B$ is $(D,E,F)$-almost diagonal with parameters
\begin{equation}\label{ad old scalar}
D>J,\quad
E>\frac{n}{2}+s,\quad\text{and}\quad
F>J-\frac{n}{2}-s,
\end{equation}
where $J:=\frac{n}{\min\{1,p/r_w,q\}}$.
Then $B$ is bounded on $\dot f^s_{p,q}(w)$.
\end{theorem}

\begin{remark}
We will compare the above two theorems in the case where
$p\leq r_w\min\{1,q\}$. If $q\in[1,\infty]$, this is the same as $p\leq r_w$.
If further $p\in(1,\infty)$, it follows from the open-endedness of the $A_r$ condition
that $p\leq r_w$ is equivalent to $w\notin A_p$. This is a natural situation
to consider in the sense that, if in fact $w\in A_p$, then we could apply
the results of our earlier work \cite{bhyyp2} instead of Theorem \ref{our}.
Indeed, when $p\leq r_w\min\{1,q\}$,
condition \eqref{ad old scalar} reads as
$$
D>\frac{nr_w}{p},\quad
E>\frac{n}{2}+s,\quad\text{and}\quad
F>\frac{nr_w}{p}-\frac{n}{2}-s,
$$
while \eqref{ad new scalar} reads as
$$
D>\frac{n+d^{\mathrm{upper}}_\infty(w)}{p},\quad
E>\frac{n}{2}+s,\quad\text{and}\quad
F>\frac{n+d^{\mathrm{upper}}_\infty(w)}{p}-\frac{n}{2}-s.
$$
Therefore, the following conclusion shows that
Theorem \ref{our} improves Theorem \ref{bownik} in this case.
\end{remark}

\begin{theorem}\label{dim and critical}
The following statements hold.
\begin{enumerate}[\rm(i)]
\item For any $w\in A_\infty(\mathbb R^n)$, one has
$d^{\mathrm{upper}}_\infty(w)\leq n(r_w-1)$.

\item For every $d\in[0,\infty)$, there exists $w\in A_\infty(\mathbb R^n)$ such that
$d^{\mathrm{upper}}_\infty(w)= n(r_w-1)=d$.

\item If $n\geq 2$, then there exists $w\in A_\infty(\mathbb R^n)$ such that
$d^{\mathrm{upper}}_\infty(w)< n(r_w-1)$.
\end{enumerate}
\end{theorem}

To prove this theorem, we need several lemmas.

\begin{lemma} \label{Ap and Ainfty}
Let $p\in (1,\infty)$ and $w\in A_p(\mathbb R^n)$.
Then the following statements hold.
\begin{enumerate}[\rm(i)]
\item For any $d\in[0,n)$, $w$ has $A_\infty$-upper dimension $d$
if and only if $w^{-\frac{1}{p-1}}$ has $A_{p'}$-dimension $\frac{d}{p-1}$.

\item There exists $d\in[0,n)$ such that
$w$ has $A_\infty$-upper dimension $d(p-1)$.
\end{enumerate}
\end{lemma}

\begin{proof}
We first show (i).
From H\"older's inequality and $w\in A_p(\mathbb R^n)$,
we deduce that, for any cube $Q\subset\mathbb R^n$,
$$
1\leq\fint_Q w(x)\,dx\left[\fint_Q w(x)^{-\frac{p'}{p}}\,dx\right]^{\frac{p}{p'}}
\leq[w]_{A_p(\mathbb R^n)}
$$
and hence
$$
w(Q)\sim |Q|\left[\fint_Q w(x)^{-\frac{p'}{p}}\,dx\right]^{-\frac{p}{p'}} .
$$
This, together with \cite[Proposition 2.46(ii)]{bhyy}, further implies that
$w^{-\frac{1}{p-1}}$ has $A_{p'}$-dimension $\frac{d}{p-1}$
if and only if, for any cube $Q\subset\mathbb R^n$ and any $i\in\mathbb Z_+$,
\begin{equation}\label{mid}
\frac{w(2^iQ)}{w(Q)}\lesssim 2^{i(d+n)}.
\end{equation}
By this and \cite[Proposition 7.3]{bhyy2}, we find that \eqref{mid} holds
if and only if $w$ has $A_\infty$-upper dimension $d$.
This finishes the proof of (i).

Now, we prove (ii).
It is well known that $w\in A_p(\mathbb R^n)$
if and only if $w^{-\frac{1}{p-1}}\in A_{p'}(\mathbb R^n)$.
From this and \cite[Proposition 2.27]{bhyy} with $m=1$, we infer that
there exists $d\in[0,n)$ such that
$w^{-\frac{1}{p-1}}$ has $A_{p'}$-dimension $d$,
which, combined with (i), further implies that
$w$ has $A_\infty$-upper dimension $d(p-1)$.
This finishes the proof of (ii) and hence Lemma \ref{Ap and Ainfty}.
\end{proof}

The following examples demonstrate clearly the potential advantages
of the $A_\infty$-upper dimension in handling anisotropic weights.
It retains information in each direction,
which the critical index cannot achieve.

\begin{lemma}\label{exa}
Let $n\geq 2$ and $\vec a:=\{a_i\}_{i=1}^n\subset\mathbb R$.
For any $x:=(x_1,\ldots,x_n)\in\mathbb R^n$,
\begin{align}\label{wa}
w_{\vec a}(x):=\prod_{i=1}^n|x_i|^{a_i}.
\end{align}
Then the following statements hold.
\begin{enumerate}[\rm(i)]
\item $w_{\vec a}\in A_\infty(\mathbb R^n)$
if and only if $a_i\in(-1,\infty)$ for every $i\in\{1,\ldots,n\}$;

\item $w_{\vec a}\in A_1(\mathbb R^n)$
if and only if $a_i\in(-1,0]$ for every $i\in\{1,\ldots,n\}$;

\item if $p\in(1,\infty)$, then $w_{\vec a}\in A_p(\mathbb R^n)$
if and only if  $a_i\in(-1,p-1)$ for every $i\in\{1,\ldots,n\}$;

\item if $w_{\vec a}\in A_\infty(\mathbb R^n)$, then
$r_{w_{\vec a}}=1+\max_{i\in\{1,\ldots,n\}}(a_i)_+$;

\item if $w_{\vec a}\in A_\infty(\mathbb R^n)$, then
$d^{\mathrm{upper}}_\infty(w_{\vec a})=\sum_{i=1}^n(a_i)_+$
is an $A_\infty$-upper dimension of $w_{\vec a}$.
\end{enumerate}
\end{lemma}

\begin{proof}
We first show (i). If $a_i\in(-1,\infty)$ for every $i\in\{1,\ldots,n\}$,
then, for every cube $Q:=I_1\times\cdots\times I_n\subset\mathbb R^n$,
\begin{align*}
&\fint_Qw_{\vec a}(x)\,dx
\exp\left(\fint_Q\log\left[w_{\vec a}(x)^{-1}\right]\,dx\right)\\
&\quad=\prod_{i=1}^n\fint_{I_i}|x_i|^{a_i}\,dx_i
\exp\left(\fint_{I_i}\log|x_i|^{-a_i}\,dx_i\right)
\leq\prod_{i=1}^n[|\cdot|^{a_i}]_{A_\infty(\mathbb R)}<\infty,
\end{align*}
where in the last step we used a well-known fact that
\begin{equation}\label{Ainfty iff}
|\cdot|^{a_i}\in A_\infty(\mathbb R) \Longleftrightarrow a_i\in(-1,\infty).
\end{equation}
Therefore, $w_{\vec a}\in A_\infty(\mathbb R^n)$.
On the other hand, if $w_{\vec a}\in A_\infty(\mathbb R^n)$,
then, for every $i\in\{1,\ldots,n\}$ and every interval $I\subset\mathbb R$,
we find by Jensen's inequality that
\begin{align*}
[w_{\vec a}]_{A_\infty(\mathbb R^n)}
&\geq\fint_{I^n}w_{\vec a}(x)\,dx
\exp\left(\fint_{I^n}\log\left[w_{\vec a}(x)^{-1}\right]\,dx\right)\\
&=\prod_{j=1}^n\fint_{I}|x_j|^{a_j}\,dx_j
\exp\left[\fint_{I}\log\left(|x_j|^{-a_j}\right)\,dx_j\right]
\geq\fint_{I}|x_i|^{a_i}\,dx_i
\exp\left[\fint_{I}\log\left(|x_i|^{-a_i}\right)\,dx_i\right]
\end{align*}
and hence $|\cdot|^{a_i}\in A_\infty(\mathbb R)$.
This, together with \eqref{Ainfty iff},
further implies that $a_i\in(-1,\infty)$,
which completes the proof of (i).
Applying an argument similar to that used in the proof of (i),
we obtain (ii) and (iii).

Next, we prove (iv).
From the definition of the critical index and (iii), we deduce that
\begin{align*}
r_{w_{\vec a}}
=\inf\left\{r\in(1,\infty):\ a_i\in(-1,r-1)\text{ for every }i\in\{1,\ldots,n\} \right\}
=1+\max_{i\in\{1,\ldots,n\}}(a_i)_+.
\end{align*}
This finishes the proof of (iv).

Finally, we show (v).
By the assumption that $w_{\vec a}\in A_\infty$ and (i), we conclude that
$a_i\in(-1,\infty)$ for every $i\in\{1,\ldots,n\}$.
We first prove that
$w_{\vec a}$ has $A_\infty$-upper dimension $\sum_{i=1}^n(a_i)_+$.
From Jensen's inequality and $w_{\vec a}\in A_\infty$,
we infer that, for any cube $Q\subset\mathbb R^n$,
$$
1\leq\fint_Q w_{\vec a}(x)\,dx
\exp\left(\fint_Q\log\left[w_{\vec a}(x)^{-1}\right]\,dx\right)
\leq[w_{\vec a}]_{A_\infty}
$$
and hence
$$
\frac{|Q|}{w_{\vec a}(Q)}\sim \exp\left(\fint_Q\log\left[w_{\vec a}(x)^{-1}\right]\,dx\right).
$$
This, combined with \cite[Lemma 2.41]{bhyy}, further implies that,
for any $\lambda\in[1,\infty)$ and any cube $Q:=I_1\times\cdots\times I_n\subset\mathbb R^n$,
\begin{align}\label{estimate}
&\fint_{\lambda Q}w_{\vec a}(x)\,dx
\exp\left(\fint_Q\log\left[w_{\vec a}(x)^{-1}\right]\,dx\right)\notag\\
&\quad\sim\fint_{\lambda Q}w_{\vec a}(x)\,dx
\left[\fint_Qw_{\vec a}(x)\,dx\right]^{-1}
=\prod_{i=1}^n\fint_{\lambda I_i}|x_i|^{a_i}\,dx_i
\left[\fint_{I_i}|x_i|^{a_i}\,dx_i\right]^{-1}\notag\\
&\quad\sim\prod_{i=1}^n\left(\frac{|c_{I_i}|+\lambda|I_i|}{|c_{I_i}|+|I_i|}\right)^{a_i}
\leq \lambda^{\sum_{i=1}^n(a_i)_+}
\end{align}
and hence $w_{\vec a}$ has $A_\infty$-upper dimension $\sum_{i=1}^n(a_i)_+$.
Now, we show $d^{\mathrm{upper}}_\infty(w_{\vec a})=\sum_{i=1}^n(a_i)_+$.
Assume that
$w_{\vec a}$ has $A_\infty$-upper dimension $d\in[0,\infty)$.
By this and \eqref{estimate}, we conclude that,
for any $\lambda\in[1,\infty)$ and any cube $Q:=I_1\times\cdots\times I_n\subset\mathbb R^n$,
\begin{align*}
\prod_{i=1}^n\left(\frac{|c_{I_i}|+\lambda|I_i|}{|c_{I_i}|+|I_i|}\right)^{a_i}
\sim\fint_{\lambda Q}w_{\vec a}(x)\,dx
\exp\left(\fint_Q\log\left[w_{\vec a}(x)^{-1}\right]\,dx\right)
\lesssim \lambda^d.
\end{align*}
In the above estimate, for any $i\in\{1,\ldots,n\}$, let
$$I_i:=\begin{cases}
[-1,1]&\text{if }a_i\in(0,\infty),\\
[-1,1]+\lambda&\text{if }a_i\in(-1,0],
\end{cases}\ \mathrm{then}\
\lambda^{\sum_{i=1}^n(a_i)_+}
\lesssim\prod_{i=1}^n\left(\frac{|c_{I_i}|+\lambda|I_i|}{|c_{I_i}|+|I_i|}\right)^{a_i}
\lesssim \lambda^d$$
and hence $d\geq \sum_{i=1}^n(a_i)_+$.
This, together with the fact that $w_{\vec a}$ has $A_\infty$-upper dimension $\sum_{i=1}^n(a_i)_+$,
further implies that $d^{\mathrm{upper}}_\infty(w_{\vec a})=\sum_{i=1}^n(a_i)_+$,
which completes the proof of (v) and hence Lemma \ref{exa}.
\end{proof}

Next, we can prove Theorem \ref{dim and critical}.

\begin{proof}[Proof of Theorem \ref{dim and critical}]
We first show (i).
From the definition of $r_w$, we deduce that,
for any $\varepsilon\in(0,\infty)$, the weight satisfies $w\in A_{r_w+\varepsilon}$.
By this and Lemma \ref{Ap and Ainfty}(ii), we obtain that
$w$ has $A_\infty$-upper dimension $d_{\varepsilon}(r_w+\varepsilon-1)$
for some $d_{\varepsilon}\in[0,n)$ and hence
$$
d^{\mathrm{upper}}_\infty(w)
\leq d_{\varepsilon}(r_w+\varepsilon-1)
<n(r_w+\varepsilon-1).
$$
Letting $\varepsilon\in(0,\infty)$ and $\varepsilon\to0$,
we obtain $d^{\mathrm{upper}}_\infty(w)\leq n(r_w-1)$.
This finishes the proof of (i).

We then prove (iii).
Let $n\geq 2$, and let $\vec a\in(-1,\infty)^n$ be any vector with $\min_{i\in\{1,\ldots,n\}}(a_i)_+<\max_{i\in\{1,\ldots,n\}}(a_i)_+$.
Let $w_{\vec a}$ be the same as in \eqref{wa}.
Then, by (iv) and (v) of Lemma \ref{exa}, we obtain
\begin{equation}\label{d<n(r-1)}
d^{\mathrm{upper}}_\infty(w_{\vec a})
=\sum_{i=1}^n(a_i)_+
<n\max_{i\in\{1,\ldots,n\}}(a_i)_+
=n(r_{w_{\vec a}}-1).
\end{equation}
This finishes the proof of (iii).

Finally, we prove (ii). Let now $\vec a=\frac{d}{n}\mathbf{1}$, where $\mathbf{1}=(1,\ldots,1)\in\mathbb R^n$. Then $\min_{i\in\{1,\ldots,n\}}(a_i)_+=\max_{i\in\{1,\ldots,n\}}(a_i)_+=d/n$, and the same computation \eqref{d<n(r-1)}, only with ``$<$'' replaced by ``$=$'', shows that $d^{\mathrm{upper}}_\infty(w_{\vec a})=n(r_{w_{\vec a}}-1)=d$, as claimed.
This finishes the proof of Theorem \ref{dim and critical}.
\end{proof}

\subsection{Inhomogeneous Sequence Spaces}
\label{Inhomogeneous}

In this short section, we formulate the inhomogeneous versions of the main results of Section \ref{Homogeneous} for convenience of reference below. The proofs of these results are immediate by specialising the results of Section \ref{Homogeneous} to vectors and matrices indexed by dyadic cubes as before, but having non-zero coefficients on the small cubes $Q\in\mathscr Q_+$ only.

Let $B:=\{b_{Q,P}\}_{Q,P\in\mathscr{Q}_+}\subset\mathbb{C}$.
For any sequence $\vec t:=\{\vec t_R\}_{R\in\mathscr{Q}_+}\subset\mathbb{C}^m$,
we define $B\vec t:=\{(B\vec t)_Q\}_{Q\in\mathscr{Q}_+}$ by setting,
for any $Q\in\mathscr{Q}_+$,
$$
\left(B\vec t\right)_Q:=\sum_{R\in\mathscr{Q}_+}b_{Q,R}\vec t_R
$$
if the above summation is absolutely convergent.
The following definition of almost diagonal operators in this context is an obvious modification of Definition \ref{def:AD}.

\begin{definition}
Let $D,E,F\in\mathbb{R}$.
An infinite matrix $B:=\{b_{Q,R}\}_{Q,R\in\mathscr{Q}_+}\subset\mathbb{C}$
is said to be \emph{$(D,E,F)$-almost diagonal}
if there exists a positive constant $C$ such that, for any $Q,R\in\mathscr{Q}_+$,
$
|b_{Q,R}|\leq C b_{Q,R}^{DEF},
$
where $b_{Q,R}^{DEF}$ is the same as in \eqref{bDEF}.
\end{definition}

As a simple application of Theorem \ref{ad BF2},
we obtain the following conclusion; we omit the details.

\begin{theorem}\label{ad BF2 in}
Let $s\in\mathbb R$, $\tau\in[0,\infty)$, $p\in(0,\infty)$, $q\in(0,\infty]$, and $W\in A_{p,\infty}$.
Let $(D,E,F)\in\mathbb{R}^3$ satisfy \eqref{ad new2}.
If $B:=\{b_{Q,P}\}_{Q,P\in\mathscr{Q}_+}$ is $(D,E,F)$-almost diagonal,
then $B$ is bounded on $a^{s,\tau}_{p,q}(W)$.
\end{theorem}

\begin{remark}
In Theorem \ref{ad BF2 in}, by Remark \ref{ad conv 2},
we obtain a slightly stronger estimate than the boundedness of $B$;
that is, there exists a positive constant $C$ such that,
for any $\vec t:=\{\vec t_R\}_{R\in\mathscr Q_+}\in a^{s,\tau}_{p,q}(W)$,
\begin{equation*}
\left\|\left\{2^{js}\sum_{Q\in\mathscr{Q}_j}\widetilde{\mathbf{1}}_Q(\cdot)
\sum_{R\in\mathscr{Q}_+}\left|H_j(\cdot)b_{Q,R}\vec{t}_R\right|
\right\}_{j\in\mathbb Z_+}\right\|_{LA_{p,q}^{\tau}}
\leq C\left\|\vec t\right\|_{a^{s,\tau}_{p,q}(W)},
\end{equation*}
where $ LA_{p,q}^{\tau}\in\{LB_{p,q}^\tau,LF_{p,q}^\tau\}$ and,
for any $j\in\mathbb Z_+$, $H_j=W^{\frac{1}{p}}$ or,
for any $j\in\mathbb Z_+$, $H_j=A_j$ with $A_j$ in \eqref{Aj}.
\end{remark}

\begin{definition}
Let $s\in\mathbb R$, $\tau\in[0,\infty)$, $p\in(0,\infty)$, $q\in(0,\infty]$,
and $W\in A_{p,\infty}$.
An infinite matrix $B:=\{b_{Q,R}\}_{Q,R\in\mathscr{Q}_+}\subset\mathbb{C}$
is said to be \emph{$a^{s,\tau}_{p,q}(W)$-almost diagonal}
if it is $(D,E,F)$-almost diagonal with $D,E,F$ in \eqref{ad new2}.
\end{definition}

Applying an argument similar to that used in the proof of \cite[Corollary 9.6]{bhyyp2},
we obtain the following conclusion; we omit the details.

\begin{proposition}
Let $s\in\mathbb R$, $\tau\in[0,\infty)$, $p\in(0,\infty)$, $q\in(0,\infty]$,
and $W\in A_{p,\infty}$.
Let $A:=\{a_{Q,R}\}_{Q,R\in\mathscr Q_+}\subset\mathbb C$
and $B:=\{b_{R,P}\}_{R,P\in\mathscr Q_+}\subset\mathbb C$
be $a^{s,\tau}_{p,q}(W)$-almost diagonal. Then
\begin{equation*}
A\circ B:=\left\{\sum_{R\in\mathscr{Q}_+}a_{Q,R}b_{R,P}\right\}_{Q,P\in\mathscr Q_+}
\end{equation*}
is also $a^{s,\tau}_{p,q}(W)$-almost diagonal.
\end{proposition}

\section{Molecular Characterization and Its Applications}
\label{molecules and more}
This section is split into three subsections. Applying Theorems \ref{phi W}
and \ref{ad BF2 in}, in Subsection \ref{molecular characterization} we establish the smooth
molecular decomposition of $A^{s,\tau}_{p,q}(W)$.
As a first application of this molecular characterization,
we obtain the boundedness of pseudo-differential operators
on $A^{s,\tau}_{p,q}(W)$ in Subsection \ref{psiDO}.
In Section \ref{wavelet decomposition},
using the boundedness of almost diagonal operators on $a^{s,\tau}_{p,q}(W)$
and both the $\varphi$-transform and the molecular characterizations
of $A^{s,\tau}_{p,q}(W)$,
we obtain the wavelet decomposition of $A^{s,\tau}_{p,q}(W)$.
Applying this and the molecular characterization of $A^{s,\tau}_{p,q}(W)$ again,
we further establish the atomic characterization of $A^{s,\tau}_{p,q}(W)$.

\subsection{Molecular Characterization}
\label{molecular characterization}

As an application of Theorem \ref{ad BF2 in},
we establish the smooth molecular decomposition of $A^{s,\tau}_{p,q}(W)$,
but we first need some preparations.

For any $r\in\mathbb R$, let
\begin{equation}\label{ceil}
\begin{cases}
\lfloor r\rfloor:=\max\{k\in\mathbb Z:\ k\leq r\},\\
\lfloor\!\lfloor r\rfloor\!\rfloor:=\max\{k\in\mathbb Z:\ k< r\},
\end{cases}
\begin{cases}
\lceil r\rceil:=\min\{k\in\mathbb Z:\ k\geq r\},\\
\lceil\!\lceil r\rceil\!\rceil:=\min\{k\in\mathbb Z:\ k>r\},
\end{cases}
\end{equation}
and
\begin{equation}\label{r**}
\begin{cases}
r^*:=r-\lfloor r\rfloor\in[0,1),\\
r^{**}:=r-\lfloor\!\lfloor r\rfloor\!\rfloor\in(0,1].
\end{cases}
\end{equation}
Thus, by their definitions, we obviously have, for any $r\in\mathbb R\setminus\mathbb Z$
$$
\lfloor r\rfloor=\lfloor\!\lfloor r\rfloor\!\rfloor<r<\lceil r\rceil=\lceil\!\lceil r\rceil\!\rceil
\text{ and }
r^{*}=r^{**}\in(0,1),
$$
for any $r\in\mathbb Z$
$$
\lfloor r\rfloor=\lceil r\rceil=r,\
\lfloor\!\lfloor r\rfloor\!\rfloor=r-1,\
\lceil\!\lceil r\rceil\!\rceil=r+1,\
r^*=0,\text{ and }
r^{**}=1,
$$
and always $\lceil r\rceil=\lfloor\!\lfloor r\rfloor\!\rfloor+1$
and $\lceil\!\lceil r\rceil\!\rceil=\lfloor r\rfloor+1$.

\begin{definition}
Let $K,M\in[0,\infty)$ and $L,N\in\mathbb{R}$.
For any $K\in[0,\infty)$, $Q\in\mathscr{Q}$, and $x\in\mathbb{R}^n$, let
\begin{equation*}
u_K(x):=(1+|x|)^{-K}\quad\text{and}\quad
(u_K)_Q(x):=|Q|^{-\frac12}u_K\left(\frac{x-x_Q}{\ell(Q)}\right).
\end{equation*}
A function $m_Q$ is called a \emph{(smooth) $(K,L,M,N)$-molecule on a cube $Q$}
if, for every $x,y\in\mathbb R^n$ and every multi-index $\gamma\in\mathbb{Z}_+^n$
in the specified ranges below, it satisfies
$$
|m_Q(x)|\leq(u_{K})_Q(x),\quad
\int_{\mathbb R^n}x^\gamma m_Q(x)\,dx=0\quad\text{ if }\quad |\gamma|\leq L,
$$
$$
|\partial^\gamma m_Q(x)|\leq[\ell(Q)]^{-|\gamma|}(u_{M})_Q(x)\quad\text{ if }\quad |\gamma|<N,
$$
and
$$
|\partial^\gamma m_Q(x)-\partial^\gamma m_Q(y)|
\leq[\ell(Q)]^{-|\gamma|}\left[\frac{|x-y|}{\ell(Q)}\right]^{N^{**}}
\sup_{|z|\leq|x-y|}(u_{M})_Q(x+z)
$$
if $|\gamma|=\lfloor\!\lfloor N\rfloor\!\rfloor$,
where $N^{**}$ and $\lfloor\!\lfloor N\rfloor\!\rfloor$ are the same as, respectively, in \eqref{r**} and \eqref{ceil}.
\end{definition}

We only consider smooth molecules and hence drop the word ``smooth'' for brevity.
Notice that the length $|\gamma|$ of a multi-index is a non-negative integer and hence
any condition for multi-indices $\gamma$ of length
(less than or) equal to a negative number is void.

The following lemma is exactly \cite[Lemma 3.7]{bhyyp3}.

\begin{lemma}\label{like B3}
Let $m_Q$ be a $(K_m,L_m,M_m,N_m)$-molecule on a cube $Q$
and $b_P$ be a $(K_b,L_b,M_b,N_b)$-molecule on a cube $P$,
where $K_m,M_m,K_b,M_b\in(n,\infty)$ and $L_m,N_m,L_b,N_b$ are real numbers.
Then, for any $\alpha\in(0,\infty)$,
there exists a positive constant $C$ such that
$|\langle m_Q,b_P\rangle|\leq C b_{Q,P}^{MGH}$,
where $b_{Q,P}^{MGH}$ is the same as in \eqref{bDEF},
$M:=K_m\wedge M_m\wedge K_b\wedge M_b\in(n,\infty)$,
$G:=\frac{n}{2}+[N_b\wedge\lceil\!\lceil
L_m\rceil\!\rceil\wedge(K_m-n-\alpha)]_+$,
and $H:=\frac{n}{2}+[N_m\wedge\lceil\!\lceil
L_b\rceil\!\rceil\wedge(K_b-n-\alpha)]_+$
with $\lceil\!\lceil\cdot\rceil\!\rceil$ the same as in \eqref{ceil}.
\end{lemma}

\begin{theorem}\label{adMol}
Let $s\in\mathbb R$, $\tau\in[0,\infty)$, $p\in(0,\infty)$, $q\in(0,\infty]$,
and $W\in A_{p,\infty}$.
Let $\{m_Q\}_{Q\in\mathscr Q_+}$ be a family of $(K_m,L_m,M_m,N_m)$-molecules
and $\{b_P\}_{P\in\mathscr Q_+}$ another family of $(K_b,L_b,M_b,N_b)$-molecules,
each on the cube indicated by its subscript.
Then the infinite matrix $\{\langle m_P,b_Q\rangle\}_{Q,P\in\mathscr Q_+}$
is $a^{s,\tau}_{p,q}(W)$-almost diagonal provided that
$$
\begin{cases}
K_m>\widetilde J+\left(\widetilde J-n-\widetilde s\right)_-,\
L_m\geq\widetilde{s}\\
M_m>\widetilde{J},\
N_m>\widetilde{J}-n-\widetilde{s}
\end{cases}
\text{and}\quad
\begin{cases}
K_b>\widetilde J+(\widetilde{s})_-,\
L_b\geq\widetilde J-n-\widetilde s\\
M_b>\widetilde J,\
N_b>\widetilde s,
\end{cases}
$$
where $\widetilde J$ and $\widetilde s$ are the same as in \eqref{tauJ2}.
\end{theorem}

\begin{proof}
The proof consists of a routine verification that, feeding the assumptions of the theorem into Lemma \ref{like B3}, the said lemma gives $|\langle m_Q,b_P\rangle|\leq C b_{Q,P}^{MGH}$ with $(M,G,H)=(D,E,F)$ satisfying the condition \eqref{ad new2}. We omit the details of this verification, which are as in \cite[Theorem 3.8]{bhyyp3}.
\end{proof}

\begin{remark}\label{remark 5.4}
The conclusion of Theorem \ref{adMol} still holds with
$\{b_Q\}_{Q\in\mathscr Q_0}$ and $\{m_P\}_{P\in\mathscr Q_0}$ replaced, respectively,
by a family of $(K_b,-1,M_b,N_b)$-molecules and a family of $(K_m,-1,M_m,N_m)$-molecules,
each on the cube indicated by its subscript, i.e., it is not necessary to require any cancellation conditions on the largest cubes $\mathscr Q_0\subset\mathscr Q_+$.

Indeed, referring to the sketch of the proof of the theorem above, we need to
verify that Lemma \ref{like B3} still provides the desired conclusions.

If both $\ell(Q),\ell(P)<1$, then both $m_Q$ and $b_P$ are unmodified, and Lemma \ref{like B3} gives the same estimate $|\langle m_Q,b_P\rangle|\leq C b_{Q,P}^{MGH}$ as before.

Suppose then that $\ell(Q)<\ell(P)=1$. Then the new $b_P$ has $\lceil\!\lceil L_b\rceil\!\rceil=\lceil\!\lceil-1\rceil\!\rceil=0$, and hence Lemma \ref{like B3} only gives $|\langle m_Q,b_P\rangle|\leq C b_{Q,P}^{M,G,\frac{n}{2}}$. However, referring back to the definition of $b_{Q,P}^{DEF}$, we note that the last parameter $F$ only plays a role when $\ell(P)<\ell(Q)$, which is not the case now that $\ell(Q)<\ell(P)=1$. Hence indeed $b_{Q,P}^{M,G,\frac{n}{2}}=b_{Q,P}^{MGH}$, and we are done in this case.

The case that $\ell(P)<\ell(Q)=1$ is completely symmetric to the previous one.

Finally, if $\ell(Q)=\ell(P)=1$, then the new $m_Q$ and $b_P$ both satisfy $\lceil\!\lceil L_m\rceil\!\rceil=\lceil\!\lceil L_b\rceil\!\rceil=\lceil\!\lceil-1\rceil\!\rceil=0$, so that Lemma \ref{like B3} only gives $|\langle m_Q,b_P\rangle|\leq C b_{Q,P}^{M,\frac{n}{2},\frac{n}{2}}$. But, again referring back to the definition of $b_{Q,P}^{DEF}$, we note both last two parameters $E$ and $F$ have no impact when $\ell(P)=\ell(Q)$. Thus,
$b_{Q,P}^{M,\frac{n}{2},\frac{n}{2}}=b_{Q,P}^{MGH}$, and we are done in all cases.
This finishes the proof of the above claim.
\end{remark}

We further introduce the analysis and the synthesis molecules as follows.

\begin{definition}\label{as mol}
Let $s\in\mathbb R$, $\tau\in[0,\infty)$, $p\in(0,\infty)$, $q\in(0,\infty]$,
and $W\in A_{p,\infty}$.
Let $\widetilde J$ and $\widetilde s$ be the same as in \eqref{tauJ2}. Then
\begin{enumerate}[\rm(i)]
\item A $(K,L,M,N)$-molecule on a cube $Q$ is called
an \emph{$A^{s,\tau}_{p,q}(W)$-synthesis molecule} on $Q$ if
\begin{align}\label{synthesis molecule}
K>\widetilde J+(\widetilde{s})_-,\quad
\begin{cases}
L=-1&\text{if }\ell(Q)=1,\\
L\geq\widetilde J-n-\widetilde s&\text{if }\ell(Q)<1,
\end{cases}\quad
M>\widetilde J,\quad\text{and}\quad
N>\widetilde s;
\end{align}

\item A $(K,L,M,N)$-molecule on a cube $Q$ is called
an \emph{$A^{s,\tau}_{p,q}(W)$-analysis molecule} on $Q$ if
\begin{align}\label{analysis molecule}
K>\widetilde J+\left(\widetilde J-n-\widetilde s\right)_-,\quad
\begin{cases}
L=-1&\text{if }\ell(Q)=1,\\
L\geq\widetilde{s}&\text{if }\ell(Q)<1,
\end{cases}\quad
M>\widetilde{J},\quad\text{ and }\quad
N>\widetilde{J}-n-\widetilde{s}.
\end{align}
\end{enumerate}

A family of functions $\{m_Q\}_{Q\in\mathscr{Q}_+}$ is called
\emph{a family of $A^{s,\tau}_{p,q}(W)$-synthesis (resp. analysis) molecules} if
there exist $K,L,M,N\in\mathbb R$ satisfying \eqref{synthesis molecule}
[resp. \eqref{analysis molecule}] such that,
for any $Q\in\mathscr{Q}_+$, $m_Q$ is a $(K,L,M,N)$-molecule.
\end{definition}

\begin{corollary}\label{83}
Let $s\in\mathbb R$, $\tau\in[0,\infty)$, $p\in(0,\infty)$, $q\in(0,\infty]$, and $W\in A_{p,\infty}$.
Let $\Phi,\Psi\in\mathcal{S}$ satisfy \eqref{19}
and $\varphi,\psi\in\mathcal{S}$ satisfy \eqref{20}.
Suppose that, for both $i\in\{1,2\}$, $\{m_Q^{(i)}\}_{Q\in\mathscr{Q}_+}$
are families of $A^{s,\tau}_{p,q}(W)$-analysis molecules and
$\{b_Q^{(i)}\}_{Q\in\mathscr{Q}_+}$
are families of $A^{s,\tau}_{p,q}(W)$-synthesis molecules.
Then
\begin{enumerate}[\rm(i)]
\item all the matrices
$\{\langle m_P^{(1)},b_Q^{(1)}\rangle\}_{P,Q\in\mathscr{Q}_+}$,
$\{\langle m_P^{(1)},\psi_Q\rangle\}_{P,Q\in\mathscr{Q}_+}$,
and $\{\langle \varphi_P,b_Q^{(1)}\rangle\}_{P,Q\in\mathscr{Q}_+}$
are $a^{s,\tau}_{p,q}(W)$-almost diagonal,
where $\varphi_0$ and $\psi_0$ are replaced, respectively, by $\Phi$ and $\Psi$;
\item if $\vec t=\{\vec t_R\}_{R\in\mathscr Q_+}\in a^{s,\tau}_{p,q}(W)$,
then, for every $P\in\mathscr Q_+$, the series
$$
\vec s_P:=\sum_{Q,R\in\mathscr Q_+}
\left\langle m_P^{(1)},b_Q^{(1)}\right\rangle
\left\langle m_Q^{(2)},b_R^{(2)}\right\rangle\vec t_R
$$
converges unconditionally
and $\vec s:=\{\vec s_P\}_{P\in\mathscr Q_+}$ satisfies
$
\|\vec s\|_{a^{s,\tau}_{p,q}(W)}
\leq C\|\vec t\|_{a^{s,\tau}_{p,q}(W)}
$,
where the positive constant $C$ is independent of $\vec t$,
$\{m_Q\}_{Q\in\mathscr Q_+}$, and $\{b_Q\}_{Q\in\mathscr Q_+}$.
\end{enumerate}
\end{corollary}

\begin{proof}
For each $Q\in\mathscr{Q}_+$, the functions $\varphi_Q$ and $\psi_Q$ are
harmless constant multiples of an $A^{s,\tau}_{p,q}(W)$-analysis molecule
and an $A^{s,\tau}_{p,q}(W)$-synthesis molecule on $Q$.
This, combined with Theorem \ref{adMol} and Remark \ref{remark 5.4},
implies (i). Applying (i) and an argument
used in the proof of \cite[Corollary 3.15(ii)]{bhyyp3}, we obtain (ii).
This finishes the proof of Corollary \ref{83}.
\end{proof}

For later reference, it is useful to observe a condition, under which
no cancellation conditions are needed for synthesis molecules.
This observation is stated in the following lemma.

\begin{lemma}\label{no can}
Let $s\in\mathbb R$, $\tau\in[0,\infty)$, $p\in(0,\infty)$, $q\in(0,\infty]$,
and $W\in A_{p,\infty}$.
Then $A^{s,\tau}_{p,q}(W)$-synthesis molecules have no cancellation conditions if and only if
\begin{equation}\label{no can1}
s>J_\tau-n-n\left(\tau-\frac{1}{p}\right)_+ +\frac{d_{p,\infty}^{\mathrm{upper}}(W)}{p}.
\end{equation}
If $A^{s,\tau}_{p,q}(W)=B^{s,\tau}_{p,q}(W)$, this can be rewritten as
\begin{equation}\label{no canB}
s>\frac{d_{p,\infty}^{\mathrm{upper}}(W)}{p}+
\begin{cases}
\displaystyle n\left(\frac{1}{p}-\tau\right) &\displaystyle \text{if}\quad\tau>\frac1p\quad\text{or}\quad(\tau,q)=\left(\frac1p,\infty\right), \\
\displaystyle n\left(\frac{1}{p}-1\right)_+ & \text{otherwise}.
\end{cases}
\end{equation}
\end{lemma}

\begin{proof}
It is immediate from Definition \ref{as mol} that no cancellation is equivalent to $\widetilde J-n-\widetilde s<0$, where $\widetilde J$ and $\widetilde s$ are the same as in \eqref{tauJ2}. Thus,
\begin{align*}
\widetilde J-n-\widetilde s
&=\left\{J_\tau+\left[(n\widehat\tau)\wedge\frac{d_{p,\infty}^{\mathrm{lower}}(W)}{p}\right]
+\frac{d_{p,\infty}^{\mathrm{upper}}(W)}{p}\right\}-n-\left(s+n\widehat\tau\right) \\
&=J_\tau-n-\left[n\widehat\tau-\frac{d_{p,\infty}^{\mathrm{lower}}(W)}{p}\right]_+
+\frac{d_{p,\infty}^{\mathrm{upper}}(W)}{p}-s.
\end{align*}
Using the definition of $\widehat\tau$, also in \eqref{tauJ2}, we further obtain
\begin{equation*}
E:=n\widehat\tau-\frac{d_{p,\infty}^{\mathrm{lower}}(W)}{p}
=\left[n\left(\tau-\frac{1}{p}\right)+\frac{d_{p,
\infty}^{\mathrm{lower}}(W)}{p}\right]_+-\frac{d_{p,\infty}^{\mathrm{lower}}(W)}{p}
=:[F]_+-\frac{d_{p,\infty}^{\mathrm{lower}}(W)}{p}.
\end{equation*}
If $\tau-\frac{1}{p}\geq 0$, then $[F]_+=F$, and we find $E=n(\tau-\frac{1}{p})$ after cancelling out $\pm d_{p,\infty}^{\mathrm{lower}}(W)/p$. If  $\tau-\frac{1}{p}< 0$, then $F<d_{p,\infty}^{\mathrm{lower}}(W)/p$, hence $E<0$, and therefore $E_+=0$. Combining the cases, we see that $E_+=n(\tau-\frac{1}{p})_+$. Hence $\widetilde J-n-\widetilde s<0$ if and only if \eqref{no can1} holds, as claimed.

By Definition \ref{d4.4}, only Triebel--Lizorkin spaces can be critical, so for a Besov-type space $A^{s,\tau}_{p,q}(W)=B^{s,\tau}_{p,q}(W)$, there are only the supercritical and the
subcritical options, and the cases for $J_\tau$ simplify to
\begin{equation*}
J_\tau-n=\begin{cases} n-n=0 &\displaystyle \text{if}\quad \tau>\frac1p\quad\text{or}\quad (\tau,q)=\left(\frac1p,\infty\right), \\
J-n=\displaystyle\frac{n}{\min\{1,p\}}-n=n\left(\frac{1}{p}-1\right)_+ &\text{otherwise}.\end{cases}
\end{equation*}
Substituting back, the condition \eqref{no canB} follows,
which completes the proof of Lemma \ref{no can}.
\end{proof}

By Corollary \ref{83}(ii) and an argument
used in the proof of \cite[Lemma 3.16]{bhyyp3},
we obtain the following conclusion; we omit the details.

\begin{lemma}\label{88}
Let $s\in\mathbb R$, $\tau\in[0,\infty)$, $p\in(0,\infty)$, $q\in(0,\infty]$, and $W\in A_{p,\infty}$.
Let $\vec f\in A^{s,\tau}_{p,q}(W)$ and $m_P$ be an
$A^{s,\tau}_{p,q}(W)$-analysis molecule on $P\in\mathscr{Q}$.
Let $\Phi,\Psi\in\mathcal{S}$ satisfy \eqref{19},
$\varphi,\psi\in\mathcal{S}$ satisfy \eqref{20},
and all of them satisfy \eqref{21}.
Then
$$
\left\langle\vec f,m_P\right\rangle
:=\sum_{R\in\mathscr{Q}_0}\left\langle\vec f,\Phi_R\right\rangle\langle\Psi_R,m_P\rangle
+\sum_{j=1}^\infty\sum_{R\in\mathscr{Q}_j}
\left\langle\vec f,\varphi_R\right\rangle\langle\psi_R,m_P\rangle
$$
converges absolutely and its value is independent of
the choices of $\Phi,\varphi,\Psi$, and $\psi$.
\end{lemma}

Repeating an argument used in the proof of \cite[Theorem 3.17]{bhyyp3}
with \cite[Corollary 3.15 and Lemma 3.16]{bhyyp3} replaced, respectively,
by Corollary \ref{83} and Lemma \ref{88},
we obtain the following conclusion; we omit the details.

\begin{theorem}\label{89}
Let $s\in\mathbb R$, $\tau\in[0,\infty)$, $p\in(0,\infty)$, $q\in(0,\infty]$, and $W\in A_{p,\infty}$.
\begin{enumerate}[{\rm (i)}]
\item\label{89ana} If $\{m_Q\}_{Q\in\mathscr{Q}_+}$
is a family of $A^{s,\tau}_{p,q}(W)$-analysis molecules,
each on the cube indicated by its subscript,
then there exists a positive constant $C$ such that,
for any $\vec{f}\in A^{s,\tau}_{p,q}(W)$,
$$
\left\|\left\{\left\langle\vec f,
m_Q\right\rangle\right\}_{Q\in\mathscr{Q}_+}\right\|_{a^{s,\tau}_{p,q}(W)}
\leq C\left\|\vec{f}\right\|_{A^{s,\tau}_{p,q}(W)}.
$$
\item\label{89syn} If $\{b_Q\}_{Q\in\mathscr{Q}_+}$
is a family of $A^{s,\tau}_{p,q}(W)$-synthesis molecules,
each on the cube indicated by its subscript,
then, for any $\vec t:=\{\vec t_Q\}_{Q\in\mathscr{Q}_+}
\in a^{s,\tau}_{p,q}(W)$,
there exists $\vec f\in A^{s,\tau}_{p,q}(W)$ such that
$\vec f=\sum_{Q\in\mathscr{Q}_+}\vec t_Qb_Q$
in $(\mathcal{S}')^m$ and there exists a positive constant $C$,
independent of $\{\vec t_Q\}_{Q\in\mathscr{Q}_+}$
and $\{b_Q\}_{Q\in\mathscr{Q}_+}$, such that
$\|\vec f\|_{A^{s,\tau}_{p,q}(W)}
\leq C\|\vec t\|_{a^{s,\tau}_{p,q}(W)}$.
\end{enumerate}
\end{theorem}

Applying an argument used in the proof of \cite[Proposition 3.18]{bhyyp3},
we obtain the following conclusion; we omit the details.

\begin{proposition}
Let $s\in\mathbb R$, $\tau\in[0,\infty)$,
$p\in(0,\infty)$, $q\in(0,\infty]$, and $W\in A_{p,\infty}$.
Then $(\mathcal{S})^m\subset A^{s,\tau}_{p,q}(W)$.
Moreover, there exist $M\in\mathbb{Z}_+$
and a positive constant $C$ such that,
for any $\vec f\in(\mathcal{S})^m$,
$$
\left\|\vec f\right\|_{A^{s,\tau}_{p,q}(W)}
\leq C\left\|\vec f\right\|_{\vec{S}_M},\ \mathrm{where}\
\left\|\vec f\right\|_{\vec{S}_M}
:=\sup_{\gamma\in\mathbb{Z}_+^n,\,|\gamma|\leq M}
\sup_{x\in\mathbb{R}^n}\left|\left(\partial^\gamma\vec f\right)(x)\right|(1+|x|)^{n+M+|\gamma|}.
$$
\end{proposition}

\subsection{Boundedness of Pseudo-Differential Operators}
\label{psiDO}

As a first application of the molecular characterization
(Theorem \ref{89}), we obtain the boundedness of pseudo-differential
operators on $A^{s,\tau}_{p,q}(W)$.
Indeed, this will be easy, given the work already done above and
elsewhere: in \cite{ysy10}, it has been shown that appropriate scaled versions
$|Q|^{\frac{u}{n}}T(\varphi_Q)$ of the images $T(\varphi_Q)$ of the basic
resolution functions $\varphi_Q$ under a pseudo-differential operator $T$
are indeed molecules, and it suffices to combine this with
the molecular characterization of $A^{s,\tau}_{p,q}(W)$ from above.

Turning to the details, we first recall the concepts of the class $S_{1,1}^u$ and its related pseudo-differential operators.

\begin{definition}
Let $u\in\mathbb R$.
The \emph{class $S_{1,1}^u$} is defined to be the set of all functions
$a\in C^\infty(\mathbb{R}^n\times(\mathbb{R}^n\setminus\{\mathbf{0}\}))$
such that, for any $\alpha,\beta\in\mathbb{Z}_+^n$,
$$
\sup_{x\in\mathbb{R}^n,\,\xi\in\mathbb{R}^n\setminus\{\mathbf{0}\}}
(1+|\xi|)^{-u-|\alpha|+|\beta|}
\left|\partial_x^\alpha\partial_\xi^\beta a(x,\xi)\right|
<\infty.
$$
\end{definition}

\begin{definition}
Let $u\in\mathbb R$ and $a\in S_{1,1}^u$.
The \emph{pseudo-differential operator $a(x,D)$ with symbol $a$}
is defined by setting,
for any $f\in\mathcal{S}$ and $x\in\mathbb{R}^n$,
$$
a(x,D)(f):=\int_{\mathbb{R}^n}a(x,\xi)\widehat{f}(\xi)e^{ix\cdot\xi}\,d\xi.
$$
If $a(x,\xi)=a(\xi)$ is independent of $x\in\mathbb R^n$, then the corresponding $a(x,D)=a(D)$ is also called a \emph{Fourier multiplier operator}.
\end{definition}

Let $u\in\mathbb{Z}_+$ and $a\in\dot{S}_{1,1}^u$.
Then $a(x,D)$ maps $\mathcal{S}$ into $\mathcal{S}$
(see, for instance, \cite[p.\,257, Remark 2]{t92}).
In particular, its \emph{formal adjoint} $a(x,D)^{\#}$ is well defined by setting,
for any $f\in\mathcal{S}'$ and $\phi\in\mathcal{S}$,
\begin{equation}\label{formal adj}
\left\langle a(x,D)^{\#}f,\phi\right\rangle
:=\langle f,a(x,D)\phi\rangle.
\end{equation}

Applying an argument similar to that used in the proof of
\cite[Theorem 5.1]{ysy10}, we obtain the following theorem.
For the sake of completeness, we give the details of its proof here.

\begin{theorem}\label{pseudo}
Let $s\in\mathbb R$, $\tau\in[0,\infty)$, $p\in(0,\infty)$, $q\in(0,\infty]$,
and $W\in A_{p,\infty}$.
Let $\widetilde J$ and $\widetilde s$ be the same as in \eqref{tauJ2}.
Let $u\in\mathbb R$, $a\in S_{1,1}^u$,
and $a(x,D)$ be a pseudo-differential operator with symbol $a$.
Assume that its formal adjoint $a(x,D)^{\#}$ satisfies,
for all $\beta\in\mathbb{Z}_+^n$ with $|\beta|\leq\widetilde J-n-\widetilde s$, the condition
\begin{equation}\label{227}
a(x,D)^{\#}\left(x^\beta\right)\in\mathcal{P},
\end{equation}
where \eqref{227} is void when $\widetilde J-n-\widetilde s<0$.
Then $a(x,D)$ can be extended to a continuous linear mapping
from $A_{p,q}^{s+u,\tau}(W)$ to $A^{s,\tau}_{p,q}(W)$.
\end{theorem}

\begin{proof}
For brevity, in what follows, we write $T:=a(x,D)$.
Let $\vec f\in A_{p,q}^{s+u,\tau}(W)$.
Let $\Phi,\Psi\in\mathcal{S}$ satisfy \eqref{19} and
$\varphi,\psi\in\mathcal{S}$ satisfy \eqref{20} such that \eqref{21} holds.
Then, by Lemma \ref{7}, we find that
$$
f=\sum_{Q\in\mathscr{Q}_+}
\left\langle f,\varphi_Q\right\rangle\psi_Q
$$
in $(\mathcal{S}')^m$, where $\varphi_0$ and $\psi_0$ are
replaced, respectively, by $\Phi$ and $ \Psi$.
Therefore, if $\ell(Q)=1$, then $\varphi_Q=\Phi_Q$ and $\psi_Q=\Psi_Q$. Let
\begin{align*}
T\left(\vec f\right)
:=\sum_{Q\in\mathscr{Q}_+}
\left\langle\vec f,\varphi_Q\right\rangle T\left(\psi_Q\right)
=\sum_{Q\in\mathscr{Q}_+}
\left[|Q|^{-\frac{u}{n}}\left(S_\varphi\vec f\right)_Q\right]
\left[|Q|^{\frac{u}{n}}T\left(\psi_Q\right)\right]
\end{align*}
in $(\mathcal{S}')^m$.
Now, we show that $T:\ A_{p,q}^{s+u,\tau}(W) \to A^{s,\tau}_{p,q}(W)$
is bounded. From Theorem \ref{phi W}, we infer that
\begin{equation}\label{102}
\left\|\left\{|Q|^{-\frac{u}{n}}\left(S_\varphi\vec f\right)_Q
\right\}_{Q\in\mathscr{Q}_+}\right\|_{a^{s,\tau}_{p,q}(W)}
=\left\|S_\varphi\vec f\right\|_{a_{p,q}^{s+u,\tau}(W)}
\lesssim\left\|\vec f\right\|_{A_{p,q}^{s+u,\tau}(W)}.
\end{equation}
By the proof of \cite[Theorem 5.1]{ysy10},
we conclude that, for any $Q\in\mathscr{Q}$, the function
$|Q|^{\frac{u}{n}}T(\varphi_Q)$ is a harmless constant multiple of
an $A^{s,\tau}_{p,q}(W)$-synthesis molecule on $Q$.
This, together with \eqref{102} and Theorem \ref{89}(ii),
further implies that $T(\vec f)$ is well defined and
$$
\left\|T\left(\vec f\right)\right\|_{A^{s,\tau}_{p,q}(W)}
\lesssim\left\|\left\{|Q|^{-\frac{u}{n}}\left(S_\varphi\vec f\right)_Q
\right\}_{Q\in\mathscr{Q}_+}\right\|_{a_{p,q}^{s,\tau}(W)}
\lesssim\left\|\vec f\right\|_{A_{p,q}^{s+u,\tau}(W)},
$$
which completes the proof of Theorem \ref{pseudo}.
\end{proof}

\begin{remark}
When $m=1$ and $W\equiv 1$,
Theorem \ref{pseudo} contains \cite[Theorem 5.1]{ysy10},
where an upper bound is imposed on $\tau$.
\end{remark}

In what follows, we use $C_{\mathrm{c}}^\infty$
[resp. $C_{\mathrm{c}}^\infty(\mathbb R^n\setminus\{\mathbf{0}\})$]
to denote the set of all infinitely differentiable functions on $\mathbb R^n$
(resp. $\mathbb R^n\setminus\{\mathbf{0}\}$) with compact support.

\begin{corollary}\label{Fmult}
Let $s\in\mathbb R$, $\tau\in[0,\infty)$, $p\in(0,\infty)$, $q\in(0,\infty]$,
and $W\in A_{p,\infty}$.
Let $u\in\mathbb R$, $a\in S_{1,1}^u$ be independent of $x\in\mathbb R^n$,
and $a(D)$ be the Fourier multiplier operator with symbol $a$.
Then $a(D)$ can be extended to a continuous linear mapping
from $A_{p,q}^{s+u,\tau}(W)$ to $A^{s,\tau}_{p,q}(W)$.
\end{corollary}

\begin{proof}
This is immediate from Theorem \ref{pseudo} once we verify that condition \eqref{227} is automatically satisfied by Fourier multiplier operators $a(x,D)=a(D)$. To this end, we need to show that, if $f\in\mathcal S'$ is a polynomial, then $a(D)^{\#}f\in\mathcal S'$, as defined in \eqref{formal adj}, is also a polynomial. For this we use a well-known characterization: $f\in\mathcal S'$ is a polynomial if and only if its Fourier transform $\widehat f$ is supported at the origin (see \cite[Corollary 2.4.2]{g14c} for ``if''; the other direction is easy).

Let $\psi\in C_{\mathrm c}^\infty(\mathbb R^n\setminus\{\mathbf{0}\})$. Then
\begin{equation*}
\left\langle \left(a(D)^{\#}f\right)^{\wedge}, \psi\right\rangle
=\left\langle a(D)^{\#}f, \widehat\psi\right\rangle
=\left\langle f, a(D)\widehat\psi\right\rangle
=\left\langle\widehat f, \left(a(D)\widehat\psi\right)^{\vee}\right\rangle,
\end{equation*}
where, for any $\xi\in\mathbb R^n\setminus\{\mathbf{0}\}$,
\begin{equation*}
\left(a(D)\widehat\psi\right)^{\vee}(\xi)
=a(\xi)\left(\widehat\psi\right)^{\wedge}(\xi)=a(\xi)\psi(-\xi)
\end{equation*}
is supported in $\mathbb R^n\setminus\{\mathbf 0\}$, since $\psi$ has this property. Recalling $\widehat f$ that supported at $\{\mathbf 0\}$, it follows that $\langle \widehat f, (a(D)\widehat\psi)^{\vee}\rangle=0$, and hence $\langle (a(D)^{\#}f)^{\wedge}, \psi\rangle=0$. Since this holds for all $\psi\in C_{\mathrm c}^\infty(\mathbb R^n\setminus\{\mathbf{0}\})$, we conclude that $(a(D)^{\#}f)^{\wedge}$ is supported on $\{\mathbf 0\}$ and hence $a(D)^{\#}f$ is a polynomial, as claimed.
This finishes the proof of Corollary \ref{Fmult}.
\end{proof}

\subsection{Wavelet and Atomic Decompositions}
\label{wavelet decomposition}

In this section, we establish the wavelet decomposition of
matrix-weighted Besov-type and Triebel--Lizorkin-type spaces via the Daubechies wavelet.
As a consequence, we obtain an atomic characterization of these spaces.

In what follows, for any $\mathscr{N}\in\mathbb{N}$,
we use $C^{\mathscr{N}}$ to denote the set of all
$\mathscr{N}$ times continuously differentiable functions on $\mathbb{R}^n$.

\begin{definition}
Let $\mathscr{N}\in\mathbb{N}$.
Then the functions $\{\theta^{(i)}\}_{i=0}^{2^n-1}$
are called \emph{Daubechies wavelets of class $C^{\mathscr N}$}
if each $\theta^{(i)}\in C^{\mathscr N}$ is real-valued with bounded support and if
$$
\left\{\theta^{(0)}_Q:\ Q\in\mathscr{Q}_0\right\}\cup
\left\{\theta^{(i)}_Q :\ i\in\{1,\ldots,2^n-1\},\ Q\in\mathscr{Q}_+\right\}
$$
is an orthonormal basis of $L^2$.
\end{definition}

The following wavelet basis was constructed by Daubechies
(see, for instance, \cite{d88}).

\begin{lemma}
For every $\mathscr{N}\in\mathbb{N}$,
there exist Daubechies wavelets of class $C^{\mathscr N}$.
\end{lemma}

\begin{remark}
Let $\{\theta^{(i)}\}_{i=0}^{2^n-1}$ be Daubechies wavelets of class $C^{\mathscr N}$.
By \cite[Corollary 5.5.2]{d92}, we find that,
for every $i\in\{1,\ldots,2^n-1\}$ and $\alpha\in\mathbb{Z}_+^n$ with $|\alpha|\leq\mathscr{N}$,
$
\int_{\mathbb{R}^n}x^\alpha\theta^{(i)}(x)\,dx=0.
$
\end{remark}

\begin{corollary}\label{88 corollary}
Let $s\in\mathbb R$, $\tau\in[0,\infty)$,
$p\in(0,\infty)$, $q\in(0,\infty]$, and $W\in A_{p,\infty}$.
Let $\widetilde J$ and $\widetilde s$ be the same as in \eqref{tauJ2},
let $\mathscr{N}\in\mathbb{N}$ satisfy
\begin{equation}\label{NDau}
\mathscr{N}>\max\left\{\widetilde J-n-\widetilde s,\widetilde s\right\},
\end{equation}
and let $\{\theta^{(i)}\}_{i=0}^{2^n-1}$ be Daubechies wavelets of class $C^{\mathscr{N}}$.
\begin{enumerate}[\rm(i)]
\item For $i=0$ and any $Q\in\mathscr{Q}_0$
or for any $i\in\{1,\ldots,2^n-1\}$ and $Q\in\mathscr Q_+$,
the wavelet $\theta^{(i)}_Q$ is a constant multiple of
both an $A^{s,\tau}_{p,q}(W)$-analysis molecule on $Q$
and an $A^{s,\tau}_{p,q}(W)$-synthesis molecule on $Q$
and the constant is independent of $Q\in\mathscr Q$.
\item Let $\Phi$ and $\Psi$ satisfy \eqref{19},
$\varphi$ and $\psi$ satisfy \eqref{20},
and all of them satisfy \eqref{21}.
Let $\vec f\in A^{s,\tau}_{p,q}(W)$.
Then, for $i=0$ and any $Q\in\mathscr{Q}_0$
or for any $i\in\{1,\ldots,2^n-1\}$ and $Q\in\mathscr Q_+$,
$$
\left\langle\vec f,\theta^{(i)}_Q\right\rangle
:=\sum_{R\in\mathscr{Q}_0}\left\langle\vec f,\Phi_R\right\rangle
\left\langle\Psi_R,\theta^{(i)}_Q\right\rangle
+\sum_{j=1}^\infty\sum_{R\in\mathscr{Q}_j}
\left\langle\vec f,\varphi_R\right\rangle
\left\langle\psi_R,\theta^{(i)}_Q\right\rangle
$$
converges absolutely and its value is independent of the choice of $\varphi$ and $\psi$.
\end{enumerate}
\end{corollary}

\begin{proof}
Repeating an argument used in the proof of \cite[Corollary 4.8(i)]{bhyyp3},
we obtain (i). From (i) and Lemma \ref{88}, we deduce (ii),
which completes the proof of Corollary \ref{88 corollary}.
\end{proof}

Next, we establish the wavelet decomposition of $A^{s,\tau}_{p,q}(\mathbb{A})$
via Daubechies wavelets.

\begin{theorem}\label{wavelet 2}
Let $s\in\mathbb R$, $\tau\in[0,\infty)$,
$p\in(0,\infty)$, $q\in(0,\infty]$, and $W\in A_{p,\infty}$.
Let $\widetilde J$ and $\widetilde s$ be the same as in \eqref{tauJ2},
$\mathscr{N}\in\mathbb{N}$ satisfy \eqref{NDau},
and $\{\theta^{(i)}\}_{i=0}^{2^n-1}$ be Daubechies wavelets of class $C^{\mathscr{N}}$.
Then, for any $\vec f\in A^{s,\tau}_{p,q}(W)$,
\begin{equation}\label{204}
\vec f=\sum_{Q\in\mathscr{Q}_0}\left\langle\vec f,\theta^{(0)}_Q\right\rangle\theta^{(0)}_Q
+\sum_{i=1}^{2^n-1}\sum_{Q\in\mathscr{Q}_+}
\left\langle\vec f,\theta^{(i)}_Q\right\rangle\theta^{(i)}_Q
\end{equation}
in $(\mathcal{S}')^m$ and
$
\|\vec f\|_{A^{s,\tau}_{p,q}(W)}
\sim\|\vec f\|_{A_{p,q}^{s,\tau}(W)_\mathrm{w}}
$,
where the positive equivalence constants are independent of $\vec f$ and
\begin{align*}
\left\|\vec f\right\|_{A_{p,q}^{s,\tau}(W)_\mathrm{w}}
:=\left\|\left\{\left\langle\vec f,\theta^{(0)}_Q \right\rangle\mathbf{1}_{Q\in\mathscr{Q}_0}
\right\}_{Q\in\mathscr{Q}_+}\right\|_{a^{s,\tau}_{p,q}(W)}
+\sum_{i=1}^{2^n-1}\left\|\left\{\left\langle\vec f,\theta^{(i)}_Q \right\rangle
\right\}_{Q\in\mathscr{Q}_+}\right\|_{a^{s,\tau}_{p,q}(W)}.
\end{align*}
\end{theorem}

\begin{proof}
We first prove \eqref{204}. Let $\Phi$ and $\Psi$ satisfy \eqref{19},
$\varphi$ and $\psi$ satisfy \eqref{20}, and all of them satisfy \eqref{21}.
Let $\vec f\in A^{s,\tau}_{p,q}(W)$. By Theorem \ref{phi A}, we find that
$
S_\varphi\vec f
=\{\langle\vec f,\varphi_Q\rangle\}_{Q\in\mathscr{Q}_+}
\in a^{s,\tau}_{p,q}(W),
$
where $\varphi_0$ is replaced by $\Phi$
and hence $\varphi_Q=\Phi_Q$ if $\ell(Q)=1$.
From Corollary \ref{88 corollary}, we infer that,
for $i=0$ and any $Q\in\mathscr{Q}_0$
or for any $i\in\{1,\ldots,2^n-1\}$ and $Q\in\mathscr Q_+$,
the wavelet $\theta^{(i)}_Q$ is a constant multiple of
both an $A^{s,\tau}_{p,q}(W)$-analysis molecule on $Q$
and an $A^{s,\tau}_{p,q}(W)$-synthesis molecule on $Q$.
These, combined with Corollary \ref{83}(ii), further imply that,
for any $\phi\in\mathcal{S}$,
\begin{align*}
\mathrm{I}
:=\sum_{Q\in\mathscr{Q}_0}\sum_{R\in\mathscr{Q}_+}
\left\langle\vec f,\varphi_R\right\rangle
\left\langle\psi_R,\theta^{(0)}_Q\right\rangle
\left\langle\theta^{(0)}_Q,\phi\right\rangle
+\sum_{i=1}^{2^n-1}\sum_{Q,R\in\mathscr{Q}_+}
\left\langle\vec f,\varphi_R\right\rangle
\left\langle\psi_R,\theta^{(i)}_Q\right\rangle
\left\langle\theta^{(i)}_Q,\phi\right\rangle\notag
\end{align*}
converges unconditionally, where $\psi_0$ is replaced by $\Psi$.

On the one hand, from the fact that
$\{\theta_Q^{(0)}:\ Q\in\mathscr Q_0\}\cup
\{\theta_Q^{(i)}:\ Q\in\mathscr Q_+,\ i\in\{1,\ldots,2^n-1\}\}$
is an orthonormal basis of $L^2$ and Lemma \ref{7}, it follows that
\begin{align*}
\mathrm{I}
&=\sum_{R\in\mathscr{Q}_+}\left\langle\vec f,\varphi_R\right\rangle
\left[\sum_{Q\in\mathscr{Q}_0}
\left\langle\psi_R,\theta^{(0)}_Q\right\rangle
\left\langle\theta^{(0)}_Q,\phi\right\rangle
+\sum_{i=1}^{2^n-1}\sum_{Q\in\mathscr{Q}_+}
\left\langle\psi_R,\theta^{(i)}_Q\right\rangle
\left\langle\theta^{(i)}_Q,\phi\right\rangle\right]\\
&=\sum_{R\in\mathscr{Q}_+}\left\langle\vec f,\varphi_R\right\rangle\langle\psi_R,\phi\rangle
=\left\langle\vec f,\phi\right\rangle.
\end{align*}
On the other hand, by Corollary \ref{88 corollary}(ii), we find that
\begin{align*}
\mathrm{I}
&=\sum_{Q\in\mathscr{Q}_0}\left[\sum_{R\in\mathscr{Q}_+}
\left\langle\vec f,\varphi_R\right\rangle
\left\langle\psi_R,\theta^{(0)}_Q\right\rangle\right]
\left\langle\theta^{(0)}_Q,\phi\right\rangle\\
&\quad+\sum_{i=1}^{2^n-1}\sum_{Q\in\mathscr{Q}_+}\left[\sum_{R\in\mathscr{Q}_+}
\left\langle\vec f,\varphi_R\right\rangle
\left\langle\psi_R,\theta^{(i)}_Q\right\rangle\right]
\left\langle\theta^{(i)}_Q,\phi\right\rangle\\
&=\sum_{Q\in\mathscr{Q}_0}
\left\langle\vec f,\theta^{(0)}_Q\right\rangle
\left\langle\theta^{(0)}_Q,\phi\right\rangle
+\sum_{i=1}^{2^n-1}\sum_{Q\in\mathscr{Q}_+}
\left\langle\vec f,\theta^{(i)}_Q\right\rangle
\left\langle\theta^{(i)}_Q,\phi\right\rangle.
\end{align*}
These finish the proof of \eqref{204}.

Now, we show that $\|\vec f\|_{A^{s,\tau}_{p,q}(W)_{\mathrm{w}}}
\lesssim\|\vec f\|_{A^{s,\tau}_{p,q}(W)}$.
From Corollary \ref{88 corollary}, we deduce that,
for $i=0$ and any $Q\in\mathscr{Q}_0$
or for any $i\in\{1,\ldots,2^n-1\}$ and $Q\in\mathscr Q_+$,
the wavelet $\theta^{(i)}_Q$ is a constant multiple of
an $A^{s,\tau}_{p,q}(W)$-analysis molecule on $Q$.
This, together with Theorem \ref{89}(i), further implies that
\begin{align}\label{66x}
\left\|\left\{\left\langle\vec f,\theta^{(0)}_Q \right\rangle\mathbf{1}_{Q\in\mathscr{Q}_0}
\right\}_{Q\in\mathscr{Q}_+}\right\|_{a^{s,\tau}_{p,q}(W)}
\lesssim\left\|\vec f\right\|_{A^{s,\tau}_{p,q}(W)}
\end{align}
and, for any $i\in\{1,\ldots,2^n-1\}$,
\begin{align}\label{66y}
\left\|\left\{\left\langle\vec f,\theta^{(i)}_Q \right\rangle
\right\}_{Q\in\mathscr{Q}_+}\right\|_{a^{s,\tau}_{p,q}(W)}
\lesssim\left\|\vec f\right\|_{A^{s,\tau}_{p,q}(W)}.
\end{align}
Therefore, $\|\vec f\|_{A^{s,\tau}_{p,q}(W)_{\mathrm{w}}}
\lesssim\|\vec f\|_{A^{s,\tau}_{p,q}(W)}$.

Finally, we prove that $\|\vec f\|_{A^{s,\tau}_{p,q}(W)}
\lesssim\|\vec f\|_{A^{s,\tau}_{p,q}(W)_{\mathrm{w}}}$.
For any $i\in\{0,\ldots,2^n-1\}$, let
\begin{align*}
\vec f^{(i)}:=
\begin{cases}
\displaystyle\sum_{Q\in\mathscr{Q}_0}\left\langle\vec f,
\theta^{(0)}_Q\right\rangle\theta^{(0)}_Q & \mathrm{if}\ i=0,\\
\displaystyle\sum_{Q\in\mathscr{Q}_+}\left\langle\vec f,
\theta^{(i)}_Q\right\rangle\theta^{(i)}_Q
&\mathrm{if}\ i\in\{1,\ldots,2^n-1\}.
\end{cases}
\end{align*}
From Corollary \ref{88 corollary}, we infer that,
for $i=0$ and any $Q\in\mathscr{Q}_0$
or for any $i\in\{1,\ldots,2^n-1\}$ and $Q\in\mathscr Q_+$,
the wavelet $\theta^{(i)}_Q$ is a constant multiple of
an $A^{s,\tau}_{p,q}(W)$-synthesis molecule on $Q$. This,
combined with \eqref{204}, \eqref{66x}, and \eqref{66y}, further implies that
\begin{align*}
\left\|\vec f\right\|_{A^{s,\tau}_{p,q}(W)}
\lesssim\sum_{i=0}^{2^n-1}\left\|\vec f^{(i)}\right\|_{A^{s,\tau}_{p,q}(W)}
\lesssim\left\|\vec f\right\|_{A_{p,q}^{s,\tau}(W)_\mathrm{w}},
\end{align*}
which completes the proof of Theorem \ref{wavelet 2}.
\end{proof}

It is known that there already exist wavelet decompositions
and their applications for many other function spaces;
see, for instance, \cite{fr04,is09,lyysu,lhy20,lyz24,r13,ro03,ysy10}.

Using wavelet and molecular characterizations of $A^{s,\tau}_{p,q}(W)$,
we can establish its atomic characterization.

\begin{definition}
Let $r\in(0,\infty)$ and $L,N\in\mathbb R$.
A function $ a_Q $ is called an \emph{$(r,L,N)$-atom} on a cube $Q$ if
\begin{enumerate}[\rm(i)]
\item $\operatorname{supp}a_Q\subset rQ$;
\item $\int_{\mathbb{R}^n}x^\gamma a_Q(x)\,dx=0$
for any $\gamma\in\mathbb{Z}_+^n$ and $|\gamma|\leq L$;
\item $|D^\gamma a_Q(x)|\leq|Q|^{-\frac12-\frac{|\gamma|}{n}}$
for any $x\in\mathbb{R}^n$ and $\gamma\in\mathbb{Z}_+^n$ with $|\gamma|\leq N$.
\end{enumerate}
\end{definition}

\begin{theorem}\label{atom}
Let $s\in\mathbb R$, $\tau\in[0,\infty)$,
$p\in(0,\infty)$, $q\in(0,\infty]$, and $W\in A_{p,\infty}$.
Let $L,N\in\mathbb{R}$ satisfy
\begin{equation}\label{900}
L\geq\widetilde J-n-\widetilde s\quad\text{and}\quad N>\widetilde s,
\end{equation}
where $\widetilde J$ and $\widetilde s$ are the same as in \eqref{tauJ2}.
Then the following two conclusions hold.
\begin{enumerate}[{\rm (i)}]
\item There exists $r\in(0,\infty)$, depending only on $L\vee N$, such that,
for any $\vec f\in A^{s,\tau}_{p,q}(W)$,
there exist a sequence $\vec t:=\{\vec t_Q\}_{Q\in\mathscr{Q}_+}\in a^{s,\tau}_{p,q}(W)$,
$(r,-1,N)$-atoms $\{a_Q\}_{Q\in\mathscr{Q}_0}$,
and $(r,L,N)$-atoms $\{a_Q\}_{j\in\mathbb N,\,Q\in\mathscr{Q}_j}$,
each on the cube indicated by its subscript,
such that $\vec f=\sum_{Q\in\mathscr{Q}_+}\vec t_Q a_Q$ in $(\mathcal{S}')^m$.
Moreover, there exists a constant $C$,
independent of $\vec f$, such that
$$
\left\|\vec t \right\|_{a^{s,\tau}_{p,q}(W)}
\leq C\left\|\vec{f}\right\|_{A^{s,\tau}_{p,q}(W)}.
$$
\item Conversely, if $r\in(0,\infty)$,
$\{a_Q\}_{Q\in\mathscr{Q}_0}$ is a family of $(r,-1,N)$-atoms,
and $\{a_Q\}_{j\in\mathbb N,\,Q\in\mathscr{Q}_j}$ is a family of $(r,L,N)$-atoms,
each on the cube indicated by its subscript,
then, for each $\vec t:=\{\vec t_Q\}_{Q\in\mathscr{Q}_+}\in a^{s,\tau}_{p,q}(W)$,
there exists $\vec f\in A^{s,\tau}_{p,q}(W)$ such that
$\vec f=\sum_{Q\in\mathscr{Q}_+}\vec t_Qa_Q$ in $(\mathcal{S}')^m$.
Moreover, there exists a constant $C$,
independent of $\vec t$ and $\{a_Q\}_{Q\in\mathscr{Q}_+}$, such that
$
\|\vec f\|_{A^{s,\tau}_{p,q}(W)}
\leq C\|\vec t\|_{a^{s,\tau}_{p,q}(W)}.
$
\end{enumerate}
\end{theorem}

\begin{proof}
Applying an argument similar to that used in the proof of
\cite[Theorem 4.13(i)]{bhyyp3} with \cite[Theorem 4.10]{bhyyp3}
replaced by Theorem \ref{wavelet 2}, we obtain (i). Next, we show (ii).
By the definitions of atoms and molecules, we find that
an $(r,L,N)$-atom is (a harmless constant multiple of)
a $(K,L,M,N)$-molecule for any $K,M\in\mathbb R$.
This, together with \eqref{900}, further implies that
an $(r,L,N)$-atom is (a harmless constant multiple of)
an $A^{s,\tau}_{p,q}(W)$-synthesis molecule.
From this and Theorem \ref{89}(ii), we deduce (ii),
which completes the proof of Theorem \ref{atom}.
\end{proof}

\begin{remark}\label{no can2}
Notice that, both atoms in Theorem \ref{atom}
and $A^{s,\tau}_{p,q}(W)$-synthesis molecules in
Definition \ref{as mol}(i) have no cancellation
are equivalent to $\widetilde{J}-n-\widetilde{s}<0$.
From this and Lemma \ref{no can}, we infer that
atoms in Theorem \ref{atom}(i) have no cancellation
if and only if \eqref{no can1} hold.
If $A^{s,\tau}_{p,q}(W)=B^{s,\tau}_{p,q}(W)$,
\eqref{no can1} can be rewritten as \eqref{no canB}.
\end{remark}

\section{Boundedness of Classical Operators}
\label{classical}

This section is devoted to give more applications of the obtained theorems
in previous sections. In Subsection \ref{trace theorems},
using the boundedness of almost diagonal operators on $a^{s,\tau}_{p,q}(W)$
and both the molecular and the wavelet characterizations of $A^{s,\tau}_{p,q}(W)$,
we establish the trace theorem of $A^{s,\tau}_{p,q}(W)$.
In Subsection \ref{Pointwise Multipliers},
applying the boundedness of pseudo-differential operators on $A^{s,\tau}_{p,q}(W)$
and the atomic characterization of $A^{s,\tau}_{p,q}(W)$,
we find a class of pointwise multipliers of $A^{s,\tau}_{p,q}(W)$.
Finally, in the last Subsection \ref{C-Z operators},
we obtain the boundedness of Calder\'on--Zygmund operators on $A^{s,\tau}_{p,q}(W)$
by using the boundedness of almost diagonal operators on $a^{s,\tau}_{p,q}(W)$
and the $\varphi$-transform characterization of $A^{s,\tau}_{p,q}(W)$.

\subsection{Trace Theorems}
\label{trace theorems}

In this section, we focus on the boundedness of trace operators
and extension operators for $A^{s,\tau}_{p,q}(W)$.
Let us make some conventions on symbols first, which are only used in this section.
Let $n\geq2$ and $\mathscr{Q}_+(\mathbb{R}^n)$
[resp. $\mathscr{Q}_+(\mathbb{R}^{n-1})$]
be the set of all dyadic cubes on $\mathbb{R}^n$
(resp. $\mathbb{R}^{n-1}$).
Since the underlying space under consideration may vary,
we use $a^s_{p,q}(\mathbb{R}^n)$,
$A^{s,\tau}_{p,q}(W,\mathbb{R}^n)$,
and $a^{s,\tau}_{p,q}(W,\mathbb{R}^n)$
instead of $a^s_{p,q}$, $A^{s,\tau}_{p,q}(W)$,
and $a^{s,\tau}_{p,q}(W)$ in this section.
We denote a point $x\in\mathbb{R}^n$ by $x=(x',x_n)$,
where $x'\in\mathbb{R}^{n-1}$ and $x_n\in\mathbb{R}$.
We also denote $\lambda\in\{0,1\}^n$ by $\lambda=(\lambda',\lambda_n)$,
where $\lambda'\in\{0,1\}^{n-1}$ and $\lambda_n\in\{0,1\}$.
Let $\mathbf{0}_n$ be the origin of $\mathbb R^n$.
To establish the trace theorem for $A^{s,\tau}_{p,q}(W,\mathbb{R}^n)$,
we need more details about the Daubechies wavelet (see, for instance, \cite{d88}).

\begin{lemma}\label{wavelet basis 3}
For any $\mathscr{N}\in\mathbb{N}$,
there exist two real-valued $C^{\mathscr{N}}(\mathbb{R})$ functions
$\varphi$ and $\psi$ with bounded support such that,
for any $n\in\mathbb{N}$,
$$
\left\{\theta^{(\mathbf{0}_n)}_Q:\ Q\in\mathscr{Q}_0\right\}
\cup\left\{\theta^{(\lambda)}_Q:\ Q\in\mathscr{Q}_+,\ \lambda\in\Lambda_n:=\{0,1\}^n\setminus\{\mathbf{0}_n\}\right\}
$$
forms an orthonormal basis of $L^2(\mathbb R^n)$,
where, for each $\lambda:=(\lambda_1,\ldots,\lambda_n)\in\{0,1\}^n$
and $x:=(x_1,\ldots,x_n)\in\mathbb{R}^n$, we define
$
\theta^{(\lambda)}(x):=\prod_{i=1}^n\phi^{(\lambda_i)}(x_i)
$
with $\phi^{(0)}:=\varphi$ and $\phi^{(1)}:=\psi$.
\end{lemma}

\begin{remark}\label{k0}
In Lemma \ref{wavelet basis 3}, from \cite[Remark 5.2]{bhyyp3}, it follows that
there exists $k_0\in\mathbb{Z}$ such that $\varphi(-k_0)\neq0$.
\end{remark}

Let $W\in A_{p,\infty}(\mathbb{R}^n,\mathbb{C}^m)$ and
$V\in A_{p,\infty}(\mathbb{R}^{n-1},\mathbb{C}^m)$.
Let $\mathscr{N}\in\mathbb{N}$ satisfy \eqref{NDau}
for the parameters corresponding to both $B^{s,\tau}_{p,q}(W,\mathbb{R}^n)$
and $B^{s-\frac{1}{p},\frac{n}{n-1}\tau}_{p,q}(V,\mathbb{R}^{n-1})$,
and let $\{\theta^{(\lambda)}\}_{\lambda\in\{0,1\}^n}
\subset C^{\mathscr{N}}(\mathbb{R}^n)$
and $\{\theta^{(\lambda')}\}_{\lambda'\in\{0,1\}^{n-1}}
\subset C^{\mathscr{N}}(\mathbb{R}^{n-1})$
be the same as in Lemma \ref{wavelet basis 3}.
For any $I\in\mathscr{Q}_+(\mathbb{R}^{n-1})$ and $k\in\mathbb{Z}$, let
$
Q(I,k):=I\times[\ell(I)k,\ell(I)(k+1)).
$
By their construction, it is easy to prove that,
for any $Q\in\mathscr{Q}_+(\mathbb{R}^n)$,
there exist a unique $I\in\mathscr{Q}_+(\mathbb{R}^{n-1})$ and a unique $k\in\mathbb{Z}$
such that $ Q=Q(I,k)$, and we let $I(Q):=I$.

Now, we can define the trace and the extension operators.
For any $\lambda:=(\lambda',\lambda_n)\in\Lambda_n$,
any $Q:=Q(I,k)\in\mathscr{Q}_+(\mathbb{R}^n)$
with $I\in\mathscr{Q}_+(\mathbb{R}^{n-1})$ and $k\in\mathbb{Z}$,
and any $x'\in\mathbb{R}^{n-1}$, let
\begin{equation*}
\left[\operatorname{Tr}\theta^{(\lambda)}_Q\right](x')
:=\theta^{(\lambda)}_Q(x',0)
=[\ell(Q)]^{-\frac12}\theta^{(\lambda')}_{I(Q)}(x')\phi^{(\lambda_n)}(-k).
\end{equation*}
For any functions $g$ and $h$, respectively,
on $\mathbb{R}^{n-1}$ and $\mathbb{R}$, let
$(g\otimes h)(x):=g(x')h(x_n)$
for any $x:=(x',x_n)\in\mathbb{R}^n$.
For any $\lambda'\in\Lambda_{n-1}$,
$I\in\mathscr{Q}_+(\mathbb{R}^{n-1})$,
and $x:=(x',x_n)\in\mathbb{R}^n$, let
\begin{align*}
\left[\operatorname{Ext}\theta^{(\lambda')}_I\right](x)
:=&\,\frac{[\ell(I)]^{\frac12}}{\varphi(-k_0)}
\left[\theta^{(\lambda')}\otimes\varphi\right]_{Q(I,k_0)}(x) \\
=&\,\frac{[\ell(I)]^{\frac12}}{\varphi(-k_0)}
\theta^{((\lambda',0))}_{Q(I,k_0)}(x)
=\frac{1}{\varphi(-k_0)}\theta^{(\lambda')}_I(x')\varphi\left(\frac{x_n}{\ell(I)}-k_0 \right),
\end{align*}
where $\varphi$ and $k_0$ are the same as, respectively,
in Lemma \ref{wavelet basis 3} and Remark \ref{k0}.
For $\lambda=\mathbf{0}_n$ and any $Q\in\mathscr{Q}_0(\mathbb R^n)$ or
$\lambda'=\mathbf{0}_{n-1}$ and any $I\in\mathscr{Q}_0(\mathbb R^{n-1})$,
we have analogous definitions.
Then, for any $x'\in\mathbb{R}^{n-1}$,
\begin{equation*}
\left(\operatorname{Tr}\circ\operatorname{Ext}\right)\left[\theta^{(\lambda')}_I\right](x')
=\frac{[\ell(I)]^{\frac12}}{\varphi(-k_0)}
\left[\operatorname{Tr}\theta^{((\lambda',0))}_{Q(I,k_0)}\right](x')
=\theta^{(\lambda')}_{I}(x').
\end{equation*}

In analogy with our results in the homogeneous $A_p$-weighted situation \cite{bhyyp3}, we obtain the following results about the mapping properties of the trace and extension operators. With all the tools from the previous sections available, the proofs are very similar to those of \cite[Theorems 5.6, 5.10, 5.13, and 5.16]{bhyyp3}. However, we take this opportunity to revisit the argument in order to streamline some details. In particular, in contrast to the formulation of four separate results and proofs for $\{\operatorname{Tr},\operatorname{Ext}\}\times\{B,F\}$ in \cite{bhyyp3}, we now state and prove just two results, one for the trace and the other for the extension, each covering all relevant spaces at once.

The main difference between the $A_p$ situation of \cite{bhyyp3} and the present $A_{p,\infty}$ setting is that the lower bound for the admissible range of the smoothness parameter $s$ in the trace theorem will now have the additional term $d_{p,\infty}^{\mathrm{upper}}(V)/p$ that was not needed in the case of $A_p$ weights in \cite{bhyyp3}.

\begin{theorem}\label{thm Tr}
Let $\tau\in[0,\infty)$, $p\in(0,\infty)$, and $q\in(0,\infty]$.
Let $W\in A_{p,\infty}(\mathbb{R}^n,\mathbb{C}^m)$ and
$V\in A_{p,\infty}(\mathbb{R}^{n-1},\mathbb{C}^m)$.
Let $s\in(\frac{1}{p}+E+\frac{d_{p,\infty}^{\mathrm{upper}}(V)}{p},\infty)$, where
\begin{align}\label{E B}
E:=(n-1)\begin{cases}
\displaystyle \frac{1}{p}-\frac{n}{n-1}\tau
&\displaystyle \text{if }\frac{n}{n-1}\tau>\frac{1}{p},
\text{ or }\left(\frac{n}{n-1}\tau,q\right)=\left(\frac{1}{p},\infty\right)\text{ and } A=B,\\
\displaystyle \left(\frac{1}{p}-1\right)_+&\text{otherwise}.
\end{cases}
\end{align}
Then the trace operator $\operatorname{Tr}$ can be extended to a
continuous linear operator
\begin{equation*}
\operatorname{Tr}:\ A^{s,\tau}_{p,q}(W,\mathbb{R}^n)
\to B^{s-\frac{1}{p},\frac{n}{n-1}\tau}_{p,r}(V,\mathbb{R}^{n-1}),
\quad\text{where}\quad
r:=\begin{cases} q & \text{if}\quad A=B, \\ p & \text{if}\quad A=F,\end{cases}
\end{equation*}
if and only if there exists a positive constant $C$ such that,
for any $I\in\mathscr{Q}_+(\mathbb{R}^{n-1})$ and $\vec z\in\mathbb{C}^m$,
\begin{equation}\label{116}
\fint_I\left|V^{\frac{1}{p}}(x')\vec z\right|^p\,dx'
\leq C\fint_{Q(I,0)}\left|W^{\frac{1}{p}}(x)\vec z\right|^p\,dx.
\end{equation}
\end{theorem}

\begin{remark}\label{rem Tr}
By Lemma \ref{no can}, the assumption on $s$ in Theorem \ref{thm Tr} is precisely the requirement that no cancellation conditions are required for molecules of the target space $B^{s-\frac1p,\frac{n}{n-1}\tau}_{p,r}(V,\mathbb R^{n-1})$. Indeed, this is immediate by substituting the parameters $\check s:=s-\frac{1}{p}$, $\check\tau:=\frac{n}{n-1}\tau$, $\check n:=n-1$, $\check p:=p$, $\check q:=r$, and $\check W:=V$ in place of $s,\tau,n,p,q,W$ in \eqref{no canB} of Lemma \ref{no can}. [Note that only the finiteness or not of the parameter $\check q$ plays a role in this condition; if $A=F$, then $\check q=r=p\in(0,\infty)$ is necessarily finite, which explains the condition ``$(\frac{n}{n-1}\tau,q)=(\frac1p,\infty)$ and $A=B$'' in \eqref{E B}.]

This makes the assumption on $s$ somewhat natural: while at least some form of size and smoothness of a function are inherited by its restriction to a hyperplane, one clearly cannot expect anything like this for cancellation.
\end{remark}

\begin{theorem}\label{thm Ext}
Let $s\in\mathbb{R}$, $\tau\in[0,\infty)$, $p\in(0,\infty)$, $q\in(0,\infty]$,
$W\in A_{p,\infty}(\mathbb{R}^n,\mathbb{C}^m)$,
and $V\in A_{p,\infty}(\mathbb{R}^{n-1},\mathbb{C}^m)$.
Assume that there exists a positive constant $C$ such that,
for any $I\in\mathscr{Q}_+(\mathbb{R}^{n-1})$
and $\vec z\in\mathbb{C}^m$,
\begin{align}\label{116x}
\fint_{Q(I,0)}\left|W^{\frac{1}{p}}(x)\vec z\right|^p\,dx
\leq C\fint_I\left|V^{\frac{1}{p}}(x')\vec z\right|^p\,dx'.
\end{align}
Then the extension operator $\operatorname{Ext}$ can be extended to a
continuous linear operator
\begin{align*}
\operatorname{Ext}:\
B^{s-\frac{1}{p},\frac{n}{n-1}\tau}_{p,r}(V,\mathbb{R}^{n-1})
\to A^{s,\tau}_{p,q}(W,\mathbb{R}^n),\quad\text{where}\quad
r:=\begin{cases} q & \text{if}\quad A=B, \\ p & \text{if}\quad A=F.\end{cases}
\end{align*}
Furthermore, if $s\in(\frac{1}{p}+E+\frac{d_{p,\infty}^{\mathrm{upper}}(V)}{p},\infty)$,
where $E$ is the same as in \eqref{E B},
and if \eqref{116} also holds,
then $\operatorname{Tr}\circ\operatorname{Ext}$ is the identity on
$B^{s-\frac{1}{p},\frac{n}{n-1}\tau}_{p,q}(V,\mathbb{R}^{n-1})$.
\end{theorem}

\begin{remark}\label{116y}
The conditions \eqref{116} and \eqref{116x} are clearly equivalent,
respectively, to $|A_{I,V}\vec z| \lesssim |A_{Q(I,0),W}\vec z| $
and $|A_{Q(I,0),W}\vec z|\lesssim   |A_{I,V}\vec z|$,
where $A_{I,V}$ and $A_{Q,W}$ are reducing operators of order $p$, respectively,
for $V$ on $I$ and for $W$ on $Q$ and the implicit positive
constants are independent of $I\in\mathscr{Q}_+(\mathbb{R}^{n-1})$
and $\vec z\in\mathbb{C}^m$.

On the other hand, the weak doubling property of the reducing operators
of $W\in A_{p,\infty}$ implies that, for any $k\in\mathbb Z$,
$|A_{Q(I,0),W}\vec z|\sim_k  |A_{Q(I,k),W}\vec z|$ with the positive
equivalence constants depending on $k$, but independent of
$I\in\mathscr{Q}_+(\mathbb{R}^{n-1})$
and $\vec z\in\mathbb{C}^m$.
\end{remark}

The proofs of Theorems \ref{thm Tr} and \ref{thm Ext} follow the same general strategy as their counterparts in the $A_p$ settings given in \cite[Theorems 5.6, 5.10, 5.13, and 5.16]{bhyyp3}, so we will only provide an outline to indicate the changes from using the weaker $A_{p,\infty}$ condition here. Note that our Theorem \ref{thm Tr} corresponds to the union of \cite[Theorems 5.6 and 5.13]{bhyyp3}, where we have combined the Besov and the Triebel--Lizorkin cases into a unified statement. Similarly, our Theorem \ref{thm Ext} corresponds to the union of \cite[Theorems 5.10 and 5.16]{bhyyp3}.

A key ingredient of both the trace and the extension theorems is the following norm equivalence on the level of (unweighted) sequence spaces. It is implicit in the proofs of \cite[Theorems 5.6, 5.10, 5.13, and 5.16]{bhyyp3}, but we now take the opportunity to state it explicitly.

\begin{lemma}\label{seq Tr}
For any $k\in\mathbb Z$ and
$u=\{u_I\}_{I\in\mathscr Q(\mathbb R^{n-1})}\subset\mathbb C$, define
$u^{(k)}:=\{u_Q^{(k)}\}_{Q\in\mathscr Q(\mathbb R^{n})}\subset\mathbb C$ by setting
\begin{equation*}
u^{(k)}_Q:=\begin{cases} u_I & \text{if}\quad Q=Q(I,k)\quad\text{for some}\quad I\in\mathscr Q(\mathbb R^{n-1}), \\
0 & \text{otherwise}.\end{cases}
\end{equation*}
For any $s\in\mathbb R$, $\tau\in[0,\infty)$, $p\in(0,\infty)$, and $q\in(0,\infty]$, one has
\begin{equation*}
\left\| u^{(k)} \right\|_{\dot b^{s,\tau}_{p,q}(\mathbb R^n)}\leq
\left\|\left\{[\ell(I)]^{-\frac12}u_I\right\}_{I\in\mathscr Q(\mathbb R^{n-1})}\right\|_{\dot b^{s-\frac1p,\frac{n}{n-1}\tau}_{p,q}(\mathbb R^{n-1})}
\lesssim_k
\left\| u^{(k)} \right\|_{\dot b^{s,\tau}_{p,q}(\mathbb R^n)},
\end{equation*}
where the implicit positive constant depends on $k$, but independent of $u$.
\end{lemma}

The constant in the last estimate is allowed to depend on $k$ because Lemma \ref{seq Tr} will only be used for boundedly many values of $k$ such that $|k|\leq N$ for some fixed number $N$.

\begin{proof}[Proof of lemma \ref{seq Tr}]
Let $\check s:=s-\frac1p$, $\check\tau:=\frac{n}{n-1}\tau$, $\check n:=n-1$,
and $\check u_I:=[\ell(I)]^{-\frac12}u_I$. The norm in
$\dot b^{\check s,\check\tau}_{p,q}(\mathbb R^{\check n})$ is the supremum over
$J\in\mathscr Q(\mathbb R^{\check n})$ of
$[\ell(J)]^{-\check n\check\tau}\check U(J)=[\ell(J)]^{-n\tau}\check U(J)$, where
\begin{align*}
\check U(J)
:=
\left[\sum_{j=j_J}^\infty\left\{\sum_{I\in\mathscr Q_j(\mathbb R^{\check n}),I\subset J}|\check u_I|^p
[\ell(I)]^{\check n-(\check s+\frac{\check n}{2})p}
\right\}^{\frac{q}{p}}\right]^{\frac{1}{q}}
=\left[\sum_{j=j_J}^\infty\left\{
\sum_{I\in\mathscr Q_j(\mathbb R^{n-1}),\,I\subset J}|u_I|^p [\ell(I)]^{n-(s+\frac{n}{2})p}
\right\}^{\frac{q}{p}}\right]^{\frac{1}{q}}
\end{align*}
after simple algebra with the exponent. On the other hand, the norm in $\dot b^{s,\tau}_{p,q}(\mathbb R^n)$ is the supremum over $P\in\mathscr Q(\mathbb R^n)$ of $[\ell(P)]^{-n\tau}U(P)$, where
\begin{equation*}
U(P):=
\left[\sum_{j=j_P}^\infty\left\{\sum_{I\in\mathscr Q_j(\mathbb R^{n-1}),\,Q(I,k)\subset P}|u_I |^p [\ell(I)]^{n-(s+\frac{n}{2})p}
\right\}^{\frac{q}{p}}\right]^{\frac{1}{q}}.
\end{equation*}

Now any $P\in\mathscr Q(\mathbb R^n)$ has the form $P=Q(J,\kappa)$ for some $J=J_P\in\mathscr Q(\mathbb R^{n-1})$ and $\kappa\in\mathbb Z$. If $Q(I,k):=I\times[k,k+1)\ell(I)\subset Q(J,\kappa):=J\times[\kappa,\kappa+1)\ell(J)$, then $I\subset J$. Then it is immediate that $[\ell(P)]^{-n\tau}U(P)\leq [\ell(J_P)]^{-n\tau}\check U(J_P)$, and taking the supremum over $P\in\mathscr Q(\mathbb R^n)$ proves the first estimate claimed in the lemma.

For the other direction, we apply \cite[Lemma 5.5]{bhyyp3},
which says the following: For any $J\in\mathscr Q(\mathbb R^{n-1})$
and $k\in\mathbb Z$, there exists $P_J\in\mathscr Q(\mathbb R^n)$
with $\ell(P_J)\leq 2(|k|+1)\ell(J)$ such that all
$I\in\mathscr Q(\mathbb R^{n-1})$ with $I\subset J$ satisfy $Q(I,k)\subset P_J$.
Thus, $\check U(J)\leq U(P_J)$ and
$[\ell(J)]^{-n\tau}\leq [2(|k|+1)]^{n\tau}[\ell(P_J)]^{-n\tau}$,
hence $[\ell(J)]^{-n\tau}\check U(J)\lesssim_k [\ell(P_J)]^{-n\tau} U(P_J)$,
and taking the supremum over $J\in\mathscr Q(\mathbb R^{n-1})$ proves
the second claimed estimate, which completes the proof of Lemma \ref{seq Tr}.
\end{proof}

Another proof ingredient in the Triebel--Lizorkin-type case
is presented in the following lemma.

\begin{lemma}\label{F no q}
Let $u^{(k)}$ be the same as in Lemma \ref{seq Tr} with $k\in\mathbb Z$.
Then, for all $s\in\mathbb R$, $\tau\in[0,\infty)$, $p\in(0,\infty)$, and $q,r\in(0,\infty]$, we have $\|u^{(k)}\|_{\dot f^{s,\tau}_{p,q}(\mathbb R^n)}
\sim \|u^{(k)}\|_{\dot f^{s,\tau}_{p,r}(\mathbb R^n)}$,
i.e., the Triebel--Lizorkin type norms of $u^{(k)}$ are essentially independent of the microscopic parameter.
\end{lemma}

\begin{proof}
By \cite[Lemma 5.15]{bhyyp3}, if $E_Q\subset Q\in\mathscr Q(\mathbb R^n)$ are measurable subsets with  $|E_Q|\geq\delta |Q|$ for some fixed $\delta\in(0,1)$, then, for any $t\in\mathbb C^{\mathscr Q}$, we have
\begin{equation*}
\|t\|_{\dot f^{s,\tau}_{p,q}(\mathbb R^n)}
\sim\sup_{P\in\mathscr Q(\mathbb R^n)}|P|^{-n\tau}
\left\|\left\{ 2^{j(s+\frac{n}{2})}\sum_{Q\in\mathscr Q_j(\mathbb R^n),\,Q\subset P} t_Q\mathbf{1}_{E_Q}\right\}_{j=j_P}^\infty\right\|_{L^p\ell^q}.
\end{equation*}
We apply this to $t=u^{(k)}$ with
\begin{equation*}
E_Q:=I\times\left[k+\frac13,k+\frac23\right)\ell(I)\quad\text{if}\quad Q=Q(I,k).
\end{equation*}
[The choice of other $E_Q$ is irrelevant because the coefficients $u^{(k)}_Q$ vanish.] As observed in \cite[right after (5.25)]{bhyyp3}, the sets $E_{Q(I.k)}$ are pairwise disjoint when $I\in\mathscr Q(\mathbb R^{n-1})$ varies but $k\in\mathbb Z$ is fixed. Now
\begin{equation*}
\left\|u^{(k)}\right\|_{\dot f^{s,\tau}_{p,q}(\mathbb R^n)}
\sim\sup_{P\in\mathscr Q(\mathbb R^n)}|P|^{-n\tau}
\left\|\left\{ 2^{j(s+\frac{n}{2})}\sum_{I\in\mathscr Q_j(\mathbb R^{n-1}),Q(I,k)\subset P} u_I
\mathbf{1}_{E_{Q(I,k)}}\right\}_{j=j_P}^\infty\right\|_{L^p\ell^q},
\end{equation*}
and the sequence of which we take the $\ell^q$ norm on
the right-hand side has at most one non-zero element. For such sequences,
all $\ell^q$ norms with $q\in(0,\infty]$ are clearly equal.
This finishes the proof of Lemma \ref{F no q}.
\end{proof}

\begin{corollary}\label{dis Tr}
Let $u^{(k)}$ be the same as in Lemma \ref{seq Tr} with $k\in\mathbb Z$.
Then, for all $s\in\mathbb R$, $\tau\in[0,\infty)$, $p\in(0,\infty)$, and $q,r\in(0,\infty]$,
one has
\begin{equation*}
\left\|\left\{[\ell(I)]^{-1/2}u_I\right\}_{I\in\mathscr Q(\mathbb R^{n-1})}\right\|_{\dot b^{s-\frac1p,\frac{n}{n-1}\tau}_{p,r}(\mathbb R^{n-1})}\sim_k
\left\|u^{(k)}\right\|_{\dot a^{s,\tau}_{p,q}(\mathbb R^n)},\quad
\text{where}\quad r:=\begin{cases} q & \text{if}\quad a=b, \\ p & \text{if}\quad a=f\end{cases}
\end{equation*}
and the positive equivalence constants depend on $k$, but independent of $u$.
\end{corollary}

\begin{proof}
For $a=b$, this is a restatement of Lemma \ref{seq Tr}, so let us consider the case $a=f$ and $r=p$. Then, using Lemma \ref{seq Tr}, the coincidence of Besov-type and Triebel--Lizorkin-type spaces when $p=q$, and Lemma \ref{F no q}, we obtain
\begin{align*}
\left\|\left\{[\ell(I)]^{-1/2}u_I\right\}_{I\in\mathscr Q(\mathbb R^{n-1})}\right\|_{\dot b^{s-\frac1p,\frac{n}{n-1}\tau}_{p,p}(\mathbb R^{n-1})}
\sim_k \left\|u^{(k)}\right\|_{\dot b^{s,\tau}_{p,p}(\mathbb R^n)}
= \left\|u^{(k)}\right\|_{\dot f^{s,\tau}_{p,p}(\mathbb R^n)}
\sim \left\|u^{(k)}\right\|_{\dot f^{s,\tau}_{p,q}(\mathbb R^n)},
\end{align*}
which is the claimed equivalence in the remaining case $a=f$.
This finishes the proof of Corollary \ref{dis Tr}.
\end{proof}

Now we are ready to show Theorem \ref{thm Tr}.

\begin{proof}[Proof of Theorem \ref{thm Tr}]
For the sufficiency of \eqref{116}, we need to estimate the $B^{s-\frac1p,\frac{n}{n-1}\tau}_{p,r}(V,\mathbb R^{n-1})$ norm of
\begin{equation}\label{Tr f expa}
\operatorname{Tr}\vec f=\sum_{\lambda\in\{0,1\}^n}\sum_{Q\in\mathscr Q_\lambda(\mathbb R^n)}
\left\langle \vec f,\theta^\lambda_Q\right\rangle \theta^\lambda_Q(\cdot,0),\quad\text{where}\quad
\mathscr Q_\lambda:=\begin{cases}\mathscr Q_0 & \text{if}\quad\lambda = \mathbf 0_n, \\
\mathscr Q_+ & \text{otherwise}.\end{cases}
\end{equation}
Parameterising $Q=Q(I,k)$ with $I\in\mathscr Q_+(\mathbb R^{n-1})$ and $k\in\mathbb Z$, we note that $\theta_{Q(I,k)}^\lambda(\cdot,0)$ is non-zero only if $|k| \leq N$ for some fixed $N\in\mathbb N$. Moreover, one can check that $[\ell(Q)]^{\frac12}\theta_Q^\lambda(\cdot,0)$ is a $B^{s-\frac1p,\frac{n}{n-1}\tau}_{p,r}(V,\mathbb R^{n-1})$ synthesis molecule on $Q$; here we used the observation (Remark \ref{rem Tr}) that our assumption on $s$ guarantees that no cancellation conditions are required for such molecules. Hence, by
the previous identity and the quasi-triangle inequality,
the molecular characterization [Theorem \ref{89}\eqref{89syn}],
the equivalence of pointwise-weighted and average-weighted sequence spaces (Theorem \ref{37}),
assumption \eqref{116} relating $V$ and $W$ via Remark \ref{rem Tr},
the discrete trace estimate of Corollary \ref{dis Tr},
another application of Theorem \ref{37},
and finally the wavelet characterization (Theorem \ref{wavelet 2}),
we obtain the following chain of estimates
\begin{align}\label{Tr pf}
&\left\|\operatorname{Tr}\vec f\right\|_{B^{s-\frac1p,\frac{n}{n-1}\tau}_{p,r}(V,\mathbb R^{n-1})}\notag\\
&\quad\lesssim\sum_{\lambda\in\{0,1\}^n}\sum_{|k|\leq N}
\left\|\sum_{I\in\mathscr Q_\lambda(\mathbb R^{n-1})}
[\ell(I)]^{-\frac12} \left\langle \vec f,\theta^\lambda_{Q(I,k)}\right\rangle [\ell(I)]^{\frac12}\theta^\lambda_{Q(I,k)}(\cdot,0)\right\|_{B^{s-\frac1p,
\frac{n}{n-1}\tau}_{p,r}(V,\mathbb R^{n-1})} \notag\\
&\quad\lesssim\sum_{\lambda\in\{0,1\}^n}\sum_{|k|\leq N}
\left\|\left\{   [\ell(I)]^{-\frac12}\left\langle \vec f,\theta^\lambda_{Q(I,k)}\right\rangle\right\}_{I\in\mathscr Q_\lambda(\mathbb R^{n-1})}
\right\|_{b^{s-\frac1p,\frac{n}{n-1}\tau}_{p,r}(V,\mathbb R^{n-1})} \notag\\
&\quad\lesssim\sum_{\lambda\in\{0,1\}^n}\sum_{|k|\leq N}
\left\|\left\{   [\ell(I)]^{-\frac12} \left|A_{V,I}
\left\langle \vec f,\theta^\lambda_{Q(I,k)}\right\rangle \right|\right\}_{I\in\mathscr Q_\lambda(\mathbb R^{n-1})}
\right\|_{b^{s-\frac1p,\frac{n}{n-1}\tau}_{p,r}(\mathbb R^{n-1})} \notag\\
&\quad\lesssim\sum_{\lambda\in\{0,1\}^n}\sum_{|k|\leq N}
\left\|\left\{   [\ell(I)]^{-\frac12} \left|A_{W,Q(I,k)}
\left\langle \vec f,\theta^\lambda_{Q(I,k)}\right\rangle \right|\right\}_{I\in\mathscr Q_\lambda(\mathbb R^{n-1})}
\right\|_{b^{s-\frac1p,\frac{n}{n-1}\tau}_{p,r}(\mathbb R^{n-1})} \notag\\
&\quad\lesssim\sum_{\lambda\in\{0,1\}^n}
\left\|\left\{\left|A_{W,Q}\left\langle\vec f,\theta^\lambda_{Q}\right\rangle \right|\right\}_{Q\in\mathscr Q_\lambda(\mathbb R^{n})}
\right\|_{a^{s,\tau}_{p,q}(\mathbb R^{n})} \notag\\
&\quad\lesssim\sum_{\lambda\in\{0,1\}^n}
\left\|\left\{\left\langle\vec f,\theta^\lambda_{Q}\right\rangle \right\}_{Q\in\mathscr Q_\lambda(\mathbb R^{n})}
\right\|_{a^{s,\tau}_{p,q}(W,\mathbb R^{n})}
\lesssim  \left\| \vec f\right\|_{A^{s,\tau}_{p,q}(W,\mathbb R^n)}.
\end{align}

The necessity of \eqref{116} is seen by ``testing'' the estimate $\|\operatorname{Tr}\vec f\|_{B^{s-\frac1p,\frac{n}{n-1}\tau}_{p,r}(V,\mathbb R^{n-1})} \lesssim  \| \vec f\|_{A^{s,\tau}_{p,q}(W,\mathbb R^n)}$ with functions that only have a single term with $Q=Q(I,k_0)$ in the expansion \eqref{Tr f expa}. Then one checks that there is actually ``$\sim$'' in place of ``$\lesssim$'' in all steps of \eqref{Tr pf}, except possibly the one, where $V$ is replaced by $W$. Hence the only way to have the required trace inequality is to still have ``$\lesssim$'' in this step, and this is seen to be equivalent to \eqref{116} by Remark \ref{rem Tr}.
This finishes the proof of Theorem \ref{thm Tr}.
\end{proof}

The case of the extension operator is very similar,
which is presented as follows.

\begin{proof}[Proof of Theorem \ref{thm Ext}]
Recalling that each $Q\in\mathscr Q(\mathbb R^n)$ has the form $Q=I\times[k,k+1)\ell(I)$, where the pair $(I,k)\in\mathscr Q(\mathbb R^{n-1})\times\mathbb Z$ is in
bijective correspondence with $Q\in\mathscr Q(\mathbb R^n)$,
we denote these unique quantities by $I=I_Q$ and $k=k_Q$. We need to estimate the $A^{s,\tau}_{p,r}(W,\mathbb R^{n})$ norm of
\begin{align}\label{Ext f expa}
\operatorname{Ext}\vec f
&=\sum_{\lambda'\in\{0,1\}^{n-1}}\sum_{I\in\mathscr Q_{\lambda'}(\mathbb R^{n-1})}
\left\langle \vec f,\theta^{\lambda'}_I\right\rangle \frac{[\ell(I)]^{\frac12}}{\varphi(-k_0)}
\left[\theta^{\lambda'}\otimes\varphi\right]_{Q(I,k_0)}\notag\\
&=\sum_{\lambda'\in\{0,1\}^{n-1}}\sum_{Q\in\mathscr Q_{(\lambda',0)}(\mathbb R^{n})} \delta_{k_Q,k_0}
\left\langle \vec f,\theta^{\lambda'}_{I_Q}\right\rangle \frac{[\ell(Q)]^{\frac12}}{\varphi(-k_0)}
\left[\theta^{\lambda'}\otimes\varphi\right]_{Q},
\end{align}
where $\mathscr Q_{\lambda'}$ is defined as in \eqref{Tr f expa}. Here $[\theta^{\lambda'}\otimes\varphi]_{Q}$ is an $A^{s,\tau}_{p,r}(W,\mathbb R^{n})$ synthesis molecule on $Q$. (In contrast to the proof of Theorem \ref{thm Tr}, this is now a direct consequence of the properties of the Daubechies wavelets.)

Hence, by the previous identity and the quasi-triangle inequality,
the molecular characterization [Theorem \ref{89}\eqref{89syn}],
the equivalence of pointwise-weighted and average-weighted sequence spaces (Theorem \ref{37}),
assumption \eqref{116x} relating $W$ and $V$ via Remark \ref{rem Tr},
the discrete trace estimate of Corollary \ref{dis Tr},
another application of Theorem \ref{37},
and finally the wavelet characterization (Theorem \ref{wavelet 2}),
we obtain the following chain of estimates
\begin{align*}
\left\|\operatorname{Ext}\vec f\right\|_{A^{s,\tau}_{p,r}(W,\mathbb R^{n})}
&\lesssim\sum_{\lambda'\in\{0,1\}^{n-1}}
\left\|\sum_{I\in\mathscr Q_{\lambda'}(\mathbb R^{n-1})}
[\ell(I)]^{\frac12}\left\langle \vec f,\theta^{\lambda'}_I\right\rangle \left[\theta^{\lambda'}\otimes\varphi\right]_{Q(I,k_0)}
\right\|_{A^{s,\tau}_{p,r}(W,\mathbb R^{n})} \\
&\lesssim\sum_{\lambda'\in\{0,1\}^{n-1}}
\left\|\left\{ \delta_{k_Q,k_0} [\ell(I_Q)]^{\frac12}\left\langle \vec f,\theta^{\lambda'}_{I_Q}
\right\rangle\right\}_{Q\in\mathscr Q_{(\lambda',0)}(\mathbb R^{n})}
\right\|_{a^{s,\tau}_{p,q}(W,\mathbb R^{n})} \\
&\lesssim\sum_{\lambda'\in\{0,1\}^{n-1}}
\left\|\left\{  \delta_{k_Q,k_0} [\ell(I_Q)]^{\frac12} \left| A_{W,Q}\left\langle \vec f,\theta^{\lambda'}_{I_Q}
\right\rangle \right| \right\}_{Q\in\mathscr Q_{(\lambda',0)}(\mathbb R^{n})}
\right\|_{a^{s,\tau}_{p,q}(\mathbb R^{n})} \\
&\lesssim\sum_{\lambda'\in\{0,1\}^{n-1}}
\left\|\left\{  \delta_{k_Q,k_0} [\ell(I_Q)]^{\frac12}
\left| A_{V,I_Q}\left\langle \vec f,\theta^{\lambda'}_{I_Q}
\right\rangle \right| \right\}_{Q\in\mathscr Q_{(\lambda',0)}(\mathbb R^{n})}
\right\|_{a^{s,\tau}_{p,q}(\mathbb R^{n})}   \\
&\lesssim\sum_{\lambda'\in\{0,1\}^{n-1}}
\left\|\left\{  \left| A_{V,I}\left\langle \vec f,\theta^{\lambda'}_{I}
\right\rangle \right| \right\}_{I\in\mathscr Q_{\lambda'}(\mathbb R^{n-1})}
\right\|_{b^{s-\frac1p,\frac{n}{n-1}\tau}_{p,r}(\mathbb R^{n-1})}  \\
&\lesssim\sum_{\lambda'\in\{0,1\}^{n-1}}
\left\|\left\{\left\langle \vec f,\theta^{\lambda'}_{I}\right\rangle \right\}_{I\in\mathscr Q_{\lambda'}(\mathbb R^{n-1})}
\right\|_{b^{s-\frac1p,\frac{n}{n-1}\tau}_{p,r}(V,\mathbb R^{n-1})}
\lesssim\left\| \vec f\right\|_{B^{s-\frac1p,\frac{n}{n-1}\tau}_{p,q}(V,\mathbb R^{n-1})}.
\end{align*}

For the identity $\operatorname{Tr}(\operatorname{Ext}\vec f)=\vec f$, we omit the details and refer the reader to the analogous argument in \cite[Proof of Theorem 5.10]{bhyyp3}.
This finishes the proof of Theorem \ref{thm Ext}.
\end{proof}

\subsection{Pointwise Multipliers}
\label{Pointwise Multipliers}

In this section, we give a class of pointwise multipliers of $A^{s,\tau}_{p,q}(W)$.
Let us begin with some concepts.
Let $\mathscr{X}$ be a quasi-Banach spaces of functions (distributions).
Then the basic question consists in descriptions of
the \emph{associated multiplier space} $M(\mathscr{X})$ given by
\begin{align*}
M(\mathscr{X}):=\{g\in L^\infty:\
g\cdot f\in\mathscr{X} \text{ for any } f\in\mathscr{X}\}.
\end{align*}
Here, we consider the specific spaces where $\mathscr{X}:=A^{s,\tau}_{p,q}(W)$.

For any $\mathscr{N}\in\mathbb{N}$, let
\begin{align*}
\mathscr C^{\mathscr{N}}
:=\left\{f\in C^{\mathscr{N}}:\ \partial^\alpha f\in L^\infty\text{ for all }\alpha\in\mathbb Z_+^n
\text{ with }|\alpha|\leq\mathscr{N} \right\}.
\end{align*}
The following theorem is the main result of this section.

\begin{theorem}\label{pointwise1}
Let $s\in\mathbb R$, $\tau\in[0,\infty)$,
$p\in(0,\infty)$, $q\in(0,\infty]$, and $W\in A_{p,\infty}$.
If $\mathscr N\in\mathbb N\cap(\widetilde s+
4(\lfloor\frac{\widetilde J-n-\widetilde s}2\rfloor+1)_+,\infty)$,
where $\widetilde J$ and $\widetilde s$ are the same as in \eqref{tauJ2},
then $\mathscr C^{\mathscr{N}}\subset M(A^{s,\tau}_{p,q}(W))$.
Moreover, there exists a positive constant $C$ such that,
for any $g\in\mathscr{C}^{\mathscr{N}}$ and $\vec f\in A^{s,\tau}_{p,q}(W)$,
\begin{align}\label{pointwise1 equ}
\left\|g\cdot\vec f\right\|_{A^{s,\tau}_{p,q}(W)}
\leq C\left(\sum_{\alpha\in\mathbb Z_+^n,\,|\alpha|\leq\mathscr{N}}
\left\|\partial^\alpha g\right\|_{L^\infty}\right)
\left\|\vec f\right\|_{A^{s,\tau}_{p,q}(W)}.
\end{align}
\end{theorem}

To show Theorem \ref{pointwise1}, we need the following two lemmas.
The first one is a simple application of Corollary \ref{Fmult}; we omit the details.

\begin{lemma}\label{pointwise lem2}
Let $s\in\mathbb R$, $\tau\in[0,\infty)$,
$p\in(0,\infty)$, $q\in(0,\infty]$, and $W\in A_{p,\infty}$.
Then, for any $\gamma\in\mathbb Z_+^n$,
the operator $\partial^\gamma:\ A^{s+|\gamma|,\tau}_{p,q}(W)
\to A^{s,\tau}_{p,q}(W)$ is continuous.
\end{lemma}

The Laplace operator is defined by setting
$\Delta:=\sum_{i=1}^n\frac{\partial^2}{\partial x_i^2}$.

\begin{lemma}\label{pointwise lem}
Let $s\in\mathbb R$, $\tau\in[0,\infty)$,
$p\in(0,\infty)$, $q\in(0,\infty]$, and $W\in A_{p,\infty}$.
Let $l\in\mathbb Z_+$.
Then, for any $\vec f\in A^{s,\tau}_{p,q}(W)$,
there exists $\vec g\in A^{s+2l,\tau}_{p,q}(W)$ such that
$\vec f=[I+(-\Delta)^l]\vec g$ and
$
\|\vec f\|_{A^{s,\tau}_{p,q}(W)}
\sim\|\vec g\|_{A^{s+2l,\tau}_{p,q}(W)}
$,
where the positive equivalence constants are independent of $\vec f$ and $\vec g$.
\end{lemma}

\begin{proof}
Consider the symbols $a_{\pm}(\xi):=(1+|\xi|^{2l})^{\pm 1}$, where $\xi\in\mathbb R^n$.
It is easy to check that $a_{\pm}\in S^{\pm 2l}_{1,1}$.
From this and Corollary \ref{Fmult}, we deduce that the Fourier multiplier $a_{\pm}(D)$
is a continuous linear mapping from
$A^{s,\tau}_{p,q}(W)$ to $A^{s\mp 2l,\tau}_{p,q}(W)$.
It is also clear that $a_+(D)=I+(-\Delta)^l$.

Given $\vec f\in A^{s,\tau}_{p,q}(W)$, we choose $\vec g:=a_-(D)\vec f$. Then $\vec f=a_+(D)\vec g=(I+(-\Delta)^l)\vec g$, and
\begin{equation*}
\left\|\vec f\right\|_{A^{s,\tau}_{p,q}(W)}
\lesssim  \left\|\vec g\right\|_{A^{s+2l,\tau}_{p,q}(W)}
\lesssim  \left\|\vec f\right\|_{A^{s,\tau}_{p,q}(W)}
\end{equation*}
by the said boundedness properties of $a_+(D)$ and $a_-(D)$,
which completes the proof of Lemma \ref{pointwise lem}.
\end{proof}

Recall that, for any $\sigma\in\mathbb R$, the \emph{lifting operator} $I_\sigma$
(see, for instance, \cite[Section 2.3.8]{t83}) is defined by setting,
for any $f\in\mathcal{S}'$,
$$
I_\sigma(f):=\left[\left(1+|\cdot|^2\right)^{\frac{\sigma}{2}}\widehat{f}\right]^\vee.
$$
It is well known that $I_\sigma$ maps $\mathcal{S}'$ onto itself.
Applying an argument similar to that used in
the proof of Lemma \ref{pointwise lem} and
using the symbols $a_{\pm}:=(1+|\cdot|^2)^{\pm\sigma/2}\in S_{1,1}^{\pm\sigma}$,
we can prove the following {\em lifting property}; we omit the details.

\begin{proposition}\label{257}
Let $s,\sigma\in\mathbb{R}$, $\tau\in[0,\infty)$,
$p\in(0,\infty)$, and $q\in(0,\infty]$.
Let $W\in A_{p,\infty}$.
Then $I_\sigma$ maps $A_{p,q}^{s,\tau}(\mathbb{A})$
isomorphically onto $A_{p,q}^{s-\sigma,\tau}(\mathbb{A})$.
Moreover, for any $\vec f\in(\mathcal{S}')^m$, we have
$
\|\vec f\|_{A_{p,q}^{s,\tau}(\mathbb{A})}
\sim\|I_\sigma\vec f\|_{A_{p,q}^{s-\sigma,\tau}(\mathbb{A})},
$
where the positive equivalence constants are independent of $\vec f$.
\end{proposition}

Next, we prove Theorem \ref{pointwise1}.

\begin{proof}[Proof of Theorem \ref{pointwise1}]
Let $g\in\mathscr{C}^{\mathscr{N}}$ and $\vec f\in A^{s,\tau}_{p,q}(W)$ be fixed.
To show \eqref{pointwise1 equ}, we consider the following two cases on $\widetilde s$.

\emph{Case 1)} $\widetilde s>\widetilde J-n$.
In this case, $\mathscr{N}>\widetilde{s}$.
Let $L,N\in\mathbb R$ satisfy $L\in[\widetilde J-n-\widetilde s,0)$
and $N\in(\widetilde s,\mathscr{N}]$ and $r$ be the same as in Theorem \ref{atom}.
By Theorem \ref{atom}(i), we find that there exist
a sequence $\vec t:=\{\vec t_Q \}_{Q\in\mathscr{Q}_+}\in a^{s,\tau}_{p,q}(W)$
and $(r,L,N)$-atoms $\{a_Q\}_{Q\in\mathscr{Q}_+}$, each on the cube indicated by its subscript,
such that
$
\vec f=\sum_{Q\in\mathscr{Q}_+}\vec t_Q a_Q
$
in $(\mathcal{S}')^m$ and hence it is natural to define
$$
g\cdot\vec f
:=\sum_{Q\in\mathscr{Q}_+}\vec t_Q b_Q
:=\sum_{Q\in\mathscr{Q}_+}\vec t_Q ga_Q
$$
in $(\mathcal{S}')^m$.
From Theorem \ref{atom} again, we infer that,
to prove \eqref{pointwise1 equ}, it is enough to show that
each $b_Q$ is a constant multiple of an $(r,L,N)$-atom for $A^{s,\tau}_{p,q}(W)$.
Observe that, with the above choice of $L$, $(r,L,N)$-atoms
have no cancellation condition. Thus, we only need to focus
on the regularity condition of atoms.
For any $\gamma\in\mathbb Z_+^n$ with $|\gamma|\leq N$
and for any $Q\in\mathscr{Q}_+$,
\begin{align*}
\left\|\partial^\gamma b_Q\right\|_{L^\infty}
&=\left\|\partial^\gamma ga_Q\right\|_{L^\infty}
\leq\sum_{\beta\in\mathbb Z_+^n,\,\beta\leq\gamma}
\left\|\partial^\beta g\right\|_{L^\infty}
\left\|\partial^{\gamma-\beta} a_Q\right\|_{L^\infty}\\
&\leq\left(\sum_{\beta\in\mathbb Z_+^n,\,\beta\leq\gamma}
\left\|\partial^\beta g\right\|_{L^\infty}\right)
|Q|^{-\frac12-\frac{|\gamma-\beta|}n}
\leq\left(\sum_{\beta\in\mathbb Z_+^n,\,|\beta|\leq\mathscr{N}}
\left\|\partial^\beta g\right\|_{L^\infty}\right)
|Q|^{-\frac12-\frac{|\gamma|}n}
\end{align*}
and hence $b_Q$ is a constant multiple of an $(r,L,N)$-atom
for $A^{s,\tau}_{p,q}(W)$.
This finishes the proof of \eqref{pointwise1 equ} in this case.

\emph{Case 2)} $\widetilde s\leq\widetilde J-n$.
In this case, $\mathscr{N}>\widetilde s+
4(\lfloor\frac{\widetilde J-n-\widetilde s}2\rfloor+1)$.
Let $l:=\lfloor\frac{\widetilde J-n-\widetilde s}2\rfloor+1\in\mathbb N$.
Then $\widetilde s+2l>\widetilde J-n$ and $\mathscr N-2l> \widetilde{s}+2l$.
By Lemma \ref{pointwise lem}, we find that
there exists $\vec h\in A^{s+2l,\tau}_{p,q}(W)$ such that
$\vec f=[I+(-\Delta)^l]\vec h$ and
\begin{align}\label{12}
\left\|\vec f\right\|_{A^{s,\tau}_{p,q}(W)}
\sim\left\|\vec h\right\|_{A^{s+2l,\tau}_{p,q}(W)}.
\end{align}
Applying an argument similar to that used in \cite[p.\,204]{t92}, we obtain
\begin{align*}
g\cdot\vec f=\left[I+(-\Delta)^l\right]\left(g\cdot\vec h\right)
+\sum_{\alpha\in\mathbb Z_+^n,\,|\alpha|<2l}
\partial^\alpha\left(g_\alpha\cdot\vec h\right),
\end{align*}
where $g_\alpha$ is the summation of terms of type $\partial^\beta g$
with $\beta\in\mathbb Z_+^n$ and $|\beta|\leq2l$.
This, together with Lemma \ref{pointwise lem2}
and the simple fact that $A^{s_1,\tau}_{p,q}(W)$
is continuously embedded in $A^{s_2,\tau}_{p,q}(W)$ if $s_1>s_2$,
further implies that
\begin{align*}
\left\|g\cdot\vec f\right\|_{A^{s,\tau}_{p,q}(W)}
\lesssim\sum_{\alpha\in\mathbb Z_+^n,\,|\alpha|\leq2l}
\left\|g_\alpha\cdot\vec h\right\|_{A^{s+2l,\tau}_{p,q}(W)}.
\end{align*}
By Case 1) and \eqref{12}, we find that,
for any $\alpha\in\mathbb Z_+^n$ with $|\alpha|\leq2l$,
\begin{align*}
\left\|g_\alpha\cdot\vec h\right\|_{A^{s+2l,\tau}_{p,q}(W)}
\lesssim\left(\sum_{\beta\in\mathbb Z_+^n,\,|\beta|\leq\mathscr{N}-2l}
\left\|\partial^\beta g_\alpha\right\|_{L^\infty}\right)
\left\|\vec h\right\|_{A^{s+2l,\tau}_{p,q}(W)}
\lesssim\left(\sum_{\beta\in\mathbb Z_+^n,\,|\beta|\leq\mathscr{N}}
\left\|\partial^\beta g\right\|_{L^\infty}\right)
\left\|\vec f\right\|_{A^{s,\tau}_{p,q}(W)}.
\end{align*}
Therefore,
\begin{align*}
\left\|g\cdot\vec f\right\|_{A^{s,\tau}_{p,q}(W)}
\lesssim\left(\sum_{\beta\in\mathbb Z_+^n,\,|\beta|\leq\mathscr{N}}
\left\|\partial^\beta g\right\|_{L^\infty}\right)
\left\|\vec f\right\|_{A^{s,\tau}_{p,q}(W)}.
\end{align*}
This finishes the proof of \eqref{pointwise1 equ} in this case and hence Theorem \ref{pointwise1}.
\end{proof}

\begin{remark}
When $m=1$ and $W\equiv1$, Theorem \ref{pointwise1} contains \cite[Theorem 6.1]{ysy10},
where $\tau$ has an upper bound.
\end{remark}

\subsection{Calder\'on--Zygmund Operators}
\label{C-Z operators}

In this section, we establish the boundedness of Calder\'on--Zygmund operators on $A^{s,\tau}_{p,q}(W)$ under fairly general assumptions essentially like those of \cite[Section 6]{bhyyp3} (which, in turn, were inspired by the earlier work of \cite{ftw88,tor}), but incorporating an extra decay condition [see Definition \ref{InCZO}\eqref{InCZOnew} below] required by the present inhomogeneous situation.

To this end, we first discuss the problem of extending a given operator
$T:\ \mathcal{S}\to\mathcal{S}'$ to $\widetilde T:\ A^{s,\tau}_{p,q}(W)\to A^{s,\tau}_{p,q}(W)$.
Let an infinite matrix $B:=\{b_{Q,P}\}_{Q,P\in\mathscr{Q}_+}\subset\mathbb{C}$ be given.
For any sequence $t:=\{t_R\}_{R\in\mathscr{Q}_+}\subset\mathbb{C}$,
we define $Bt:=\{(Bt)_Q\}_{Q\in\mathscr{Q}_+}$ by setting,
for any $Q\in\mathscr{Q}_+$,
$$
\left(Bt\right)_Q:=\sum_{R\in\mathscr{Q}_+}b_{Q,R}t_R
$$
if this summation is absolutely convergent.

\begin{definition}
Let $L,N\in\mathbb R$.
A function $a_Q\in C_{\mathrm{c}}^\infty$ is called an \emph{$(L,N)$-atom} on a cube $Q$ if
\begin{enumerate}[\rm(i)]
\item $\operatorname{supp}a_Q\subset3Q$;
\item $\int_{\mathbb{R}^n}x^\gamma a_Q(x)\,dx=0$
for any $\gamma\in\mathbb{Z}_+^n$ and $|\gamma|\leq L$;
\item $|D^\gamma a_Q(x)|\leq|Q|^{-\frac12-\frac{|\gamma|}{n}}$
for any $x\in\mathbb{R}^n$ and $\gamma\in\mathbb{Z}_+^n$ and $|\gamma|\leq N$.
\end{enumerate}
\end{definition}

Applying an argument similar to that used in the proof of \cite[Proposition 6.5]{bhyyp3},
we obtain the following conclusion; we omit the details.

\begin{proposition}\label{ext}
Let $s\in\mathbb R$, $\tau\in[0,\infty)$,
$p\in(0,\infty)$, $q\in(0,\infty]$, and $W\in A_{p,\infty}$.
Let $L,N\in(0,\infty)$.
If $T\in\mathcal L(\mathcal S,\mathcal S')$
maps $(-1,N)$-atoms to $a^{s,\tau}_{p,q}(W)$-synthesis molecules
if they are on some $Q\in\mathscr Q_0$
and maps $(L,N)$-atoms to $a^{s,\tau}_{p,q}(W)$-synthesis molecules
if they are on some $Q\in\mathscr Q$ with $\ell(Q)<1$,
then there exists an operator $\widetilde T\in\mathcal L(A^{s,\tau}_{p,q}(W))$
that agrees with $T$ on $(\mathcal{S})^m$.
\end{proposition}

We now turn to the actual discussion of Calder\'on--Zygmund operators.
The following classical definitions are standard.
Let $\mathcal D :=C_{\mathrm{c}}^\infty $ with the usual inductive limit topology.
We denote by $\mathcal D' $ the space of
all continuous linear functionals on $\mathcal D$,
equipped with the weak-$*$ topology.
If $T\in\mathcal L(\mathcal S,\mathcal S')$,
then, by the well-known Schwartz kernel theorem,
we find that there exists $\mathcal K\in\mathcal S'(\mathbb R^n\times\mathbb R^n)$
such that
$\langle T\varphi,\psi\rangle
=\langle\mathcal K,\varphi\otimes\psi\rangle$ for all $\varphi,\psi\in\mathcal S$.
This $\mathcal K$ is called the \emph{Schwartz kernel} of $T$.

\begin{definition}\label{WBP}
Let $T\in\mathcal L(\mathcal S,\mathcal S')$
and let $\mathcal K\in\mathcal S'(\mathbb R^n\times\mathbb R^n)$ be its Schwartz kernel.
\begin{enumerate}[\rm(i)]
\item We say that $T$ satisfies the \emph{weak boundedness property}
and write $T\in\operatorname{WBP}$ if,
for every bounded subset $\mathcal B$ of $\mathcal D$,
there exists a positive constant $C=C(\mathcal B)$ such that,
for any $\varphi,\eta\in\mathcal B$, $h\in\mathbb R^n$, and $r\in(0,\infty)$,
\begin{equation*}
\left|\left\langle T \left(\varphi\left(\frac{\cdot-h}{r}\right)\right),
\eta \left(\frac{\cdot-h}{r}\right)\right\rangle\right|\leq Cr^n.
\end{equation*}
\item For any $\ell\in(0,\infty)$, we say that $T$ has a \emph{Calder\'on--Zygmund kernel} of order $\ell$
and write $T\in\operatorname{CZO}(\ell)$
if the restriction of $\mathcal K$ on the set $\{(x,y)\in\mathbb R^n\times\mathbb R^n:\ x\neq y\}$
is a continuous function with continuous partial derivatives
in the $x$ variable up to order $\lfloor\!\lfloor\ell\rfloor\!\rfloor$ satisfying that
there exists a positive constant $C$ such that,
for any $\gamma\in\mathbb{Z}_+^n$ with $|\gamma|\leq\lfloor\!\lfloor\ell\rfloor\!\rfloor$
and for any $x,y\in\mathbb{R}^n$ with $x\neq y$,
\begin{equation*}
\left|\partial_x^{\gamma}\mathcal K(x,y)\right|
\leq C|x-y|^{-n-|\gamma|}
\end{equation*}
and, for all $\gamma\in\mathbb{Z}_+^n$ with $|\gamma|=\lfloor\!\lfloor\ell\rfloor\!\rfloor$
and for any $x,y,h\in\mathbb{R}^n$ with $|h|<\frac12|x-y|$,
\begin{equation*}
\left|\partial_x^{\gamma}\mathcal K(x,y)-\partial_x^\gamma\mathcal K(x+h,y)\right|
\leq C|x-y|^{-n-\ell}|h|^{\ell^{**}},
\end{equation*}
where $\lfloor\!\lfloor\ell\rfloor\!\rfloor $ and $\ell^{**}$
are the same as, respectively, in \eqref{ceil} and \eqref{r**}
with $r$ replaced by $\ell$. For any $\ell\in(-\infty,0]$,
we interpret $T\in\operatorname{CZO}(\ell)$ as a void condition.
\end{enumerate}
\end{definition}

\begin{remark}
For any $\ell_1,\ell_2\in\mathbb{R}$ with $\ell_1<\ell_2$,
it is easy to prove that $\operatorname{CZO}(\ell_2)\subset\operatorname{CZO}(\ell_1)$.
\end{remark}

For the formulation of important cancellation conditions, we need to define the action of Calder\'on--Zygmund operators on polynomials, which lie outside their initial domain $\mathcal S$ of definition.
For this purpose, we quote the following result which is a special case of \cite[Lemma 2.2.12]{tor} (see also the comment right after the proof of the said lemma).

\begin{lemma}\label{tor2.2.12}
Let $\ell\in(0,\infty)$ and $T\in\operatorname{CZO}(\ell)$,
and let $\{\phi_j\}_{j\in\mathbb{N}}\subset\mathcal D$ be a sequence of functions such that $\sup_{j\in\mathbb{N}}\|\phi_j\|_{L^\infty}<\infty $ and,
for each compact set $K\subset\mathbb R^n$, there exists $j_K\in\mathbb{N}$ such that $\phi_j(x)=1$ for all $x\in K$ and $j\geq j_K$.
Then the limit
\begin{equation}\label{Tphijf}
\langle T(f),g\rangle:=\lim_{j\to\infty}\langle T(\phi_jf),g\rangle
\end{equation}
exists for all monomials $f(y)=y^\gamma$ with $|\gamma|\leq\lfloor\!\lfloor\ell\rfloor\!\rfloor$ and all $g\in\mathcal D_{\lfloor\!\lfloor\ell\rfloor\!\rfloor}$, where
\begin{equation*}
\mathcal D_{\lfloor\!\lfloor\ell\rfloor\!\rfloor}
:=\left\{g\in\mathcal D:\
\int_{\mathbb R^n}x^\gamma g(x)\,dx=0\text{ if }
\gamma\in\mathbb{Z}_+^n\text{ with }|\gamma|\leq\lfloor\!\lfloor\ell\rfloor\!\rfloor \right\}
\end{equation*}
and $\lfloor\!\lfloor\ell\rfloor\!\rfloor $ is the same as in \eqref{ceil}
with $r$ replaced by $\ell$. Moreover, the limit \eqref{Tphijf}
is independent of the choice of the sequence $\{\phi_j\}_{j\in\mathbb{N}}$.
\end{lemma}

Based on Lemma \ref{tor2.2.12}, we can give the following definition.

\begin{definition}
Let $\ell\in(0,\infty)$. For any $T\in\operatorname{CZO}(\ell)$
and any $f(y)=y^\gamma$ with $y\in\mathbb R^n$,
where $\gamma\in\mathbb{Z}_+^n$ satisfies $|\gamma|\leq\lfloor\!\lfloor\ell\rfloor\!\rfloor$,
we define $T(y^\gamma)=Tf:\ \mathcal D_{\lfloor\!\lfloor\ell\rfloor\!\rfloor}\to\mathbb{C}$ given by \eqref{Tphijf}.
\end{definition}

We next recall joint smoothness conditions (simultaneously in the $x$
and the $y$ variables) in the form introduced in \cite[Definition 6.11]{bhyyp3};
they are related to, but not exactly the same as, those used in \cite{ftw88,tor}.

\begin{definition}\label{def joint CZK}
Let $ E,F\in\mathbb{R}$, $T\in\mathcal L(\mathcal S,\mathcal S')$,
and $\mathcal K\in\mathcal S'(\mathbb R^n\times\mathbb R^n)$ be its Schwartz kernel.
We say that $T\in\operatorname{CZK}^0(E;F)$
if the restriction of $\mathcal K$ to
$\{(x,y)\in\mathbb R^n\times\mathbb R^n:\ x\neq y\}$
is a continuous function such that all the derivatives below
exist as continuous functions and there exists a positive constant $C$ such that,
for any $\alpha\in\mathbb{Z}_+^n$ with $|\alpha|\leq\lfloor\!\lfloor E\rfloor\!\rfloor_+$
and for any $x,y\in\mathbb R^n$ with $x\neq y$
\begin{equation*}
|\partial_x^\alpha\mathcal K(x,y)|
\leq C|x-y|^{-n-|\alpha|},
\end{equation*}
for any $\alpha\in\mathbb{Z}_+^n$ with $|\alpha|=\lfloor\!\lfloor E\rfloor\!\rfloor$
and for any $x,y,u\in\mathbb R^n$ with $|u|<\frac12|x-y|$
\begin{equation*}
|\partial_x^\alpha\mathcal K(x,y)-\partial_x^\alpha\mathcal K(x+u,y)|
\leq C|u|^{E^{**}}|x-y|^{-n-E}
\end{equation*}
and, for any $\alpha,\beta\in\mathbb{Z}_+^n$ with $|\alpha|\leq\lfloor\!\lfloor E\rfloor\!\rfloor_+$
and $|\beta|=\lfloor\!\lfloor F-|\alpha|\rfloor\!\rfloor$
and for any $x,y,v\in\mathbb R^n$ with $|v|<\frac12|x-y|$
$$
\left|\partial_x^\alpha\partial_y^\beta\mathcal K(x,y)
-\partial_x^\alpha\partial_y^\beta\mathcal K(x,y+v)\right|
\leq C|v|^{(F-|\alpha|)^{**}}|x-y|^{-n-|\alpha|-(F-|\alpha|)},
$$
where $\lfloor\!\lfloor E\rfloor\!\rfloor$ and $E^{**}$
are the same as, respectively, in \eqref{ceil} and \eqref{r**}
with $r$ replaced by $E$.

We say that $T\in\operatorname{CZK}^1(E;F)$
if $T\in\operatorname{CZK}^0(E;F)$ and, in addition,
for any $\alpha,\beta\in\mathbb{Z}_+^n$ with
$|\alpha|=\lfloor\!\lfloor E\rfloor\!\rfloor$ and $|\beta|=\lfloor\!\lfloor F-E\rfloor\!\rfloor$
and for any $x,y,u,v\in\mathbb R^n$ with $|u|+|v|<\frac12|x-y|$,
\begin{align}\label{CZKxy}
&\left|\partial_x^\alpha\partial_y^\beta\mathcal K(x,y)
-\partial_x^\alpha\partial_y^\beta\mathcal K(x+u,y)
-\partial_x^\alpha\partial_y^\beta\mathcal K(x,y+v)
+\partial_x^\alpha\partial_y^\beta\mathcal K(x+u,y+v)\right|\notag\\
&\quad\leq C|u|^{E^{**}}|v|^{(F-E)^{**}}
|x-y|^{-n-E-(F-E)}.
\end{align}
We may write just $\operatorname{CZK}(E;F)$
if the parameter values are such that \eqref{CZKxy} is void
and hence $\operatorname{CZK}^0(E;F)$ and $\operatorname{CZK}^1(E;F)$ coincide.
\end{definition}

We refer the reader to \cite[paragraph after Definition 6.11, and Remark 6.13]{bhyyp3} for a discussion of the meaning of the conditions of Definition \ref{def joint CZK} in various ranges of the parameters.

In the following definition, analogous to \cite[Definition 6.17]{bhyyp3}, we introduce a short-hand notation
that combines the various assumptions that we have discussed above.
For any $T\in\mathcal L(\mathcal S,\mathcal S')$,
let $T^*\in\mathcal L(\mathcal S,\mathcal S')$ be defined by setting
$
\langle T^*\psi,\varphi\rangle=\langle T\varphi,\psi\rangle
$
for any $\varphi,\psi\in\mathcal S$.

\begin{definition}\label{InCZO}
Let $\sigma\in\{0,1\}$ and $ E,F,G,H\in\mathbb{R}$.
We say that $T\in\operatorname{InCZO}^\sigma(E,F,G,H)$
if $T\in\mathcal L(\mathcal S,\mathcal S')$
and its Schwartz kernel $\mathcal K\in\mathcal S'(\mathbb R^n\times\mathbb R^n)$ satisfy
\begin{enumerate}[\rm(i)]
\item $T\in\operatorname{WBP}$, where $\operatorname{WBP}$
is the same as in Definition \ref{WBP}(i);
\item $\mathcal K\in\operatorname{CZK}^\sigma(E;F)$;
\item\label{T10} $T(y^\gamma)=0$ for any $\gamma\in\mathbb{Z}_+^n$ with $|\gamma|\leq G$;
\item\label{T20} $T^*(x^\theta)=0$ for any $\theta\in\mathbb{Z}_+^n$ with $|\theta|\leq H$;
\item\label{InCZOnew} there exists a positive constant $C$ such that,
for any $\alpha\in\mathbb{Z}_+^n$ with $|\alpha|\leq\lfloor\!\lfloor E\rfloor\!\rfloor+1$
and for any $x,y\in\mathbb{R}^n$ with $|x-y|>1$,
$$
|\partial_x^\alpha\mathcal{K}(x,y)|\leq C|x-y|^{-(n+F)}.
$$
\end{enumerate}
\end{definition}

\begin{remark}
In Definition \ref{InCZO},
condition (v) comes from \cite[Theorems 3.2.48 and 3.2.49]{tor}.
If this condition is removed, then
$\operatorname{InCZO}^\sigma(E,F,G,H)$ reduces to
$\operatorname{CZO}^\sigma(E,F,G,H)$ as defined in \cite[Definition 6.17]{bhyyp3}.

The extra decay required by Definition \ref{InCZO}\eqref{InCZOnew} unfortunately excludes all usual ``global'' singular integrals like the Riesz transforms with $\mathcal K_i(x,y)=c_n(x_i-y_i)|x-y|^{-n-1}$, which in some directions decay only as fast as $|x-y|^{-n}$ and no faster. On the other hand, the condition of Definition \ref{InCZO}\eqref{InCZOnew} is clearly satisfied by the local versions of such operators with kernel $\chi(x-y)\mathcal K_i(x,y)$ for any smooth cut-off $\chi\in C_{\mathrm c}^\infty$.

The cancellation conditions \eqref{T10} and \eqref{T20} of Definition \ref{InCZO} are so-called ``vanishing paraproduct'' or ``special $T(1)$'' conditions, in contrast to more general conditions such as $T(1),T^*(1)\in\operatorname{BMO}$ in the original $T(1)$ theorem of \cite{DJ84} describing the boundedness of Calder\'on--Zygmund operators on $L^2$. A larger part of the literature on extensions of the $T(1)$ theorem to function spaces of Besov and Triebel--Lizorkin type (cf. \cite{ftw88,tor}) seems to have been set up under the vanishing paraproduct assumptions; see, however, \cite{DGW,you} for results under more general BMO-type conditions. Extending this theory to the matrix-weighted setting could be of interest for further investigation.
\end{remark}

The following theorem, analogous to \cite[Theorem 6.18]{bhyyp3}, is the main result of this section.

\begin{theorem}\label{T1 BF}
Let $s\in\mathbb R$, $\tau\in[0,\infty)$, $p\in(0,\infty)$, and $q\in(0,\infty]$.
Let $W\in A_{p,\infty}$ and $\mathbb{A}:=\{A_Q\}_{Q\in\mathscr{Q}_+}$ be
a sequence of reducing operators of order $p$ for $W$.
Let $\widetilde J,\widetilde s$ be the same as in \eqref{tauJ2}.
Let $T\in\operatorname{InCZO}^\sigma(E,F,G,H)$, where
$\sigma\in\{0,1\}$ and $E,F,G,H\in\mathbb{R}$ satisfy
\begin{align}\label{T1Amol1}
\sigma\geq\mathbf{1}_{[0,\infty)}(\widetilde s),\
E>(\widetilde s)_+,\
F>\widetilde J-n+(\widetilde s)_-,\
G\geq\lfloor\widetilde s\rfloor_+,\text{ and }
H\geq\left\lfloor\widetilde J-n-\widetilde s\right\rfloor.
\end{align}
Then there exists an operator $\widetilde T\in\mathcal L(A^{s,\tau}_{p,q}(W))$
that agrees with $T$ on $\mathcal S$.
\end{theorem}

As in \cite{bhyyp3}, the proof of Theorem \ref{T1 BF} is based on describing conditions under which an operator $T$ maps atoms into molecules of given parameters. The novelty in the present inhomogeneous situation is the lack of cancellation conditions, in the atomic characterization of the inhomogeneous spaces (Theorem \ref{atom}), for atoms supported by the largest cubes $Q\in\mathscr Q_0$. Thus, the result of \cite[Proposition 6.19]{bhyyp3}, where such cancellation conditions were critically exploited, needs to be complemented by a new argument for the non-cancellative atoms. It is here that the new condition of Definition \ref{InCZO}\eqref{InCZOnew} will be needed.

\begin{proposition}\label{T1EFGH}
Let $\sigma\in\{0,1\}$, $ E,F,G,H\in\mathbb{R}$,
$K,M,N\in\mathbb{R}$, and $Q\in\mathscr{Q}_0$.
Suppose that $T\in\operatorname{InCZO}^\sigma(E,F,G,H)$.
Then $T$ maps sufficiently regular non-cancellative atoms on $Q$ to $(K,-1,M,N)$-molecules on $Q$ provided that
\begin{align}\label{T1EFGHcond}
\sigma\geq\mathbf{1}_{(0,\infty)}(N),\quad
\begin{cases}
E\geq N,\\
E>\lfloor N\rfloor_+,
\end{cases}
F\geq(K\vee M)-n,\quad\text{and}\quad
G\geq\lfloor N\rfloor_+.
\end{align}
\end{proposition}

\begin{proof}
We note that the conditions \eqref{T1EFGHcond} are the same as those in \cite[Proposition 6.19]{bhyyp3}, except that two conditions $F>\lfloor L \rfloor$ and $H\geq\lfloor L \rfloor$ presented in \cite[Proposition 6.19]{bhyyp3} are now omitted as redundant because we are now in the situation with the cancellation parameter $L=-1$ corresponding to no vanishing moments.

Without essential loss of generality,
we may consider atoms and molecules on $Q_{0,\mathbf{0}}$ only.
Let $a$ be a $(-1,N_a)$-atom, where
$N_a\in[(\lfloor\!\lfloor E\rfloor\!\rfloor\wedge\lfloor G\rfloor)+1,\infty)$.
(This quantifies the assumption ``sufficiently regular'' in the statement of the proposition.)

\textbf{Size and derivatives:}\quad
We need to estimate $\partial^\alpha Ta$ for any
$\alpha\in\mathbb{Z}_+^n $ with $|\alpha|\leq\lfloor\!\lfloor N\rfloor\!\rfloor_+$,
where $\lfloor\!\lfloor N\rfloor\!\rfloor $ is the same as in \eqref{ceil}.

The argument in \cite[paragraph between (6.21) and (6.22)]{bhyyp3}, which shows the uniform bound $|\partial^\alpha Ta(x)|\lesssim 1$ for all $|\alpha|\leq\lfloor\!\lfloor E\rfloor\!\rfloor\wedge\lfloor G\rfloor$, remains valid without changes.
However, to obtain the required decay for large $x$, the argument of \cite[paragraph between (6.21) and (6.22)]{bhyyp3} uses vanishing moments and needs to be replaced here. To this end, by the support condition $\operatorname{supp} a\subset 3Q_{0,\mathbf{0}}$ and Definition \ref{InCZO}(v),
we find that, for any $\alpha\in\mathbb{Z}_+^n$ with $|\alpha|\leq\lfloor\!\lfloor E\rfloor\!\rfloor+1$
and for any $x\in\mathbb{R}^n$ with $|x|>4\sqrt{n}$,
\begin{align}\label{T1mole0}
|\partial^\alpha Ta(x)|
&=\left|\int_{\mathbb R^n}\partial_x^\alpha \mathcal K(x,y)a(y)\,dy\right|
=\left|\int_{3Q_{0,\mathbf{0}}}\partial_x^\alpha \mathcal K(x,y)a(y)\,dy\right|\notag\\
&\lesssim\int_{3Q_{0,\mathbf{0}}}|x-y|^{-n-F}\,dy
\sim|x|^{-n-F}.
\end{align}
Combining this with the uniform bound $|\partial^\alpha Ta(x)|\lesssim 1$,
we obtain,
for any $\alpha\in\mathbb{Z}_+^n$ with $|\alpha|\leq\lfloor\!\lfloor E\rfloor\!\rfloor\wedge \lfloor G\rfloor$
and for any $x\in\mathbb R^n$, the estimate
\begin{equation*}
|\partial^\alpha Ta(x)|\lesssim(1+|x|)^{-n-F}.
\end{equation*}

Since $\lfloor\!\lfloor E\rfloor\!\rfloor\wedge\lfloor G\rfloor\geq\lfloor N\rfloor_+\geq 0$, this bound holds for any $|\alpha|\leq\lfloor N\rfloor_+$ and in particular for $\alpha=\mathbf 0$. Since $F+n\geq K\vee M$, we conclude that,
for any $\alpha\in\mathbb{Z}_+^n$ with $|\alpha|\leq\lfloor\!\lfloor N\rfloor\!\rfloor$
and for any $x\in\mathbb{R}^n$,
\begin{equation}\label{T1mole2}
|Ta(x)|\lesssim(1+|x|)^{-K}
\quad\text{and}\quad
|\partial^\alpha Ta(x)|\lesssim(1+|x|)^{-M},
\end{equation}
which are the required size and derivative bounds of a $(K,-1,M,N)$-molecule.

\textbf{Differences:}\quad
Next, we need to estimate $\partial^\alpha Ta(x)-\partial^\alpha Ta(x+h)$
for any $x,h\in\mathbb R^n$ and $\alpha\in\mathbb{Z}_+^n$
with $|\alpha|=\lfloor\!\lfloor N\rfloor\!\rfloor$.
There is nothing to do if $N\leq0$, so we may in the remainder of the
present proof assume that $N>0$.

For any $|h|\geq 1$, the estimate
\begin{equation}\label{mole diff}
|\partial^\alpha Ta(x)-\partial^\alpha Ta(x+h)|
\lesssim|h|^{N^{**}}\sup_{z\in\mathbb R^n,\,|z|\leq h}(1+|x+z|)^{-M}
\end{equation}
follows from the already checked the bound \eqref{T1mole2} by the triangle inequality, so we are only concerned with $|h|<1$.

As in the case of size and derivatives above, the uniform bound
\begin{equation*}
|\partial^\alpha Ta(x)-\partial^\alpha Ta(x+h)|\lesssim|h|^{N^{**}}
\end{equation*}
can be simply borrowed from \cite[Proof of Proposition 6.19]{bhyyp3}, as no vanishing moments are used in getting this estimate. (Note, however, that this estimate is the most technical part of \cite[Proof of Proposition 6.19]{bhyyp3}, starting with \cite[(6.26)]{bhyyp3} and continuing until the end of the said proof.)

Hence it remains to establish the decay estimate \eqref{mole diff}
for any $|h|<1\ll |x|$. To this end, for any $i\in\{1,\ldots,n\}$, let
$e_i:=(0,\ldots,0,1,0,\ldots,0)\in\mathbb{Z}_+^n$
where $1$ is at the $i$-th position.
With $|\alpha|=\lfloor\!\lfloor N\rfloor\!\rfloor$, we apply the already checked bound \eqref{T1mole0} with each $\alpha+e_i$ in place of $\alpha$, which is justified because $|\alpha+e_i|=|\alpha|+1=\lfloor\!\lfloor N\rfloor\!\rfloor+1\leq\lfloor\!\lfloor E\rfloor\!\rfloor+1$ is still in the range where \eqref{T1mole0} is valid. Recalling also that $n+F\geq M$, we find that,
for any $x,h\in\mathbb{R}^n$ with $|x|>4\sqrt{n}+2$ and $|h|<1$,
\begin{align*}
|\partial^\alpha Ta(x)-\partial^\alpha Ta(x+h)|
&=\left|\int_0^1h\cdot\nabla\partial^\alpha Ta(x+th)\,dt\right|
\lesssim|h|\int_0^1\max_{i\in\{1,\ldots,n\}}|\partial^{\alpha+e_i}Ta(x+th)|\,dt\\
&\lesssim|h|\int_0^1|x+th|^{-M}\,dt
\lesssim|h|^{N^{**}}|x|^{-M}.
\end{align*}
This completes the proof of the molecular difference estimate \eqref{mole diff} in all cases. Together with the size and the derivative bounds \eqref{T1mole2} already established, this shows that $Ta$ is a $(K,-1,M,N)$-molecule, as claimed,
which completes the proof of Proposition \ref{T1EFGH}.
\end{proof}

\begin{corollary}\label{T1Amol}
Let $\sigma\in\{0,1\}$ and $E,F,G,H\in\mathbb{R}$.
Let $W\in A_{p,\infty}$
and $\widetilde{J},\widetilde{s}$ be the same as in \eqref{tauJ2}.
Let $T\in\operatorname{InCZO}^\sigma(E,F,G,H)$.
Then $T$ maps sufficiently regular atoms to
$A^{s,\tau}_{p,q}(W)$-synthesis molecules provided that \eqref{T1Amol1} is satisfied.
\end{corollary}

\begin{proof}
Since $\operatorname{InCZO}^\sigma(E,F,G,H)\subset\operatorname{CZO}^\sigma(E,F,G,H)$, the case of cancellative atoms is contained in \cite[Corollary 6.21]{bhyyp3}, based on \cite[Corollary 6.20]{bhyyp3}. The case of non-cancellative atoms follows in the same  way, by using Proposition \ref{T1EFGH} in place of \cite[Proposition 6.19]{bhyyp3} in the proof of  \cite[Corollary 6.20]{bhyyp3}; we omit the details.
This finishes the proof of Corollary \ref{T1Amol}.
\end{proof}

Next, we show Theorem \ref{T1 BF}.

\begin{proof}[Proof of Theorem \ref{T1 BF}]
The present assumptions are the same as those of Corollary \ref{T1Amol},
which guarantee that $T$ maps sufficiently regular atoms to
$A^{s,\tau}_{p,q}(W)$-synthesis molecules.
Then Proposition \ref{ext} provides the existence of
an operator $\widetilde T\in\mathcal L(A^{s,\tau}_{p,q}(W))$ as claimed.
This finishes the proof of Theorem \ref{T1 BF}.
\end{proof}

\bigskip

\noindent Fan Bu

\medskip

\noindent Laboratory of Mathematics and Complex Systems (Ministry of Education of China),
School of Mathematical Sciences, Beijing Normal University, Beijing 100875, The People's Republic of China

\smallskip

\noindent{\it E-mail: }\texttt{fanbu@mail.bnu.edu.cn}

\bigskip

\noindent Tuomas Hyt\"onen

\medskip

\noindent Department of Mathematics and Systems Analysis,
Aalto University, P.O. Box 11100, FI-00076 Aalto, Finland

\smallskip

\noindent{\it E-mail: }\texttt{tuomas.p.hytonen@aalto.fi}

\bigskip

\noindent Dachun Yang (Corresponding author) and Wen Yuan

\medskip

\noindent Laboratory of Mathematics and Complex Systems (Ministry of Education of China),
School of Mathematical Sciences, Beijing Normal University, Beijing 100875, The People's Republic of China

\smallskip

\noindent{\it E-mails: }\texttt{dcyang@bnu.edu.cn} (D. Yang)

\noindent\phantom{{\it E-mails: }}\texttt{wenyuan@bnu.edu.cn} (W. Yuan)

\end{document}